\documentclass[download]{msjmemoirs}

\usepackage{amsmath,amsthm,amssymb}
\usepackage{bbm}
\usepackage{graphicx}
\usepackage{epic,eepic}
\usepackage[all]{xy}
\usepackage{tikz}
\usepackage{mathrsfs} 
\usepackage{aligned-overset}
\usepackage{anyfontsize}
\usepackage[normalem]{ulem} 

\usepackage{makeidx}
\makeindex

 \usepackage{float}

\allowdisplaybreaks
\numberwithin{equation}{chapter}
\numberwithin{figure}{chapter}

\newtheorem{thm}{\bf Theorem}[chapter]
\newtheorem{prop}[thm]{\bf Proposition}
\newtheorem{conj}[thm]{\bf Conjecture}
\newtheorem{cor}[thm]{\bf Corollary}
\newtheorem{lem}[thm]{\bf Lemma}

\theoremstyle{definition}

\newtheorem{rem}[thm]{\bf Remark}
\newtheorem{ex}[thm]{\bf Example}

\newtheorem{defn}[thm]{\bf Definition}
\newtheorem{prob}[thm]{\bf Problem}

\newtheorem{const}[thm]{\bf Construction}

\def\notes{\medskip \par\noindent{\centerline{\bf Notes}}\par\nopagebreak\medskip\par\nopagebreak}

\def\bbC{\mathbb{C}}
\def\bbP{\mathbb{P}}
\def\bbQ{\mathbb{Q}}
\def\bbR{\mathbb{R}}
\def\bbT{\mathbb{T}}

\def\bbZ{\mathbb{Z}}

\def\rmsf{\mathrm{sf}}

\def\bfC{\mathbf{C}}
\def\bfF{\mathbf{F}}
\def\bfG{\mathbf{G}}

\def\bfX{\mathbf{X}}
\def\bfY{\mathbf{Y}}
\def\bfSigma{\mathbf{\Sigma}}
\def\bfUpsilon{\mathbf{\Upsilon}}
\def\bfa{\mathbf{a}}
\def\bfb{\mathbf{b}}
\def\bfc{\mathbf{c}}

\def\bfe{\mathbf{e}}
\def\bfg{\mathbf{g}}

\def\bfp{\mathbf{p}}

\def\bfn{\mathbf{n}}
\def\bfm{\mathbf{m}}

\def\bfu{\mathbf{u}}
\def\bfv{\mathbf{v}}
\def\bfw{\mathbf{w}}
\def\bfx{\mathbf{x}}
\def\bfy{\mathbf{y}}
\def\bfz{\mathbf{z}}
\def\bfzero{\mathbf{0}}

\def\rmid{\mathrm{id}}
\def\frakd{\mathfrak{d}}
\def\frakD{\mathfrak{D}}
\def\frakg{\mathfrak{g}}
\def\frakj{\mathfrak{j}}
\def\frakp{\mathfrak{p}}
\def\frakq{\mathfrak{q}}

\def\frakS{\mathfrak{S}}
\def\fraks{\mathfrak{s}}
\def\calA{\mathcal{A}}

\def\calC{\mathcal{C}}

\def\calF{\mathcal{F}}
\def\calG{\mathcal{G}}

\def\calL{\mathcal{L}}

\def\calO{\mathcal{O}}
\def\calP{\mathcal{P}}

\def\calY{\mathcal{Y}}

\def\scrL{\mathscr{L}}
\def\scrH{\mathscr{H}}

\def\haty{\hat{y}}
\def\yhat{\hat{y}}
\def\ve{\varepsilon}

\def\sgn{\mathrm{sgn}}

\def\Li2{\mathrm{Li_2}}
\def\tL{\tilde{L}}
\def\op{\mathrm{op}}
\def\id{\mathrm{id}}
\def\diag{\mathrm{diag}}
\def\ve{\varepsilon}
\def\trop{\mathrm{trop}}
\def\rmpr{\mathrm{pr}}
\def\rmAd{\mathrm{Ad}}

\def\l{\ell}
\def\d{\delta}
\newcommand{\overunder}[2]{
\!\!\begin{array}{c}
\scriptstyle{#1}\\[-9pt]
-\!\!\!-\\[-.1in]
\scriptstyle{#2}
\end{array}
}
\newcommand{\rp}{\color{black}}

\VolumeNumber{999}
\VolumeTitle{Cluster Algebras \\and\\ Dilogarithm Identities}
\VolumeNote{}
\VolumePublishedYear{202x}

\ListOfKeywords{%
  cluster algebra, dilogarithm identity.}
\ListOfMSC{%
  Primary 13F60}

\begin{volumeauthor}
  \name{Tomoki Nakanishi}
  \address{Nagoya University\\
   Graduate School of Mathematics\\Chikusa-ku, Nagoya, 464-8602, Japan.}
  \email{nakanisi@math.nagoya-u.ac.jp}
\end{volumeauthor}
  
\begin{document}

 \frontmatter
  \InsertVolumeTitle

\fontsize{11pt}{12.5pt}\selectfont

\begin{preface}
This is a reasonably self-contained exposition of 
the fascinating interplay between
cluster algebras and the dilogarithm in the past two decades.

The theme of the monograph is the dilogarithm identity associated with each period in a cluster pattern of cluster algebra theory.
They provide a vast generalization of  Euler's and Abel's identities. We present several proofs, variations, and generalizations of them
with various methods and techniques.
The quantum dilogarithm identities are also treated from the unified point of view with the classical ones.

Let us give a brief overview of the subject.
The dilogarithm 
 is  defined by a power series or an integral expression as
\begin{align*}
\mathrm{Li}_2(x)=\sum_{n=1}^{\infty} {\frac{x^n}{n^2}}
=-\int_0^x \frac{\log (1-y)}{y}\, dy,
\end{align*}
and it was studied first by Euler \cite{Euler68}.
The function has had a rich history since then and is related to various fields of mathematics and physics \cite{Lewin81, Kirillov95, Zagier07}.
One of the main features of the dilogarithm is that it satisfies a wide variety of functional equations.
 We call them \emph{dilogarithm identities} (DIs).

In the 1980s,  it was discovered 
by Bazhanov, Kirillov, and Reshetikhin  \cite{Kirillov86,
Bazhanov89}
that the dilogarithm plays a key role
in studying some mathematical physics models known as the \emph{integrable lattice models}.
Meanwhile, 
Zamolodchikov  \cite{Zamolodchikov91b} introduced
systems of functional/algebraic equations called the \emph{$Y$-systems} (of type ADE)
in the study of the \emph{factorizable $S$-matrix models}.
Also,
the periodicity of the solutions of the $Y$-systmes was conjectured \cite{Zamolodchikov91}.
Then, the $Y$-systems were generalized by Kuniba-Nakanishi \cite{Kuniba92} and Ravanini-Tateo-Valleriani \cite{Ravanini93}
based on root systems.
Following these developments, a family of DIs associated with these $Y$-systems
was conjectured by Gliozzi-Tateo \cite{Gliozzi95}.
 Remarkably, these conjectural DIs provide a vast generalization of the celebrated Euler's and Abel's identities
based on  root systems,
where the latter identities correspond to the simplest case of types $A_1$ and $A_2$,
respectively.
This conjecture got the immediate attention by several experts at that time.
However, 
the only special case of type $A$ were proved by Gliozzi-Tateo \cite{Gliozzi96} and Frenkel-Szenes,
\cite{Frenkel95},
 and the rest were left as mysteries.
This was a situation in B.C. (= Before Cluster algebra).

The scene changed 
after the introduction of \emph{cluster algebras} (in ``A.D.'') by  Fomin-Zelevinsky \cite{Fomin02} around 2000.
Cluster algebras are a certain class of commutative algebras.
The original motivation was
to study the structure of the coordinate rings of certain algebraic varieties
  in Lie theory (e.g., the Grassmannian, the double Bruhat cells).
 Remarkably, the above $Y$-systems were recognized as
 a part of the cluster algebra structure.
 Based on this observation, the periodicity  of the $Y$-systems of type $ADE$
  was proved by Fomin-Zelevinsky \cite{Fomin03b}.
  Then, using this result, the DIs for the $Y$-systems of type $ADE$ were proved
  by Chapoton \cite{Chapoton05}.
    This was the first  important step to connect the DIs
   and cluster algebras.
  However,  to prove the DIs for the $Y$-systems  in full generality, 
 it was necessary to wait for further technical development of cluster algebra theory,
 especially the separation formula  \cite{Fomin07} and the categorification results \cite{Amiot09,Plamondon10b}.
  Then, based on these new results, the periodicity and  DIs for $Y$-systems
  were proved in full generality by
Keller  \cite{Keller10} and  Nakanishi \cite{Nakanishi09}, respectively.
Similar DIs for variants of $Y$-systems  that had been conjectured in B.C.
 were also proved by the  same method \cite{Inoue10a,Inoue10b,Nakanishi10b}.
After these developments, it was recognized by \cite{Nakanishi10c}
  that
   \emph{a DI is associated with each period in a cluster pattern},
  and all above DIs are special instances of such DIs.
  These DIs play the role of the ``leitmotif'' throughout the entire text.
  
  Next, we describe the contents of the text.
  
  Part I is the foundation of the whole text and presents the above results in detail.
  In Chapter 1, we introduce the dilogarithm and its variations.
  Then, we present Euler's and  Abel's identities, which are the two most basic DIs of all.
  Following \cite{Faddeev94}, we also refer to Abel's identity as the \emph{pentagon identity}.
  In Chapter 2, a quick course on cluster algebras
is given.
   In Chapter 3,   
   by reversing the history in an updated point of view,
   we first formulate and prove
   the \emph{DI associated with a period of a $Y$-pattern} (or a cluster pattern) (Theorem \ref{thm:DI1}).
This is the leitmotif throughout the entire text, as already mentioned.
Then, in Chapter 4, we formulate the $Y$-systems  in terms of $Y$-patterns
in cluster algebra theory.
We also observe the similarity between the $Y$-systems and the actions of the
 Coxeter elements on the root systems.
Then, combining them with the results in Chapters 2 and 3,
the DIs for $Y$-systems are proved.
In the proof, we use the \emph{tropicalization} as a key technique.

In the rest, consisting of three parts, we present further developments of the interrelation between
cluster algebras and DIs from various points of view,
which are closely connected with each other.

In Part II, the classical mechanical method is presented.
In Chapter 5 
each mutation of $y$-variables in a $Y$-pattern is  decomposed into two parts,
the tropical and the nontropical parts
(the Fock-Goncharov decomposition).
After introducing the mutation-compatible Poisson bracket,
the Euler dilogarithm appears as
the Hamiltonian of a mutation.
This gives a more intrinsic and direct connection between cluster algebras and the dilogarithm.
In Chapter 6
 the canonical coordinates (in the sense of classical mechanics)
are introduced.
Then, by the Legendre transformation, the modified Rogers dilogarithm
appears as the Lagrangian.
Based on this picture, we give an alternative proof of Theorem \ref{thm:DI1}.
The DIs are now interpreted as the invariance of the action integral under the periodicity.
Moreover, this statement is regarded as a version of Noether's theorem in classical mechanics.

In Part III, the cluster scattering diagram method is presented.
This is based on the recent development of cluster algebra theory
by the scattering diagram method by Gross-Hacking-Keel-Kontsevich
\cite{Gross14}.
In Chapter 7 an algebraic formulation of the dilogarithm is given.
For a given skew-symmetric matrix, a certain group $G$ is defined.
Then, a family of elements in $G$ called the \emph{dilogarithm elements}
are introduced. They satisfy the \emph{pentagon relation} in the group $G$, which is regarded
as an algebraic version of the pentagon identity of the dilogarithm.
In Chapter 8  a quick course on cluster scattering diagrams is given.
In Chapter 9
we extend Theorem \ref{thm:DI1}
to the \emph{DI associated with a loop in a cluster scattering diagram (CSD)}.
We give two proofs. The first one is the extension of the method in Part II.
The second one is solely based on the pentagon relation.
This also gives yet another proof of Theorem \ref{thm:DI1}; moreover,
it shows the \emph{infinite reducibility} of these DIs.

In Part IV, the quantum dilogarithm identities (QDIs) are presented.
In Chapter 10 we introduce the \emph{quantum dilogarithm} 
and show the pentagon relation
following Faddeev, Kashaev, and Volkov \cite{Faddeev93,Faddeev94}.
In Chapter 11 the quantum mutation of $y$-variables is given
following Fock-Goncharov \cite{Fock03}.
The quantum dilogarithm naturally appears as the Hamiltonian
of the quantum mutation.
Thus, it is indeed the quantization of the classical mechanical 
picture in Part II.
At the same time, the algebraic picture in Part III also has the
quantization. Namely, we have the \emph{quantum dilogarithm elements}
and the pentagon relation among them.
In Chapter 12 we present the QDI associated with a period of a $Y$-pattern,
which is completely parallel to Theorem \ref{thm:DI1}.
Also, we have the quantum version of the 
DI associated with a loop in a CSD in Part III.
However, not much is known about the structure of quantum cluster scattering diagrams (QCSDs).
We present some examples of QDIs and open problems.

In the beginning, we mentioned that the exposition is
``reasonably self-contained".
What we mean is as follows.
\begin{itemize}
\item
The text is aimed at researchers and graduate students interested in the interrelation between cluster algebras and the dilogarithm.
We expect the reader is already familiar with one of the two subjects to some extent, but not necessarily both.
Naturally, we assume the reader is a newcomer to both.

\item
To make  the text in reasonable length,
we skip the proofs of the standard facts on cluster algebras
and cluster scattering diagrams, which are collected in Chapters \ref{ch:quick1}
and \ref{ch:CSD1}, respectively.
The proofs and extensive explanations are found
in the monograph by the author \cite{Nakanishi22a}, which we regard as the companion
of this text.
\item Other than the above, we give the proofs for \emph{most} facts
unless they are too complicated or too far from the context.
\end{itemize}

In \cite{Zagier07}  Zagier impressively described the nature of the dilogarithm as follows:
``\emph{Almost all of its appearances in mathematics, and almost all the formulas relating to it, have something of the fantastical in them as if this function alone among all others possessed a sense of humor.}"
I completely agree with it and wish to add  words:
``\emph{The function is so shy that it hides its algebraic nature under the mask of the integral.}''
I hope that the reader may share this feeling through this exposition.

I thank 
Michael Gekhtman,
Rei Inoue,
Osamu Iyama,
Rinat Kashaev,
Bernhard Keller, 
Atuso Kuniba,
Dylan Rupel,
Salvatore Stella,
Junji Suzuki,
Roberto Tateo,
and
Andrei Zelevinsky
for fruitful collaboration in various works related to this fascinating subject.
I thank Ivan Ip for useful communications.
This work is supported in part by JSPS grant No. JP22H01114.

\bigskip

    \begin{flushright}
      Tomoki Nakanishi\\
      \textit{March, 2025, Nagoya}   
    \end{flushright}

\end{preface}


\tableofcontents
\mainmatter

\part{Dilogarithm, cluster algebras, and $Y$-systems}
\label{part:dilog1}
\chapter{Prologue: Mysterious dilogarithm identities}
\label{ch:prologue}

Cluster algebras were introduced around 2000 by Fomin and Zelevinsky.
Our story began in B.C. (= Before Cluster algebra).
After the series of works by
 Bazhanov, Kirillov, and Reshetikhin in the study of some mathematical physics problems in the 1980s,
mysterious dilogarithm identities 
(the DIs for $Y$-systems)
were conjectured
by Gliozzi and Tateo in the 1990s.
This introductory chapter presents this historical background.

\section{Euler Dilogarithm}

To begin with,
we present some basic facts about the dilogarithm,
its variants, and dilogarithm identities.
The subject has a long and rich history started by Euler \cite{Euler68}, and it is related to various fields of mathematics and physics.
We will concentrate on the material we are going to treat here.
We recommend the references by Lewin \cite{Lewin81}, Kirillov \cite{Kirillov95}, and Zagier \cite{Zagier07} for further information.

First of all, the \emph{dilogarithm}\index{dilogarithm} is defined as a power series
\begin{align}
\label{eq:Li1}
\mathrm{Li}_2(x)=\sum_{n=1}^{\infty} {\frac{x^n}{n^2}},
\end{align}
where the notation $\mathrm{Li}$ stands for the \emph{logarithmic integral} by the reason we see  very soon.
Here, we also call it the \emph{Euler dilogarithm}\index{Euler dilogarithm}\index{dilogarithm!Euler ---} to distinguish it from its variants.
The power series \eqref{eq:Li1} has the convergent radius 1, so it defines a complex analytic function
on the disk $|x|<1$. 
More generally, for any positive integer $k$,
the {\em polylogarithm of order $k$}\index{polylogarithm} is defined by
\begin{align}
\mathrm{Li}_k(x)=\sum_{n=1}^{\infty} {\frac{x^n}{n^k}}.
\end{align}

By \eqref{eq:Li1}, we have
\begin{align}
\label{eq:dLx1}
x \frac{d}{dx} \mathrm{Li}_2(x)=\sum_{n=1}^{\infty} {\frac{x^n}{n}}
=\mathrm{Li}_1(x)
=-\log(1-x)
\quad
(|x|<1).
\end{align}
Thus, we obtain an integral expression
\begin{align}
\label{eq:Li2}
\mathrm{Li}_2(x)=-\int_0^x \frac{\log (1-y)}{y}\, dy,
\end{align}
by which $\mathrm{Li}_2(x)$ is analytically continued to the whole complex plain $\bbC$.
We have to be careful about the branches, though.
There is obvious multivaluedness around $x=1$ by the branches of $\log(1-x)$ in the integrand.
On the other hand, the function is smooth at $x=0$ as a convergent power series.
Nevertheless, if the integration contour starts from $x=0$, goes around $x=1$, say, {\rp counterclockwise},
and comes back to $x=0$, then the integration \eqref{eq:Li2} picks up an
additional term
$2\pi \sqrt{-1}/y$ in the integral. Since it has a simple pole at $y=0$, the function $\mathrm{Li}_2(x)$ now has multivaluedness also around $x=0$.
For our purpose, however, it is enough to concentrate on the interval $(-\infty, 1]$ in the real line $\bbR$,
where the expression \eqref{eq:Li2} is regarded as  an ordinary real integral.
Then, there is no such issue of multivaluedness.
From now on, we consider $\mathrm{Li}_2(x)$ only on  the interval $(-\infty, 1]$.

The following two special values are important:
\begin{align}
\label{eq:sp1}
\mathrm{Li}_2(0)=0,
\quad
\mathrm{Li}_2(1)=\zeta(2)=\frac{\pi^2}{6}=1.6449\dots,
\end{align}
where $\zeta(s)$ is the Riemann zeta function.
The first one is obvious, while the second one is the famous answer to
the Basel problem by Euler.

\begin{figure}
\centering
\includegraphics[width=75mm]{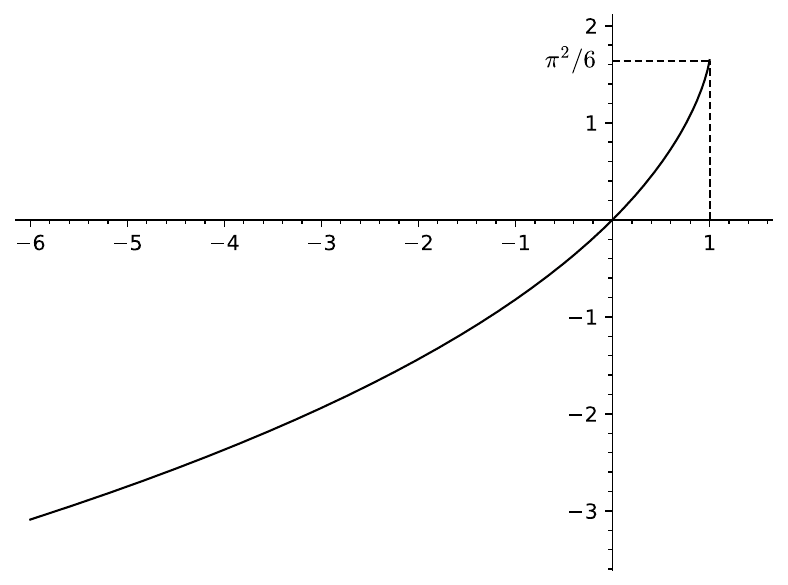}
\vskip-5pt
\caption{The graph of the Euler dilogarithm $\mathrm{Li}_2(x)$.}
\label{fig:Li1}
\end{figure}

The function $\mathrm{Li}_2(x)$ is monotonically increasing on $(-\infty, 1]$,
because its derivative $-\log(1-x)/x$ is positive on  $(-\infty, 1)$.
The graph is presented in Figure \ref{fig:Li1},
where we see nothing spectacular.
However, the function is known to have  one very specific feature
that is not expected from its definition.
Namely, \emph{it satisfies various functional identities.}
Here, we vaguely call them  {\em dilogarithm identities}\index{dilogarithm identity}\index{DI|see{dilogarithm identity}} (DIs).

Let us present the two most basic examples of DIs.
The first one is  due to Euler:  (\emph{Euler's identity})\index{Euler's identity}
\begin{align}
\label{eq:euler1}
\Li2(x)+ \mathrm{Li}_2(1-x) = \frac{\pi^2}{6}-\log x \log(1-x)
\quad
(0\leq x \leq 1).
\end{align}
Here, we understand that,
at $x=0$ and $1$, 
\begin{align}
\log 0 \log 1 := \lim_{x\rightarrow 0^+}\log x \log(1-x)=0.
\end{align}
The second one is regarded as a two-variable extension of \eqref{eq:euler1}:
\begin{align}
\label{eq:pent1}
\begin{split}
&
\Li2(x)+ \mathrm{Li}_2(y)
+ \Li2\biggl(\frac{1-x}{1-xy}\biggr)
+ \Li2(1-xy)
+ \Li2\biggl(\frac{1-y}{1-xy}\biggr)
\\
 =\ &\frac{\pi^2}{2}-\log x  \log(1-x)
 -\log y  \log(1-y)
 \\
 &\qquad
-\log\biggl(\frac{1-x}{1-xy}\biggr)
\log\biggl(\frac{1-y}{1-xy}\biggr)
\quad
(0\leq x, y < 1).
\end{split}
\end{align}
This identity was discovered independently by Spence (1809), Abel (1826, published in 1881), Hill (1830), Kummer (1840), and Schaeffer (1846)
in various equivalent forms \cite{Lewin81,Pirio21}.
It is usually called  \emph{Abel's identity}\index{Abel's identity}, or the \emph{five-term relation}\index{five-term relation}.
Recently, it is also called the \emph{pentagon identity}\index{pentagon!identity} (e.g., \cite{Faddeev94}),
due to a hidden relation to the geometry of a pentagon.
To see some clue, let us set the variables appearing in the left-hand side (LHS) of \eqref{eq:pent1} as
\begin{align}
\label{eq:X1}
X_1=x,
\quad
X_2=y,
\quad
X_3=\frac{1-x}{1-xy},
\quad
X_4={1-xy},
\quad
X_5=\frac{1-y}{1-xy},
\quad
\end{align}
and also $X_{i+5}=X_i$.
 Then, for any $i\in \bbZ$, the following recurrence  relation holds:
\begin{align}
\label{eq:X2}
X_{i+2}=\frac{1-X_{i}}{1-X_i X_{i+1}}.
\end{align}
Moreover, the identity \eqref{eq:pent1} is written in the cyclic form
\begin{align}
\begin{split}
\sum_{i=1}^5 \Li2(X_i)&=
\frac{\pi^2}{2}
- \log X_1 \log X_3
- \log X_2 \log X_4
- \log X_3 \log X_5
\\
&\qquad
- \log X_4 \log X_1
- \log X_5 \log X_2.
\end{split}
\end{align}
Thus, there is a cyclic symmetry of order 5 behind the identity \eqref{eq:pent1}.

We will prove the identities \eqref{eq:euler1} and \eqref{eq:pent1} very soon in some alternative forms.
Before that, let us play a little with them.
First, by using \eqref{eq:euler1}, one can rewrite \eqref{eq:pent1} as follows:
\begin{align}
\label{eq:pent2}
\begin{split}
&\quad\
\Li2(x)+ \mathrm{Li}_2(y)
- \Li2\biggl(\frac{x(1-y)}{1-xy}\biggr)
-  \Li2(xy)
-  \Li2\biggl(\frac{y(1-x)}{1-xy}\biggr)
\\
&= \log\biggl(\frac{1-x}{1-xy}\biggr)
\log\biggl(\frac{1-y}{1-xy}\biggr),
\end{split}
\end{align}
where the constant term $\pi^2/2$ disappeared.
Next, we set
\begin{align}
X=\frac{x(1-y)}{1-xy},
\quad
Y=\frac{y(1-x)}{1-xy}.
\end{align}
Then, we have
\begin{align}
\frac{X}{1-Y}=x,
\quad
\frac{Y}{1-X}=y,
\quad
\frac{XY}{(1-X)(1-Y)}=xy.
\end{align}
Putting them into \eqref{eq:pent2}, we obtain another alternative
expression of the pentagon identity.
\begin{align}
\label{eq:pent3}
\begin{split}
&\quad\
\Li2\biggl(\frac{X}{1-Y}\biggr)
+ \Li2\biggl(\frac{Y}{1-X}\biggr)
- \Li2(X)
-  \Li2\biggl(\frac{XY}{(1-X)(1-Y)}\biggr)
-  \Li2(Y)
\\
&= \log(1-X)\log(1-Y).
\end{split}
\end{align}
This is the expression obtained by Abel \cite{Lewin81}.

\section{Rogers dilogarithm}
\label{sec:Rogers1}

The following variant or the symmetrization of the Euler dilogarithm was
introduced by Rogers \cite{Rogers07}:
\begin{align}
\label{eq:L1}
L(x):=&\ \Li2(x) + \frac{1}{2}\log x  \log(1-x)
\\
\label{eq:L2}
=& \ -\frac{1}{2}\int_{0}^x
\biggl\{
\frac{\log(1-y)}{y}
+
\frac{\log y}{1-y}
\biggr\}
\, dy.
\end{align}
It is called the \emph{Rogers dilogarithm}\index{Rogers dilogarithm}\index{dilogarithm!Rogers ---} (also called  \emph{Rogers' $L$-function}\index{Rogers' $L$-function}).
We concentrate on the interval $[0,1]$, so that   there is no
issue of multivaluedness again.
By \eqref{eq:sp1}, we have the following important special values:
\begin{align}
\label{eq:sp2}
L(0)=0,
\quad
L(1)=\frac{\pi^2}{6}.
\end{align}
The function $L(x)$ is monotonically increasing on $[0,1]$.
The graph of $L(x)$ is presented in Figure \ref{fig:Li2}.
We observe that it is close to the linear function $y=(\pi^2/6)x$.

\begin{figure}
\centering
\includegraphics[width=60mm]{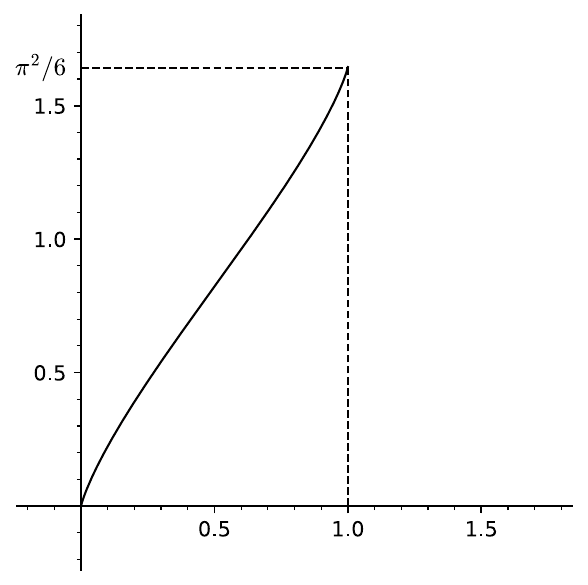}
\vskip-5pt
\caption{The graph of the Rogers dilogarithm $L(x)$.}
\label{fig:Li2}
\end{figure}

A special feature of the function $L(x)$ is
that \emph{it cleans up some DIs}, and it is the main reason why the function was introduced by Rogers.
For examples, the identities \eqref{eq:euler1} and  \eqref{eq:pent1}
are rewritten in the following form:
\begin{gather}
\label{eq:euler2}
L(x)+ L(1-x) = \frac{\pi^2}{6}
\quad
(0\leq x \leq 1),
\\
\label{eq:pent4}
\begin{split}
L(x)+ L(y)
+L\biggl(\frac{1-x}{1-xy}\biggr)
+ L(1-xy)
+ L\biggl(\frac{1-y}{1-xy}\biggr)
 = \frac{\pi^2}{2} 
 \\
\quad
(0\leq x, y < 1).
 \end{split}
\end{gather}
Observe that the log terms in the RHSs disappear in both identities.
Using \eqref{eq:euler2}, the pentagon identity \eqref{eq:pent4} is also written as
\begin{gather}
\label{eq:pent5}
\begin{split}
L(x)+ L(y)
-L\biggl(\frac{x(1-y)}{1-xy}\biggr)
- L(xy)
- L\biggl(\frac{y(1-x)}{1-xy}\biggr)=0
 \\
\quad
(0\leq x, y < 1),
 \end{split}
\end{gather}
where the constant term also disappeared.
This is the counterpart of the form \eqref{eq:pent2}.

Now we prove  Euler's identity in the form \eqref{eq:euler2}.
Of course, it is quickly proved by the integration by parts.
However, we prefer the following method. (The reason will be clear very soon.)
The proof consists of two steps.

\begin{itemize}
\item
{\bf Step 1 (Constancy)}. \emph{Verify the {constancy} of the LHS}.
Namely,
\begin{align}
\frac{d}{dx}(\mathrm{LHS})=0.
\end{align}
This is true due to
  the symmetry of the integrand of  \eqref{eq:L2} between $y$ and  $1-y$.
\item
{\bf Step 2 (Constant term)}.
\emph{Determine the constant value of the LHS}. To do it, we simply set $x=0$
in \eqref{eq:euler2}.
Then, by \eqref{eq:sp2}, we obtain $\mathrm{LHS}=L(0)+L(1)=\pi^2/6$.
This is why we stressed the importance of \eqref{eq:sp2}.
\end{itemize}

We apply the same method to prove the pentagon identity in the form \eqref{eq:pent4}.
First, we prove the constancy
(good exercise for  undergraduate calculus!)
\begin{align}
\frac{\partial}{\partial x}(\mathrm{LHS})=\frac{\partial}{\partial y}(\mathrm{LHS})=0.
\end{align}
Next, we set $x=y=0$.
Then, we have $\mathrm{LHS}=2L(0)+3L(1)=3\pi^2/6=\pi^2/2$.
This also clarifies the meaning of the constant $\pi^2/2$.

The point of the method is that we do not carry out  
any integration for the integral \eqref{eq:L2}; rather,
we differentiate it
and do some algebraic manipulation for its
integrand.
In other words, the structure behind these DIs are \emph{algebraic}
rather than \emph{analytic}.
So, we call the above method the \emph{algebraic method}\index{algebraic method}.
We are going to apply it to prove various much more complicated DIs.

The readers may already start  enjoying
the ``sense of humor" of the dilogarithm  
in the words of Zagier, quoted in the preface.

\section{Modified Rogers dilogarithm}

Here, we introduce yet another variant of the dilogarithm,
which especially fits our problem.
Let
\begin{align}
\label{eq:L3}
\tL(x):=&\  L\biggl(\frac{x}{1+x}\biggr) \quad (0\leq x)
\\
\label{eq:L4}
=& \ \frac{1}{2}\int_{0}^x
\biggl\{
\frac{\log(1+y)}{y}
-
\frac{\log y}{1+y}
\biggr\}
\, dy
\\
\label{eq:L5}
=&\ {}- \Li2(-x)  - \frac{1}{2}\log x  \log(1+x).
\end{align}
The importance of the function $\tL(x)$ has been gradually recognized
since the 1990s
 through the study of DIs, as we explain very soon.
 A little surprisingly,
the expressions \eqref{eq:L4} and \eqref{eq:L5} seem
less well-known in the literature
until recently  (e.g., \cite{Nakanishi10c,Nakanishi16});
they are
quickly and
transparently
  proved by the algebraic method in the previous section.
The function $\tL(x)$ does not have a name in the literature.
Since this is inconvenient, here we call it the \emph{modified Rogers dilogarithm}\index{modified Rogers dilogarithm}\index{dilogarithm!modified Rogers ---}.

Again, by \eqref{eq:sp2},  we have the following important special values:
\begin{align}
\label{eq:sp3}
\tL(0)=0,
\quad
\tL(\infty):=\lim_{x\rightarrow \infty} \tL(x)= \frac{\pi^2}{6}.
\end{align}
The function is  monotonically increasing on $[0,\infty)$,
and its graph is presented in Figure \ref{fig:Li3}.

\begin{figure}
\centering
\includegraphics[width=90mm]{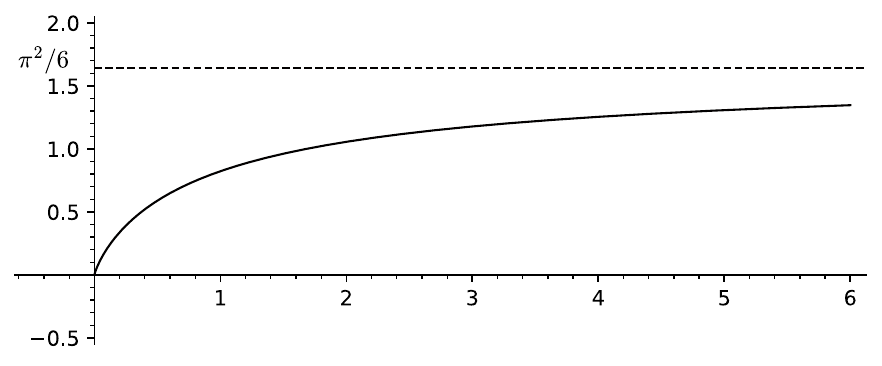}
\vskip-5pt
\caption{The graph of the modified Rogers dilogarithm $\tL(x)$.}
\label{fig:Li3}
\end{figure}

Using the identity
\begin{align}
\frac{x}{1+x} + \frac{x^{-1}}{1+x^{-1}} =
\frac{x}{1+x} + \frac{1}{1+x} =
1,
\end{align}
one can  rewrite Euler's identity \eqref{eq:euler2} 
as
\begin{gather}
\label{eq:euler3}
\tL(y)+ \tL(y^{-1}) = \frac{\pi^2}{6}
\quad
(0\leq y),
\end{gather}
where we use the variable $y$  instead of $x$ to fit the context appearing later.

To rewrite the pentagon identity \eqref{eq:pent4} with $\tL(x)$,
we introduce independent variables $y_1$ and $y_2$, and we set
\begin{align}
\label{eq:x1}
\begin{split}
&Y_1=y_1,
\quad
Y_2=y_2(1+y_1),
\quad
Y_3=y_1^{-1}(1+y_2+y_1y_2),
\\
&
Y_4=y_1^{-1}y_2^{-1}(1+y_2),
\quad
Y_5=y_2^{-1},
\end{split}
\end{align}
and $Y_{i+5}=Y_i$ ($i\in \bbZ$).
They satisfy the recurrence relation
\begin{align}
\label{eq:x2}
Y_{i+2}=Y_i^{-1}(1+Y_{i+1}).
\end{align}
Moreover, if we set
\begin{align}
X_i=\frac{Y_i}{1+Y_i} \quad (i\in \bbZ),
\end{align}
we have
\begin{align}
\begin{split}
\frac{1-X_i}{1-X_i X_{i+1}}
&=
\frac{1+Y_{i+1}}{(1+Y_i)(1+Y_{i+1})-Y_i Y_{i+1}}
\\
&=
\frac{Y_i^{-1}(1+Y_{i+1})}{Y_i^{-1}(1+Y_i+ Y_{i+1})}
=
\frac{Y_{i+2}}{1+Y_{i+2}}=X_{i+2}.
\end{split}
\end{align}
This is the recurrence relation in  \eqref{eq:X2}.
Thus, with variables $Y_1$, \dots, $Y_5$,
and $y_1$, \dots, $y_5$, the pentagon identity \eqref{eq:pent4}
is rewritten as
\begin{gather}
\label{eq:pent6}
\begin{split}
\sum_{i=1}^5 \tL(Y_i)&=
\tL(y_1)+ \tL(y_2(1+y_1))
+\tL(y_1^{-1}(1+y_2+y_1y_2))
\\
&\qquad + \tL(y_1^{-1}y_2^{-1}(1+y_2))
+ \tL(y_2^{-1}) = \frac{\pi^2}{2} 
\quad (0<  y_1,\,y_2).
 \end{split}
\end{gather}
Note that in the limit $y_1,\, y_2 \rightarrow 0$,
we have
$Y_1,\, Y_2 \rightarrow 0$ and $Y_3,\, Y_4,\, Y_5 \rightarrow \infty$.
This is consistent with the constant $\pi^2/2=3 \pi^2/6$
in the RHS.
Thanks to Euler's identity \eqref{eq:euler3},
the pentagon identity \eqref{eq:pent6} is also written without a constant term as
\begin{gather}
\label{eq:pent7}
\begin{split}
&\ \tL(Y_1)+\tL(Y_2)
-\tL(Y_3)-\tL(Y_4)
-\tL(Y_5)
\\
= &\
\tL(y_1)+ \tL(y_2(1+y_1))
-\tL(y_1(1+y_2+y_1y_2)^{-1})
\\
&\qquad - \tL(y_1y_2(1+y_2)^{-1})
- \tL(y_2) =0.
 \end{split}
\end{gather}
This is the counterpart of the form \eqref{eq:pent5}.

Though
we do not see why  the variables $Y_i$ in \eqref{eq:x1} work well at this moment,
they further reveals  the algebraic structure  \eqref{eq:x2},
which is simpler than \eqref{eq:X2},  behind the
pentagon identity.

The identities \eqref{eq:euler3}, \eqref{eq:pent6}, and \eqref{eq:pent7}  in these forms
are the prototypical examples of DIs we are going to study.

\section{Periodicities and DIs for $Y$-systems}

So far, we have  seen only two DIs, namely,
 Euler's identity and the pentagon (Abel's) identity.
Here, we introduce a family of infinitely many DIs, 
which we call the \emph{DIs for $Y$-systems}.
It turns out that
they are  a vast generalization of the above  two identities
in view of \emph{root systems}.

In the series of works
\cite{Kirillov86,
Bazhanov89, Kirillov90, Kirillov89,Bazhanov90},
 Bazhanov, Kirillov, and Reshetikhin studied
some mathematical physics problems on
certain
statistical models on the two-dimensional lattice
by the \emph{thermodynamic Bethe ansatz (TBA) method}.
The \emph{Yang-Baxter equation} \cite{Takhtajan79, Kulish81} and
the \emph{quantum groups}  \cite{Drinfeld85, Jimbo85, Chari95a}
are behind the scenes as 
their algebraic background.
As a part of the results, they naturally 
 reached a remarkable conjecture on a family of (non-functional) dilogarithm identities.
 Let us present the conjecture without explaining its mathematical physics meaning
 because we are interested in the mathematical aspect of the identities in this text.
 
 \begin{figure}
 \begin{center}
 \begin{tikzpicture}[scale=1]
 \draw (0,0) circle (2pt);
 \draw (0.7,0) circle (2pt);
 \draw (2.8,0) circle (2pt);
 \draw (3.5,0) circle (2pt);
 \draw (0.08,0) --(0.62,0);
 \draw (0.78,0) --(1.32,0);
 \draw (2.18,0) --(2.72,0);
 \draw (2.88,0) --(3.42,0);
 \draw [fill] (1.55,0) circle (0.5pt);
 \draw [fill] (1.7,0) circle (0.5pt);
 \draw [fill] (1.85,0) circle (0.5pt);
 \draw (0,-0.4) node {\small $1$};
 \draw (0.7,-0.4) node {\small $2$};
 \draw (2.8,-0.4) node {\small $r-1$};
 \draw (3.5,-0.4) node {\small $r$};
 \draw (-0.9,0) node {$A_r$};
 \draw (6,0) circle (2pt);
 \draw (6.7,0) circle (2pt);
 \draw (8.8,0) circle (2pt);
 \draw (9.5,0.4) circle (2pt);
 \draw (9.5,-0.4) circle (2pt);
 \draw (6.08,0) --(6.62,0);
 \draw (6.78,0) --(7.32,0);
 \draw (8.18,0) --(8.72,0);
 \draw (8.87,0.04) --(9.44,0.36);
 \draw (8.87,-0.04) --(9.44,-0.36);
 \draw [fill] (7.55,0) circle (0.5pt);
 \draw [fill] (7.7,0) circle (0.5pt);
 \draw [fill] (7.85,0) circle (0.5pt);
 \draw (6,-0.4) node {\small $1$};
 \draw (6.7,-0.4) node {\small $2$};
 \draw (8.8,-0.4) node {\small $r-2$};
 \draw (9.5,-0.7) node {\small $r$};
 \draw (9.5,0.7) node {\small $r-1$};
 \draw (5.1,0) node {$D_r$};
 \draw (0,-1.8) circle (2pt);
 \draw (0.7,-1.8) circle (2pt);
 \draw (1.4,-1.8) circle (2pt);
 \draw (2.1,-1.8) circle (2pt);
 \draw (2.8,-1.8) circle (2pt);
 \draw (1.4,-1.1) circle (2pt);
 \draw (0.08,-1.8) --(0.62,-1.8);
 \draw (0.78,-1.8) --(1.32,-1.8);
 \draw (1.48,-1.8) --(2.02,-1.8);
 \draw (2.18,-1.8) --(2.72,-1.8);
 \draw (1.4,-1.72) --(1.4,-1.18);
 \draw (0,-2.2) node {\small $1$};
 \draw (0.7,-2.2) node {\small $2$};
 \draw (1.4,-2.2) node {\small $3$};
 \draw (2.1,-2.2) node {\small \rp$4$};
 \draw (2.8,-2.2) node {\small \rp$5$};
 \draw (1.7,-1.1) node {\small \rp$6$};
 \draw (-0.9,-1.8) node {$E_6$};
 \draw (0,-3.6) circle (2pt);
 \draw (0.7,-3.6) circle (2pt);
 \draw (1.4,-3.6) circle (2pt);
 \draw (2.1,-3.6) circle (2pt);
 \draw (2.8,-3.6) circle (2pt);
 \draw (3.5,-3.6) circle (2pt);
 \draw (4.2,-3.6) circle (2pt);
 \draw (2.8,-3.6) circle (2pt);
 \draw (2.8,-2.9) circle (2pt);
 \draw (0.08,-3.6) --(0.62,-3.6);
 \draw (0.78,-3.6) --(1.32,-3.6);
 \draw (1.48,-3.6) --(2.02,-3.6);
 \draw (2.18,-3.6) --(2.72,-3.6);
 \draw (2.88,-3.6) --(3.42,-3.6);
 \draw (3.58,-3.6) --(4.12,-3.6);
 \draw (2.8,-3.52) --(2.8,-2.98);
 \draw (0,-4) node {\small $1$};
 \draw (0.7,-4) node {\small $2$};
 \draw (1.4,-4) node {\small $3$};
 \draw (2.1,-4) node {\small $4$};
 \draw (2.8,-4) node {\small $5$};
 \draw (3.5,-4) node {\small $6$};
 \draw (4.2,-4) node {\small $7$};
 \draw (3.1,-2.9) node {\small $8$};
 \draw (-0.9,-3.6) node {$E_8$};
 \draw (6,-1.8) circle (2pt);
 \draw (6.7,-1.8) circle (2pt);
 \draw (7.4,-1.8) circle (2pt);
 \draw (8.1,-1.8) circle (2pt);
 \draw (8.8,-1.8) circle (2pt);
 \draw (9.5,-1.8) circle (2pt);
 \draw (7.4,-1.1) circle (2pt);
 \draw (6.08,-1.8) --(6.62,-1.8);
 \draw (6.78,-1.8) --(7.32,-1.8);
 \draw (7.48,-1.8) --(8.02,-1.8);
 \draw (8.18,-1.8) --(8.72,-1.8);
 \draw (8.88,-1.8) --(9.42,-1.8);
 \draw (7.4,-1.72) --(7.4,-1.18);
 \draw (6,-2.2) node {\small $1$};
 \draw (6.7,-2.2) node {\small $2$};
 \draw (7.4,-2.2) node {\small $3$};
 \draw (8.1,-2.2) node {\small $4$};
 \draw (8.8,-2.2) node {\small $5$};
 \draw (9.5,-2.2) node {\small $6$};
 \draw (7.7,-1.1) node {\small $7$};
 \draw (5.1,-1.8) node {$E_7$};
 \end{tikzpicture}
 \end{center}
\vskip-10pt
\caption{Simply-laced Dynkin diagrams of finite type.}
\label{fig:Dynkin1}
\end{figure}
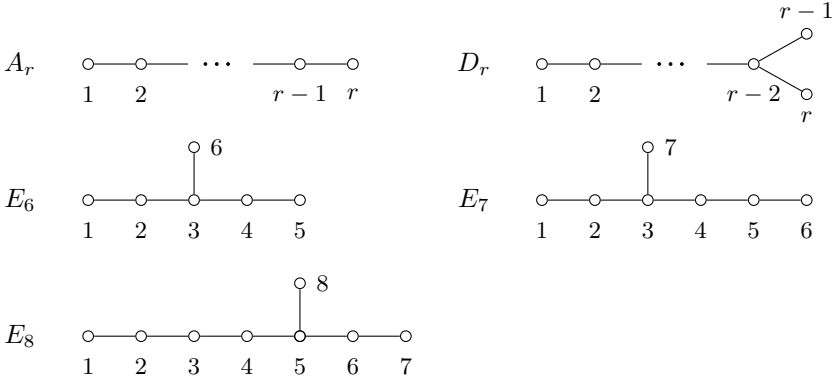

Let $X$ be a simply-laced (connected) Dynkin diagram of finite type
\cite{Bourbaki02, Kac90},
namely, $X=A_r$ ($r\geq 1$), $D_r$ ($r\geq 4$), $E_6$, $E_7$, or $E_8$
as in Figure \ref{fig:Dynkin1}.
The number $r$ of the vertices of $X$ is called the \emph{rank}\index{rank!of Dynkin diagram}.
Let $I=\{1,\, \dots, \, r\}$ be the index set for the vertices of $X$
as in the figure.
Let $\ell\geq 2$ be any integer,
which we call the \emph{level}.
For any pair $(X,\ell)$, 
there is  an integrable lattice model in two dimensions
called the \emph{restricted solid-on-solid (RSOS) model of type $X$ with level $\ell$} \cite{Jimbo88,Bazhanov89,Bazhanov90}.
Let   $\{y_m^{(a)} \in \bbR_{>0} \mid a\in I;\, m=1,\, \dots,\, \ell-1 \}$
be  a positive real 
solution of the following system of algebraic relations:
\begin{align}
\label{eq:ysys1}
(y_m^{(a)})^2 = 
\frac{\displaystyle\prod_{b\in I:\, b\sim a} (1+y_m^{(b)})}{(1+y_{m-1}^{(a)}{}^{-1})(1+y_{m+1}^{(a)}{}^{-1})},
\end{align}
where $b\sim a$ means that $b$ is adjacent to $a$ in the Dynkin diagram $X$.
For example, for $X=E_6$, we have $3 \sim 2$, 4, {\rp 6}.
Also, we set $y_{0}^{(a)}{}^{-1}=y_{\ell}^{(a)}{}^{-1}=0$ if they appear 
in the denominator of the RHS.

Remarkably, the special value of the modified Rogers dilogarithm $\tL(y_m^{(a)})$  emerges in the calculation
of a certain thermodynamic quantity for the RSOS model.
Motivated by this background, the following summation formula was conjectured
 by Kirillov, Bazhanov, and Reshetikhin:
\begin{conj}[\cite{Kirillov89, Bazhanov90}]
\label{conj:KBR1}
For the above $y_m^{(a)}$'s,
the following equality holds:
\begin{align}
\label{eq:y2}
\sum_{a=1}^r\sum_{m=1}^{\ell-1} \tL(y_m^{(a)}) = \frac{ r (\ell-1)  h}{h+\ell} 
\frac{\pi^2}{6},
\end{align}
where  $h=h(X)$ is the Coxeter number of $X$, namely,
$h=r+1$ for $A_r$, $2r-2$ for $D_r$, and 12, 18, 30 for $E_6$, $E_7$, $E_8$.
\end{conj}

\begin{rem}
(1).
To be precise, the original conjecture was given in terms of the
solution of another system of algebraic relations called the \emph{$Q$-systems},
and the equality is given by the Rogers dilogarithm $L(x)$.
It was translated as above by Kuniba-Nakanishi-Suzuki \cite{Kuniba94a}.
\par
(2).
Later, it was shown by Nahm-Keegan   \cite{Nahm09} that there exists a unique positive real solution of \eqref{eq:ysys1}.
\end{rem}

Remarkably,
for $X=A_r$ and any $\ell$, Kirillov \cite{Kirillov89} gave the explicit expression of $y_m^{(a)}$
and proved the equality based on it.
 Unfortunately, this approach does not work for the other cases,
and they were left as conjectures in B.C.

After Conjecture \ref{conj:KBR1}, 
a  functional extension (affinization) of the system of the
algebraic relation \eqref{eq:ysys1}
called the \emph{$Y$-system}
was introduced by Zamolodchikov \cite{Zamolodchikov91} for $\ell=2$,
and by Kuniba-Nakanishi \cite{Kuniba92} and Ravanini-Tateo-Valleriani \cite{Ravanini93} for general $\ell$.
Here, we present the version in \cite{Ravanini93}, where the symmetry 
of the numerator and denominator in \eqref{eq:ysys1} is  manifest.

\begin{defn}[$Y$-system]
\label{defn:Ysystem1}\index{$Y$-system}
Let $(X, X')$ be a pair of simply-laced Dynkin diagrams of finite type with
  index sets $I=\{1,\, \dots,\, r \}$ and $I'=\{1,\, \dots,\, r' \}$ for their vertices, respectively.
We introduce variables $Y_{a,a'}(u)$ for $(a,a')\in I\times I'$
and $u\in \bbZ$,
and consider the following system of functional/algebraic relations
(the \emph{$Y$-system of type $(X, X')$}):
\begin{align}
\label{eq:Ysys1}
Y_{a,a'}(u+1) Y_{a,a'}(u-1)= 
\frac{\displaystyle\prod_{b\in I:\, b\sim a} (1+Y_{b,a'}(u))}
{\displaystyle\prod_{b'\in I':\, b'\sim a'} (1+Y_{a,b'}(u)^{-1})},
\end{align}
where the symbol $\sim$ denotes  the adjacency
in  $X$ in the numerator
and in  $X'$ in the denominator.
The product is 1 if it is empty.
\end{defn}

In view of the $Y$ system \eqref{eq:Ysys1},
 we identify the former system \eqref{eq:ysys1} 
 with the \emph{constant $Y$-system}\index{constant $Y$-system}\index{$Y$-system!constant ---}
 of type $(X, A_{\ell-1})$.
 Namely, it is satisfied by the \emph{constant solution}
 $y_{m}^{(a)}=Y_{a,m}(u)$  of the $Y$-system
 with respect to $u$.
In the original context in mathematical physics, $u$ is a \emph{complex} parameter.
However, for our purpose, we may restrict $u$ as an \emph{integer parameter}.
Then, the relation \eqref{eq:Ysys1}
is regarded as a discrete dynamical system with 
the discrete-time $u$ and a given initial condition at $u=0$ and $1$.

Observe that
$Y$-systems of types $(X, X')$ and $(X', X)$
are equivalent
under the correspondence $Y_{a,a'}(u)^{-1}
\leftrightarrow Y_{a',a}(u)$.
We call it the \emph{level-rank duality}\index{level-rank duality}\index{duality!level-rank}
in view of the original context of the constant $Y$-system \eqref{eq:ysys1}.

The following conjecture was given by Zamolodchikov, Ravanini-Tateo-Valleriani, 
and Gliozzi-Tateo, again motivated by the underlying mathematical physics 
models:
\begin{conj}
\label{conj:DIY1}
Let $\{ Y_{a,a'} (u) \in \bbR_{>0}   \mid (a,a')\in I\times I',  u\in \bbZ \}$ be any  positive real
solution of the $Y$-system of type $(X, X')$.
Let $r,\, r'$ and $h,\, h'$ be the ranks and the Coxeter numbers of $X$ and $X'$, respectively.
\par
(1) Periodicity \cite{Ravanini93} (\cite{Zamolodchikov91} for $X'=A_1$).
\begin{align}
\label{eq:Y2}
Y_{a,a'}(u+2(h+h'))=Y_{a,a'}(u).
\end{align}
\par
(2) Dilogarithm identity  \cite{Gliozzi95}.
\begin{align}
\label{eq:Y3}
\sum_{u=0}^{2(h+h')-1} \sum_{(a,a')\in I\times I'}
\tL(Y_{a,a'}(u)) & = 2 { r r' h} \frac{\pi^2}{6}.
\end{align}
\end{conj}

Let us call the conjectural identity \eqref{eq:Y3} the \emph{DI for the $Y$-system
of type $(X, X')$}\index{dilogarithm identity!for $Y$-system}.
Some remarks are in order.

\begin{enumerate}
\item
The sum in \eqref{eq:Y3} is
 over the fundamental domain of the periodicity in \eqref{eq:Y2}.

\item
The general solution of  the $Y$-system \eqref{eq:Ysys1}
is expressed as subtraction-free rational functions in the initial variables $Y_{a,a'}(u)$ at
$u=0$ and $1$.
Thus, the equality \eqref{eq:Y3} is regarded as a functional
identity with respect to the initial variables $Y_{a,a'}(0)$ and $Y_{a,a'}(1)$
that take values in $\bbR_{>0}$.
\item
Alternatively, the initial variables $Y_{a,a'}(u)$ can be taken at
$u=-1$ and $0$. 
\item
By the level-rank duality,
the identity  \eqref{eq:Y3} is equivalent to 
the identity
\begin{align}
\label{eq:Y4}
\sum_{u=0}^{2(h+h')-1} \sum_{(a,a')\in I\times I'}
\tL(Y_{a,a'}(u)^{-1}) & = 2 { r r' h'} \frac{\pi^2}{6}.
\end{align}
Alternatively,
 two identities \eqref{eq:Y3} and \eqref{eq:Y4} are shown to be equivalent 
by Euler's identity \eqref{eq:euler3}
because
there are $2rr'(h+h')$ terms in
the sums of \eqref{eq:Y3}  and \eqref{eq:Y4}, respectively.

\item
The original conjectural (non-functional) identity \eqref{eq:y2} follows from the DI  \eqref{eq:Y3}.
Indeed,  $\{ y_m^{(a)}\in \bbR_{>0}\} $ satisfying \eqref{eq:ysys1} is 
a constant solution $y_m^{(a)}=Y_{a,m}(u)$ of the $Y$-system \eqref{eq:Ysys1}
of type $(X, A_{\ell-1})$ as mentioned above.
We have $r'=\ell-1$ and $h'=\ell$ therein.
Then, in the LHS of \eqref{eq:Y3},
the sum over $u$ is replaced with
the multiplication by $2(h+\ell)$. 
This is exactly the equality \eqref{eq:y2}.

\end{enumerate}

Let us see the two simplest examples for Conjecture \ref{conj:DIY1}.

\begin{ex}
 $(X, X')=(A_1,A_1)$.
We have $r=r'=1$ and $h=h'=2$.
We set $Y(u):=Y_{1,1}(u)$.
Then,
the $Y$-system \eqref{eq:Ysys1} looks as
\begin {align}
\label{eq:YY1}
Y(u+1)Y(u-1)=1.
\end{align}
Thus, we have the periodicity
\begin{align}
\label{eq:Y4p}
Y(u+4)=Y(u+2)^{-1}=Y(u).
\end{align}
Meanwhile,
\eqref{eq:Y2} claims the periodicity
\begin{align}
\label{eq:Y4p2}
Y(u+8)=Y(u).
\end{align}
Thus, the periodicity \eqref{eq:Y2} holds.
Next, the identity \eqref{eq:Y3} claims that
\begin{align}
\label{eq:Y5}
\sum_{u=0}^{7}  \tL(Y(u)) = 4 \frac{\pi^2}{6}.
\end{align}
By \eqref{eq:Y4p},
this  is reduced to
\begin{align}
\label{eq:Y51}
2\{   \tL(Y(0))  +  \tL(Y(0)^{-1}) +   \tL(Y(1))  +  \tL(Y(1)^{-1})\}
= 2 \frac{\pi^2}{6}.
\end{align}
Moreover,
since $Y(0)$ and $Y(1)$ are independent variables,
the identity \eqref{eq:Y51} decomposes into two identical identities
\begin{align}
  \tL(Y(0))  +  \tL(Y(0)^{-1}) =   \tL(Y(1))  +  \tL(Y(1))^{-1}) = \frac{\pi^2}{6}.
\end{align}
This is Euler's identity in the form \eqref{eq:euler3}.
\end{ex}

\begin{ex}
 $(X, X')=(A_2,A_1)$.
 We have $r=2$, $r'=1$ and $h=3$, $h'=2$.
We set $Y_1(u):=Y_{1,1}(u)$ and $Y_2(u):=Y_{2,1}(u)$.
Then,
the $Y$-system \eqref{eq:Ysys1} looks as
\begin {align}
\label{eq:YY2}
Y_1(u+1)Y_1(u-1)&=1+Y_2(u),
\\
Y_2(u+1)Y_2(u-1)&=1+Y_1(u).
\end{align}
The system decouples into two independent systems,
one for $Y_a(u)$'s with even $a+u$ (the \emph{even sector}),
and the other for $Y_a(u)$'s with odd $a+u$ (the \emph{odd sector}).
Let us concentrate on the even sector,
and we set
\begin{align}
Y(u):=
\begin{cases}
Y_1(u) & \text{$u$ is odd,}
\\
Y_2(u) & \text{$u$ is even.}
\end{cases}
\end{align}
Then, the $Y$-system on the even sector is unified into a single relation
\begin{align}
Y(u+1)Y(u-1)=1+Y(u).
\end{align}
We have already encountered this recurrence relation in \eqref{eq:x2},
where we identify $Y_i$ therein with $Y(i-1)$.
Moreover, a general solution was already given in \eqref{eq:x1},
which yields the periodicity
\begin{align}
\label{eq:Yper1}
Y(u+5)=Y(u).
\end{align}
Meanwhile,
\eqref{eq:Y2} claims the periodicity
\begin{align}
\label{eq:Y6}
Y(u+10)=Y(u).
\end{align}
Thus, the periodicity \eqref{eq:Y2} holds.
Next, the identity \eqref{eq:Y3} restricting to the even sector claims
that
\begin{align}
\label{eq:Y7}
\sum_{u=0}^{9}  \tL(Y(u)) = 6 \frac{\pi^2}{6},
\end{align}
where the RHS of \eqref{eq:Y3} was divided by 2.
By the periodicity \eqref{eq:Yper1},
it is reduced to
\begin{align}
\sum_{u=0}^{4}  \tL(Y(u))  =  3 \frac{\pi^2}{6}.
\end{align}
This is the pentagon identity in the form \eqref{eq:pent6}.
\end{ex}

Thus, we saw that the two simplest  cases  of the conjectural DIs  \eqref{eq:Y3}
are  Euler's and the pentagon (Abel's) identities.
In other words, DIs  \eqref{eq:Y3}
yield a vast generalization of these classic  identities
in view of root systems.
This was indeed a remarkable and novel discovery in the long history of the dilogarithm.

Of course, we need to prove Conjecture \ref{conj:DIY1}.
The conjecture consists of two parts.
The first part \eqref{eq:Y2}
 seems easier.
Indeed,  for very  low ranks, it is
proved by explicitly solving the $Y$-system.
Also,
 for higher ranks,
it is possible to test and confirm the periodicity
  \emph{numerically}
with randomly chosen initial values.
In general, however, the problem turned out to be not easy.
It  was proved only in the case $(X,X')=(A_r, A_1)$ by Gliozzi-Tateo \cite{Gliozzi96}
and Frenkel-Szenes \cite{Frenkel95} by giving the explicit solution (with different expressions to each other) of the $Y$-system.
 Similarly, the second part \eqref{eq:Y2} was proved only 
 in the same case by them \cite{Gliozzi96,Frenkel95}
 using their explicit solutions.
 This was the situation in B.C.

The scene changed after the introduction of cluster algebras (in ``A.D.").
It was recognized that there is a cluster algebraic structure behind
the $Y$-system \eqref{eq:Ysys1}.
Using the structure, Conjecture \ref{conj:DIY1}
and their variants for other $Y$-systems were proved in full generality.
So, we may say that
 it is exactly the cluster algebra structure
that is responsible for Conjecture \ref{conj:DIY1}.
We are going to reveal this hidden and exciting relationship in the rest of Part I.

\section{Reduction problem}
\label{sec:reduction1}

Let us conclude this introductory chapter by presenting the following important result
by Wojtkowiak \cite{Wojtkowiak96}.

\begin{thm}[{\cite[Theorem 4.4]{Wojtkowiak96}}]
\label{thm:reduction1}
Any functional identity of the dilogarithm with rational functions of one variable
as arguments is reduced to the trivial one by the successive application of
the pentagon identity.
\end{thm}

Based on this fact, there is a conjecture that (or a question if)  the same is true for 
the DIs with
rational functions of \emph{more than one variable}
as arguments (e.g.,  \cite[\S2.A]{Zagier07}).
Let us call it the \emph{reduction problem}\index{reduction!problem}.
There are several case studies on the problem (e.g., \cite{Kirillov95}).
No counterexample is known, 
but there is also no systematic study of the problem, as far as we know.
For
the DIs for $Y$-systems (and all DIs considered in this text),
we will answer the question (in a little unexpected way) later in Section \ref{sec:construction1}.

\notes
Many  DIs
are presented in   \cite{Lewin81} and  \cite{Kirillov95}.
Most of them seem to be of a different kind from the one we consider in this text.
We do not know how wide the cluster algebraic structure covers general DIs.

There is
another class of  mathematical physics  models called the \emph{factorizable $S$-matrix models} (\cite{Klassen90, Zamolodchikov91b,Ravanini93}),
which are also based on the Yang-Baxter equation.
The dilogarithm sums  \eqref{eq:y2}  and the corresponding $Y$-systems appeared
also in the study of their thermodynamic property.

As for $Y$-systems,
see the review \cite{Kuniba10} for  the background  in mathematical physics
and further information.
There are several variants of $Y$-systems other than the
one presented here,
and the conjectures on their periodicities and DIs for them were also given  in B.C.,
for example,  
 for the nonsimply-laced $Y$-systems \cite{Kuniba93a, Kuniba92,Kuniba94a}
and  the sine-Gordon $Y$-systems \cite{Tateo95,Gliozzi96}.

\chapter{Quick course on cluster algebras}
\label{ch:quick1}

Let us return in A.D.~(= after cluster algebra).
Cluster algebras  introduced by Fomin-Zelevinsky \cite{Fomin02}
are a class of commutative algebras
originated in Lie theory.
In this chapter, we present several basic notions and fundamental results on cluster algebras.
All contents are well-known in cluster algebra theory.
We skip proofs
because the readers can apply these results to the forthcoming problems
without knowing the proofs.
Proofs of all results  in this chapter  except for Theorem \ref{thm:finite1} 
are conveniently found in the companion monograph
 \cite{Nakanishi22a}.
 More specific instruction is found in Notes of the chapter.

\section{Semifield}

Let us begin with the notion of a semifield following \cite{Fomin02}.

\begin{defn}[Semifield]
A multiplicative abelian group $\mathbb{P}$
equipped with 
 a binary operation
$\oplus$
 is called a \emph{semifield}\index{semifield} if 
the following properties hold:
For any $a,\, b,\, c\in \mathbb{P}$,
\begin{align}
\label{1eq:sf1}
a\oplus b &= b\oplus a,\\
(a\oplus b)\oplus c &= a\oplus (b\oplus c),\\
\label{1eq:sf3}
(a\oplus b) c &= ac\oplus bc.
\end{align}
The operation $\oplus$ is called the \emph{addition\/} in $\mathbb{P}$.
Note that there is no \emph{subtraction\/} in $\mathbb{P}$.
\end{defn}

For example, the set of positive real numbers $\bbR_{>0}$ is a semifield
by the usual multiplication and addition.

The following two examples are especially important in cluster algebra theory.

\begin{ex}
\label{1ex:sf1}
(a) \emph{Universal semifield}\index{universal semifield}\index{semifield!universal ---} 
$\mathbb{Q}_{\mathrm{sf}}(\bfu)$.
Let $\bfu=(u_1,\dots, u_n)$ be an $n$-tuple of 
 variables.
Let $\bbQ(\bfu)$ be the  rational function field of $\bfu$
over $\bbQ$.
We say that a
rational function $f(\bfu)\in \bbQ(\bfu)$ 
has a \emph{subtraction-free expression}\index{subtraction-free expression} 
if it is expressed as
$f(\bfu)=p(\bfu)/q(\bfu)$, where both $p(\bfu)$ and $q(\bfu)$
are \emph{nonzero} polynomials in $\bfu$ whose coefficients
are \emph{nonnegative} integers.
For example, $f(\bfu)=u_1^2-u_1+1=(u_1^3+1)/(u_1+1)$ has a subtraction-free expression.
Let   $\mathbb{Q}_{\mathrm{sf}}(\bfu)$
be the set of all rational functions in $\bfu$
 having subtraction-free expressions.
Then, $\mathbb{Q}_{\mathrm{sf}}(\bfu)$
is a semifield by the usual 
multiplication and addition in $\bbQ(\bfu)$.
\par
(b) \emph{Tropical semifield}\index{tropical!semifield}\index{semifield!tropical ---}
 $\mathrm{Trop}(\bfu)$.
Let $\bfu=(u_1,\dots, u_n)$ be an $n$-tuple of 
 variables.
Let $\mathrm{Trop}(\bfu)$ be the set of all
Laurent monomials in $\bfu$ with coefficient 1,
which is a multiplicative abelian group by the usual
multiplication.
We define the addition $\oplus$ by
\begin{align}
\label{1eq:ts1}
\prod_{i=1}^n u_i^{a_i} \oplus
\prod_{i=1}^n u_i^{b_i}
:=
\prod_{i=1}^n u_i^{\min(a_i,b_i)} .
\end{align}
Then,  $\mathrm{Trop}(\bfu)$ becomes a semifield.
The addition $\oplus$ is called the \emph{tropical sum}\index{tropical!sum}.
\end{ex}

\begin{lem}[\cite{Fomin03a}]
\label{lem:domain1}
Let $\bbZ\bbP$ be the group algebra of $\bbP$ (as a multiplicative group).
Then, $\bbZ\bbP$ is a domain, i.e., there is no zero divisor other than 0.
\end{lem}
Thanks to the lemma, we have the fraction field $\bbQ\bbP$ of $\bbZ\bbP$.

\section{Seeds and  mutations}

Here, we introduce  the most fundamental notions in cluster algebras,
namely, \emph{seed} and \emph{mutation}.

\begin{defn}[Skew-symmetrizable matrix]
\label{defn:skew1}
An integer  square matrix $B=(b_{ij})_{i,j=1}^n$
is said to be \emph{skew-symmetrizable}\index{skew-symmetrizable matrix}
if there is a diagonal matrix $D=\mathrm{diag}(d_1,\dots,d_n)$
whose diagonal entries $d_i$ are positive rational numbers
such that $DB$ is skew-symmetric,
i.e,
\begin{align}
\label{1eq:ss1}
d_i b_{ij}= - d_jb_{ji}.
\end{align}
The above matrix $D$  is called a \emph{(left) skew-symmetrizer} \index{skew-symmetrizer}
of $B$,
which is not unique to $B$.
\end{defn}

For example, any skew-symmetric matrix is skew-symmetrizable with
a skew-symmetrizer $D=I$.
For any skew-symmetrizable matrix $B$, $b_{ii}=0$  holds by
\eqref{1eq:ss1}.

\begin{defn}[Seed/$Y$-seed]
Let $n$ be any positive integer,
and  let $\bbP$ be any semifield.
Let $\calF$  be a field that is isomorphic to
the rational function field of $n$-variables
with coefficients in the field $\bbQ\bbP$.
\begin{itemize}
\item
A \emph{(labeled) seed\index{seed}\index{seed!labeled ---} 
with coefficients in $\bbP$
(or a seed in $\calF$)}
is a triplet
$\Sigma=(\bfx,\bfy,B)$
such that
$\bfx=(x_1,\dots,x_n)$ is an
$n$-tuple of algebraically independent and generating elements
in $\calF$,
$\bfy=(y_1,\dots,y_n)$ is an
$n$-tuple of any  elements
in $\bbP$,
and $B=(b_{ij})_{i,j=1}^n$
is an $n\times n$ skew-symmetrizable (integer) matrix.
\item In the above, a pair $(\bfy, B)$ is called
a \emph{$Y$-seed in $\bbP$}\index{$Y$-seed}.
\item
We call $\bfx$, $\bfy$, and $B$
the \emph{cluster}\index{cluster},
the \emph{coefficient tuple}\index{coefficient!tuple},
and the \emph{exchange matrix}\index{exchange matrix} 
of $\Sigma$, respectively.
The variables $x_i$ and $y_i$ are  called
a \emph{cluster variable}\index{cluster!variable}
 and a \emph{coefficient}\index{coefficient}, respectively.
In this text, we mainly call them  \emph{$x$-}\index{$x$-variable} and  \emph{$y$-variables}\index{$y$-variable}, respectively.
\item
We call $n$, $\bbP$, and $\calF$ the \emph{rank}\index{rank!of cluster pattern},
the \emph{coefficient semifield}\index{semifield!coefficient},
and  the \emph{ambient field}\index{ambient field}
of a seed $\Sigma$, respectively, 
and also of the forthcoming cluster patterns, cluster algebras, etc.
\end{itemize}
\end{defn}

For any seed $\Sigma=(\bfx,\bfy,B)$,
we attach an $n$-tuple 
$\hat\bfy=(\hat{y}_1,\dots,\hat{y}_n)$
of elements in the ambient field $\calF$ defined by
\begin{align}
\label{2eq:yhat1}
\hat{y}_i
= y_i \prod_{j=1}^n x_j^{b_{ji}}.
\end{align}
They are called  \emph{$\hat{y}$-variables}\index{$\yhat$-variables}
and play an important role in cluster algebra theory.

For any integer $a$,
we define 
\begin{align}
[a]_+:=\max(a,0).
\end{align}
We have a  useful equality
\begin{align}
\label{1eq:pos1}
a&=[a]_+ - [-a]_+.
\end{align}

\begin{defn}[Seed mutation]
\label{2defn:mut1}
For any seed $\Sigma=(\bfx,\bfy,B)$
and $k=1,\, \dots,\, n$,
we define a new seed 
$\Sigma'=(\bfx',\bfy',B')$
of the same kind
by the following rule:
\begin{align}
\label{2eq:xmut1}
x'_i
&=
\begin{cases}
\displaystyle
x_k^{-1}\Biggl(\, \prod_{j=1}^n x_j^{[-b_{jk}]_+}
\Biggr)
\frac{
 1+\hat{y}_k}
 {1\oplus y_k}
 & i=k,
\\
x_i
&i\neq k,
\end{cases}
\\
\label{2eq:ymut1}
y'_i
&=
\begin{cases}
\displaystyle
y_k^{-1}
& i=k,
\\
y_i y_k^{[b_{ki}]_+} (1\oplus y_k)^{-b_{ki}}
&i\neq k,
\end{cases}
\\
\label{2eq:bmut1}
b'_{ij}&=
\begin{cases}
-b_{ij}
&
\text{$i=k$ or $j=k$,}
\\
b_{ij}+
b_{ik} [b_{kj}]_+
+
[-b_{ik}]_+b_{kj}
&
i,\, j\neq k,
\end{cases}
\end{align}
where $\hat{y}_k$ in \eqref{2eq:xmut1} is a $\hat{y}$-variable defined in  \eqref{2eq:yhat1}.
The seed $\Sigma'$ is called the \emph{mutation of $\Sigma$ in direction $k$}\index{mutation!of seed},
and denoted by $\mu_k(\Sigma)=\mu_k(\bfx,\bfy,B)$.
\end{defn}

One can easily check the following facts.
\begin{enumerate}
\item
Any skew-symmetrizer $D$ of $B$ is also a skew-symmetrizer of $B'$.
Therefore, $B'$ is also skew-symmetrizable.
\item
We have $\Sigma=\mu_k(\Sigma')$. Namely, each mutation $\mu_k$ is involutive.
\item
$x'_1$, \dots, $x'_n$ are algebraically independent and generating elements
in $\calF$. Thus, $\Sigma'$ is indeed a seed.
\item We have $\det B' = \det B$. In particular, $B'$ is nonsingular if and only if $B$ is so.
\item 
We say that a square matrix is \emph{indecomposable}\index{indecomposable matrix} if it is not a direct sum of two square matrices
up to the simultaneous change of the row and column indices.
Then, $B'$ is indecomposable if and only if $B$ is indecomposable.
\end{enumerate}

The following fact exhibits a certain duality between $x$- and $y$-variables.
\begin{prop}[{\cite[Prop.~3.9]{Fomin07}}]
\label{prop:yhat1}
The $\haty$-variables in 
\eqref{2eq:yhat1} mutate as $y$-variables in $\calF$. Namely, we have
\begin{align}
\label{2eq:yhatmut1}
\haty'_i
&=
\begin{cases}
\displaystyle
\haty_k^{-1}
& i=k,
\\
\haty_i \haty_k^{[b_{ki}]_+} (1+ \haty_k)^{-b_{ki}}
&i\neq k.
\end{cases}
\end{align}
\end{prop}
For example, for $i\neq k$,
by using the fact $b_{kk}=0$, we have
\begin{align*}
\hat{y}'_i
&=
y'_i
\prod_{j=1}^n x'_j{}^{b'_{ji}}\\
&=y_i y_k^{[b_{ki}]_+} (1\oplus y_k)^{-b_{ki}}
\Biggl(
\prod_{\scriptstyle j=1\atop \scriptstyle  j\neq k}^n x_j^{b_{ji}+
b_{jk} [b_{ki}]_+
+
[-b_{jk}]_+b_{ki}}
\Biggr)
\\
&\quad\ \times
\Biggl(
x_k^{-1}
\Biggl(\,
\prod_{j=1}^n x_j^{[-b_{jk}]_+}
\Biggr)
\frac{
 1+\hat{y}_k}
 {1\oplus y_k}
\Biggr)^{-b_{ki}}
\\
&=\hat{y}_i\hat{y}_k^{[b_{ki}]_+}
(1+\hat{y}_k)^{-b_{ki}}.
\end{align*}

We define the action of the symmetric group $S_n$  on seeds.
\begin{defn}[$S_n$-action]
\label{defn:saction1}
\index{$S_n$-action (on seeds)}
For a seed $\Sigma=(\bfx, \bfy, B)$ and a permutation $\nu$ of $\{ 1,\, \dots,\, n\}$,
we define the action of $\nu$ on $\Sigma$ by
\begin{align}
\label{1eq:saction1}
\nu\Sigma = (\nu \bfx, \nu \bfy, \nu B),
\end{align}
where
$\nu \bfx=\bfx'$, 
$\nu \bfy=\bfy'$, 
$\nu B=B'$ are defined by
\begin{align}
x'_i=x_{\nu^{-1}(i)},
\quad
y'_i=y_{\nu^{-1}(i)},
\quad
b'_{ij}=b_{\nu^{-1}(i)\nu^{-1}(j)}.
\quad
\end{align}
This is a left action of the symmetric group $S_n$ of degree $n$; namely,
we have $\tau (\nu \Sigma)=\tau\nu (\Sigma)$
for $\nu, \, \tau\in S_n$.
 
 One can easily confirm that
 the action is compatible with mutations in the following sense:
 \begin{align}
     \label{eq:comp1}
   \mu_{\nu(k)}(  \nu \Sigma_t)&=\nu (\mu_k (\Sigma_t)).
   \end{align}
\end{defn}

\section{Cluster patterns and $Y$-patterns}

Let $\bbT_n$ be the \emph{$n$-regular tree} graph\index{$n$-regular tree};
 Namely,  it is a connected tree graph such that exactly $n$ edges are attached to each vertex of $\bbT_n$.
Moreover,  edges are labeled by $1$, \dots, $n$
so that the edges attached to each vertex 
have different labels.
By abusing the notation,
the set of vertices of $\bbT_n$ is also denoted by 
 $\bbT_n$.
 See Figure \ref{fig:regular1}.
 We say that a pair of vertices in $t$ and $t'$ in $\bbT_n$
 are $k$-\emph{adjacent}, or $t'$ is $k$-\emph{adjacent to $t$}\index{$k$-adjacent},
 if they are connected with an edge labeled by $k$.
 
 \begin{figure}
 \begin{center}
 \begin{tikzpicture}
 \draw (0,0)--(4,0);
 \draw [fill] (0,0) circle (1.5pt);
 \draw [fill] (1,0) circle (1.5pt);
 \draw [fill] (2,0) circle (1.5pt);
 \draw [fill] (3,0) circle (1.5pt);
 \draw [fill] (4,0) circle (1.5pt);
 \draw [fill] (-0.7,0) circle (0.5pt);
 \draw [fill] (-0.5,0) circle (0.5pt);
 \draw [fill] (-0.3,0) circle (0.5pt);
 \draw [fill] (4.7,0) circle (0.5pt);
 \draw [fill] (4.5,0) circle (0.5pt);
 \draw [fill] (4.3,0) circle (0.5pt);
 \draw (0.5,0.3) node {2};
 \draw (1.5,0.3) node {1};
 \draw (2.5,0.3) node {2};
 \draw (3.5,0.3) node {1};
 \draw (2,-0.5) node {$n=2$};
  \draw (1.5,1.5)--(2.5,1.5);
   \draw [fill] (1.5,1.5) circle (1.5pt);
 \draw [fill] (2.5,1.5) circle (1.5pt);
 \draw (2,1.8) node {1};
  \draw (2,1) node {$n=1$};
\draw (7,0.75)--(8,0.75);
\draw (6,0.25)--(7,0.75);
\draw (6,1.25)--(7,0.75);
\draw (9,0.25)--(8,0.75);
\draw (9,1.25)--(8,0.75);
\draw (9,1.25)--(10,0.9);
\draw (9,1.25)--(10,1.6);
\draw (9,0.25)--(10,0.6);
\draw (9,0.25)--(10,-0.1);
 \draw [fill] (7,0.75) circle (1.5pt);
 \draw [fill] (8,0.75) circle (1.5pt);
 \draw [fill] (6,0.25) circle (1.5pt);
 \draw [fill] (6,1.25) circle (1.5pt);
 \draw [fill] (9,0.25) circle (1.5pt);
 \draw [fill] (9,1.25) circle (1.5pt);
 \draw [fill] (10,1.6) circle (1.5pt);
 \draw [fill] (10,0.9) circle (1.5pt);
 \draw [fill] (10,0.6) circle (1.5pt);
 \draw [fill] (10,-0.1) circle (1.5pt);
 \draw [fill] (5.7,0.25) circle (0.5pt);
 \draw [fill] (5.5,0.25) circle (0.5pt);
 \draw [fill] (5.3,0.25) circle (0.5pt);
 \draw [fill] (5.7,1.25) circle (0.5pt);
 \draw [fill] (5.5,1.25) circle (0.5pt);
 \draw [fill] (5.3,1.25) circle (0.5pt);
 \draw [fill] (10.7,1.25) circle (0.5pt);
 \draw [fill] (10.5,1.25) circle (0.5pt);
 \draw [fill] (10.3,1.25) circle (0.5pt);
 \draw [fill] (10.7,0.25) circle (0.5pt);
 \draw [fill] (10.5,0.25) circle (0.5pt);
 \draw [fill] (10.3,0.25) circle (0.5pt);
  \draw (7.5,1.05) node {1};
  \draw (6.5,1.25) node {2};
  \draw (6.5,0.7) node {3};
  \draw (8.5,1.25) node {2};
  \draw (8.5,0.7) node {3};
  \draw (9.6,1.7) node {1};
  \draw (9.6,1.25) node {3};
  \draw (9.6,0.7) node {1};
  \draw (9.6,0.25) node {2};
 \draw (8,-0.5) node {$n=3$};
\end{tikzpicture}
\end{center}
\vskip-10pt
\caption{The $n$-regular  trees $\bbT_n$
 for $n=1$, $2$, $3$.}
\label{fig:regular1}
\end{figure}
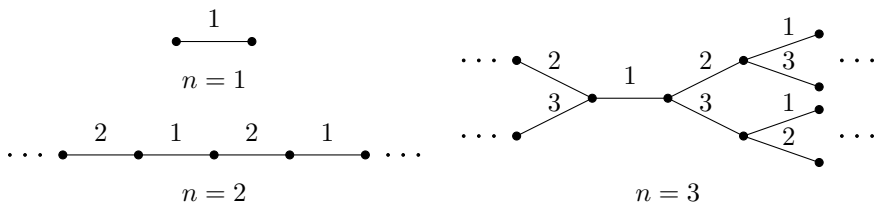

\begin{defn}[Cluster pattern/$Y$-pattern/$B$-pattern]
Fix the rank $n$,
the coefficient semifield $\bbP$,
and the ambient field $\calF$ for seeds.
\begin{itemize}
\item
A collection of seeds $\mathbf{\Sigma}=\{ \Sigma_t
=(\bfx_t,\bfy_t,B_t)
\}_{ t\in \bbT_n}
$
 indexed by $\bbT_n$ 
is called a \emph{cluster pattern}\index{cluster!pattern}  with coefficients in $\bbP$
if,
for any vertices $t,\,t'\in \bbT_n$ that are $k$-adjacent,
the equality $\Sigma_{t'}=\mu_k( \Sigma_t)$ holds.
\item
The collection of $Y$-seeds
 $\bfUpsilon=\{ \Upsilon_t=(\bfy_t,B_t) \}
_{t\in \bbT_n}
$ extracted from a cluster pattern $\bfSigma$ is called 
a \emph{$Y$-pattern}\index{$Y$-pattern} in $\bbP$,
or the  \emph{$Y$-pattern}
of $\bfSigma$.
\item
The collection of exchange matrices
 $\mathbf{B}=\{ B_t\}
_{t\in \bbT_n}
$ extracted from a cluster pattern $\bfSigma$ is called 
a \emph{$B$-pattern},
or the \emph{$B$-pattern}\index{$B$-pattern} of $\bfSigma$.
\end{itemize}
\end{defn}
We arbitrarily choose a distinguished
 vertex (the \emph{initial vertex})\index{initial!vertex} $t_0$ in $\bbT_n$. 
Any cluster pattern
 $\mathbf{\Sigma}=\{ \Sigma_t
=(\bfx_t,\bfy_t,B_t)\}_{t\in \bbT_n }$ is uniquely determined from
the \emph{initial seed}\index{initial!seed} $\Sigma_{t_0}$ at  the initial vertex $t_0$
by repeating mutations along $\bbT_n$.

In the above, if $B_t$ is skew-symmetric for some $t$, it is skew-symmetric for any $t$.
Such a cluster pattern, $Y$-pattern, or a $B$-pattern is said to be \emph{skew-symmetric}\index{skew-symmetric!pattern}.
 Similarly, if $B_t$ is nonsingular for some $t$, it is nonsingular for any $t$.
Such a cluster pattern, $Y$-pattern, or a $B$-pattern is said to be \emph{nonsingular}\index{nonsingular (pattern)}.

For  a seed $\Sigma_t=(\bfx_t,\bfy_t,B_t)$
in a cluster pattern $\mathbf{\Sigma}$,
we use the notation
\begin{align}
\label{2eq:xyt1}
\bfx_t=(x_{1;t},\dots,x_{n;t}),
\quad
\bfy_t=(y_{1;t},\dots,y_{n;t}),
\quad
B_t=(b_{ij;t})_{i,j=1}^n.
\end{align}
For the initial seed $\Sigma_{t_0}=(\bfx_{t_0},\bfy_{t_0},B_{t_0})$, we often drop the indices $t_0$ as
\begin{align}
\label{2eq:xyinit1}
\bfx_{t_0}=\bfx=(x_{1},\dots,x_{n}),
\quad
\bfy_{t_0}=\bfy=(y_{1},\dots,y_{n}),
\quad
B_{t_0}=B=(b_{ij})_{i,j=1}^n,
\end{align}
if there is no confusion.
We also use similar notations for $\hat{y}$-variables $\hat{\bfy}_t
=(\hat{y}_{1;t}, \dots, \hat{y}_{n;t})$
and the initial
$\hat{y}$-variables $\hat{\bfy}_{t_0}=\hat{\bfy}
=(\hat{y}_1,\dots,\hat{y}_n)$.

Every $x$-variable $x_{i;t}$ is regarded as
a rational function of  the initial
$x$-variables $\bfx$ by repeating the mutations \eqref{2eq:xmut1}.
The following stronger fact served as the guiding principle for the formulation of cluster algebras by Fomin-Zelevinsky \cite{Fomin02}.
\begin{thm}[{Laurent phenomenon \cite[Thm.~3.1]{Fomin02}}]\index{Laurent!phenomenon}
\label{thm:Laurent1}
Every $x$-variable $x_{i;t}$ is expressed as a Laurent polynomial in $\bfx$ with
  coefficients in $\bbZ\bbP$.
\end{thm}

Now, we give the main notion in cluster algebra theory.

\begin{defn}[Cluster algebra]
For a cluster pattern $\mathbf{\Sigma}$,
the associated \emph{cluster algebra}\index{cluster!algebra} $\calA(\bfSigma)$
is the $\bbZ\bbP$-subalgebra of the ambient field $\calF$
generated by all $x$-variables $x_{i,t}$ ($i=1,\, \dots,\, n;\, t\in \bbT_n$).
\end{defn}

We vaguely refer to the totality of the objects and notions related to cluster algebras and/or  cluster patterns as
 \emph{cluster algebraic structure}.
In this text, we focus on cluster patterns rather than cluster algebras.
Moreover, we mainly work with $Y$-patterns.
In fact, for the most part, we do not use $x$-variables.
Nevertheless, we include $x$-variables in the presentation here,
because 
it provides a better perspective on how  cluster algebra theory
is related to the dilogarithm.

\begin{defn}
\label{defn:finite1}
We say that a cluster pattern is \emph{of finite type}\index{finite type!cluster pattern}
if there are only finitely many distinct seeds.
We say that a $Y$-pattern is \emph{of finite type}\index{finite type!$Y$-pattern}
if it is the $Y$-pattern of a cluster pattern of finite type.
\end{defn}

\begin{figure}
 \begin{center}
 \begin{tikzpicture}[scale=1]
 \draw (0,0) circle (2pt);
 \draw (0.7,0) circle (2pt);
 \draw (2.8,0) circle (2pt);
 \draw (3.5,0) circle (2pt);
 \draw (0.08,0) --(0.62,0);
 \draw (0.78,0) --(1.32,0);
 \draw (2.18,0) --(2.72,0);
 \draw (2.8,0.07) --(3.5,0.07);
 \draw (2.8,-0.07) --(3.5,-0.07);
 \draw (3.05,0.2) --(3.25,0);
 \draw (3.05,-0.2) --(3.25,0);
 \draw [fill] (1.55,0) circle (0.5pt);
 \draw [fill] (1.7,0) circle (0.5pt);
 \draw [fill] (1.85,0) circle (0.5pt);
 \draw (0,-0.4) node {\small $1$};
 \draw (0.7,-0.4) node {\small $2$};
 \draw (2.8,-0.4) node {\small \rp$r-1$};
 \draw (3.5,-0.4) node {\small \rp$r$};
 \draw (-0.9,0) node {$B_{\rp r}$};
 \draw (6,0) circle (2pt);
 \draw (6.7,0) circle (2pt);
 \draw (7.4,0) circle (2pt);
 \draw (8.1,0) circle (2pt);
 \draw (6.08,0) --(6.62,0);
 \draw (7.48,0) --(8.02,0);
  \draw (6.7,0.07) --(7.4,0.07);
 \draw (6.7,-0.07) --(7.4,-0.07);
 \draw (6.95,0.2) --(7.15,0);
 \draw (6.95,-0.2) --(7.15,0);
 \draw (6,-0.4) node {\small $1$};
 \draw (6.7,-0.4) node {\small $2$};
 \draw (7.4,-0.4) node {\small $3$};
 \draw (8.1,-0.4) node {\small $4$};
 \draw (5.1,0) node {$F_4$};
 \draw (0,-1.5) circle (2pt);
 \draw (0.7,-1.5) circle (2pt);
 \draw (2.8,-1.5) circle (2pt);
 \draw (3.5,-1.5) circle (2pt);
 \draw (0.08,-1.5) --(0.62,-1.5);
 \draw (0.78,-1.5) --(1.32,-1.5);
 \draw (2.18,-1.5) --(2.72,-1.5);
 \draw (2.8,-1.43) --(3.5,-1.43);
 \draw (2.8,-1.57) --(3.5,-1.57);
 \draw (3.25,-1.3) --(3.05,-1.5);
 \draw (3.25,-1.7) --(3.05,-1.5);
 \draw [fill] (1.55,-1.5) circle (0.5pt);
 \draw [fill] (1.7,-1.5) circle (0.5pt);
 \draw [fill] (1.85,-1.5) circle (0.5pt);
 \draw (0,-1.9) node {\small $1$};
 \draw (0.7,-1.9) node {\small $2$};
 \draw (2.8,-1.9) node {\small \rp$r-1$};
 \draw (3.5,-1.9) node {\small \rp$r$};
 \draw (-0.9,-1.5) node {$C_{\rp r}$};
 \draw (6,-1.5) circle (2pt);
 \draw (6.7,-1.5) circle (2pt);
 \draw (6.08,-1.5) --(6.62,-1.5);
  \draw (6,-1.43) --(6.7,-1.43);
 \draw (6,-1.57) --(6.7,-1.57);
 \draw (6.25,-1.3) --(6.45,-1.5);
 \draw (6.25,-1.7) --(6.45,-1.5);
 \draw (6,-1.9) node {\small $1$};
 \draw (6.7,-1.9) node {\small $2$};
 \draw (5.1,-1.5) node {$G_2$};
 \end{tikzpicture}
 \end{center}
\vskip-10pt
\caption{Nonsimply-laced Dynkin diagrams of finite type}
\label{fig:Dynkin2}
\end{figure}
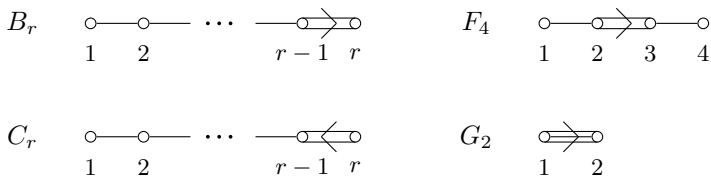

For any skew-symmetrizable integer matrix $B=(b_{ij})$,
we define a symmetrizable integer matrix $A(B)=(a_{ij})$ of the same size
by
\begin{align}
a_{ij}=
\begin{cases}
2 & i= j,
\\
-|b_{ij}| & i\neq j.
\end{cases}
\end{align}
The matrix $A(B)$ is called the \emph{Cartan counterpart}\index{Cartan counterpart} of $B$.
It is well known that the (indecomposable) Cartan matrices of finite type
are classified by the Dynkin diagrams of types $A_{\rp r}$, $B_{\rp r}$, $C_{\rp r}$, $D_{\rp r}$,
$E_6$, $E_7$, $E_8$, $F_4$, and $G_2$.
The simply-laced ones (types $ADE$) already appeared in Figure \ref{fig:Dynkin1}.
So, we only present the nonsimply-laced ones in Figure \ref{fig:Dynkin2}.
The following theorem shows a strong similarity
between cluster algebras and root systems.

\begin{thm}
[Finite type classification {\cite[Thm.~1.4]{Fomin03a}}]
\label{thm:finite1}\index{finite type!classification}
Suppose that the exchange matrices  of a cluster pattern $\bfSigma$ 
are indecomposable.
Then,
a cluster pattern $\bfSigma$ is of finite type 
if and only if it contains an exchange matrix $B_t$ such that
its Cartan counterpart  $A(B_t)$ is a Cartan matrix of finite type.
Moreover,
 the Dynkin type of $A(B_t)$ is uniquely determined from the underlying $B$-pattern.
\end{thm}

\section{Examples: Cluster patterns of types $A_1$ and $A_2$}

Let us present the two simplest examples of cluster patterns.

\begin{ex}[{Type $A_1$}]
Consider the simplest example with $n=1$.
The initial seed is given by $\Sigma_{t_0}=\Sigma=(x_1,y_1, B=(0))$.
The Cartan counterpart $A(B)=(2)$ of $B$ is the Cartan matrix of type $A_1$.
So, the corresponding cluster pattern $\bfSigma$  is called a cluster pattern of type $A_1$.
The only  seed in $\bfSigma$ other than $\Sigma$ is $\Sigma'=\mu_1(\Sigma)
=(x'_1,y'_1, B'=(0))$,
where $\haty_1= y_1$, and 
\begin{align}
x'_1 = \frac{1+y_1}{1\oplus y_1}x_1^{-1},
\quad
y'_1 = y_1^{-1}.
\end{align}
\end{ex}

\begin{ex}[{Type $A_2$}]
\label{ex:typeA22}
Consider the next simplest example with $n=2$.
We consider the following arrangement of
a cluster pattern on $\bbT_2$.
\begin{align}
\label{eq:seedp1}
\cdots
\
\overunder{2}
\
\Sigma_{t_{-2}}
\
\overunder{1}
\
\Sigma_{t_{-1}}
\
\overunder{2}
\
\Sigma_{t_{0}}
\
\overunder{1}
\
\Sigma_{t_{1}}
\
\overunder{2}
\
\Sigma_{t_{2}}
\
\overunder{1}
\
\cdots.
\end{align}
Let $t_0$ be the initial vertex. Below we use simplified notations such as $
\Sigma_{t_s}=\Sigma_s$, $B_{t_s}=B_s$, $\bfx_{t_s}=\bfx_s$, $ \bfy_{t_s}=\bfy_s$, etc.
We consider the
 initial seed $\Sigma_{0}=\Sigma=(\bfx,\bfy, B)$ with
\begin{align}
\label{eq:BA2}
B=\begin{pmatrix}
0 & -1
\\
1 & 0
\end{pmatrix}
,
\quad
A(B)=\begin{pmatrix}
2 & -1
\\
-1 & 2
\end{pmatrix}
.
\end{align}
The matrix $A(B)$ is the Cartan matrix of type $A_2$.
Thus, the corresponding cluster pattern $\bfSigma$ is called a cluster pattern of type $A_2$.
The initial $\hat{y}$-variables
\eqref{2eq:yhat1} are given by
\begin{align}
\label{eq:yhat2}
\hat{y}_1=y_1 x_2,
\quad
\hat{y}_2=y_2 x_1^{-1}.
\end{align}
By the mutations \eqref{2eq:bmut1},
the exchange matrices are given by
\begin{align}
\label{eq:Bs1}
B_{s}=
\begin{cases}
B & \mbox{$s$: even},\\
-B & \mbox{$s$: odd}.\\
\end{cases}
\end{align}
Accordingly, we have, 
for even $s$,
\begin{gather}
\begin{cases}
\displaystyle
x_{1;{s+1}}=x_{1;s}^{-1}\frac{ 1+\hat{y}_{1;s}}{1\oplus y_{1;s}},
\\
x_{2;{s+1}}=x_{s},
\end{cases}
\quad
\begin{cases}
y_{1;{s+1}}=y_{1;s}^{-1},
\\
y_{2;{s+1}}=y_{2;s} (1\oplus y_{1;s}),
\end{cases}
\end{gather}
and,
for odd $s$,
\begin{gather}
\label{1eq:sodd1}
\begin{cases}
x_{1;{s+1}}=x_{1;s},
\\
\displaystyle
x_{2;{s+1}}=x_{2;s}^{-1}\frac{ 1+\hat{y}_{2;s}}{1\oplus y_{2;s}},
\end{cases}
\quad
\begin{cases}
y_{1;{s+1}}=y_{1;s} (1\oplus y_{2;s}),
\\
y_{2;{s+1}}=y_{2;s}^{-1}.
\end{cases}
\end{gather}
Let us calculate the mutations of $x$- and $y$-variables.
To make the effort minimal,
we first calculate $y$-variables, 
then  calculate $x$-variables
with the help of \eqref{2eq:yhatmut1} and \eqref{eq:yhat2}.
Then,
we have the following result: (If you have never calculated it by yourself before,
we strongly recommend doing it now!)
\begin{alignat}{3}
\label{eq:A2mut1}
&
\begin{cases}
\displaystyle
 x_{1;1}=x_1^{-1}\frac{1+ \hat{y}_1}{1\oplus y_1},
\\
  x_{2;1}=x_2,
  \end{cases}
  &\quad
  &
  \begin{cases}
 y_{1;1}=y_1^{-1},\\ 
   y_{2;1}=y_2 (1\oplus y_1),
 \end{cases}
 \\
 &
\begin{cases}
 \displaystyle
 x_{1;2}=x_1^{-1}\frac{1+ \hat{y}_1}{1\oplus y_1},
\\
 \displaystyle
  x_{2;2}=x_2^{-1}\frac{1+ \hat{y}_2+\hat{y}_1\hat{y}_2}{1\oplus y_2\oplus y_1y_2},
  \end{cases}
  &\quad
  &
  \begin{cases}
 y_{1;2}=y_1^{-1}(1\oplus y_2\oplus y_1y_2),\\ 
   y_{2;2}=y_2^{-1} (1\oplus y_1)^{-1},
 \end{cases}
 \\
  &
\begin{cases}
 \displaystyle
 x_{1;3}=x_1 x_2^{-1}\frac{1+ \hat{y}_2}{1\oplus y_2},
\\
 \displaystyle
  x_{2;3}=x_2^{-1}\frac{1+ \hat{y}_2+\hat{y}_1\hat{y}_2}{1\oplus y_2\oplus y_1y_2},
  \end{cases}
  &\quad
  &
  \begin{cases}
 y_{1;3}=y_1(1\oplus y_2\oplus y_1y_2)^{-1},\\ 
   y_{2;3}=y_1^{-1}y_2^{-1} (1\oplus y_2),
 \end{cases}
 \\
  &
\begin{cases}
 \displaystyle
 x_{1;4}=x_1 x_2^{-1} \frac{1+ \hat{y}_2}{1\oplus y_2},
\\
 \displaystyle
  x_{2;4}=x_1,
  \end{cases}
  &\quad
  &
  \begin{cases}
 y_{1;4}=y_2^{-1},\\ 
   y_{2;4}=y_1y_2 (1\oplus y_2)^{-1},
 \end{cases}
  \\
  \label{eq:A2mut2}
  &
\begin{cases}
 \displaystyle
 x_{1;5}=x_2,
\\
 \displaystyle
  x_{2;5}=x_1,
  \end{cases}
  &\quad
  &
  \begin{cases}
 y_{1;5}=y_2,\\ 
   y_{2;5}=y_1.
 \end{cases}
 \end{alignat}
 We observe the following properties in the above.
 \begin{itemize}
 \item
 (Laurent phenomenon).
 As claimed in Theorem \ref{thm:Laurent1},
 all $x$-variables are expressed as Laurent polynomials in the initial $x$-variables $\bfx$
 due to the systematic reduction of rational functions starting at $s=3$.
 \item 
 (Periodicity/Synchronicity).
 Even though  $x$- and $y$-variables
 mutate differently, they look closely related.
 In particular,
 they enjoy the same (half) periodicity at $s=5$.
Moreover, together with the result \eqref{eq:Bs1},
we have the following (half) periodicity of seeds
\begin{align}
\label{eq:pentp1}
\Sigma_{5}=\tau_{12}\Sigma_0,
\end{align}
where $\tau_{12}$ is the transposition of 1 and 2.
This periodicity is interpreted as the periodicity of the triangulations of a pentagon
under the alternating flips of diagonals \cite[\S12.3]{Fomin03a}.
Thus, we call it the \emph{pentagon periodicity}\index{pentagon!periodicity} of the cluster pattern $\bfSigma$.
\end{itemize}
\end{ex}

\section{Separation formulas and tropicalization}
\label{sec:separation1}

In \cite{Fomin07} it was clarified that
there is an important structure in seeds that is compatible with the tropicalization.

\subsection{Separation formulas}
We first introduce  related notions
called $C$-matrices, $G$-matrices, and $F$-polynomials.
Let $\mathbf{\Sigma}$ be any cluster pattern,
and let $\mathbf{B}$ be the $B$-pattern of $\mathbf{\Sigma}$.
Let $t_0\in \mathbb{T}_n$ be a given initial vertex.

We  introduce a collection of 
matrices $\bfC=\bfC^{t_0}=\{ C_t\}_{ t\in \bbT_n}$
called a \emph{$C$-pattern}\index{$C$-pattern},
which is uniquely determined  from  $\mathbf{B}$ and $t_0$, therefore, eventually only from $t_0$ and $B_{t_0}$.

\begin{defn}[$C$-matrices]
\label{2defn:Cmat1}
\index{$C$-matrix}
The \emph{$C$-matrices $C_t=(c_{ij;t})_{i,j=1}^n$ $(t\in \bbT_n)$
of a cluster pattern $\mathbf{\Sigma}$ with an initial vertex $t_0$}
are $n\times n$
integer matrices 
uniquely determined 
by the following initial condition
and the mutation rule:
\begin{align}
\label{2eq:Cmat1}
C_{t_0}&=I,
\\
 \label{2eq:cmutmat2}
 c_{ij;t'}&=
 \begin{cases}
 -c_{ik;t}
 &
 j= k,
 \\
 c_{ij;t} + c_{ik;t} [b_{kj;t}]_+
 + [-c_{ik;t}]_+ b_{kj;t}
 &
 j \neq k,
 \end{cases}
 \end{align}
where $t$ and $t'$ are $k$-adjacent.
Each column vector $\bfc_{i;t}=(c_{ji;t})_{j=1}^n$ of a matrix $C_t$ is called
a \emph{$c$-vector}\index{$c$-vector}.
\end{defn}

The mutation \eqref{2eq:cmutmat2}, in particular, implies
\begin{align}
\label{2eq:ck1}
\bfc_{k;t'}=-\bfc_{k;t}.
\end{align}
One can easily check that the mutation \eqref{2eq:cmutmat2} is involutive,
using \eqref{1eq:pos1}, \eqref{2eq:bmut1}, and \eqref{2eq:ck1}.

Under the same assumption of $C$-matrices,
we  introduce another collection of 
matrices $\bfG=\bfG^{t_0}=\{ G_t\}_{ t\in \bbT_n}$
called a \emph{$G$-pattern}\index{$G$-pattern},
which is  uniquely determined  from $\mathbf{B}$, $t_0$,
and $\bfC^{t_0}$ defined above,
therefore, eventually only from $t_0$ and $B_{t_0}$.

\begin{defn}[$G$-matrices]
\label{defn:Gmat1}
The \emph{$G$-matrices\index{$G$-matrix} $G_t=(g_{ij;t})_{i,j=1}^n$ $(t\in \bbT_n)$
of a cluster pattern $\mathbf{\Sigma}$  with an initial vertex $t_0$}
are $n\times n$
integer matrices  uniquely determined 
by the following initial condition
and the mutation rule:
\begin{align}
\label{2eq:Gmat1}
G_{t_0}&=I,
\\
\label{2eq:gmut1}
 g_{ij;t'}&=
 \begin{cases}
 \displaystyle
 -g_{ik;t}
 + \sum_{\ell=1}^n g_{i\ell;t} [-b_{\ell k;t}]_+
 -  \sum_{\ell=1}^nb_{i\ell;t_0}  [-c_{\ell k;t}]_+ 
 &
 j= k,
 \\
 g_{ij;t}  &
 j \neq k,
 \end{cases}
 \end{align}
where $t$ and $t'$ are $k$-adjacent.
Each column vector $\bfg_{i;t}=(g_{ji;t})_{j=1}^n$ of a matrix $G_t$ is called
a \emph{$g$-vector}\index{$g$-vector}.
\end{defn}

The following relation
between $C$- and $G$-matrices
can be  proved by 
 the induction on $t$ along $\bbT_n$ from $t_0$:
 (\emph{the first duality}) \index{duality!first}\index{first duality}
\begin{align}
\label{2eq:dual0}
G_t B_t = B_{t_0} C_t.
\end{align}
Using this relation, 
one can   show that the mutation \eqref{2eq:gmut1} is involutive.

Let 
$\mathbf{y}=(y_1,\dots,y_n)$  be an $n$-tuple of  variables.
(They are just formal variables at this moment,
but they will be related with the initial $y$-variables $\mathbf{y}_{t_0}$
of $\mathbf{\Sigma}$ soon. So, the notation is compatible.)
Under the same assumption of $C$-matrices,
we  introduce a collection of 
rational functions  in $\bfy$ having subtraction-free expressions
 $\bfF=\bfF^{t_0}=\{ F_{i;t}(\bfy) \}_{ i=1,\, \dots,\, n;\, t\in \bbT_n}$
 called an \emph{$F$-pattern}\index{$F$-pattern},
 which is  uniquely determined  from $\mathbf{B}$, $t_0$,
and $\bfC^{t_0}$, therefore, eventually only from $t_0$ and $B_{t_0}$.

\begin{defn}[$F$-polynomials]
\label{2defn:Fpoly1}
\index{$F$-polynomial}
The \emph{$F$-polynomials $F_{i;t}(\bfy)\in \bbQ_{\rmsf}(\bfy)$ $(i=1,\, \dots,\, n;\, t\in \bbT_n)$
of a cluster pattern $\mathbf{\Sigma}$  with an initial vertex $t_0$}
 are uniquely determined 
by the following initial condition
and the mutation rule:
\begin{align}
\label{2eq:Finit1}
F_{i;t_0}(\bfy)&=1,
\\
 \label{2eq:Fmut1}
 F_{i;t'}(\bfy)&=
 \begin{cases}
\frac
 {
  \displaystyle
 M_{k;t}(\bfy)
 }
{ \displaystyle
 F_{k;t}(\bfy)
 }
   &
 i= k,
 \\
 F_{i;t}(\bfy)  &
 i \neq k,
 \end{cases}
 \end{align}
where $t$ and $t'$ are $k$-adjacent,
and $M_{k,t}(\bfy)$ is a rational function of $\bfy$ (actually a polynomial in $\bfy$ as we see below)
defined by 
\begin{align}
\label{2eq:M1}
 M_{k;t}(\bfy)
 &=
    \prod_{j=1}^{n}
  y_j^{[c_{jk;t}]_+}
    \prod_{j=1}^{n}
  F_{j;t}(\bfy)^{[b_{jk;t}]_+}
+
    \prod_{j=1}^{n}
  y_j^{[-c_{jk;t}]_+}
    \prod_{j=1}^{n}
  F_{j;t}(\bfy)^{[-b_{jk;t}]_+}.
\end{align}
\end{defn}

We have
\begin{align}
 M_{k;t}(\bfy)
 =
  M_{k;t'}(\bfy).
\end{align}
Therefore, 
 the mutation 
\eqref{2eq:Fmut1} is involutive.

As the name suggests, rational functions $F_{i;t}(\bfy)
\in \bbQ_{\rmsf}(\bfy)$ 
are reduced to polynomials in $\bfy$.
This  is essentially a restatement of the Laurent phenomenon in
Theorem \ref{thm:Laurent1}.

\begin{prop}[{\cite[Prop.~3.6]{Fomin07}}]
\label{2prop:Fpoly1}
Any $F$-polynomial $F_{i;t}(\bfy)$ is a polynomial  in $\bfy$
with coefficients in $\bbZ$.
\end{prop}

We stress that $C$- and  $G$-matrices and $F$-polynomials
depend not only on $t$ but also on the choice of the initial vertex $t_0$.

Now, we are ready to present key formulas in cluster algebra theory.

\begin{thm}[{Separation formulas \cite[Prop. 3.15, Cor. 6.13]{Fomin07}}]
\index{separation formula}
\label{2thm:sep1}
Let $\bfx$, $\bfy$, $\hat\bfy$ be the
initial variables at $t_0$.
Then, the following formulas hold.
\begin{align}
\label{eq:sep1}
x_{i;t}&=
\biggl(\,
\prod_{j=1}^n
x_j^{g_{ji;t}}
\biggr)
\frac{
F_{i;t}(\hat{\bfy})
}
{
F_{i;t}{\vert _{\bbP}}({\bfy})
}
,\\
\label{eq:sep2}
y_{i;t}&=
\biggl(\,
\prod_{j=1}^n
y_j^{c_{ji;t}}
\biggr)
\prod_{j=1}^n
F_{j;t}{\vert _{\bbP}}(\bfy)^{b_{ji;t}},
\end{align}
where $F_{i;t}{\vert _{\bbP}}({\bfy})\in \bbP$
is obtained from the $F$-polynomial $F_{i;t}(\bfy)$
 by replacing $+$ in
a subtraction-free expression of
$F_{i;t}({\bfy})$ with $\oplus$ for $\bbP$.
\end{thm}

Thanks to the above formulas, the study of cluster patterns is reduced to
the study of $C$- and $G$-matrices, and $F$-polynomials.

\begin{ex}[Type $A_2$]
\label{ex:typeA23}
Let us consider the cluster pattern of type $A_2$
in Example \ref{ex:typeA22}.
By comparing the result therein with the separation formulas 
\eqref{eq:sep1} and \eqref{eq:sep2},
one can quickly read off $C$- and $G$-matrices and $F$-polynomials
 as follows:
\begin{alignat}{5}
  \label{eq:A2mut3}
C_{0}&=
\begin{pmatrix}
1 & 0\\
0 & 1
\end{pmatrix},
\
&
G_{0}&=
\begin{pmatrix}
1 & 0\\
0 & 1
\end{pmatrix},
\
&
&
\begin{cases}
F_{1;0}(\bfy)=1,\\
F_{2;0}(\bfy)=1,\\
\end{cases}
\\
C_{1}&=
\begin{pmatrix}
-1 & 0\\
0 & 1
\end{pmatrix},
\
&
G_{1}&=
\begin{pmatrix}
-1 & 0\\
0 & 1
\end{pmatrix},
\
&&
\begin{cases}
F_{1;1}(\bfy)=1+y_1,\\
F_{2;1}(\bfy)=1,\\
\end{cases}
\\
C_{2}&=
\begin{pmatrix}
-1 & 0\\
0 & -1
\end{pmatrix},
\quad
&
G_{2}&=
\begin{pmatrix}
-1 & 0\\
0 & -1
\end{pmatrix},
\quad
&&
\begin{cases}
F_{1;2}(\bfy)=1+y_1,\\
F_{2;2}(\bfy)=1+y_2+y_1y_2,
\end{cases}
\\
C_{3}&=
\begin{pmatrix}
1 & -1\\
0 & -1
\end{pmatrix},
\
&
G_{3}&=
\begin{pmatrix}
1 & 0\\
-1 & -1
\end{pmatrix},
\
&&
\begin{cases}
F_{1;3}(\bfy)=1+y_2,\\
F_{2;3}(\bfy)=1+y_2+y_1y_2,\\
\end{cases}
\\
C_{4}&=
\begin{pmatrix}
0 & 1\\
-1 & 1
\end{pmatrix},
\
&
G_{4}&=
\begin{pmatrix}
1 & 1\\
-1 & 0
\end{pmatrix},
\
&&
\begin{cases}
F_{1;4}(\bfy)=1+y_2,\\
F_{2;4}(\bfy)=1,\\
\end{cases}
\\
  \label{eq:A2mut4}
C_{5}&=
\begin{pmatrix}
0 & 1\\
1 & 0
\end{pmatrix},
\
&
G_{5}&=
\begin{pmatrix}
0 & 1\\
1 & 0
\end{pmatrix},
\
&&
\begin{cases}
F_{1;5}(\bfy)=1,\\
F_{2;5}(\bfy)=1.\\
\end{cases}
\end{alignat}
 One can observe the following nontrivial properties in the above results.
 \begin{enumerate}
 \item (Unit constant property). Every $F$-polynomial has
 a constant term 1.
 \item (Laurent positivity). Every $F$-polynomial has
 no negative coefficient.
 \item (Sign-coherence of $c$-vectors).
 Every $c$-vector is a nonzero vector,
 and its nonzero components have the same sign.
\item (Second duality).
The following relation between $C$- and $G$-matrices holds:
\begin{align}
\label{1eq:dual1}
G_t^T C_t=I,
\end{align}
where $M^T$ is the transpose of a matrix $M$.
(The general form is given later in \eqref{eq:dual2}.)
\end{enumerate}
\end{ex}

\subsection{Tropicalization}
\label{sec:tropicalization1}
The
 \emph{tropicalization}\index{tropicalization} is a new and common idea in recent mathematics.
 See \cite{Speyer04} for the background of the notion and the terminology.
Remarkably, the tropicalization is naturally built into cluster algebra theory.

Let us
consider  the semifields $\mathbb{Q}_{\mathrm{sf}}(\bfy)$ and
$\mathrm{Trop}(\bfy)$ with $\bfy=(y_1,\dots,y_n)$ as common  variables.
We have a unique semifield homomorphism
\begin{align}
\label{1eq:trophom1}
\pi_{\mathrm{trop}}: \mathbb{Q}_{\mathrm{sf}}(\bfy)
\rightarrow \mathrm{Trop}(\bfy)
\end{align}
such that $\pi_{\mathrm{trop}}(y_i)=y_i$ for any $i=1$, \dots, $n$.
We call it the \emph{tropicalization (homomorphism)}\index{tropicalization!homomorphism}.
For example, for $\bfy=(y_1,y_2,y_3)$,
\begin{align}
\pi_{\mathrm{trop}}
\left(
\frac{3y_1y_2^2y_3^2 + 2y_1^2y_2y_3}{3y_2^2 + y_1^2 y_2^2+
y_1y_2^3y_3}
\right)
=
\frac{y_1y_2y_3}{y_2^2}=y_1y_2^{-1}y_3.
\end{align}
Roughly speaking, it extracts the ``leading monomial" of 
 $f(\bfy)\in \mathbb{Q}_{\mathrm{sf}}(\bfy)$.

To apply the tropicalization,
we especially consider a cluster pattern whose coefficient semifield
 is $\bbP=\mathbb{Q}_{\mathrm{sf}}(\bfy)$,
 $\bfy=(y_1,\dots,y_n)$,
and its initial $y$-variables $\bfy_{t_0}$ 
coincide with $\bfy$.
We call such coefficients  \emph{free coefficients}\index{free!coefficient}
or \emph{free $y$-variables}\index{free!$y$-variable} at $t_0$.
We also call its $Y$-pattern a \emph{free $Y$-pattern}\index{free!$Y$-pattern}.
(This terminology is not common, and it is used in \cite{Nakanishi22a}.)
Each free $y$-variable $y_{i;t}$ is  an
element of $\mathbb{Q}_{\mathrm{sf}}(\bfy)$.
Thus, we may apply the tropicalization 
$\pi_{\mathrm{trop}}$  in \eqref{1eq:trophom1} to $y_{i;t}$.
Then, we have the following result.
\begin{prop}[{\cite[Prop.~5.2]{Fomin07}}]
\label{prop:tropF1}
We have
\begin{align}
\label{eq:ty1}
\pi_{\mathrm{trop}}(y_{i;t}) & 
=
{\rp y^{\bfc_{i;t}}:=}
\prod_{j=1}^n
y_j^{c_{ji;t}},
\\
\label{eq:tF1}
\pi_{\mathrm{trop}}(F_{i;t}(\bfy)) & = 1.
\end{align}
\end{prop}
Namely, in the  separation formula \eqref{eq:sep2},
the monomial  with powers given by the $c$-vector $\bfc_{i;t}$ is
the \emph{tropical part}\index{tropical part!of $y$-variable} of the free $y$-variable  $y_{i;t}$,
while the product of $F$-polynomials is the \emph{nontropical part}\index{nontropical part!of $y$-variable}.
In other words, the tropicalization
is the operation of ``forgetting $F$-polynomials" in  \eqref{eq:sep2}.

For a cluster pattern 
whose coefficient semifield
 is $\bbP=\rm{Trop}(\bfy)$ and $\bfy_{t_0}=\bfy$,
 the coefficients are called \emph{principal coefficients}\index{principal!coefficients} at $t_0$.
 As we have seen, the principal coefficients at $t_0$ are obtained by applying 
 the tropicalization 
to the  {free coefficients} at $t_0$.
So, we also call them  \emph{tropical $y$-variables}\index{tropical!$y$-variable}.

As for the separation formula \eqref{eq:sep1} for $x$-variables,
 the above interpretation of the tropicalization \emph{with respect to
the initial $x$-variables $\bfx$} does not hold
because the $F$-polynomial $F_{i;t}(\hat{\bfy})$  therein is
a polynomial in $\hat{\bfy}$, not in $\bfx$.
Nevertheless, in view of the duality structure between the separation  formulas
\eqref{eq:sep1} and \eqref{eq:sep2},
we regard 
the monomial 
\begin{align}
{\rp x^{\bfg_{i;t}}:=}
\prod_{j=1}^n
x_j^{g_{ji;t}}
\end{align}
as
the \emph{tropical part} of the $x$-variable  $x_{i;t}$\index{tropical part!of $x$-variable},
in the sense that  it is formally obtained by ``forgetting $F$-polynomials (= setting
$F$-polynomials to $1$)" in \eqref{eq:sep1}.
Correspondingly, we regard the factor $
F_{i;t}(\hat{\bfy})/
F_{i;t}{\vert _{\bbP}}({\bfy})
$ in \eqref{eq:sep1}
the \emph{nontropical part} of the $x$-variable  $x_{i;t}$\index{nontropical part!of $x$-variable}.

\section{Advanced results}
\label{sec:advanced1}

The results below are the consequences of
Theorems \ref{thm:sign1} and \ref{thm:Lpos1}.
Both theorems had been important and long-standing conjectures
until they were proved by Gross-Hacking-Keel-Kontsevich \cite{Gross14}
by the scattering diagram method,
which we explain later in Chapter \ref{ch:CSD1}.
Thus, they are regarded as advanced results.

\subsection{Sign-coherence of $c$-vectors and duality}
\label{sec:sign1}

We continue considering the situation for Theorem \ref{2thm:sep1}.
\begin{defn}
We say that a vector $v\in \bbZ^n$ is \emph{positive} (resp. \emph{negative})\index{positive!(vector)}\index{negative (vector)}
if it is a nonzero vector and all nonzero components are positive (resp.~negative).
\end{defn}

The following theorem was conjectured by Fomin-Zelevinsky \cite{Fomin07}
and proved  by Gross-Hacking-Keel-Kontsevich \cite{Gross14}
in full generality
after several affirmative partial results
(e.g., \cite{Derksen10,Plamondon10b,Nagao10}).

\begin{thm}[{Sign-coherence of $c$-vectors \cite[Cor.~5.5]{Gross14}}]
\label{thm:sign1}\index{sign-coherence}
Every $c$-vector $c_{i;t}$ is either positive or negative.
\end{thm}

This is one of the cornerstones 
of the cluster algebra theory
because many nontrivial and important results follow from it.
Here, we limit ourselves to the ones that we will use in the text.

It was shown by  \cite{Fomin07} that the  theorem is
equivalent to the following one.
(In fact, \cite{Gross14} proved Theorem \ref{thm:Fpol1} instead of Theorem \ref{thm:sign1})

\begin{thm}[{Unit constant property \cite[Cor.~5.5]{Gross14}}]
\label{thm:Fpol1}\index{unit constant property}
Every $F$-polynomial $F_{i;t}(\bfy)$ has constant term 1.
\end{thm}

Note that the property \eqref{eq:tF1} is weaker than  Theorem \ref{thm:Fpol1}.
For example, we have $\pi_{\trop}(y_1+y_2)=1$.

Thanks to Theorem \ref{thm:sign1},
the following notion is well-defined.

\begin{defn}[Tropical sign]
\label{defn:tropsign1}\index{tropical!sign}
To each $c$-vector $c_{i;t}$, we assign the 
\emph{tropical sign} $\ve_{i;t}\in \{1, -1\}$ by
\begin{align}
\ve_{i;t}=
\begin{cases}
1 & \text{$c_{i;t}$ is positive},
\\
-1 & \text{$c_{i;t}$ is negative}.
\end{cases}
\end{align}
\end{defn}

We stress that the  sign $\ve_{i;t}$ depends  on
the choice of the initial vertex $t_0$.

The following duality between $C$- and $G$-matrices 
was shown by Nakanishi-Zelevinsky \cite{Nakanishi11a} under the assumption of  Theorem \ref{thm:sign1}.

\begin{thm}
[{\emph{Second duality} \cite[Thm.~1.2]{Nakanishi11a}}]
\label{thm:dual2}\index{second duality}\index{duality!second}
Let $D$ be a common skew-symmetrizer of $B_{t}$'s.
Then, the following equality holds:
\begin{align}
\label{eq:dual2}
D^{-1} G_t^T DC_t=I,
\end{align}
or, equivalently,
\begin{align}
\label{eq:dual3}
D^{-1} C_t^T DG_t=I.
\end{align}
\end{thm}
We have the following easy consequences of
Theorem \ref{thm:dual2}.
\begin{prop}
\label{prop:CG1}
(1) (Unimodularity \cite[Prop.~4.2]{Nakanishi11a})\index{unimodularity}.
\begin{align}
\label{eq:uni1}
|C_t|=|G_t|= \pm 1.
\end{align}
\par (2)
\cite[Eq.~(2.9)]{Nakanishi11a}).
The matrix $B_t$ is determined by $B_{t_0}$ and $C_t$
as follows:
\begin{align}
\label{eq:db1}
DB_t= C_t^T (DB_{t_0}) C_t.
\end{align}
\end{prop}

\subsection{Laurent positivity, detropicalization, and synchronicity}

The following theorem,
which is another cornerstone of the cluster algebra theory,
 was also conjectured by Fomin-Zelevinsky \cite{Fomin02,Fomin07}
and proved by Gross-Hacking-Keel-Kontsevich \cite{Gross14}
in full generality
after several affirmative partial results
(e.g., \cite{Musiker09,Kimura12,Lee15}).

\begin{thm}
[{{Laurent positivity \cite[Cor.~0.4]{Gross14}}}]
\label{thm:Lpos1}\index{Laurent!positivity}
Every $F$-poly\-no\-mi\-al $F_{i;t}(\bfy)$
has no negative coefficient.
\end{thm}

In accordance with the $S_n$-action on seeds in Definition \ref{defn:saction1},
we define the $S_n$-action on $C$- and $G$-matrices and $F$-polynomials as
 \begin{alignat}{3}
 \label{2eq:sigmax1}
C'&=\nu C_t, &\quad
c'_{ij}&=c_{i\nu^{-1}(j);t},\\
 \label{2eq:sigmag1}
G'&=\nu G_t, &\quad
g'_{ij}&=g_{i\nu^{-1}(j);t},\\
 \label{2eq:sigmaF1}
\bfF'&=\nu \bfF_t, &\quad
F'_{i}&=F_{\nu^{-1}(i);t}.
\end{alignat}
{\rp This is compatible with the separation formulas \eqref{eq:sep1} and \eqref{eq:sep2}.}
The following theorem was proved by Cao-Huang-Li \cite{Cao17} 
for $\nu=\rmid$ using
 Theorems \ref{thm:sign1}, \ref{thm:Fpol1}, \ref{thm:dual2},  \ref{thm:Lpos1}
and the separation formulas all together.
It was generalized to any permutation $\nu$ by Nakanishi
\cite{Nakanishi19}.

\begin{thm}
[{Detropicalization \cite[Thm.~2.5]{Cao17},
\cite[Thm.~4.7]{Nakanishi19}}]
\label{thm:detrop1}\index{detropicalization}
For  any  $t,\,t'\in \bbT_n$
 and any permutation $\nu\in S_n$,
 the following facts hold:
\begin{align}
\label{eq:detrop3}
G_{t}= \nu G_{t'} \quad & \Longrightarrow \quad
\bfF_{t}=\nu\bfF_{t'},\quad
 \bfx_t=\nu \bfx_{t'},
\\
\label{eq:detrop4}
C_{t}=\nu C_{t'} \quad & \Longrightarrow \quad
\bfF_{t}=\nu\bfF_{t'},
\quad \bfy_t=\nu \bfy_{t'}.
\end{align}
\end{thm}
In other words, the $\bfx_t$ and $\bfy_t$ are uniquely
determined by their tropical parts $G_t$ and $C_t$, respectively.

One can complete Theorem \ref{thm:detrop1}
as the following equivalences of periodicities.

\begin{thm}[{Synchronicity \cite[Thm.~5.3]{Nakanishi19}}]
\index{synchronicity}
\label{1thm:synchro1}
Let $\bfSigma$ be
 a cluster pattern with {free coefficients} at $t_0$.
 Then,
for  given  $t,\,t'\in \bbT_n$ and a given permutation $\nu\in S_n$,
the following five conditions are equivalent:
\begin{itemize}
\item[(a).] $G_{t}=\nu G_{t'}$.
\item[(b).] $C_{t}=\nu C_{t'}$.
\item[(c).] $\bfx_{t}=\nu \bfx_{t'}$.
\item[(d).] $\bfy_{t}=\nu \bfy_{t'}$.
\item[(e).] $\Sigma_{t}=\nu \Sigma_{t'}$.
\end{itemize}
\end{thm}

\begin{rem}
The assumption of \emph{free} coefficients guarantees the implication
from the condition (d) to the other conditions. The rest of the implications hold
for any coefficients. 
In particular,
the implication (a) $\Longleftrightarrow$ (c) implies that the periodicity of 
$x$-variables is independent of their coefficients.
\end{rem}

\notes

Here is a bibliographic guide for learning the basics of cluster algebra theory.
The foundation of cluster algebra theory was
established by a series of papers ``Cluster Algebras I--IV" by Fomin-Zelevinsky together with Berenstein \cite{Fomin02,Fomin03a,Berenstein05,Fomin07}.
Also, the papers on $Y$-systems \cite{Fomin03b}
and the quantum cluster algebras \cite{Berenstein05b}
 may be included in the series.
Other foundational works on cluster algebras and cluster varieties are given
by Fock-Goncharov \cite{Fock03,Fock03b,Fock05,Fock07}.
A  comprehensive treatment is given in the ongoing writing
\cite{Fomin16b}.
Reviews, monographs, and textbooks with different emphases are available 
(e.g., \cite{Carter06}, \cite{Keller08}, \cite{Gekhtman10}, \cite{Keller12}, \cite{Marsh13}, \cite{Williams12},
\cite{Plamondon16}, \cite{Glick17}, \cite{Nakanishi22a}).

Proofs of all results in this chapter  except for Theorem \ref{thm:finite1}  
are conveniently found in the companion monograph  \cite{Nakanishi22a}
as specified below.
\begin{itemize}
\item
Lemma \ref{lem:domain1}:
See  \cite[Prop.~I.1.11]{Nakanishi22a}.
\item
Proposition \ref{prop:yhat1}:
See  \cite[Prop.~I.2.6]{Nakanishi22a}.
\item
Eq.~\eqref{eq:comp1}:
See  \cite[Prop.~I.2.9]{Nakanishi22a}.

\item
Theorem \ref{thm:Laurent1}:
See  \cite[\S I.3.1]{Nakanishi22a}.
\item
Eq.~\eqref{2eq:dual0}:
See \cite[Props. I.4.10,
II.1.17]{Nakanishi22a}.
\item
Proposition \ref{2prop:Fpoly1}: See \cite[Thm.~I.3.23]{Nakanishi22a}.
\item
Theorem \ref{2thm:sep1} and Proposition \ref{prop:tropF1}:
See \cite[Thms. I.4.16, I.4.14]{Nakanishi22a}.
\item
Theorems \ref{thm:sign1} and \ref{thm:Fpol1}:
See \cite[\S II.3.2]{Nakanishi22a}.
\item
Theorem \ref{thm:dual2}:
See
\cite[Thm.~I.4.24]{Nakanishi22a}.

\item
Proposition
\ref{prop:CG1}:
See \cite[Props.~II.2.3, II.2.6]{Nakanishi22a}.

\item
Theorem \ref{thm:Lpos1}:
See \cite[\S II.6.4, \S III.7.3]{Nakanishi22a}.
\item
Theorem 
\ref{thm:detrop1}:
See \cite[Thm.~II.7.1]{Nakanishi22a}.
\item
Theorem
\ref{1thm:synchro1}:
The essential part of the proof   is the implication (c) $\Longrightarrow$ (a),
and it is found in \cite[Thm.~II.7.2]{Nakanishi22a}.
\end{itemize}
As for the proof of Theorem \ref{thm:finite1},
see \cite{Gekhtman10} or \cite[Ch.5]{Fomin16b}.

\chapter{Dilogarithm identities in $Y$-patterns}
\label{ch:dilogarithm1}

We present the dilogarithm identity associated with 
any period of a free $Y$-pattern.
This is the \emph{leitmotif} throughout the text,
and we will give several proofs and variations of the identity with
various methods, techniques, and perspectives.
Here, we give the first proof (actually two alternative proofs)
via the constancy condition by Frenkel-Szenes.
We call it the \emph{algebraic method}
because it is most directly related to the algebraic structure of
seeds and mutations in a cluster pattern compared with other methods.

\section{Periodicity}
\label{sec:periodicity1}

Let us introduce the central idea of the whole text, namely,
\emph{periodicity} in a cluster pattern and a $Y$-pattern.
Consider a cluster pattern $\bfSigma$ with \emph{free coefficients} at $t_0$ and its $Y$-pattern.
Consider sequences of mutations therein
\begin{align}
\label{eq:mseq1}
&\Sigma(0) 
\
{\buildrel {k_0} \over \rightarrow}
\
\Sigma(1) 
\
{\buildrel {k_1} \over \rightarrow}
\
\cdots
\
{\buildrel {k_{P-1}} \over \rightarrow}
\
\Sigma(P),
\\
\label{eq:mseq2}
&\Upsilon(0) 
\
{\buildrel {k_0} \over \rightarrow}
\
\Upsilon(1) 
\
{\buildrel {k_1} \over \rightarrow}
\
\cdots
\
{\buildrel {k_{P-1}} \over \rightarrow}
\
\Upsilon(P),
\end{align}
where 
$\buildrel {k} \over \rightarrow$ stands for the mutation in direction $k$.
The seed $\Sigma(0)$ is not necessarily the initial seed $\Sigma_{t_0}$.
We set
$\Sigma(s)=(\bfx(s), \bfy(s), B(s))$, 
$\Upsilon(s)=(\bfy(s),B(s))$.

\begin{defn}[Periodicity]
\label{defn:period1}
For a permutation $\nu\in S_n$,
we say that 
sequences of mutations \eqref{eq:mseq1} and \eqref{eq:mseq2} are
\emph{$\nu$-periodic}\index{$\nu$-periodic}
if
\begin{align}
\label{eq:sigmap1}
\Sigma(P)=\nu \Sigma(0),
\\
\label{eq:sigmap2}
\Upsilon(P)=\nu \Upsilon(0)
\end{align}
hold, respectively.
\end{defn}
Thanks to the synchronicity in Theorem \ref{1thm:synchro1},
two conditions \eqref{eq:sigmap1} and \eqref{eq:sigmap2}
are equivalent.
Thus, we can use the two notions interchangeably.
Below, we focus on the periodicity of  (free) $Y$-patterns.
Since there is no confusion with $+$ in the ambient field $\calF$,
we write the addition $\oplus$ in $\bbQ_{\rmsf}(\bfy)$ as $+$,
which is indeed the ordinary addition of rational functions by definition.

We give a necessary condition  for a permutation $\nu$
having
 the $\nu$-periodicity.
 
\begin{defn}
 We say that a permutation $\nu\in S_n$ is {\em compatible with 
 a diagonal matrix $D=\mathrm{diag}(d_1,\dots,d_n)$}\index{compatible!(for permutation)}
 if 
 \begin{align}
 \label{eq:comd1}
 d_{\nu(i)}=d_i 
 \quad (i=1,\, \dots,\  n).
 \end{align}
 \end{defn}

\begin{prop}[{\cite[Prop.~4.3]{Nakanishi16}}]
\label{prop:compat1}
Suppose that the sequence \eqref{eq:mseq2} is $\nu$-periodic.
Let $D$ be any common skew-symmetrizer for the
exchange matrices $B(s)$ of  \eqref{eq:mseq2}.
Then, $\nu$ is compatible with $D$.
\end{prop}
\begin{proof}
Without loss of generality, one may assume that $B(0)$ is decomposed into
a block diagonal form such that each block
is indecomposable.
By \eqref{2eq:bmut1},
mutations preserve the block diagonal form of exchange matrices.
Moreover,
by \eqref{2eq:ymut1},
mutations do not mix $y$-variables whose indices belong to different blocks.
On the other hand,
due to
 the $\nu$-periodicity,
we have
$y_{i}(P)
=y_{\nu^{-1}(i)}(0)$.
Thus,
$\nu$ only permutes the indices in the same block of $B(0)$.
Therefore, we may further assume that $B(0)$ is indecomposable.
Then, it is easy to see that its skew-symmetrizer $D$ is unique up to a positive multiplicative constant.
In particular, there is a unique minimal integer skew-symmetrizer
$D_0=\mathrm{diag}(d'_1,\dots,d'_n)$.
Due to the $\nu$-periodicity,
we have
$B(P)
=\nu B(0)$.
It follows that
 the matrix $\nu D_0=\mathrm{diag}(d'_{\nu^{-1}(1)},\dots,d'_{\nu^{-1}(n)})$
 is also the minimal integer skew-symmetrizer of $B(0)$.
Due to the uniqueness $\nu D_0=D_0$, we have $d'_{\nu^{-1}(i)}=d'_i$ for any $i$.
Therefore, $d_{\nu(i)}=d_i$ holds.
\end{proof}

\section{Tropical limit}

Let $\bfy(s)=(y_1(s),\dots,y_n(s))$.
Then, each $y$-variable $y_i(s)$ is 
a rational function of the initial $y$-variables $\bfy_{t_0}=\bfy=(y_1,\dots,y_n)$
having a subtraction-free expression.
In particular, if $\bfy$ takes values in 
$\bbR_{>0}^n$, all $y$-variables $y_i(s)$ also take values in $\bbR_{>0}$.
More formally speaking, for any $\bfa=(a_1,\dots,a_n)\in \bbR_{>0}^n$,
we have a specialization homomorphism 
\begin{align}
\label{eq:special1}
\pi_{\bfa}: \bbQ_{\mathrm{sf}}(\bfy)\rightarrow \bbR_{>0},
\quad
y_i \mapsto a_i.
\end{align}
Since the function $f(\bfy)\in \bbQ_{\mathrm{sf}}(\bfy)$ does not  have any pole
in $\bbR_{>0}^n$
 due to the subtraction-free property,
 the map is defined.

Let us write the $C$-matrix and the $F$-polynomials for $\Upsilon(s)$
as $C(s)$ and $F_i(s)$,
where we drop the argument $\bfy$ for the $F$-polynomials for simplicity.
Let $\varepsilon_i(s)$ denote the tropical sign of $y_i(s)$
in Definition \ref{defn:tropsign1}.

The following fact is a consequence of 
the results in Section \ref{sec:sign1}.

\begin{lem}
\label{lem:limit1}
In the limit $\bfy \rightarrow \bfzero$,
each $y$-variable $y_i(s)$  converges to 0 
if $\varepsilon_i(s)=1$
and
diverges to $\infty$ if $\varepsilon_i(s)=-1$
as a function of $\bfy$ on $\bbR_{>0}^n$.
\end{lem}
\begin{proof}
In the current notation, the separation formula 
\eqref{eq:sep2} takes the form
\begin{align}
\label{eq:sep3}
y_{i}(s)&=
\biggl(\,
\prod_{j=1}^n
y_j^{c_{ji}(s)}
\biggr)
\prod_{j=1}^n
F_{j}(s)^{b_{ji}(s)}.
\end{align}
In the limit $\bfy \rightarrow \bfzero$, the nontropical part
converges to 1 thanks to Theorem \ref{thm:Fpol1}.
Meanwhile, thanks to Theorem \ref{thm:sign1},
the tropical part  converges to 0 
if $\varepsilon_i(s)=1$
and
diverges to $\infty$ if $\varepsilon_i(s)=-1$. 
\end{proof}

It is natural to call the limit $\bfy \rightarrow \bfzero$
the \emph{tropical limit}\index{tropical!limit}.

\section{DI associated with a period of a $Y$-pattern}
\label{sec:DI1}

Let us formulate the DI associated with a period of a $Y$-pattern.
For this purpose,
it is  convenient to choose a common skew-symmetrizer 
for $B(s)$ ($s=0$, \dots, $P-1$)
of the form
\begin{align}
\label{eq:Dd1}
D=\mathrm{diag}(\delta_1^{-1}, \dots, \delta_n^{-1})
\quad (\delta_i \in \bbZ_{>0}).
\end{align}
Such a skew-symmetrizer always (and not uniquely) exists.
For example, take any (rational positive) skew-symmetrizer $D=\mathrm{diag}(d_1,\dots,d_n)$ of $B$,
then divide it by the least common multiple of the numerators
of $d_1$, \dots, $d_n$.

Let us further introduce notations.
We  write $\varepsilon_s:=\varepsilon_{k_s}(s)$
$(s=0,\,\dots,\,P-1)$,
where $k_s$ is the direction of the mutation at time $s$ in \eqref{eq:mseq2}.
Then, we define integers $N_+$ and $N_-$ by
\begin{align}
\label{eq:Npm1}
N_+ & :=\sum_{s=0}^{P-1} \d_{k_s} \frac{1+\varepsilon_s}{2},
\quad
N_-  :=\sum_{s=0}^{P-1} \d_{k_s} \frac{1-\varepsilon_s}{2},
\end{align}
so that we have
\begin{align}
\label{eq:NN1}
N_+ +N_-& =\sum_{s=0}^{P-1} \d_{k_s}.
\end{align}
In other words,
$N_{\pm}$ is the total number of $s\in \{0,\,\dots,\,P-1\}$ such that
$\varepsilon_s = \pm1$ weighted by $\delta_{k_s}$.

The following theorem is the \emph{leitmotif} throughout the text.
\begin{thm}[{\cite[Thms.~6.1, 6.8]{Nakanishi10c}}]
\label{thm:DI1}
Suppose that a sequence of mutation \eqref{eq:mseq2}
is $\nu$-periodic
and that the initial  $y$-variables $\bfy$ take values in $\bbR_{>0}^n$.
 Then, the following
DIs hold:
\begin{align}
\label{eq:DI1}
\sum_{s=0}^{P-1}
\d_{k_s} \tilde L(y_{k_s}(s))
&= N_-\frac{\pi^2}{6} ,
\\
\label{eq:DI2}
\sum_{s=0}^{P-1}
\d_{k_s} \tilde L(y_{k_s}(s)^{-1})
&=N_+\frac{\pi^2}{6} ,
\\
\label{eq:DI3}
\sum_{s=0}^{P-1}
\varepsilon_s
\d_{k_s} \tilde L(y_{k_s}(s)^{\varepsilon_s})
&=0,
\end{align}
where $\tilde L$ is the modified Rogers dilogarithm in
\eqref{eq:L3}.
Moreover, the above three identities are equivalent to each other
under   Euler's identity \eqref{eq:euler3}.
\end{thm}

They are functional identities with respect to the  initial $y$-variables
$\bfy$  in $\bbR_{>0}^n$.
We call them together 
 the \emph{DI associated with a period of a free $Y$-pattern (or a cluster pattern)}\index{dilogarithm identity!for period of $Y$-pattern}.

Let us closely look at the first DI  \eqref{eq:DI1}.
According to the algebraic method in Section \ref{sec:Rogers1},
the claim of the identity is separated into two parts,
namely, the \emph{constancy} and the \emph{constant term}.
Assuming the constancy, 
the constant term is immediately obtained
by considering the tropical limit $\bfy \rightarrow \bfzero$.
Thanks to  \eqref{eq:sp3} and Lemma \ref{lem:limit1},
only the terms 
in \eqref{eq:DI1} with $\varepsilon_s=-1$
contribute to the sum.
This yields the constant term $N_-(\pi^2/6)$.
The other DIs \eqref{eq:DI2} and \eqref{eq:DI3}
are equivalent to \eqref{eq:DI1}
 by  Euler's identity \eqref{eq:euler3} and \eqref{eq:NN1}.
Thus,
we are left to prove the constancy of the LHS of
 \eqref{eq:DI1}.
 We will give two alternative proofs in
 Sections \ref{sec:first1} and \ref{sec:second1}.

 \begin{rem}
 \label{rem:DI1}
 The constancy and the constant values of the above DIs do not depend on the choice of
 the initial vertex $t_0$
 because
 changing $t_0$ only causes the change of variables 
 of the identities.
In other words, even though 
 each tropical sign $\varepsilon_s$ depends
 on $t_0$,
 the sums $N_{\pm}$ are   independent of $t_0$.
\end{rem}

Let us see some examples of the DIs in Theorem \ref{thm:DI1}.

\begin{ex}
The involution property of mutations
\begin{align}
&\Upsilon(0) 
\
{\buildrel {k} \over \rightarrow}
\
\Upsilon(1) 
\
{\buildrel {k} \over \rightarrow}
\
\Upsilon(2)=\Upsilon(0)
\end{align}
is regarded as the simplest example of periodicity.
The corresponding DIs \eqref{eq:DI1} and \eqref{eq:DI2} are both written as
\begin{align}
\tilde L(y_k(0))+\tilde L(y_k(0)^{-1})=
\frac{\pi^2}{6}.
\end{align}
This is  Euler's identity \eqref{eq:euler3}.
Meanwhile, the last one \eqref{eq:DI3} is a trivial identity
\begin{align}
\tilde L(y_k(0))-\tilde L(y_k(0))=0.
\end{align}
\end{ex}

\begin{ex}
 Let us consider the pentagon periodicity
\eqref{eq:pentp1}
 for the  free $Y$-pattern of type $A_2$
 in Example \ref{ex:typeA22}.
The sequence of mutations  is written as
\begin{align}
\label{eq:A2seq1}
&\Upsilon(0) 
\
{\buildrel {1} \over \rightarrow}
\
\Upsilon(1) 
\
{\buildrel {2} \over \rightarrow}
\
\Upsilon(2)
\
{\buildrel {1} \over \rightarrow}
\
\Upsilon(3) 
\
{\buildrel {2} \over \rightarrow}
\
\Upsilon(4)
\
{\buildrel {1} \over \rightarrow}
\
\Upsilon(5)=\tau_{12} \Upsilon(0)
.
\end{align}
We choose a skew-symmetrizer $D=I$.
We set $\bfy(0)=\bfy_{t_0}=\bfy$.
By \eqref{eq:A2mut1}--\eqref{eq:A2mut2},
the $y$-variables $y_{k_s}(s)$ concerning the DI \eqref{eq:DI1} are
given by
\begin{align}
\label{eq:ys3}
\begin{split}
y_1(0)&=y_1,\
y_2(1)=y_2(1+y_1),\
y_1(2)=y_1^{-1}(1+y_2+y_1y_2),\\
y_2(3)&=y_1^{-1}y_2^{-1}(1+y_2),\
y_1(4)=y_2^{-1}.
\end{split}
\end{align}
Also, by \eqref{eq:A2mut3}--\eqref{eq:A2mut4},
we have
\begin{gather}
\varepsilon_0=
\varepsilon_1=1,
\quad
\varepsilon_2=
\varepsilon_3=
\varepsilon_4
=-1,
\\
N_+=2,
\quad
N_-=3.
\end{gather}
Then, 
the corresponding DI \eqref{eq:DI1}
is explicitly written as
\begin{align}
\tilde L(y_{1}(0))
+\tilde L(y_{2}(1))
+\tilde L(y_{1}(2))
+\tilde L(y_{2}(3))
+\tilde L(y_{1}(4))
&=\frac{\pi^2}{2}.
\end{align}
This is exactly the pentagon identity in the form \eqref{eq:pent6}.
In fact, the variables $Y_i$ in \eqref{eq:x1} are taken from \eqref{eq:ys3}.
Also, the DI without a constant term \eqref{eq:DI3} is the one \eqref{eq:pent7}.
\end{ex}

\begin{ex} 
\label{ex:BC1}
We give two more examples of rank 2 periodicity.
See \cite[\S2.3]{Nakanishi22a} for the calculation of the mutations.
\par
(a) Type $B_2$. Under the same parametrization of the 2-regular tree $\bbT_2$ in Example \ref{ex:typeA22},
we replace $B$ in \eqref{eq:BA2} with the following one and take the following choice of   a skew-symmetrizer $D$:
\begin{align}
\label{eq:BB2}
B=\begin{pmatrix}
0 & -1
\\
2 & 0
\end{pmatrix}
,
\quad
D=\begin{pmatrix}
1 & 0
\\
0 & 1/2
\end{pmatrix}
,
\quad
\d_1=1,\ \d_2=2.
\end{align}
The Cartan counterpart $A(B)$ is the Cartan matrix of type $B_2$.
A long and straightforward (but educational) calculation, which is similar to Example \ref{ex:typeA22},
yields the periodicity of the free $Y$-pattern
\begin{align}
\label{eq:B2seq1}
&\Upsilon(0) 
\
{\buildrel {1} \over \rightarrow}
\
\Upsilon(1) 
\
{\buildrel {2} \over \rightarrow}
\
\Upsilon(2)
\
{\buildrel {1} \over \rightarrow}
\
\Upsilon(3) 
\
{\buildrel {2} \over \rightarrow}
\
\Upsilon(4)
\
{\buildrel {1} \over \rightarrow}
\
\Upsilon(5)
\
{\buildrel {2} \over \rightarrow}
\
\Upsilon(6)=\Upsilon(0)
.
\end{align}
The $y$-variables $y_{k_s}(s)$ concerning the DI \eqref{eq:DI1} are
\begin{align}
\begin{split}
y_1(0)&=y_1,\
y_2(1)=y_2(1+y_1),\
y_1(2)=y_1^{-1}(1+y_2+y_1y_2)^2,\\
y_2(3)&=y_1^{-1}y_2^{-1}(1+2y_2+y_2^2+y_1y_2^2),\
y_1(4)=y_1^{-1}y_2^{-2}(1+y_2)^2,
\\
y_2(5)&=y_2^{-1}.
\end{split}
\end{align}
Also, we have
\begin{gather}
\varepsilon_0=
\varepsilon_1=1,
\quad
\varepsilon_2=
\cdots =
\varepsilon_5
=-1,
\\
N_+=3,
\quad
N_-=6.
\end{gather}

\par
(b) Type $G_2$. We replace $B$ and $D$ in \eqref{eq:BB2}
with
\begin{align}
\label{eq:BG2}
B=\begin{pmatrix}
0 & -1
\\
3 & 0
\end{pmatrix}
,
\quad
D=\begin{pmatrix}
1 & 0
\\
0 & 1/3
\end{pmatrix}
,
\quad
\d_1=1,\ \d_2=3.
\end{align}
The Cartan counterpart $A(B)$ is the Cartan matrix of type $G_2$.
Then,  we have the periodicity of  the free $Y$-pattern
\begin{align}
\label{eq:G2seq1}
&\Upsilon(0) 
\
{\buildrel {1} \over \rightarrow}
\
\Upsilon(1) 
\
{\buildrel {2} \over \rightarrow}
\
\Upsilon(2)
\
{\buildrel {1} \over \rightarrow}
\
\Upsilon(3) 
\
{\buildrel {2} \over \rightarrow}
\
\cdots
\
{\buildrel {1} \over \rightarrow}
\
\Upsilon(7)
\
{\buildrel {2} \over \rightarrow}
\
\Upsilon(8)=\Upsilon(0)
.
\end{align}
The $y$-variables $y_{k_s}(s)$ concerning the DI \eqref{eq:DI1} are
\begin{align}
\begin{split}
y_1(0)&=y_1,\
y_2(1)=y_2(1+y_1),\
y_1(2)=y_1^{-1}(1+y_2+y_1y_2)^3,\\
y_2(3)&=y_1^{-1}y_2^{-1}(1+3y_2+3y_2^2+y_2^3 + 3y_1y_2^2+2y_1y_2^3 +y_1^2 y_2^3),\\
y_1(4)&=y_1^{-2}y_2^{-3}(1+2y_2+y_2^2 +y_1y_2^2)^3,
\\
y_2(5)&=y_1^{-1}y_2^{-2}(1+3y_2 +3y_2^2+y_2^3 +y_1y_2^3),
\\
y_1(6)&=y_1^{-1}y_2^{-3}(1+y_2)^3,
\quad
y_2(7)=y_2^{-1}.
\end{split}
\end{align}
Also, we have
\begin{gather}
\varepsilon_0=
\varepsilon_1=1,
\quad
\varepsilon_2=
\cdots =
\varepsilon_7
=-1,
\\
N_+=4,
\quad
N_-=12.
\end{gather}

\end{ex}

\section{Constancy condition}

To prove the constancy of the DI 
 \eqref{eq:DI1},
we use the constancy condition given by Frenkel-Szenes \cite{Frenkel95},
following the idea of Chapoton \cite{Chapoton05}.

Let $G$ be any multiplicative abelian group.
Let $G\otimes G$ be its tensor product over $\mathbb{Z}$,
that is, the additive abelian group with the generators $f\otimes g$ ($f,\,g\in G$) and the relations
\begin{align}
\label{eq:tensor1}
(fg)\otimes h = f\otimes h + g\otimes h,
\quad
f\otimes (gh) = f\otimes g + f\otimes h.
\end{align}
We have
\begin{align}
1\otimes h = h \otimes 1 = 0,
\quad
f\otimes g^{-1}=f^{-1}\otimes g=-f\otimes g.
\end{align}
Let $S^2 G$ be the {\em symmetric subgroup} of
$G\otimes G$,
namely, the subgroup generated by all {\rp $f\otimes f$ ($f\in G)$}.
{\rp Thus, we have $f\otimes g + g \otimes f\in G$  ($f,\,g\in G)$.}
The {\em wedge product} of $G$ is defined by
$\bigwedge^2 G= G\otimes G/
S^2 G$.
The symbol $f\wedge g$ denote the equivalence class of $f\otimes g$ in
$\bigwedge^2 G$ as usual.

Let
$\mathcal{C}=\mathcal{C}(I,\bbR_{>0})$ be the set of all 
  differentiable  $\bbR_{>0}$-valued functions on the interval $I=[0,1]$ in $\mathbb{R}$.
  The set  $\mathcal{C}$ is equipped with a semifield structure 
  by the usual multiplication and the addition of functions.
Regarding it as 
a multiplicative abelian group, we have
 the wedge product 
$\bigwedge^2\mathcal{C}$ as defined above.

The following sufficient condition for the constancy of the dilogarithm sum
was given by Frenkel-Szenes \cite{Frenkel95},
whose idea originated in
the work of Bloch \cite{Bloch78}  on the Bloch-Wigner function.

\begin{thm}[{\cite[Proposition 1]{Frenkel95}}]
\label{thm:const2}
Let $f_1(u)$, \dots, $f_n(u) \in \mathcal{C}$.
Suppose that they satisfy the following  relation in $\bigwedge^2\mathcal{C}$:
\begin{align}
\label{eq:const1}
\text{\emph{(Constancy condition)}}\index{constancy condition}
\quad
\sum_{t=1}^r   f_t \wedge (1+f_t) = 0.
\end{align}
Then, the  sum of the modified Rogers dilogarithm
\begin{align}
\label{eq:sum2}
\sum_{t=1}^r
\tilde L
\left(
f_t(u)
\right)
\end{align}
is constant as a function of $u$ on $I$.
\end{thm}

\begin{proof}
By \eqref{eq:L4}, for each $t=1$, \dots, $r$, we have
\begin{align}
\label{eq:dL1}
\begin{split}
&\ 2
\frac{d}{du}
 \tilde{L}
\left( f_t(u)\right)
\\
=& \
 \frac{d}{du}
 \log f_t(u)
 \cdot
 \log (1+f _t(u))
-
\log f_t(u)\cdot \frac{d}{du}
\log (1+f_t(u)).
\end{split}
\end{align}
On the other hand, by the assumption \eqref{eq:const1},
there are some $k\geq 0$ and $g_i,\, h_i \in \mathcal{C}$ ($i=1$, \dots, $k$) such that
\begin{align}
\label{eq:sym1}
\sum_{t=1}^r
f_t \otimes (1+f_t)
=
\sum_{i=1}^k
( g_i \otimes h_i + h_i \otimes g_i).
\end{align}
For each $u$, $v\in I$, we have an additive group homomorphism
$\Psi_{u,v}:
\mathcal{C}\otimes\mathcal{C}
\rightarrow \mathbb{R}$,
$f\otimes g \mapsto \log f(u) \cdot
\log g(v)$.
Applying it on 
\eqref{eq:sym1}, we have
\begin{align}
\label{eq:sym2}
\begin{split}
&\ \sum_{t=1}^r
\log f_t(u)\cdot \log (1+f_t(v) )
\\
=&\
\sum_{i=1}^k
\big( \log g_i(u) \cdot \log h_i(v) + \log h_i(u) \cdot \log g_i(v)
\big).
\end{split}
\end{align}
Then, taking the derivative for $u$ and setting $v=u$, we have
\begin{align}
\label{eq:sym3}
\begin{split}
&\ \sum_{t=1}^r
\frac{d}{du}
 \log f_t(u)\cdot \log (1+f_t(u))\\
=&\
\sum_{i=1}^k
\left(\frac{d}{du}
  \log g_i(u) \cdot \log h_i(u) + \frac{d}{du}
 \log h_i(u) \cdot \log g_i(u)
\right).
\end{split}
\end{align}
Similarly,
taking the derivative for $v$ and setting $v=u$, we have
\begin{align}
\label{eq:sym4}
\begin{split}
&\
\sum_{t=1}^r
 \log f_t(u)\cdot\frac{d}{du}
 \log (1+f_t(u))
 \\
=&\
\sum_{i=1}^k
\left(  \log g_i(u) \cdot\frac{d}{du}
\log h_i(u) + \log h_i(u) \cdot\frac{d}{du}
 \log g_i(u)
\right).
\end{split}
\end{align}
Note that the RHSs of
\eqref{eq:sym3} and \eqref{eq:sym4} are identical.
Thus, by \eqref{eq:dL1},
\eqref{eq:sym3}, and \eqref{eq:sym4},
we obtain the equality
\begin{align}
\frac{d}{du}
\sum_{t=1}^r
 \tilde{L}
\left( f_t(u)\right)
=0.
\end{align}
\end{proof}

\begin{rem}
To be precise, the original condition in \cite{Frenkel95}
was given for the function
\begin{align}
\label{eq:const2}
g_t = \frac{f_t}{1+f_t},
\end{align}
and it has the form
\begin{align}
\label{eq:const3}
\sum_{t=1}^r g_t \wedge (1-g_t) = 0.
\end{align}
It was converted to
 the one \eqref{eq:const1}
by Chapoton \cite{Chapoton05}.
\end{rem}

To apply Theorem \ref{thm:const2}
to  the constancy
of the LHS of \eqref{eq:DI1},
we need to make some adjustments.

Firstly,
we  introduce a multiplicity $m_t \in \bbZ_{>0}$ to each function $f_t$ above.
Then, 
\eqref{eq:const1} and \eqref{eq:sum2} 
are written as
\begin{align}
\label{eq:const4}
\sum_{t=1}^r   m_t f_t \wedge (1+f_t) = 0
\end{align}
and 
\begin{align}
\label{eq:sum4}
\sum_{t=1}^r
m_t 
\tilde L
\left(
f_t(u)
\right).
\end{align}

Secondly, we regard the $y$-variables $y_{k_s}(s)$
in Theorem \ref{thm:DI1}
 as
functions of multi-variables $\bfy$ on $\bbR_{>0}^n$,
while $f_t$ in Theorem \ref{thm:const2} are
 functions of a single variable $u$ on $I=[0,1]$.
So,  following the argument of \cite{Frenkel95},
we recast Theorem \ref{thm:const2} in a suitable form
as follows.
 \begin{prop}
 Let $y_{k_s}(s)\in \bbQ_{\mathrm{sf}}(\bfy)$ $(s=0,\, \dots,\, P-1)$.
Suppose that  the following relation holds in $\bigwedge^2
\bbQ_{\mathrm{sf}}(\bfy)$:
\begin{align}
\label{eq:const51}
\sum_{s=0}^{P-1}  \d_{k_s}  y_{k_s}(s) \wedge (1+y_{k_s}(s)) = 0.
\end{align}
Then, the  sum of the modified Rogers dilogarithm
\begin{align}
\label{eq:sum21}
\sum_{s=0}^{P-1}
 \d_{k_s}
\tilde L
\left(
y_{k_s}(s)
\right)
\end{align}
is constant as a function of $\bfy$ on $\bbR_{>0}^n$.
\end{prop}
\begin{proof}
Let $\bfa_0$ and $\bfa_1$ be arbitrary points in $\bbR_{>0}^n$.
We linearly interpolate them as 
$\bfa(u)=(1-u)\bfa_0 + u \bfa_1$ ($u\in [0,1]$),
and apply the specialization homomorphism
$\pi_{\bfa(u)}\colon \bbQ_{\mathrm{sf}}(\bfy) \rightarrow \bbR_{>0}$
in \eqref{eq:special1} as
\begin{align}
f_s(u):=\pi_{\bfa(u)}(y_{k_s}(s))
\quad
(s=0,\,\dots,\, P-1).
\end{align}
Then, we have functions $f_s(u)\in \mathcal{C}=\mathcal{C}([0,1],\bbR_{>0})$.
Moreover, since $\pi_{\bfa(u)}$ induces a semifield homomorphism
$\bbQ_{\mathrm{sf}}(\bfy)\rightarrow \mathcal{C}$,
the relation \eqref{eq:const51} implies the relation
\begin{align}
\label{eq:const6}
\sum_{s=0}^{P-1}  \d_{k_s}  f_s \wedge (1+f_s) = 0
\quad
\text{in $\textstyle \bigwedge^2 \mathcal{C}$}.
\end{align}
Then, by Theorem \ref{thm:const2}
and the modification in \eqref{eq:sum4},
the sum \eqref{eq:sum21}  coincides for
the specializations at $\bfy=\bfa_0$ and $\bfy=\bfa_1$.
This is the desired constancy.
\end{proof}

In summary,
to complete the proof of Theorem  \ref{thm:DI1},
it is enough to prove the following claim.

\begin{prop}[Constancy {\cite[Prop.~6.3]{Nakanishi11c}}]
\label{prop:const2}
Suppose that a sequence of mutation \eqref{eq:mseq2}
is $\nu$-periodic.
Then, the following relation holds in $\bigwedge^2
\bbQ_{\mathrm{sf}}(\bfy)$:
\begin{align}
\label{eq:const5}
\sum_{s=0}^{P-1}  \d_{k_s}  y_{k_s}(s) \wedge (1+y_{k_s}(s)) = 0.
\end{align}
\end{prop}

\begin{ex}
Let us directly check the consistency condition \eqref{eq:const5}
for the $y$-variables \eqref{eq:ys3} in the $Y$-pattern of type $A_2$.
We have
\begin{align}
\begin{split}
&\quad\
\sum_{s=0}^{4}    y_{k_s}(s) \wedge (1+y_{k_s}(s))
\\
&= y_1\wedge (1+ y_1)
+ y_2(1+y_1) \wedge (1+y_2 + y_1y_2)
\\
&\qquad
+y_1^{-1}(1+y_2 + y_1y_2)\wedge y_1^{-1}(1+y_1)(1+y_2)
\\
&\qquad
+y_1^{-1}y_2^{-1}(1+y_2)\wedge y_1^{-1}y_2^{-1}(1+y_2 +y_1y_2)
\\
&\qquad
+
y_2^{-1}\wedge y_2^{-1}(1+y_2).
\end{split}
\end{align}
It is easy to see that all terms cancel each other.
This also gives an alternative proof of the pentagon identity in the form \eqref{eq:pent6}.
\end{ex}

In the rest of this chapter, we present \emph{two alternative  proofs} of Proposition \ref{prop:const2}, both of which  crucially depend on 
the results on cluster patterns presented in Chapter \ref{ch:quick1}.

\section{First proof of constancy}
\label{sec:first1}

Here, we give the first proof of the constancy condition
\eqref{eq:const5}. 
The basic idea 
is to 
express the $y$-variables therein
by the initial $y$-variables
through the separation formula \eqref{eq:sep2}.
Then, we manipulate the resulting expression
by the change of the summation index with the help of the periodicity.
This approach is regarded as 
a ``brute force" method
(namely, straightforward but less efficient compared with the second one).

In view of Remark \ref{rem:DI1}, 
we may set the initial vertex $t_0$
so that  $\Upsilon(0)$ coincides with the
initial $Y$-seed $\Upsilon_{t_0}$  without
loss of generality.
Let us extend the {\rp $\nu$-periodic} sequence \eqref{eq:mseq2} in the both directions
as
\begin{align}
\label{eq:mseq4}
\cdots
\
{\buildrel {k_{-2}} \over \rightarrow}
\
\Upsilon(-1) 
\
{\buildrel {k_{-1}} \over \rightarrow}
\Upsilon(0) 
\
{\buildrel {k_0} \over \rightarrow}
\
\cdots
\
{\buildrel {k_{P-1}} \over \rightarrow}
\
\Upsilon(P)
\
{\buildrel {k_{P}} \over \rightarrow}
\
\Upsilon(P+1)
\
{\buildrel {k_{P+1}} \over \rightarrow}
\
\cdots
\end{align}
so that  $\Upsilon(s+P)={\rp \nu}\Upsilon(s)$ and $k_{s+P}=k_{\rp \nu(s)}$ hold for any $s\in \bbZ$.
The {\rp $\nu$}-periodicity $\bfy(P)={\rp\nu}\bfy(0)$ implies
 \begin{align}
 \label{eq:Fperiod1}
 F_i(P)=F_{\rp\nu^{-1}(i)}(0)=1
 \quad
 (i=1,\, \dots,\, n)
 \end{align}
  by Theorems \ref{thm:detrop1} and \ref{1thm:synchro1}.
For each $s\in \bbZ$,
we define $\lambda_+(s)$ and $\lambda_-(s)$ as the minimal positive numbers
such that $k_{s \pm \lambda_{\pm}(s)}=k_s$ hold, respectively.
Thus, we have
\begin{align}
\label{eq:lam1}
s'= s + \lambda_+(s)
\
\Longleftrightarrow
\
s=s'- \lambda_-(s').
\end{align}
 For each $s\in \bbZ$, we define the set
\begin{align}
\label{eq:Ts1}
T(s):=\{ t\in \bbZ \mid s\in (t-\lambda_-(t),t)\},
\end{align}
where $(a,b)$ is the open interval between $a$ and $b$.

For simplicity, let $[f]$ denote the image 
of  $f\in \bbQ_{\mathrm{sf}}(\bfy)$ under the tropicalization $\pi_{\trop}$
in \eqref{1eq:trophom1}.
For example, by \eqref{eq:ty1}, we have
\begin{align}
\label{eq:ytrop1}
[y_i(s)]=\prod_{j=1}^n y_j^{c_{ji}(s)},
\quad
[1+y_i(s)]=\prod_{j=1}^n y_j^{-[-c_{ji}(s)]_+}.
\end{align}
By \eqref{1eq:pos1}, we also have
\begin{align}
\label{eq:ytrop2}
\left[
\frac{y_i(s)}{1+y_i(s)}
\right]
&=
\prod_{j=1}^n y_j^{[c_{ji}(s)]_+},
\quad
\left[
\frac{1}{1+y_i(s)}
\right]
=
\prod_{j=1}^n y_j^{[-c_{ji}(s)]_+}.
\end{align}

We use the following formulas.
(The formulas \eqref{eq:Fmut2} and \eqref{eq:sep5}
are only used to prove 
\eqref{eq:sep6}.)

\begin{lem}
\label{lem:const1}
(a) (Mutations).
\begin{align}
\label{eq:Fmut2}
\begin{split}
F_{k_s}(s+\lambda_+(s)) F_{k_s}(s)
&=
\left[
\frac{y_{k_s}(s)}{1+y_{k_s}(s)}
\right]
\prod_{
t\in T(s)}
F_{k_t}(t)^{[b_{k_t k_s}(s)]_+}
\\
&\qquad
+
\left[
\frac{1}{1+y_{k_s}(s)}
\right]
\prod_{
t\in T(s)}
F_{k_t} (t)^{[-b_{k_t k_s}(s)]_+},
\end{split}
\\
\begin{split}
\label{eq:ymut2}
[y_{k_s}(s)][ y_{k_s}(s-\lambda_-(s))]
&=
\prod_{t\in (s-\lambda_-(s), s)}
\bigl(
[y_{k_t}(t)]^{[b_{k_t k_s}(t)]_+}
[1+y_{k_t}(t)]^{-b_{k_t k_s}(t)}
\bigr).
\end{split}
\end{align}
\par
(b)  (Separation formulas).
\begin{align}
\label{eq:sep4}
y_{k_s}(s)
&=
[y_{k_s}(s)]
\prod_{
t\in T(s)}
F_{k_t}(t)^{b_{k_t {k_s}}(s)}
\\
\label{eq:sep5}
&=
[y_{k_s}(s)]
\frac{
\prod_{
t\in T(s)}
F_{k_t}(t)^{[b_{k_t {k_s}}(s)]_+}
}
{
\prod_{
t\in T(s)}
F_{k_t}(t)^{[-b_{k_t {k_s}}(s)]_+}
},
\\
\label{eq:sep6}
1+y_{k_s}(s)
&=
[1+y_{k_s}(s)]
\frac{
F_{k_s}(s+\lambda_+(s))F_{k_s}(s)
}
{
\prod_{
t\in T(s)}
F_{k_t}(t)^{[-b_{k_t k_s}(s)]_+}
}.
\end{align}
(c) (Skew-symmetry).
\begin{align}
\label{eq:mskew1}
\d_{k_s}b_{k_tk_s}(s)=- \d_{k_t}b_{k_sk_t}(s).
\end{align}
\end{lem}
\begin{proof}
(a).
{\rp For each $s\in \bbZ$,
if $j\neq k_s$ ($j\in \{1,\dots, n\}$) does not appear in $T(s)$,
then $j\neq k_{s'}$ for any $s'\in \bbZ$ in the sequence \eqref{eq:mseq4}.
This implies that, for such $j$, we have $F_{j}(s')=F_{j}(0)=1$ for any $s'\in \bbZ$.
Keeping it in mind,}
the formula \eqref{eq:Fmut2}  is the  mutation formula \eqref{2eq:Fmut1} 
rewritten with the expressions \eqref{eq:ytrop2} and the set $T(s)$ in \eqref{eq:Ts1}.
The formula \eqref{eq:ymut2}  follows from the mutation formula \eqref{2eq:ymut1}
by applying successive mutations at $s-\lambda_-(s)$, \dots, $s-1$
and the tropicalization.
(b). The  formula \eqref{eq:sep4} is the separation formula \eqref{eq:sep2}.
The  formula \eqref{eq:sep5} is  equivalent to \eqref{eq:sep4}.
The  formula \eqref{eq:sep6} is obtained from \eqref{eq:sep5}  and \eqref{eq:Fmut2}.
(c). This follows from \eqref{1eq:ss1}.
\end{proof}

Now, we are ready to prove the constancy condition \eqref{eq:const5}.
We put the formulas \eqref{eq:sep4} and \eqref{eq:sep6}
in the LHS of \eqref{eq:const5},
then, separate it into three parts as follows.

The first part involves only tropical parts, namely,
\begin{align}
\label{eq:const7}
\sum_{s=0}^{P-1}  \d_{k_s}  [y_{k_s}(s)] \wedge [1+y_{k_s}(s)].
\end{align}
Recall the sign-coherence in Theorem \ref{thm:sign1}.
If the tropical sign $\varepsilon_s$ is 1, we have $[1+y_{k_s}(s)] =1$.
If $\varepsilon_s$ is $-1$, we have $ [1+y_{k_s}(s)] =[y_{k_s}(s)] $.
In either case, we have $[y_{k_s}(s)] \wedge [1+y_{k_s}(s)]=0$.
Therefore, the sum \eqref{eq:const7} vanishes.

The second part involves both tropical and nontropical parts.
It consists of the following four terms.
\begin{align}
(A)
&=
\sum_{s=0}^{P-1}  \d_{k_s}  [y_{k_s}(s)] \wedge F_{k_s}(s),
\\
\begin{split}
(B)
&=
\sum_{s=0}^{P-1}  \d_{k_s}  [y_{k_s}(s)] \wedge F_{k_s}(s+\lambda_+(s))
\\
&
= 
\sum_{s=0}^{P-1}  \d_{k_s}  [y_{k_s}(s-\lambda_-(s))] \wedge F_{k_s}(s),
\end{split}
\\
\label{eq:termC1}
\begin{split}
(C)
&=
- \sum_{s=0}^{P-1}  \d_{k_s}  [y_{k_s}(s)] \wedge
\Biggl(\,
 \prod_{
t\in T(s)}
F_{k_t}(t)^{[-b_{k_t k_s}(s)]_+}
\Biggr)
\\
&
= 
 \sum_{s=0}^{P-1}  \d_{k_s} 
 \Biggl(\,
  \prod_{t\in (s-\lambda_-(s),s)}
  [y_{k_t}(t)]^{-[b_{k_t k_s}(t)]_+} 
  \Biggr)
  \wedge 
F_{k_s}(s),
\end{split}
\\
\label{eq:termD1}
\begin{split}
(D)
&=
- \sum_{s=0}^{P-1}  \d_{k_s}  [1+y_{k_s}(s)] \wedge
\Biggl(\,
 \prod_{
t\in T(s)}
F_{k_t}(t)^{b_{k_t k_s}(s)}
\Biggr)
\\
&
= 
 \sum_{s=0}^{P-1}  \d_{k_s} 
 \Biggl(\,
  \prod_{t\in (s-\lambda_-(s),s)}
 [1+y_{k_t}(t)]^{b_{k_t k_s}(t)} 
  \Biggr)
  \wedge 
F_{k_s}(s).
\end{split}
\end{align}
In the above, we did the following manipulations.
\begin{itemize}
\item
For (B), we changed the index $s$ of the summation
based on \eqref{eq:lam1}.
{\rp The first sum with $s+\lambda_+(s)\leq P-1$
and the second sum with $s-\lambda_-(s)\geq 0$ clearly match.
When $s+\lambda_+(s)\geq P$ in the first sum,
we have $F_{k_s}(s+\lambda_+(s))= F_{k_s}(P)=1$ by \eqref{eq:Fperiod1},
so that it does not contribute to the sum.
Similarly,
when $s- \lambda_-(s)< 0$ in the second sum,
we have $F_{k_{s}}(s)=F_{k_{s}}(0)=1$,
so that it does not contribute to the sum.}
\item
For (C) and (D), we first exchanged the indices $s$ and $t$
and also  used the skew-symmetry \eqref{eq:mskew1}.
(This  clarifies the necessity of the factor $\d_{k_s}$
in  \eqref{eq:const5}.)
Then, we cut off $s\geq P$, because, if so,
then $F_{k_s}(s)=F_{k_s}(P)=1$ by \eqref{eq:Fperiod1}.
See Figure \ref{fig:st1}.
We also lifted the condition $t\geq 0$, because,
if $t<0$ happens, then $F_{k_s}(s)=F_{k_s}(0)=1$
for any $s\geq 0$ such that $t\in (s-\lambda_-(s),s)$;
thus, it does not contribute to the sum.
\end{itemize}

\begin{figure}
\begin{center}
\begin{tikzpicture}
%
\draw (-1,0)--(5,0);
\draw (0,0.2)--(0,-0.2);
\draw (4,0.2)--(4,-0.2);
\draw (3,0.1)--(3,-0.1);
\draw (1.5,0.4)--(3.5,0.4);
\draw [dotted] (3,0.5)--(3,0);
\draw [dotted] (1.5,0.4)--(1.5,0);
\draw [dotted] (3.5,0.4)--(3.5,0);
\node at (0,-0.4) {0};
\node at (4,-0.4) {$P$};
\node at (3,-0.4) {$t$};
\node at (1.5,-0.4) {$s-\lambda_-(s)$};
\node at (3.5,-0.4) {$s$};
\fill [black] (1.5,0.4) circle [radius=0.06];
\fill [black] (3.5,0.4) circle [radius=0.06];
\end{tikzpicture}
\end{center}
\vskip-10pt
\caption{Schematic diagram for the manipulations 
in \eqref{eq:termC1} and \eqref{eq:termD1}.
}
\label{fig:st1}
\end{figure}
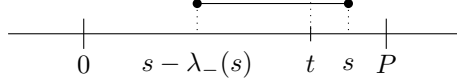

Then, thanks to the mutation formula \eqref{eq:ymut2},
 the sum $(A)+(B)+(C)+(D)$  vanishes.

The third part involves only nontropical parts,
and it requires more elaborated treatment.
It consists of the following three terms.
\begin{align}
(E)
&=
 \sum_{s=0}^{P-1}  \d_{k_s} 
\Biggl(\,
 \prod_{
t\in T(s)}
F_{k_t}(t)^{b_{k_t k_s}(s)}
\Biggr)
 \wedge
 F_{k_s}(s),
 \\
 (F)
&=
 \sum_{s=0}^{P-1}  \d_{k_s} 
\Biggl(\,
 \prod_{
t\in T(s)}
F_{k_t}(t)^{b_{k_t k_s}(s)}
\Biggr)
 \wedge
 F_{k_s}(s+\lambda_+(s)),
 \\
 \begin{split}
 (G)
&=
-
 \sum_{s=0}^{P-1}  \d_{k_s} 
\Biggl(\,
 \prod_{
t\in T(s)}
F_{k_t}(t)^{b_{k_t k_s}(s)}
\Biggr)
 \wedge
\Biggl(\,
 \prod_{
r \in T(s)}
F_{k_{r}}(r)^{[-b_{k_r k_s}(s)]_+}
\Biggr).
\end{split}
\end{align}
We do similar manipulations as before.

The first term $(E)$ is written as follows.
\begin{align*}
(E)
&=
 \sum_{
 \scriptstyle
 s, t=0
 \atop
 \scriptstyle
 s\in (t-\lambda_-(t),t)
 }^{P-1}  \d_{k_s} 
 b_{k_t k_s}(s)
F_{k_t}(t)
 \wedge
 F_{k_s}(s)
 \\
 &=
 \sum_{
 \scriptstyle
 s, t=0
 \atop
 \scriptstyle
 t \in (s-\lambda_-(s),s)
 }^{P-1}  \d_{k_s} 
 b_{k_t k_s}(t)
F_{k_t}(t)
 \wedge
 F_{k_s}(s)
 \\ 
 &=
 \frac{1}{2}
 \sum_{
 \scriptstyle
 s, t=0
 \atop
 \scriptstyle
(s-\lambda_-(s),s)\cap (t-\lambda_-(t),t)\neq \emptyset
 }^{P-1}  \d_{k_s} 
 b_{k_t k_s}(\min(s,t))
F_{k_t}(t)
 \wedge
 F_{k_s}(s).
\end{align*}
In the first line, we  cut off
$t\geq P$ as before.
In the second line, we exchanged the indices $s$ and $t$
and  also used the skew-symmetry \eqref{eq:mskew1}.
The third line is obtained by 
averaging the first and second ones
with factor $1/2$.
(The first one contributes to the case $s<t$,
while the second one contributes to the case $t<s$.)

Similarly, the second term $(F)$ is written as follows.
\begin{align*}
(F)
&=
 \sum_{
 \scriptstyle
 s, t=0
 \atop
 \scriptstyle
 s-\lambda_-(s) \in (t-\lambda_-(t),t)
 }^{P-1}  \d_{k_s} 
 b_{k_t k_s}(s-\lambda_-(s))
F_{k_t}(t)
 \wedge
 F_{k_s}(s)
 \\ 
 &=
 \frac{1}{2}
 \sum_{
 \scriptstyle
 s, t=0
 \atop
 \scriptstyle
(s-\lambda_-(s),s)\cap (t-\lambda_-(t),t)\neq \emptyset
 }^{P-1}  \d_{k_s} 
 b_{k_t k_s}(\max(s-\lambda_-(s),t-\lambda_-(t)))
 \\
 &\hskip150pt
 \times
F_{k_t}(t)
 \wedge
 F_{k_s}(s).
\end{align*}
In the second line, we took the average of the first line
and the one obtained from it by exchanging the indices $s$ and $t$.

Finally, the third term $(G)$ is written as follows.
\begin{align*}
 (G)
&=
-
 \sum_{s,t=0}^{P-1}  \d_{k_s} 
 \Biggl(\,
 \sum_{
r\in (s-\lambda_-(s),s)\cap  (t-\lambda_-(t),t)
 }
 b_{k_t k_r}(r)[b_{k_r k_s}(r)]_+
 \Biggr)
 F_{k_t}(t)
 \wedge
 F_{k_s}(s)
\\
&=
-
\frac{1}{2}
 \sum_{s,t=0}^{P-1}  \d_{k_s} 
 \Biggl(
 \sum_{
r\in (s-\lambda_-(s),s)\cap  (t-\lambda_-(t),t)
 }
 \\
&
\hskip50pt
 \biggl(
 b_{k_t k_r}(r)[b_{k_r k_s}(r)]_+
 +
[-b_{k_t k_r}(r)]_+  b_{k_r k_s}(r)
 \biggr)
 \Biggr)
 F_{k_t}(t)
 \wedge
 F_{k_s}(s).
\end{align*}
In the first line, 
we first exchange variables $s$ and $r$,
and also used the  skew-symmetry \eqref{eq:mskew1}.
Then,
we cut off $s,\, t\geq P$ as before.
Also, we lifted the condition $r\geq 0$,
because, if $r<0$ happens, then
 $F_{k_t}(t)=F_{k_s}(s)=1$
 for the relevant $t$ and $s$ as before.
 In the second line, we took the average of the first line
and the one obtained from it by exchanging the indices $s$ and $t$.

Meanwhile, 
the mutation of exchange matrices \eqref{2eq:bmut1} yields
the equality
\begin{align}
\begin{split}
& \ b_{k_t k_s}(\min(s,t))+
  b_{k_t k_s}(\max(s-\lambda_-(s),t-\lambda_-(t)))
  \\
  =&
 \sum_{
r\in (s-\lambda_-(s),s)\cap  (t-\lambda_-(t),t)
 }
 \biggl(
 b_{k_t k_r}(r)[b_{k_r k_s}(r)]_+
 +
[-b_{k_t k_r}(r)]_+  b_{k_r k_s}(r)
\biggr).
 \end{split}
\end{align}
Thus, the sum $(E)+(F)+(G)$ vanishes.
This completes the first proof of 
Proposition \ref{prop:const2}.

\section{Second proof of constancy}
\label{sec:second1}
Let us give an alternative proof of Proposition \ref{prop:const2},
which was inspired by the result of Fock-Goncharov \cite{Fock03,Fock07}.

Here we return to the usual notation for a $Y$-seed $(\bfy_t,B_t)$ and a $F$-polynomial $F_{i;t}$.
To each $t\in \bbT_n$, we assign the following element in
 $\bigwedge^2 \bbQ_{\mathrm{sf}}(\bfy)$:
\begin{align}
\label{eq:Vt1}
V_t:=
&\ \sum_{i=1}^n \d_i  F_{i;t} \wedge [y_{i;t}]
+\frac{1}{2}\sum_{i,j=1}^n  \d_i b_{ji;t} F_{i;t} \wedge F_{j;t}.
\end{align}
The following mutation formula holds.
\begin{prop}[{\cite[Prop.~6.7]{Nakanishi10c}}]
\label{prop:FG1}
Let $t,\, t'\in \bbT_n$ be $k$-adjacent.
Then, we have
\begin{align}
\label{eq:VV1}
V_{t'}-V_t = 
\d_k 
y_{k;t}\wedge (1+y_{k;t}).
\end{align}
\end{prop}
\begin{proof}
We have
the following formulas, which are parallel with the ones
in Lemma \ref{lem:const1}.
(The relation \eqref{eq:Fmut3} is only used to prove 
\eqref{eq:ymut8}.)
\begin{align}
\label{eq:Fmut3}
F_{k;t'} F_{k;t}
&=
\left[
\frac{y_{k;t}}{1+y_{k;t}}
\right]
\prod_{
i=1}^n
F_{i;t}^{[b_{i k;t}]_+}
+
\left[
\frac{1}{1+y_{k;t}}
\right]
\prod_{
i=1}^n
F_{i;t}^{[-b_{i k;t}]_+},
\\
[y_{i;t'}]&=
\begin{cases} [y_{k;t}]^{-1} & i=k,
\\
[y_{i;t}][y_{k;t}]^{[b_{ki;t}]_+} [1+y_{k;t}]^{-b_{ki;t}},
& i\neq k,
\end{cases}
\\
\label{eq:ymut7}
y_{k;t}&= 
[y_{k;t}]\prod_{i=1}^n F_{i;t}^{b_{ik;t}},
\\
\label{eq:ymut8}
1+y_{k;t}&= 
[1+y_{k;t}]
\frac{
F_{k;t'} F_{k;t}
}
{\prod_{i=1}^n F_{i;t}^{[-b_{ik;t}]_+}},
\\
\d_j b_{ij;t}&=-\d_i b_{ji;t}.
\end{align}
We put the expressions \eqref{eq:ymut7}
and \eqref{eq:ymut8} into the RHS of \eqref{eq:VV1}.
Then, separate it into three parts as before.

The first part consists of a single term
\begin{align}
\d_k [y_{k;t}]\wedge [1+y_{k;t}].
\end{align}
This vanishes by the same reason for the one \eqref{eq:const7}.

The second part consists of four terms
\begin{align}
(A)&=\d_k [y_{k;t}]\wedge F_{k;t},
\\
(B)&=\d_k [y_{k;t}]\wedge F_{k;t'},
\\
(C)&=- \d_k [y_{k;t}]\wedge \biggl(\,
 \prod_{i=1}^n F_{i;t}^{[-b_{ik;t}]_+}
 \biggr)
=
-
\sum_{i=1}^n  \d_i [b_{ki;t}]_+ [y_{k;t}]\wedge F_{i;t},
\\
(D)&=\d_k
\biggl(\,
  \prod_{i=1}^n F_{i;t}^{b_{ik;t}}
  \biggr)
   \wedge [1+y_{k;t}]
=
-
\sum_{i=1}^n  \d_i b_{ki;t} F_{i;t} \wedge  [1+y_{k;t}].
\end{align}
Meanwhile, we have
\begin{align*}
&\
 \sum_{i=1}^n \d_i  F_{i;t'} \wedge [y_{i;t'}]
-
 \sum_{i=1}^n \d_i  F_{i;t} \wedge [y_{i;t}]
 \\
=&\ 
 \sum_{ \scriptstyle
i=1
 \atop
 \scriptstyle
 i\neq k}^n \d_i  F_{i;t} \wedge [y_{k;t}]^{[b_{ki;t}]_+}
  [1+y_{k;t}]^{-b_{ki;t}}
 -   \d_k  F_{k;t'} \wedge [y_{k;t}] 
  -  \d_k  F_{k;t} \wedge [y_{k;t}]
  \\
=&\  (A)+(B)+(C)+(D),
\end{align*}
where we used the fact $b_{kk}=0$.

The third part consists of three terms
\begin{align}
(E)&=\d_k  \biggl(\,\prod_{i=1}^n F_{i;t}^{b_{ik;t}} 
\biggr)
\wedge F_{k;t}
=
\sum_{i=1}^n \d_k b_{ik;t} F_{i;t}\wedge F_{k;t},
\\
(F)&=\d_k \biggl(\, \prod_{i=1}^n F_{i;t}^{b_{ik;t}}\biggr)
 \wedge F_{k;t'}
=
\sum_{i=1}^n \d_k b_{ik;t} F_{i;t}\wedge F_{k;t'},
\\
\begin{split}
(G)&= - \d_k  \biggl(\,\prod_{i=1}^n F_{i;t}^{b_{ik;t}}\biggr)
 \wedge
 \biggl(\,
 \prod_{j=1}^n F_{j;t}^{[-b_{jk;t}]_+}
 \biggr)
=
 \sum_{i,j=1}^n \d_i [-b_{jk;t}]_+ b_{ki;t}
 F_{i;t}\wedge F_{j;t}.
 \end{split}
\end{align}
Meanwhile, we have
\begin{align*}
&\
\frac{1}{2}\sum_{i,j=1}^n  \d_i b_{ji;t'} F_{i;t'} \wedge F_{j;t'}
-
\frac{1}{2}\sum_{i,j=1}^n  \d_i b_{ji;t} F_{i;t} \wedge F_{j;t}
\\
=&\
\frac{1}{2}\sum_{
\scriptstyle i,j=1
\atop
\scriptstyle i,j\neq k}^n  \d_i
\left(
b_{jk;t}[b_{ki;t}]_+
+
[-b_{jk;t}]_+b_{ki;t}
\right)
 F_{i;t} \wedge F_{j;t}
 \\
 &\
+\frac{1}{2}
 \sum_{\scriptstyle i=1\atop\scriptstyle i\neq k }^n \d_i (-b_{ki;t} )F_{i;t}
 \wedge F_{k;t'}
 -\frac{1}{2}
 \sum_{\scriptstyle i=1\atop\scriptstyle i\neq k }^n \d_i b_{ki;t} F_{i;t}
 \wedge F_{k;t}
  \\
 &\
+\frac{1}{2}
 \sum_{\scriptstyle j=1\atop\scriptstyle j\neq k }^n \d_k (-b_{jk;t} )F_{k;t'}
 \wedge F_{j;t}
-\frac{1}{2}
 \sum_{\scriptstyle j=1\atop\scriptstyle j\neq k }^n \d_k b_{jk;t} F_{k;t}
 \wedge F_{j;t}
 \\
= &\
 (E)+(F)+(G).
\end{align*}
Therefore, the equality \eqref{eq:VV1} holds.
\end{proof}

Now let us prove Proposition   \ref{prop:const2}
using Proposition \ref{prop:FG1},
where we do not need to assume that $\Upsilon(0)=\Upsilon_{t_0}$
for the initial vertex $t_0$.
Let us go back to the sequence \eqref{eq:mseq2}
and assume the $\nu$-periodicity.
Let $V(s)$ denote the element $V_t$ in \eqref{eq:Vt1} corresponding to each $Y$-seed
$\Upsilon(s)$ therein.
Then, thanks to 
Theorem \ref{1thm:synchro1}, \eqref{eq:detrop4}, and Proposition \ref{prop:compat1},
we have the periodicity
\begin{align}
\label{eq:VV2}
V(P)=
V(0).
\end{align}
Meanwhile, by Proposition \ref{prop:FG1}, we have
\begin{align}
V(P)-V(0)=\sum_{s=0}^{P-1} \d_{k_s} y_{k_s}(s) \wedge
(1+ y_{k_s}(s) ).
\end{align}
This is 0 by \eqref{eq:VV2}.
This completes the second proof of 
Proposition \ref{prop:const2}.

Compared with the first proof, the second proof 
is more efficient; moreover, it exhibits the role of the periodicity
in Proposition   \ref{prop:const2} more transparently.
Meanwhile, the intrinsic meaning of the element $V_t$  is not clear.

In summary, we have  proved 
  Theorem \ref{thm:DI1} via Proposition   \ref{prop:const2}
 with two alternative proofs.
 We call them the \emph{algebraic method}\index{algebraic method}
because it is  directly related to the algebraic structure of
seeds and mutations in a cluster pattern.
We also note that both proofs
 depend on some advanced results on cluster algebras.
 Even though the idea of the proofs is clear,
we are not  completely satisfied with them because
   they do not clarify  \emph{any intrinsic reason why
DIs are relevant to cluster algebras}.
Thus, we continue to study Theorem \ref{thm:DI1}
in different perspectives in the following part, looking for an answer to this question.

\begin{rem}
In \cite[Lemma 6.3]{Fock03},
for a cluster pattern without coefficients,
the following element was considered
instead of $V_t$ in \eqref{eq:Vt1}:
\begin{align}
W_t&=
\frac{1}{2}
\sum_{i=1}^n \d_i x_{i;t}\wedge \hat y_{i;t}
=
\frac{1}{2}
\sum_{i,j=1}^n  \d_i b_{ji;t} x_{i;t} \wedge x_{j;t}.
\end{align}
Then,
 the  mutation formula 
\begin{align}
\label{eq:Wt1}
W_{t'}-W_t = 
\d_k 
\hat y_{k;t}\wedge (1+\hat y_{k;t}).
\end{align}
is shown without using the advanced results in
Chapter \ref{sec:advanced1} .
Moreover, when the underlying $B$-pattern is nonsingular,
the initial $\hat y$-variables are algebraically independent.
Then, the formula \eqref{eq:Wt1} implies Proposition   \ref{prop:const2}.
In the singular case, however, we still need the  results in
Chapter \ref{sec:advanced1} to resolve the issue of algebraic dependence.
\end{rem}

\notes
The constancy condition in Theorem \ref{thm:const2}
was formulated by  \cite{Frenkel95}
to prove the DI for the $Y$-system of type $(A_n,A_1)$.
It is parallel with the result on the Bloch-Wigner function by  \cite[Thm.~1.7, Remark~1.11]{Bloch78},
which is recapitulated in  \cite[Prop.~3]{Frenkel95}.
 
There have been several prototypical examples before Theorem \ref{thm:DI1}
was formulated in \cite{Nakanishi11c}.
The relation between the periodicity and
the constancy condition for the $Y$-system of type $(X,X')$ was
studied  in \cite{Caracciolo99}.
The crucial step was made by 
 \cite{Chapoton05},
 where 
the DIs for the $Y$-system of type $(X,A_1)$ was proved.
The consistency condition \eqref{eq:const1} was proved
using the explicit description of the solution of 
the above $Y$-system by \cite{Fomin03b}.
The evaluation of the constant term by taking the tropical limit
also originated in  \cite{Chapoton05}.
This method was generalized by \cite{Nakanishi09}
to prove the DI for the $Y$-system of general type $(X,X')$
in a more cluster algebraic setting,
where all ingredients of the method in this chapter
was essentially introduced.
There are  subsequent papers where the same method
was applied
to prove   the DIs for various $Y$-systems
 \cite{Inoue10a, Inoue10b, Nakanishi10b}.
After these works, it was  figured out in \cite{Nakanishi11c}
 that the periodicity
in a cluster pattern is the sole reason for these DIs.

The Bloch-Wigner function was  studied  in 
 \cite{Fock03,Fock07}
by a more abstract approach with algebraic $K$-theory.

\chapter{$Y$-systems and $Y$-patterns}

Now we come back to and prove the DIs for $Y$-systems 
in Chapter \ref{ch:prologue}.
We first formulate these $Y$-systems in terms of $Y$-patterns.
Then, we prove their periodicity using the tropicalization.
It turns out that the Coxeter elements
of root systems are behind the scenes and govern the dynamics.
By applying Theorem  \ref{thm:DI1},
we obtain the DIs for $Y$-systems,
where the counting of the constant term is also done by the tropicalization.

\section{Quivers and mutations}

It is often convenient and useful to represent
a \emph{skew-symmetric\/} integer matrix by a \emph{quiver}.

\begin{defn}[Quiver]
A (finite) \emph{quiver}\index{quiver}
 is a finite directed graph. Namely, it consists of a finite set of vertices
and a finite set of arrows between the vertices.
For a quiver with $n$ vertices, we assume that its vertices are labeled  with 1, \dots, $n$
without duplication.
The following arrows are called a \emph{loop} 
and   a \emph{$2$-cycle}, respectively:
\index{loop (for quiver)}\index{2-cycle (for quiver)} 
$$
\begin{xy}
(0,0)*\cir<2pt>{},
(30,0)*\cir<2pt>{},
(50,0)*\cir<2pt>{},
(0.5,1); (0.5,-1) **\crv{(5,10) &(20,0)&(5,-10)};
\ar@{->} (0.62,-1.2);(0.5,-1)
\ar@{->} (32,1);(48,1)
\ar@{<-} (32,-1);(48,-1)
\end{xy}
$$

\end{defn}

For a skew-symmetric integer matrix $B$, one can associate a quiver $Q(B)$ without loops and 2-cycles by the following rule:
\begin{itemize}
\item 
If $b_{ij}>0$, then assign $b_{ij}$ arrows from the vertex $i$ to the vertex $j$.
\end{itemize}
Since $b_{ii}=0$, there are no loops. Also,  there are no 2-cycles,
because $b_{ji}=-b_{ij}<0$ if $b_{ij}>0$.
For example,
\begin{align}
B=
\begin{pmatrix}
0 & -3 & 2\\
3 & 0 & 1\\
-2 & -2 & 0
\end{pmatrix},
\quad
Q(B)=
\raisebox{-15pt}{
\begin{xy}
(0,0)*\cir<2pt>{},
(20,0)*\cir<2pt>{},
(10,17)*\cir<2pt>{},
(0,-3)*{\text{\small1}},
(20,-3)*{\text{\small2}},
(7,17)*{\text{\small3}},
\ar@{<-} (2,1);(18,1)
\ar@{<-} (2,0);(18,0)
\ar@{<-} (2,-1);(18,-1)
\ar@{->} (19,2);(11,15)
\ar@{->} (1.5,1.8);(9.5,14.8)
\ar@{->} (0.5,2.2);(8.5,15.2)
\end{xy}
}
\end{align}

Conversely, one can recover a skew-symmetric integer matrix $B$ from a quiver without loops and 2-cycles by applying the above rule in the opposite direction.
It is clear that this correspondence is one-to-one.

For a quiver $Q$, the quiver obtained from $Q$ by changing  all arrows  of $Q$ to the 
ones in the opposite direction is called the \emph{opposite quiver}\index{opposite quiver} of $Q$,
and it is denoted by $Q^{\op}$.
If a skew-symmetric integer matrix $B$ corresponds to a quiver $Q$,
then $-B$ corresponds to $Q^{\op}$.

One can translate the matrix mutation \eqref{2eq:bmut1} into the quiver mutation as follows.

\begin{defn}[Quiver mutation]
\label{1defn:qmut1}
Let $Q$ be a quiver with $n$ vertices and without loops and 2-cycles.
For each $k=1$, \dots, $n$, we define a new quiver
$Q'=\mu_k(Q)$ by the following operation:
\begin{itemize}
\item 
For each  pair $i$, $j$ ($i\neq j$) such that $i$, $j\neq k$,
if
there are $p>0$ arrows from the vertex $i$ to the vertex $k$,
 and $q>0$ arrows from the vertex $k$ to the vertex $j$,
 then  add $pq$ arrows from the vertex $i$ to the vertex $j$.
\item Remove the resulting 2-cycles as many as possible.
\item Invert all arrows into and out of the vertex $k$.
\end{itemize}
The quiver $Q'$ is called the \emph{mutation of $Q$ at the vertex $k$}\index{mutation!of quiver}.
\end{defn}

Under the above identification,  a $Y$-seed $(\bfy, B)$  with a skew-symmetric matrix $B$
is replaced with $(\bfy, Q)$ for the corresponding quiver $Q$.
For the mutation $(\bfy', Q')=\mu_k(\bfy,Q)$ of a $Y$-seed,
the mutation of $y$-variables \eqref{2eq:ymut1}
is translated into  the following simple pictorial rule:
\begin{align}
\label{eq:ypic1}
y'_i 
=
\begin{cases}
y_k^{-1} & i=k,
\\
y_i,
&
\begin{xy}
(0,0)*\cir<2pt>{},
(10,0)*{\bullet},
(0,-3)*{\text{\small $i$}},
(10,-3)*{\text{\small $k$}},
\end{xy}
\ \text{(no arrow)},
\\
y_i  (1\oplus y_k)^{m}
&
\begin{xy}
(0,0)*\cir<2pt>{},
(10,0)*{\bullet},
(0,-3)*{\text{\small $i$}},
(10,-3)*{\text{\small $k$}},
(5,2)*{\text{\small $m$}},
\ar@{<-} (8,0);(2,0) 
\end{xy}
\ \text{($m$ arrows)},
\\
y_i  (1\oplus y_k^{-1})^{-m}
&
\begin{xy}
(0,0)*\cir<2pt>{},
(10,0)*{\bullet},
(0,-3)*{\text{\small $i$}},
(10,-3)*{\text{\small $k$}},
(5,2)*{\text{\small $m$}},
\ar@{->} (8,0);(2,0) 
\end{xy}
\ \text{($m$ arrows)}.
\end{cases}
\end{align}

\begin{ex}
\label{ex:QB1}
Let us consider the following quiver $Q$:
$$
Q=
\
\lower12pt
\hbox{
\begin{xy}
(0,0)*\cir<2pt>{},
(0,10)*\cir<2pt>{},
(10, 0)*\cir<2pt>{},
(10,10)*\cir<2pt>{},
(20, 0)*\cir<2pt>{},
(20,10)*\cir<2pt>{},
(0,13)*{\text{\small1}},
(10,13)*{\text{\small2}},
(20,13)*{\text{\small3}},
(0,-3)*{\text{\small4}},
(10,-3)*{\text{\small5}},
(20,-3)*{\text{\small6}},
\ar@{->} (2,10);(8,10) 
\ar@{->} (18,10);(12,10) 
\ar@{->} (8,0);(2,0) 
\ar@{->} (12,0);(18,0) 
\ar@{->} (0,2);(0,8) 
\ar@{->} (10,8);(10,2) 
\ar@{->} (20,2);(20,8) 
\end{xy}
}
$$
We consider  the sequence of mutations at the vertices $2$, $4$, $6$,
$1$, {\rp $3$, $5$} in this order.
In the first half,
the quivers mutate as follows,
where a black vertex  is  the one at which the mutation is
going to be applied:
$$
\begin{xy}
(0,0)*\cir<2pt>{},
(0,10)*\cir<2pt>{},
(10, 0)*\cir<2pt>{},
(10,10)*{\bullet},
(20, 0)*\cir<2pt>{},
(20,10)*\cir<2pt>{},
(0,13)*{\text{\small1}},
(10,13)*{\text{\small2}},
(20,13)*{\text{\small3}},
(0,-3)*{\text{\small4}},
(10,-3)*{\text{\small5}},
(20,-3)*{\text{\small6}},
(27.5,8)*{\text{\small2}}
\ar@{->} (2,10);(8,10) 
\ar@{->} (18,10);(12,10) 
\ar@{->} (8,0);(2,0) 
\ar@{->} (12,0);(18,0) 
\ar@{->} (0,2);(0,8) 
\ar@{->} (10,8);(10,2) 
\ar@{->} (20,2);(20,8) 
\ar@{->} (25,5);(30,5) 
\end{xy}
\quad
\begin{xy}
(0,0)*{\bullet},
(0,10)*\cir<2pt>{},
(10, 0)*\cir<2pt>{},
(10,10)*\cir<2pt>{},
(20, 0)*\cir<2pt>{},
(20,10)*\cir<2pt>{},
(0,13)*{\text{\small1}},
(10,13)*{\text{\small2}},
(20,13)*{\text{\small3}},
(0,-3)*{\text{\small4}},
(10,-3)*{\text{\small5}},
(20,-3)*{\text{\small6}},
(27.5,8)*{\text{\small4}}
\ar@{<-} (2,10);(8,10) 
\ar@{<-} (18,10);(12,10) 
\ar@{->} (8,0);(2,0) 
\ar@{->} (12,0);(18,0) 
\ar@{->} (0,2);(0,8) 
\ar@{<-} (10,8);(10,2) 
\ar@{->} (20,2);(20,8) 
\ar@{->} (2,8);(8,2) 
\ar@{->} (18,8);(12,2) 
\ar@{->} (25,5);(30,5) 
\end{xy}
\quad
\begin{xy}
(0,0)*\cir<2pt>{},
(0,10)*\cir<2pt>{},
(10, 0)*\cir<2pt>{},
(10,10)*\cir<2pt>{},
(20, 0)*{\bullet},
(20,10)*\cir<2pt>{},
(0,13)*{\text{\small1}},
(10,13)*{\text{\small2}},
(20,13)*{\text{\small3}},
(0,-3)*{\text{\small4}},
(10,-3)*{\text{\small5}},
(20,-3)*{\text{\small6}},
(27.5,8)*{\text{\small6}}
\ar@{<-} (2,10);(8,10) 
\ar@{<-} (18,10);(12,10) 
\ar@{<-} (8,0);(2,0) 
\ar@{->} (12,0);(18,0) 
\ar@{<-} (0,2);(0,8) 
\ar@{<-} (10,8);(10,2) 
\ar@{->} (20,2);(20,8) 
\ar@{->} (18,8);(12,2) 
\ar@{->} (25,5);(30,5) 
\end{xy}
\quad
\begin{xy}
(0,0)*\cir<2pt>{},
(0,10)*\cir<2pt>{},
(10, 0)*\cir<2pt>{},
(10,10)*\cir<2pt>{},
(20, 0)*\cir<2pt>{},
(20,10)*\cir<2pt>{},
(0,13)*{\text{\small1}},
(10,13)*{\text{\small2}},
(20,13)*{\text{\small3}},
(0,-3)*{\text{\small4}},
(10,-3)*{\text{\small5}},
(20,-3)*{\text{\small6}},
\ar@{<-} (2,10);(8,10) 
\ar@{<-} (18,10);(12,10) 
\ar@{<-} (8,0);(2,0) 
\ar@{<-} (12,0);(18,0) 
\ar@{<-} (0,2);(0,8) 
\ar@{<-} (10,8);(10,2) 
\ar@{<-} (20,2);(20,8) 
\end{xy}
$$
The last quiver is the opposite quiver $Q^{\op}$ of $Q$.
It is easy to see that the final result is independent of the order
of the mutations at 2, 4, 6.
Similarly, 
in the second half,
the quivers mutate as follows:
$$
\begin{xy}
(0,0)*\cir<2pt>{},
(0,10)*{\bullet},
(10, 0)*\cir<2pt>{},
(10,10)*\cir<2pt>{},
(20, 0)*\cir<2pt>{},
(20,10)*\cir<2pt>{},
(0,13)*{\text{\small1}},
(10,13)*{\text{\small2}},
(20,13)*{\text{\small3}},
(0,-3)*{\text{\small4}},
(10,-3)*{\text{\small5}},
(20,-3)*{\text{\small6}},
(27.5,8)*{\text{\small1}}
\ar@{<-} (2,10);(8,10) 
\ar@{<-} (18,10);(12,10) 
\ar@{<-} (8,0);(2,0) 
\ar@{<-} (12,0);(18,0) 
\ar@{<-} (0,2);(0,8) 
\ar@{<-} (10,8);(10,2) 
\ar@{<-} (20,2);(20,8) 
\ar@{->} (25,5);(30,5) 
\end{xy}
\quad
\begin{xy}
(0,0)*\cir<2pt>{},
(20,10)*{\bullet},
(0,10)*\cir<2pt>{},
(10, 0)*\cir<2pt>{},
(10,10)*\cir<2pt>{},
(20, 0)*\cir<2pt>{},
(0,13)*{\text{\small1}},
(10,13)*{\text{\small2}},
(20,13)*{\text{\small3}},
(0,-3)*{\text{\small4}},
(10,-3)*{\text{\small5}},
(20,-3)*{\text{\small6}},
(27.5,8)*{\text{\small3}}
\ar@{->} (2,10);(8,10) 
\ar@{<-} (18,10);(12,10) 
\ar@{<-} (8,0);(2,0) 
\ar@{<-} (12,0);(18,0) 
\ar@{->} (0,2);(0,8) 
\ar@{<-} (10,8);(10,2) 
\ar@{<-} (20,2);(20,8) 
\ar@{->} (8,8);(2,2) 
\ar@{->} (25,5);(30,5) 
\end{xy}
\quad
\begin{xy}
(0,0)*\cir<2pt>{},
(0,10)*\cir<2pt>{},
(10,10)*\cir<2pt>{},
(20, 0)*\cir<2pt>{},
(10,0)*{\bullet},
(20,10)*\cir<2pt>{},
(0,13)*{\text{\small1}},
(10,13)*{\text{\small2}},
(20,13)*{\text{\small3}},
(0,-3)*{\text{\small4}},
(10,-3)*{\text{\small5}},
(20,-3)*{\text{\small6}},
(27.5,8)*{\text{\small5}}
\ar@{->} (2,10);(8,10) 
\ar@{->} (18,10);(12,10) 
\ar@{->} (2,0);(8,0) 
\ar@{<-} (12,0);(18,0) 
\ar@{->} (0,2);(0,8) 
\ar@{<-} (10,8);(10,2) 
\ar@{->} (20,2);(20,8) 
\ar@{->} (8,8);(2,2) 
\ar@{->} (12,8);(18,2) 
\ar@{->} (25,5);(30,5) 
\end{xy}
\quad
\begin{xy}
(0,0)*\cir<2pt>{},
(0,10)*\cir<2pt>{},
(10, 0)*\cir<2pt>{},
(10,10)*\cir<2pt>{},
(20, 0)*\cir<2pt>{},
(20,10)*\cir<2pt>{},
(0,13)*{\text{\small1}},
(10,13)*{\text{\small2}},
(20,13)*{\text{\small3}},
(0,-3)*{\text{\small4}},
(10,-3)*{\text{\small5}},
(20,-3)*{\text{\small6}},
\ar@{->} (2,10);(8,10) 
\ar@{->} (18,10);(12,10) 
\ar@{->} (8,0);(2,0) 
\ar@{->} (12,0);(18,0) 
\ar@{->} (0,2);(0,8) 
\ar@{->} (10,8);(10,2) 
\ar@{->} (20,2);(20,8) 
\end{xy}
$$
The last quiver is the original quiver $Q$.
Thus, we have a periodicity of quivers.
Again, it is easy to see that  the final result is independent of the order
of the mutations at 1, 3, 5.
For $(\bfy',Q^{\op})=\mu_6\mu_4\mu_2(\bfy,Q)$,
by applying the pictorial rule \eqref{eq:ypic1},
we have
\begin{align}
\label{eq:ymutex1}
y'_i
=
\begin{cases}
y_i^{-1} & i =2,\, 4,\, 6,
\\
y_1(1\oplus y_2)(1\oplus y_4^{-1})^{-1} & i=1,
\\
y_3(1\oplus y_2)(1\oplus y_6^{-1})^{-1} & i=3,
\\
y_5(1\oplus y_4)(1\oplus y_6)(1\oplus y_2^{-1})^{-1} & i=5.
\end{cases}
\end{align}
Similarly, 
for $(\bfy'',Q)=\mu_5\mu_3\mu_1(\bfy',Q^{\op})$,
we have
\begin{align}
\label{eq:ymutex2}
y''_i
=
\begin{cases}
y'_i{}^{-1} & i =1,\, 3,\, 5,
\\
y'_2(1\oplus y'_1)(1\oplus y'_3)(1\oplus y'_5{}^{-1})^{-1} & i=2.
\\
y'_4(1\oplus y'_5)(1\oplus y'_1{}^{-1})^{-1} & i=4,
\\
y'_6(1\oplus y'_5)(1\oplus y'_3{}^{-1})^{-1} & i=6.
\end{cases}
\end{align}
\end{ex}

\section{$Y$-systems and   $Y$-patterns}
\label{sec:Ysystems1}

The readers may already notice
 the similarity between the mutation of $y$-variables 
\eqref{eq:ymutex1}, \eqref{eq:ymutex2}
and   $Y$-systems  \eqref{eq:Ysys1}.
Let us make this observation more precise.

In general, mutations are not commutative,
namely, $\mu_{k}\mu_{\ell}(\Sigma)\neq
\mu_{\ell}\mu_{k}(\Sigma)$.
However,  they commute under the following condition.

\begin{lem}[{e.g., \cite[\S 8]{Fomin07}}]
\label{lem:bcommut1}
For a seed $\Sigma=(\bfx,\bfy,B)$ and 
a pair $k$, $\ell$ ($k\neq \ell$),
suppose that 
\begin{align}
\label{1eq:bkl1}
b_{k \ell}=b_{ \ell k}=0
\end{align}
holds.
Then, for the seed $\Sigma'=(\bfx',\bfy',B')=\mu_k(\Sigma)$,
we also have
\begin{align}
\label{1eq:bkl2}
b'_{k \ell}=b'_{ \ell k}=0.
\end{align}
Moreover, we have
\begin{align}
\label{1eq:bkl3}
\mu_k\mu_{\ell}(\Sigma)
=
\mu_{\ell}\mu_{k}(\Sigma).
\end{align}
\end{lem}
The equality \eqref{1eq:bkl3} can be easily verified
by a direct calculation (e.g., \cite[Prop.~2.9]{Nakanishi22a}).

Now we formulate the $Y$-system \eqref{eq:Ysys1}
of type $(X,X')$
in terms of a $Y$-pattern following Keller \cite{Keller08}.
We divide the formulation into three steps.

{\bf Step 1. Quiver $Q(X,X')$.}
At each vertex $a$ of $X$ we assign a sign $\kappa_a \in \{+, -\}$
so that any adjacent vertices in $X$ have the opposite sign.
(There are two possible choices, and either will do.)
We conveniently identify the signs $\pm$ with $\pm1$ as usual.
We do the same assignment of signs for $X'$.
Then, we consider the product graph $T=X \times X'$,
where $X$ lies horizontally and $X'$ lies vertically.
For each vertex $v=(a,a')$ of $T$, we set
a sign $\kappa_v=\kappa_{a,a'}:=\kappa_a \kappa_{a'}\in \{+, -\}$.
Moreover,
we set an orientation on each edge of $T$
by the following rule:
\begin{align}
\begin{matrix}
(++)&\leftarrow &(-+)
\\
\downarrow &&\uparrow
\\
(+-)&\rightarrow & (--)
\end{matrix}
\end{align}
The resulting quiver is denoted by $Q(X,X')$.
See Figure \ref{fig:prod1} for examples.
The quiver $Q$ in Example \ref{ex:QB1}
is  also an example for $Q(A_3,A_2)$.

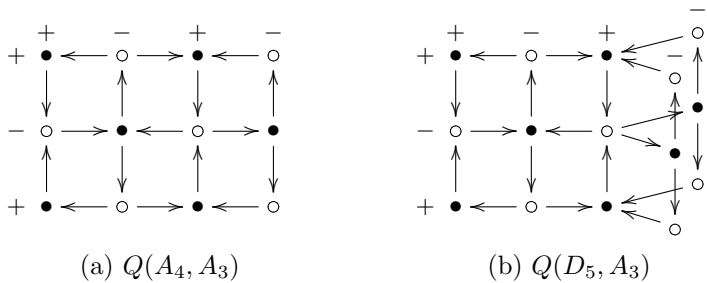
\begin{figure}
$$
\begin{xy}
(0,0)*{\bullet},
(0,10)*\cir<2pt>{},
(0,20)*{\bullet},
(10, 0)*\cir<2pt>{},
(10,10)*{\bullet},
(10,20)*\cir<2pt>{},
(20, 0)*{\bullet},
(20,10)*\cir<2pt>{},
(20,20)*{\bullet},
(30, 0)*\cir<2pt>{},
(30,10)*{\bullet},
(30,20)*\cir<2pt>{},
(0,23)*{\text{\small+}},
(10,23)*{\text{\small $-$}},
(20,23)*{\text{\small+}},
(30,23)*{\text{\small $-$}},
(-4,0)*{\text{\small+}},
(-4,10)*{\text{\small $-$}},
(-4,20)*{\text{\small+}},
(15,-8)*{\text{(a) $Q(A_4,A_3)$}},
\ar@{->} (8,20);(2,20) 
\ar@{->} (12,20);(18,20) 
\ar@{->} (28,20);(22,20) 
\ar@{->} (2,10);(8,10) 
\ar@{->} (18,10);(12,10) 
\ar@{->} (22,10);(28,10) 
\ar@{->} (8,0);(2,0) 
\ar@{->} (12,0);(18,0) 
\ar@{->} (28,0);(22,0) 
\ar@{->} (0,2);(0,8) 
\ar@{->} (10,8);(10,2) 
\ar@{->} (20,2);(20,8) 
\ar@{->} (30,8);(30,2) 
\ar@{->} (0,18);(0,12) 
\ar@{->} (10,12);(10,18) 
\ar@{->} (20,18);(20,12) 
\ar@{->} (30,12);(30,18) 
\end{xy}
\hskip50pt
\begin{xy}
(0,0)*{\bullet},
(0,10)*\cir<2pt>{},
(0,20)*{\bullet},
(10, 0)*\cir<2pt>{},
(10,10)*{\bullet},
(10,20)*\cir<2pt>{},
(20, 0)*{\bullet},
(20,10)*\cir<2pt>{},
(20,20)*{\bullet},
(29, -3)*\cir<2pt>{},
(29,7)*{\bullet},
(29,17)*\cir<2pt>{},
(32, 3)*\cir<2pt>{},
(32,13)*{\bullet},
(32,23)*\cir<2pt>{},
(0,23)*{\text{\small+}},
(10,23)*{\text{\small $-$}},
(20,23)*{\text{\small+}},
(32,26)*{\text{\small $-$}},
(29,20)*{\text{\small $-$}},
(-4,0)*{\text{\small+}},
(-4,10)*{\text{\small $-$}},
(-4,20)*{\text{\small+}},
(15,-8)*{\text{(b) $Q(D_5,A_3)$}},
\ar@{->} (8,20);(2,20) 
\ar@{->} (12,20);(18,20) 
\ar@{->} (27,18);(22,19.5) 
\ar@{->} (2,10);(8,10) 
\ar@{->} (18,10);(12,10) 
\ar@{->} (22,9.5);(27,8) 
\ar@{->} (8,0);(2,0) 
\ar@{->} (12,0);(18,0) 
\ar@{->} (27,-2);(22,-0.5) 
\ar@{->} (30,2.5);(22,0.5) 
\ar@{<-} (30,12.5);(22,10.5) 
\ar@{->} (30,22.5);(22,20.5) 
\ar@{->} (0,2);(0,8) 
\ar@{->} (10,8);(10,2) 
\ar@{->} (20,2);(20,8) 
\ar@{->} (29,5);(29,-1) 
\ar@{->} (32,11);(32,5) 
\ar@{->} (0,18);(0,12) 
\ar@{->} (10,12);(10,18) 
\ar@{->} (20,18);(20,12) 
\ar@{->} (29,9);(29,15) 
\ar@{->} (32,15);(32,21) 
\end{xy}
$$
\vskip-5pt
\caption{Examples of quiver $Q(X,X')$.
Black and white    vertices belong to $ V_+$ and $ V_-$, respectively.
}
\label{fig:prod1}
\end{figure}

{\bf Step 2. Composite mutations $\mu_+$ and  $\mu_-$.}
We consider the $Y$-pattern $\bfUpsilon=\{ \Upsilon_t\}$
of a cluster pattern with free coefficients
at $t_0$,
where
the initial $Y$-seed is given by $\Upsilon_{t_0}=\Upsilon=(\bfy, Q)$
with $Q=Q(X, X')$.
Let $V_+$ and $V_{-}$ be the sets of vertices  $v$ of the quiver $Q$ with sign
$\kappa_v$ being
 $+$
and $-$, respectively.
In  Figure \ref{fig:prod1}, the black and white vertices
belong to $V_+$ and $V_-$, respectively.
Let  $p=|V_+|$, and let $v_1$, \dots, $v_p$ be all vertices in $V_+$.
Similarly,
let  $q=|V_-|$, and let $u_1$, \dots, $u_q$ be all vertices in $V_-$.
Then, the following facts hold:
\begin {enumerate}
\item
By the construction of $Q$, there is no arrow between
the vertices $v$ and $v'$ with the same sign.
Therefore,
by Lemma \ref{lem:bcommut1},
the sequence of mutations
\begin{align}
\mu_+(\bfy,Q):=\mu_{v_p} \cdots \mu_{v_1}(\bfy,Q)
\end{align}
is independent of the order of the product.
\item
Let $(\bfy',Q')=\mu_+(\bfy,Q)$.
Then, $Q'=Q^{\op}$ holds in the same way as the quiver $Q'=\mu_6\mu_4\mu_2(Q)$
in Example \ref{ex:QB1}.
\item
For the above seed $(\bfy',Q^{\op})$,  we apply the sequence of mutations
\begin{align}
\mu_-:=\mu_{u_q} \cdots \mu_{u_1}.
\end{align}
Again, by Lemma \ref{lem:bcommut1},
it is independent of the order of the product.
\item
Let $(\bfy'',Q'')=\mu_-(\bfy',Q^{\op})$.
Then, $Q''=Q$ holds 
in the same way as the quiver $Q''=\mu_5\mu_3\mu_1(Q^{\op})$
in Example \ref{ex:QB1}.
Note that this periodicity is only for quivers,
and $\bfy''=\bfy$ does not hold.
\end{enumerate}

{\rp In the case $(X,X')=(A_1,A_1)$, 
one of $V_+$ or $V_-$ is empty.
If $V_+$ is empty, for example,
we regard $\mu_+=\rmid$. All formulas below are still applicable.}

{\bf Step 3. Identification with $Y$-system.}
For  the above $Y$-pattern,
 consider the following sequence of (composite) mutations
\begin{align}
\label{eq:Yseq1}
\begin{split}
\cdots
\
{\buildrel + \over \rightarrow}
\
(\bfy(-1), Q^{\op})
\
{\buildrel - \over \rightarrow}
&
\
(\bfy(0), Q)
\
{\buildrel + \over \rightarrow}
\
(\bfy(1), Q^{\op})
\
{\buildrel - \over \rightarrow}
\
(\bfy(2), Q)
\
{\buildrel + \over \rightarrow}
\
\cdots,
\end{split}
\end{align}
where $(\bfy(0), Q)=(\bfy, Q)$ is the initial $Y$-seed,
and
$\buildrel + \over \rightarrow$ and $\buildrel - \over \rightarrow$
represent $\mu_+$ and $\mu_-$, respectively.
Let us write $y$-variables 
as  $\bfy(s)=(y_{a,a'}(s))_{(a,a')\in I\times I'}$.
Also, let us write the addition $\oplus$ in $\bbQ_{\rmsf}(\bfy)$ as $+$
as in the previous chapter.
Then, in  the same way as \eqref{eq:ymutex1} and \eqref{eq:ymutex2},
we have,
for even $s$,
\begin{align}
\label{eq:yma1}
y_{a,a'}(s+1)=
\begin{cases}
y_{a,a'}(s)^{-1}
& \kappa_{a,a'}=+,
\\
y_{a,a'}(s)
\frac{
\displaystyle
\prod_{b\in I:\, b\sim a}(1+y_{b,a'}(s))
}
{
\displaystyle
\prod_{b'\in I':\, b'\sim a'}(1+y_{a,b'}(s)^{-1})
}
& \kappa_{a,a'}=-,
\\
\end{cases}
\end{align}
and, for odd $s$,
\begin{align}
\label{eq:yma2}
y_{a,a'}(s+1)=
\begin{cases}
\displaystyle
y_{a,a'}(s)
\frac{
\displaystyle
\prod_{b\in I:\, b\sim a}(1+y_{b,a'}(s))
}
{
\displaystyle
\prod_{b'\in I':\, b'\sim a'}(1+y_{a,b'}(s)^{-1})
}
& \kappa_{a,a'}=+,
\\
y_{a,a'}(s)^{-1}
& \kappa_{a,a'}=-,
\end{cases}
\end{align}
where the symbol $\sim$ denotes  the adjacency
in  $X$ in the numerator
and in  $X'$ in the denominator, respectively.
Now we separate the $y$-variables into the \emph{even} and the \emph{odd sectors}
$\calY_+$,  $\calY_-$ by
\begin{align}
\label{eq:ycal1}
\calY_{\pm} &:=\{ y_{a,a'}(s) \mid \kappa_{a,a'}(-1)^s=\pm1 \}.
\end{align}
Then, by \eqref{eq:yma1} and \eqref{eq:yma2},
 the variables in $\calY_+$ and  $\calY_-$ are related as
\begin{align}
\label{eq:yyinv1}
y_{a,a'}(s+1)=y_{a,a'}(s)^{-1}
\quad (y_{a,a'}(s)\in \calY_+,\  y_{a,a'}(s+1)\in \calY_-).
\end{align}
By using this relation,
both relations  \eqref{eq:yma1} and \eqref{eq:yma2} are
summarized into a single relation
\begin{align}
\label{eq:yma3}
y_{a,a'}(s+1)y_{a,a'}(s-1)=
\frac
{\displaystyle
\prod_{b\in I:\, b\sim a}(1+y_{b,a'}(s))}
{
\displaystyle
\prod_{b'\in I':\, b'\sim a'}(1+y_{a,b'}(s)^{-1})
},
\end{align}
where all $y$-variables are in the even sector $\calY_+$.
This is exactly the $Y$-system \eqref{eq:ysys1}
of type $(X,X')$.
On the other hand, by \eqref{eq:yyinv1} and \eqref{eq:yma3}, $y$-variables
in the odd sector $\calY_-$ satisfy the ``dual" relations
\begin{align}
\label{eq:yma4}
y_{a,a'}(s+1)y_{a,a'}(s-1)=
\frac
{
\displaystyle
\prod_{b'\in I':\, b'\sim a'}(1+y_{a,b'}(s))
}
{\displaystyle
\prod_{b\in I:\, b\sim a}(1+y_{b,a'}(s)^{-1})}
.
\end{align}

Meanwhile, we observe that the $Y$-system \eqref{eq:ysys1} itself
is decoupled into two sectors.
Namely,
let us separate the variables $Y_{a,a'}(u)$ therein into the even and the odd sectors
$Y_+$, $Y_-$ by
\begin{align}
\label{eq:Yp1}
Y_{\pm} &:=\{ Y_{a,a'}(u) \mid \kappa_{a,a'}(-1)^s=\pm1\}.
\end{align}
In contrast to \eqref{eq:yyinv1},
there is no relation between the variables in $Y_+$ and the ones in $Y_-$
in the  $Y$-system \eqref{eq:ysys1}.
Therefore,
the $Y$-system can be restricted on each sector.

In summary, the $Y$-system \eqref{eq:ysys1} restricted on the even sector $Y_+$
is ``embedded''
in the ambient $Y$-pattern,
where the variables $Y_{a,a'}(u)\in Y_+$ for the $Y$-system
are identified with the $y$-variables $y_{a,a'}(u)\in \calY_+$ 
in the sequence of mutations \eqref{eq:Yseq1}.

\section{DIs for $Y$-systems}

For each simply-laced Dynkin diagram $X$,
we introduce the following diagram automorphism $\omega:I \rightarrow I$
of $X$,
\begin{equation}
\label{eq:auto1}
\begin{alignedat}{2}
&\text{for $A_r$}
\quad
&&
\omega:  a \leftrightarrow r+1-a,
\\
&\text{for $D_r$ ($r$: odd)}
\quad
&&
\omega:   a \mapsto a\ (a\neq r-1,\ r),\
r-1\leftrightarrow r,
\\
&\text{for $E_6$}
\quad
&&
 \omega:  {\rp 6 \mapsto 6,\
1\leftrightarrow 5,\
2 \leftrightarrow 4, }
\\
&\text{otherwise}
\quad
&&
\omega=\id,
\end{alignedat}
\end{equation}
where we used the labels of $I$ in Figure \ref{fig:Dynkin1}.
The intrinsic meaning of $\omega$ will be clarified soon in \eqref{eq:long1}.
Note that $\omega$ is an involution for any $X$.

The following periodicity was proved for $(X,X')=(X,A_1)$
by Fomin-Zelevinsky \cite{Fomin03b, Fomin07},
and in general
by  Keller \cite{Keller08,Keller10} (for full periodicity)
and Inoue-Iyama-Kuniba-Nakanishi-Suzuki \cite{Inoue10c} (for half periodicity).

\begin{thm}[{\cite[Thm.~1.1]{Fomin03b}, \cite[Thm.~8.2]{Keller08}, \cite[Thm.\ 2.3]{Keller10}, \cite[Cor.\ 4.28]{Inoue10c}}]
\label{thm:Yperiod1}
For the sequence of mutations \eqref{eq:Yseq1},
the following periodicities hold:
\begin{alignat}{2}
\label{eq:half1}
&\text{(half periodicity)} 
&\quad y_{a,a'}(s+h+h')&=y_{\omega(a),\omega'(a')}(s),
\\
\label{eq:full1}
&\text{(full periodicity)} 
&\quad y_{a,a'}(s+2(h+h'))&=y_{a,a'}(s),
\end{alignat}
where $h$, $h'$ are the Coxeter numbers,
and $\omega$, $\omega'$ are the diagram automorphisms
\eqref{eq:auto1} of $X$ and $X'$, respectively.
(When  $\omega=\omega'=\id$,
the half periodicity \eqref{eq:half1} is already a full periodicity.)
\end{thm}

The periodicity \eqref{eq:full1}
 is equivalent to the one for the $Y$-systems 
in Conjecture \ref{conj:DIY1} under  the correspondence in the previous section.
Temporarily assume Theorem \ref{thm:Yperiod1}.
Then,
thanks to Theorem \ref{1thm:synchro1},
or by $Q(2(h+h'))=Q$ more directly,
the periodicity \eqref{eq:full1}
is lifted to the periodicity of $Y$-seeds
\begin{align}
\Upsilon(s+2(h+h'))=\Upsilon(s).
\end{align}
Thus, by  Theorem \ref{thm:DI1},
there is a DI associated with this periodicity.
Note that $y_{k_s}(s)$ appearing in the DI are 
variables in the even sector $\calY_+$.
Also, we take $D=I$ because the $Y$-pattern is skew-symmetric.
After evaluating the constant term,
we obtain the following DI.
\begin{thm}[{\cite[Thm.~2.8]{Nakanishi09}}]
\label{thm:YDI1}
Suppose that the initial $y$-variables $\bfy$ take values in $\bbR_{>0}^n$.
Then, the following DIs hold.
\begin{align}
\label{eq:YDI1}
\sum_{s=0}^{2(h+h')-1}
\sum_{
\scriptstyle
(a,a')\in I \times I'
\atop
\scriptstyle
 \kappa_{a,a'}(-1)^s=1}
\tilde L (y_{a,a'}(s))
=
hrr'
\frac{\pi^2}{6}.
\end{align}
\end{thm}

Taking into account the following fine points,
this is indeed the DI \eqref{eq:Y3} for the $Y$-system:
\begin{itemize}
\item
The constant term of \eqref{eq:YDI1} is \emph{half} of
the one of \eqref{eq:Y3}  because 
the sum in \eqref{eq:YDI1}
is restricted to the even sector $\calY_+$,
which corresponds to the even sector $Y_+$
of the $Y$-system.
\item 
By \eqref{eq:yyinv1},
the initial $y$-variables $y_{a,a'}(0)$ are identified with the initial
variables of the $Y$-system in the even sector $Y_+$  as
\begin{align}
y_{a,a'}(0)=
\begin{cases}
Y_{a,a'}(0) & \kappa_{a,a'}=+,
\\
Y_{a,a'}(-1)^{-1} & \kappa_{a,a'}=-.
\end{cases}
\end{align}
\end{itemize}

\begin{cor}{\cite{Keller08, Keller10, Nakanishi09}}
Conjecture \ref{conj:DIY1} is true.
\end{cor}

In the rest of the chapter, we present the proofs of
Theorem \ref{thm:Yperiod1} and
the counting of the constant term
in Theorem  \ref{thm:YDI1}.
The main idea is to use the tropicalization.
By  Theorem \ref{thm:detrop1},
to prove the periodicity of  $y$-variables $y_{a,a'}(u)$,
it is enough to  prove the periodicity of \emph{tropical $y$-variables}
$[y_{a,a'}(u)]$.
Remarkably, 
the dynamics of tropical $y$-variables is governed by
the \emph{Coxeter elements} of the root systems.
This mechanism was found  for the $Y$-systems of type $(X,A_1)$ in \cite{Fomin03b, Fomin07},
and  generalized to the   $Y$-systems of type $(X,X')$ in \cite{Nakanishi09} .

\section{Coxeter elements}

\label{sec:Coxeter1}

This section summarizes the facts on the {Coxeter elements} we will use.
The presentation here is a minimal one.
See \cite{Bourbaki02, Humphreys90} for further information.
 
 Let $X$ be a simply-laced  Dynkin diagram of rank $r$.
Let $\Phi=\Phi(X)\subset V$ be the  {root system} corresponding to $X$,
where  $V$ is the ambient  real vector space  spanned by $\Phi$ with an inner product $(\cdot, \cdot)$.
Let $\alpha_1$, \dots, $\alpha_r \in \Phi$ be the   simple roots
parametrized by the index set $I=\{1,\, \dots,\,  r\}$ as in Figure \ref{fig:Dynkin1}.
We normalize them as $(\alpha_a,\alpha_a)=2$.
Then,
they satisfy, for $a\neq b$,
\begin{align}
\label{eq:Cartan1}
 (\alpha_a, \alpha_b)=
\begin{cases}
-1 & a\sim b,
\\
0 & a\not\sim b,
\end{cases}
\end{align}
where $\sim$ means the adjacency in the diagram $X$ as before.
Let $\Phi_{+}$ and $\Phi_{-}$ be the sets of
\emph{positive} and \emph{negative} roots, respectively,
so that we have $\Phi=\Phi_+\sqcup \Phi_-$.
The explicit description of $\Phi_+$ is found
in \cite{Bourbaki02}.

\begin{ex}[{\cite[Plates]{Bourbaki02}}]
(1) Type $A_r$.
The positive roots are given by
\begin{align}
[i]:=\alpha_i
\ (1\leq i \leq r),
\quad
[i,j]:=\alpha_i + \cdots +\alpha_j
\ (1\leq i < j \leq r).
\end{align}
Thus, we have $|\Phi_+|=r(r+1)/2$ and $|\Phi|=r(r+1)$.

(2) Type $D_r$.
The positive roots are given by
\begin{align}
\begin{split}
[i]&:=\alpha_i
\ (1\leq i \leq r),
\\
[i,j]&:=\alpha_i + \cdots +\alpha_j
\ (1\leq i < j \leq r-2),
\\
[i;r-1]&:=\alpha_i + \cdots +\alpha_{r-2}+\alpha_{r-1}
\ (1\leq i  \leq r-2),
\\
[i;r]&:=\alpha_i + \cdots +\alpha_{r-2}+\alpha_{r}
\ (1\leq i  \leq r-2),
\\
[i;r-1,r]&:=\alpha_i + \cdots +\alpha_{r-2}+\alpha_{r-1}+\alpha_{r}
\ (1\leq i  \leq r-2),
\\
\{i,j\}&:=\alpha_i + \cdots +\alpha_{j-1}+2\alpha_{j}+
\cdots +
2 \alpha_{r-2}+\alpha_{r-1}+ \alpha_r
\\
&\hskip150pt (1\leq i < j \leq r-2),
\end{split}
\end{align}
Thus, we have $|\Phi_+|=r(r-1)$ and $|\Phi|=2r(r-1)$.

(3) Types  $E_6$, $E_7$, $E_8$.
The explicit description of the positive roots is case-by-case.
For example, for type $E_6$, besides the ones from subdiagrams of type $A$ and $D$,
there are seven positive roots
\begin{align}
\begin{picture}(22,12)(0,4)
\put(0,0){\small 1}
\put(5,0){\small 1}
\put(10,0){\small 1}
\put(15,0){\small 1}
\put(20,0){\small 1}
\put(10,8){\small 1}
\end{picture}
\quad
\begin{picture}(22,12)(0,4)
\put(0,0){\small 1}
\put(5,0){\small 1}
\put(10,0){\small 2}
\put(15,0){\small 1}
\put(20,0){\small 1}
\put(10,8){\small 1}
\end{picture}
\quad
\begin{picture}(22,12)(0,4)
\put(0,0){\small 1}
\put(5,0){\small 2}
\put(10,0){\small 2}
\put(15,0){\small 1}
\put(20,0){\small 1}
\put(10,8){\small 1}
\end{picture}
\quad
\begin{picture}(22,12)(0,4)
\put(0,0){\small 1}
\put(5,0){\small 1}
\put(10,0){\small 2}
\put(15,0){\small 2}
\put(20,0){\small 1}
\put(10,8){\small 1}
\end{picture}
\quad
\begin{picture}(22,12)(0,4)
\put(0,0){\small 1}
\put(5,0){\small 2}
\put(10,0){\small 2}
\put(15,0){\small 2}
\put(20,0){\small 1}
\put(10,8){\small 1}
\end{picture}
\quad
\begin{picture}(22,12)(0,4)
\put(0,0){\small 1}
\put(5,0){\small 2}
\put(10,0){\small 3}
\put(15,0){\small 2}
\put(20,0){\small 1}
\put(10,8){\small 1}
\end{picture}
\quad
\begin{picture}(22,12)(0,4)
\put(0,0){\small 1}
\put(5,0){\small 2}
\put(10,0){\small 3}
\put(15,0){\small 2}
\put(20,0){\small 1}
\put(10,8){\small 2}
\end{picture}
\end{align}
 Here we only record the fact
$|\Phi|=72$, $126$, $240$ for $E_6$, $E_7$, $E_8$, respectively.
\end{ex}

To each simple root $\alpha_a$,
we assign a linear map 
$s_{a}:V \rightarrow V$ called a \emph{simple reflection}\index{simple reflection} by
 (under the  normalization  $(\alpha_a,\alpha_a)=2$)
\begin{align}
\label{eq:ref1}
s_{a}(v)=v- (v,\alpha_a)\alpha_a.
\end{align}
The group $W=W(X)$ generated by all simple reflections
is caleed the \emph{Weyl group} associated with the root system $\Phi=\Phi(X)$.
We have $W(\Phi)=\Phi$; it is indeed the defining condition of a root system.

Let $w_0$ be the \emph{longest element}\index{longest element (of Weyl group)} of $W$;
namely, it has a reduced expression of the longest length in $W$.
It is  known that $w_0$ is  unique,
and its length is $|\Phi_+|$ \cite[\S1.8]{Humphreys90}.
Explicitly, it is described as follows \cite[Plates]{Bourbaki02}:
\begin{align}
\label{eq:long1}
w_0(\alpha_a) = - \alpha_{\omega(a)},
\end{align}
where $\omega$ is the diagram automorphism defined in \eqref{eq:auto1}.
(This gives an intrinsic meaning of  $\omega$.)
We have $w_0^2=\id$.

Let us introduce some distinguished elements in $W$.
\begin{defn}
A \emph{Coxeter element}\index{Coxeter!element} $c\in W$ is a product of all simple reflections,
where each reflection appears exactly once and the order of the product is arbitrary.
\end{defn}

A Coxeter element $c$ \emph{does} depend on the order of the product.
Thus, it is not unique.
However, all Coxeter elements are known to be conjugate to each other in $W$ \cite[Prop.~3.16]{Humphreys90}.
Therefore, the following definition makes sense.
\begin{defn}
The \emph{Coxeter number}\index{Coxeter!number} $h=h(X)$ of the Dynkin diagram $X$
is the order of any Coxeter element $c$.
\end{defn}

The following fact due to Coxeter is well-known and already mentioned in
Conjecture \ref{conj:KBR1}.
\begin{thm}[{\cite[Prop.~3.18]{Humphreys90}}]
The following equality holds:
\begin{align}
\label{eq:Cox1}
|\Phi|=h r.
\end{align}
Explicitly, the Coxeter number is  given 
by 
$h=r+1$ for $A_r$, $2r-2$ for $D_r$, and 12, 18, 30 for $E_6$, $E_7$, $E_8$.
In particular,  $h$ is odd if and only if $X=A_r$ ($r$: even).
\end{thm}

We are especially interested in the following choice of Coxeter elements.
We put the sign $\kappa_a$ to each vertex $a$ of $X$ as we did in Section \ref{sec:Ysystems1},
so that
any adjacent vertices in $X$ have opposite signs.
Then, for any $a$ and $b$ with $\kappa_a=\kappa_b$,
$s_a$ and $s_b$ commute.
So, we have the products $s_+$ and $s_-$ by
\begin{align}
s_{\pm}=\prod_{
a\in I;\
\kappa_a=\pm 1
} s_a,
\end{align}
where both are independent of the order of the product.
In the special case of $X=A_1$, one of $s_{\pm}$ is the identity.
Then, $s_- s_+$ and $s_+ s_-$ are Coxeter elements.
By \eqref{eq:Cartan1} and \eqref{eq:ref1},
we have
\begin{align}
\label{eq:aa2}
 s_{\kappa_a} (\alpha_a)=-\alpha_a,
 \quad
 s_{-{\kappa_a}} (\alpha_a)=\alpha_a +  \sum_{b\in I:\, b\sim a} \alpha_b.
 \end{align}
For any positive even integer $k$, the powers $(s_+ s_-)^{k/2}$
and $(s_- s_+)^{k/2}$ are defined in the ordinary way.
For any  positive  odd integer $k$, we define the powers as
\begin{align}
(s_+ s_-)^{k/2}:= \underbrace{s_- s_+ \cdots s_-}_{\text{$k$ terms}},
\quad
(s_-s_+)^{k/2}:= 
\underbrace{s_+ s_- \cdots s_+}_{\text{$k$ terms}}.
\end{align}

The following formula is regarded as the ``half periodicity'' of the Coxeter elements $s_-s_+$
and $s_+ s_-$.
\begin{prop}
[{\cite[V.~Exercise \S6.2]{Bourbaki02}, \cite[\S 3.19 Exercise 2]{Humphreys90}}]
The following formula holds.
\label{prop:Cw1}
\begin{align}
\label{eq:ssw1}
(s_-s_+)^{h/2}=(s_+s_-)^{h/2}=w_0.
\end{align}
Moreover, both products are reduced expressions of $w_0$.
\end{prop}

Indeed, they are reduced expressions because the number of the simple reflections
in each product is  $ hr/2=|\Phi_+|$.

We define a family of  roots $\alpha(a;k)\in \Phi$ ($a\in I,\ -1 \leq k \leq h$) by
\begin{align}
\alpha(a;-1)&=-\alpha_a.
\\
\label{eq:salpha2}
\alpha(a;k)&=
\begin{cases}
(s_+s_-)^{k/2}(\alpha_a)
&
\kappa_a=+,
\\
(s_-s_+)^{k/2}(\alpha_a)
&
\kappa_a=-
\end{cases}
\quad
(0\leq k \leq h).
\end{align}
The following result due to \cite{Fomin03b, Fomin07} is crucial to our problem.
So, we provide a proof.
(Our definition of $\alpha(a;k)$ is slightly different from the one in \cite{Fomin03b, Fomin07}.)
\begin{prop}
\label{prop:alpha1}
(a) { \cite[Prop.~2.5]{Fomin03b}}.
For $0\leq k \leq h-1$, $\alpha(a;k)$ is a positive root.
\par
(b).
\begin{align}
\label{eq:alpha1}
\alpha(a;h-1)=\alpha_{\omega(a)},
\quad
\alpha(a;h)=-\alpha_{\omega(a)}.
\end{align}
\par
(c)
{\cite[Eq.~(10.9)]{Fomin07}}.
For $0\leq k \leq h-1$, the following relation holds:
\begin{align}
\label{eq:alpha2}
\alpha(a;k+1)+\alpha(a;k-1)
=
\sum_{b\in I:\, b\sim a}
\alpha(b;k).
\end{align}
In particular, $\alpha(a; k)$ $(1\leq k \leq h)$ are uniquely determined
by the initial condition
\begin{align}
\label{eq:alpha3}
\alpha(a;-1)=- \alpha_a,
\quad
\alpha(a;0)= \alpha_a.
\end{align}
and the recurrence relation \eqref{eq:alpha2}.
\end{prop}
\begin{proof}
(a).
Let us concentrate on the case $h$ is even and $\kappa_a=+$.
(The other cases are similar.)
The following facts are known or checked easily:
{\rp
\begin{itemize}
\item[(i).]
If $h$ is even,  $\kappa_{\omega(a)}=\kappa_a$.
If $h$ is odd, $\kappa_{\omega(a)}=-\kappa_a$.
\item[(ii).]
A positive root $\alpha$ is sent to a negative root by
a simple reflection $s_a$ if and only if $\alpha=\alpha_a$  \cite[Prop.~1.4]{Humphreys90}.
 \item[(iii).]
For any reduced expression $w=s_{a_1}\cdots s_{a_m}$,
there is no pair $1\leq i<j\leq m$ such that
$\alpha_{a_i}=s_{a_{i+1}}\cdots s_{a_{j-1}}(\alpha_{a_j})$
\cite[Thm.~1.7]{Humphreys90}.
\end{itemize}
Since we assume $\kappa_a=+$, we have $\kappa_{\omega(a)}=+$ by (i).
Thus, we have $s_+(\alpha_{\omega(a)})=-\alpha_{\omega(a)}$ by \eqref{eq:aa2}.
Then, by \eqref{eq:long1}, the result \eqref{eq:ssw1} is rewritten as
\begin{align}
(s_+ s_-)^{(h-1)/2}(\alpha_a)=\alpha_{\omega(a)}.
\end{align}
Let $(s_+ s_-)^{(h-1)/2}=\cdots s_{b_3} s_{b_2}s_{b_1}$.
Suppose that the statement (a) fails.
Then,
by (ii), there is some  $p$ such that
$s_{b_p}\cdots s_{b_1}(\alpha_a)=\alpha_{b_{p+1}}$.
Since the elements in $s_+$ is commutative, we may assume that the leftmost element of $s_+$ is $s_a$.
Then, $ s_{b_p}\cdots s_{b_1}s_a$ is a  subproduct of $(s_- s_+)^{h/2}$.
By (iii), this contradicts the non-reducibility of $(s_-s_+)^{h/2}$ in  Proposition \ref{prop:Cw1}.}

(b). The second equality is true by Proposition \ref{prop:Cw1}.
The first equality follows from the second one.
For example, if $h$ is even, 
$\kappa_{\omega(a)}=\kappa_a$. Thus,
\begin{align}
\alpha(a;h-1)=s_{\kappa_a}(\alpha(a;h))=s_{\kappa_a}(-\alpha_{\omega(a)})=
\alpha_{\omega(a)}.
\end{align}

(c).
We prove it by the induction on $k$.
 The case $k=0$ is nothing but the second equality in
 \eqref{eq:aa2}.
Suppose that the relation \eqref{eq:alpha2} holds for some $k$ ($1\leq k \leq h-1$).
 Suppose that $\kappa_a=+$ and $k$ is even. (The other cases are similar.)
 Note that $\kappa_b=-$ for the one in the RHS of  \eqref{eq:aa2}.
 Then, we have
 \begin{align}
 \begin{split}
 s_+ (\alpha(a;k+1))&=\alpha(a;k+2),
 \quad
 s_+ (\alpha(a;k-1))=\alpha(a;k),
 \\
 s_+ (\alpha(b;k))&=\alpha({\rp b};k+1).
 \end{split}
 \end{align}
Thus, we obtain the relation \eqref{eq:alpha2} for $k+1$
from the one for $k$.
\end{proof}

\begin{ex}
\label{ex:Coxeter1}
Let us give examples to illustrate Proposition \ref{prop:alpha1}.
Let $\buildrel + \over \rightarrow$ and $\buildrel - \over \rightarrow$
represent $s_+$ and $s_-$, respectively.
We choose the sign $\kappa_1=+$ for all examples.
\par
(1) Type $A_5$. $h=6$.
We have the following sequences.
\begin{align}
\begin{split}
&
[1]
\buildrel - \over \rightarrow
[1,2]
\buildrel + \over \rightarrow
[2,3]
\buildrel - \over \rightarrow
[3,4]
\buildrel + \over \rightarrow
[4,5]
\buildrel - \over \rightarrow
[5],
\\
&
[2]
\buildrel + \over \rightarrow
[1,3]
\buildrel - \over \rightarrow
[1,4]
\buildrel + \over \rightarrow
[2,5]
\buildrel - \over \rightarrow
[3,5]
\buildrel + \over \rightarrow
[4],
\\
&
[3]
\buildrel - \over \rightarrow
[2,4]
\buildrel + \over \rightarrow
[1,5]
\buildrel - \over \rightarrow
[1,5]
\buildrel + \over \rightarrow
[2,4]
\buildrel - \over \rightarrow
[3].
\end{split}
\end{align}
The sequences for $[4]$ and $[5]$ are also found in  the above
 by reading them in reverse. 
As an example of the relation \eqref{eq:alpha2}, for $a=2$ and $k=3$,
we have $[3,5]+[1,4]=[3,4]+[1,5]$.
\par
(2) Type $A_6$. $h=7$.
We have the following sequences.
\begin{align}
\begin{split}
&
[1]
\buildrel - \over \rightarrow
[1,2]
\buildrel + \over \rightarrow
[2,3]
\buildrel - \over \rightarrow
[3,4]
\buildrel + \over \rightarrow
[4,5]
\buildrel - \over \rightarrow
[5,6]
\buildrel + \over \rightarrow
[6],
\\
&
[2]
\buildrel + \over \rightarrow
[1,3]
\buildrel - \over \rightarrow
[1,4]
\buildrel + \over \rightarrow
[2,5]
\buildrel - \over \rightarrow
[3,6]
\buildrel + \over \rightarrow
[4,6]
\buildrel - \over \rightarrow
[5],
\\
&
[3]
\buildrel - \over \rightarrow
[2,4]
\buildrel + \over \rightarrow
[1,5]
\buildrel - \over \rightarrow
[1,6]
\buildrel + \over \rightarrow
[2,6]
\buildrel - \over \rightarrow
[3,5]
\buildrel + \over \rightarrow
[4].
\end{split}
\end{align}
(3) Type $D_5$. $h=8$.
We have the following sequences.
\begin{align}
\begin{split}
&
[1]
\buildrel - \over \rightarrow
[1,2]
\buildrel + \over \rightarrow
[2,3]
\buildrel - \over \rightarrow
[3;4,5]
\buildrel + \over \rightarrow
[3;4,5]
\buildrel - \over \rightarrow
[2,3]
\buildrel + \over \rightarrow
[1,2]
\buildrel - \over \rightarrow
[1],
\\
&
[2]
\buildrel + \over \rightarrow
[1,3]
\buildrel - \over \rightarrow
[1;4,5]
\buildrel + \over \rightarrow
\{2,3\}
\buildrel - \over \rightarrow
\{2,3\}
\buildrel + \over \rightarrow
[1;4,5]
\buildrel - \over \rightarrow
[1,3]
\buildrel + \over \rightarrow
[2],
\\
&
[3]
\buildrel - \over \rightarrow
[2;4,5]
\buildrel + \over \rightarrow
\{1,3\}
\buildrel - \over \rightarrow
\{1,2\}
\buildrel + \over \rightarrow
\{1,2\}
\buildrel - \over \rightarrow
\{1,3\}
\buildrel + \over \rightarrow
[2;4,5]
\buildrel - \over \rightarrow
[3],
\\
&
[4]
\buildrel + \over \rightarrow
[3;4]
\buildrel - \over \rightarrow
[2;5]
\buildrel + \over \rightarrow
[1;5]
\buildrel - \over \rightarrow
[1;4]
\buildrel + \over \rightarrow
[2;4]
\buildrel - \over \rightarrow
[3;5]
\buildrel + \over \rightarrow
[5].
\end{split}
\end{align}
The automorphism $\omega$ is not trivial as given in \eqref{eq:auto1}.
As an example of the relation \eqref{eq:alpha2}, for $a=3$ and $k=3$,
we have $\{1,2\}+\{1,3\}=\{2,3\}+[1;5]+[1;4]$.
\par
(4) Type $D_6$. $h=10$.
We have the following sequences.
\begin{align}
\begin{split}
&
[1]
\buildrel - \over \rightarrow
[1,2]
\buildrel + \over \rightarrow
[2,3]
\buildrel - \over \rightarrow
[3,4]
\buildrel + \over \rightarrow
[4;5,6]
\buildrel - \over \rightarrow
[4;5,6]
\buildrel + \over \rightarrow
\cdots
\buildrel - \over \rightarrow
[1],
\\
&
[2]
\buildrel + \over \rightarrow
[1,3]
\buildrel - \over \rightarrow
[1,4]
\buildrel + \over \rightarrow
[2;5,6]
\buildrel - \over \rightarrow
\{3,4\}
\buildrel + \over \rightarrow
\{3,4\}
\buildrel - \over \rightarrow
\cdots
\buildrel + \over \rightarrow
[2],
\\
&
[3]
\buildrel - \over \rightarrow
[2,4]
\buildrel + \over \rightarrow
[1;5,6]
\buildrel - \over \rightarrow
\{1,4\}
\buildrel + \over \rightarrow
\{2,3\}
\buildrel - \over \rightarrow
\{2,3\}
\buildrel + \over \rightarrow
\cdots
\buildrel - \over \rightarrow
[3],
\\
&
[4]
\buildrel + \over \rightarrow
[3;5,6]
\buildrel - \over \rightarrow
\{2,4\}
\buildrel + \over \rightarrow
\{1,3\}
\buildrel - \over \rightarrow
\{1,2\}
\buildrel + \over \rightarrow
\{1,2\}
\buildrel - \over \rightarrow
\cdots
\buildrel + \over \rightarrow
[4],
\\
&
[5]
\buildrel - \over \rightarrow
[4;5]
\buildrel + \over \rightarrow
[3;6]
\buildrel - \over \rightarrow
[2;6]
\buildrel + \over \rightarrow
[1;5]
\buildrel - \over \rightarrow
[1;5]
\buildrel + \over \rightarrow
\cdots
\buildrel - \over \rightarrow
[5],
\\
&
[6]
\buildrel - \over \rightarrow
[4;6]
\buildrel + \over \rightarrow
[3;5]
\buildrel - \over \rightarrow
[2;5]
\buildrel + \over \rightarrow
[1;6]
\buildrel - \over \rightarrow
[1;6]
\buildrel + \over \rightarrow
\cdots
\buildrel - \over \rightarrow
[6].
\end{split}
\end{align}
The automorphism $\omega$ is trivial as given in \eqref{eq:auto1}.

(5) Type $E_6$. $h=12$.
We have the following sequences.
\begin{align}
\begin{split}
\begin{picture}(22,12)(0,4)
\put(0,0){\small 1}
\put(5,0){\small 0}
\put(10,0){\small 0}
\put(15,0){\small 0}
\put(20,0){\small 0}
\put(10,8){\small 0}
\end{picture}
\
\buildrel - \over \rightarrow
\begin{picture}(22,12)(0,4)
\put(0,0){\small 1}
\put(5,0){\small 1}
\put(10,0){\small 0}
\put(15,0){\small 0}
\put(20,0){\small 0}
\put(10,8){\small 0}
\end{picture}
\
&
\buildrel + \over \rightarrow
\begin{picture}(22,12)(0,4)
\put(0,0){\small 0}
\put(5,0){\small 1}
\put(10,0){\small 1}
\put(15,0){\small 0}
\put(20,0){\small 0}
\put(10,8){\small 0}
\end{picture}
\
\buildrel - \over \rightarrow
\begin{picture}(22,12)(0,4)
\put(0,0){\small 0}
\put(5,0){\small 0}
\put(10,0){\small 1}
\put(15,0){\small 1}
\put(20,0){\small 0}
\put(10,8){\small 1}
\end{picture}
\
\buildrel + \over \rightarrow
\begin{picture}(22,12)(0,4)
\put(0,0){\small 0}
\put(5,0){\small 0}
\put(10,0){\small 1}
\put(15,0){\small 1}
\put(20,0){\small 1}
\put(10,8){\small 1}
\end{picture}
\
\buildrel - \over \rightarrow
\begin{picture}(22,12)(0,4)
\put(0,0){\small 0}
\put(5,0){\small 1}
\put(10,0){\small 1}
\put(15,0){\small 1}
\put(20,0){\small 1}
\put(10,8){\small 0}
\end{picture}
\
\buildrel + \over \rightarrow
\begin{picture}(22,12)(0,4)
\put(0,0){\small 1}
\put(5,0){\small 1}
\put(10,0){\small 1}
\put(15,0){\small 1}
\put(20,0){\small 0}
\put(10,8){\small 0}
\end{picture}
\\
&
\buildrel - \over \rightarrow
\begin{picture}(22,12)(0,4)
\put(0,0){\small 1}
\put(5,0){\small 1}
\put(10,0){\small 1}
\put(15,0){\small 0}
\put(20,0){\small 0}
\put(10,8){\small 1}
\end{picture}
\
\buildrel + \over \rightarrow
\begin{picture}(22,12)(0,4)
\put(0,0){\small 0}
\put(5,0){\small 1}
\put(10,0){\small 1}
\put(15,0){\small 0}
\put(20,0){\small 0}
\put(10,8){\small 1}
\end{picture}
\
\buildrel - \over \rightarrow
\begin{picture}(22,12)(0,4)
\put(0,0){\small 0}
\put(5,0){\small 0}
\put(10,0){\small 1}
\put(15,0){\small 1}
\put(20,0){\small 0}
\put(10,8){\small 0}
\end{picture}
\
\buildrel + \over \rightarrow
\begin{picture}(22,12)(0,4)
\put(0,0){\small 0}
\put(5,0){\small 0}
\put(10,0){\small 0}
\put(15,0){\small 1}
\put(20,0){\small 1}
\put(10,8){\small 0}
\end{picture}
\
\buildrel - \over \rightarrow
\begin{picture}(22,12)(0,4)
\put(0,0){\small 0}
\put(5,0){\small 0}
\put(10,0){\small 0}
\put(15,0){\small 0}
\put(20,0){\small 1}
\put(10,8){\small 0}
\end{picture}
\ ,
%
\\
\begin{picture}(22,12)(0,4)
\put(0,0){\small 0}
\put(5,0){\small 1}
\put(10,0){\small 0}
\put(15,0){\small 0}
\put(20,0){\small 0}
\put(10,8){\small 0}
\end{picture}
\
\buildrel + \over \rightarrow
\begin{picture}(22,12)(0,4)
\put(0,0){\small 1}
\put(5,0){\small 1}
\put(10,0){\small 1}
\put(15,0){\small 0}
\put(20,0){\small 0}
\put(10,8){\small 0}
\end{picture}
\
&
\buildrel - \over \rightarrow
\begin{picture}(22,12)(0,4)
\put(0,0){\small 1}
\put(5,0){\small 1}
\put(10,0){\small 1}
\put(15,0){\small 1}
\put(20,0){\small 0}
\put(10,8){\small 1}
\end{picture}
\
\buildrel + \over \rightarrow
\begin{picture}(22,12)(0,4)
\put(0,0){\small 0}
\put(5,0){\small 1}
\put(10,0){\small 2}
\put(15,0){\small 1}
\put(20,0){\small 1}
\put(10,8){\small 1}
\end{picture}
\
\buildrel - \over \rightarrow
\begin{picture}(22,12)(0,4)
\put(0,0){\small 0}
\put(5,0){\small 1}
\put(10,0){\small 2}
\put(15,0){\small 2}
\put(20,0){\small 1}
\put(10,8){\small 1}
\end{picture}
\
\buildrel + \over \rightarrow
\begin{picture}(22,12)(0,4)
\put(0,0){\small 1}
\put(5,0){\small 1}
\put(10,0){\small 2}
\put(15,0){\small 2}
\put(20,0){\small 1}
\put(10,8){\small 1}
\end{picture}
\
\buildrel - \over \rightarrow
\begin{picture}(22,12)(0,4)
\put(0,0){\small 1}
\put(5,0){\small 2}
\put(10,0){\small 2}
\put(15,0){\small 1}
\put(20,0){\small 1}
\put(10,8){\small 1}
\end{picture}
\\
&
\buildrel + \over \rightarrow
\begin{picture}(22,12)(0,4)
\put(0,0){\small 1}
\put(5,0){\small 2}
\put(10,0){\small 2}
\put(15,0){\small 1}
\put(20,0){\small 0}
\put(10,8){\small 1}
\end{picture}
\
\buildrel - \over \rightarrow
\begin{picture}(22,12)(0,4)
\put(0,0){\small 1}
\put(5,0){\small 1}
\put(10,0){\small 2}
\put(15,0){\small 1}
\put(20,0){\small 0}
\put(10,8){\small 1}
\end{picture}
\
\buildrel + \over \rightarrow
\begin{picture}(22,12)(0,4)
\put(0,0){\small 0}
\put(5,0){\small 1}
\put(10,0){\small 1}
\put(15,0){\small 1}
\put(20,0){\small 1}
\put(10,8){\small 1}
\end{picture}
\
\buildrel - \over \rightarrow
\begin{picture}(22,12)(0,4)
\put(0,0){\small 0}
\put(5,0){\small 0}
\put(10,0){\small 1}
\put(15,0){\small 1}
\put(20,0){\small 1}
\put(10,8){\small 0}
\end{picture}
\
\buildrel + \over \rightarrow
\begin{picture}(22,12)(0,4)
\put(0,0){\small 0}
\put(5,0){\small 0}
\put(10,0){\small 0}
\put(15,0){\small 1}
\put(20,0){\small 0}
\put(10,8){\small 0}
\end{picture}
\ ,
\\
\begin{picture}(22,12)(0,4)
\put(0,0){\small 0}
\put(5,0){\small 0}
\put(10,0){\small 1}
\put(15,0){\small 0}
\put(20,0){\small 0}
\put(10,8){\small 0}
\end{picture}
\
\buildrel - \over \rightarrow
\begin{picture}(22,12)(0,4)
\put(0,0){\small 0}
\put(5,0){\small 1}
\put(10,0){\small 1}
\put(15,0){\small 1}
\put(20,0){\small 0}
\put(10,8){\small 1}
\end{picture}
\
&
\buildrel + \over \rightarrow
\begin{picture}(22,12)(0,4)
\put(0,0){\small 1}
\put(5,0){\small 1}
\put(10,0){\small 2}
\put(15,0){\small 1}
\put(20,0){\small 1}
\put(10,8){\small 1}
\end{picture}
\
\buildrel - \over \rightarrow
\begin{picture}(22,12)(0,4)
\put(0,0){\small 1}
\put(5,0){\small 2}
\put(10,0){\small 2}
\put(15,0){\small 2}
\put(20,0){\small 1}
\put(10,8){\small 1}
\end{picture}
\
\buildrel + \over \rightarrow
\begin{picture}(22,12)(0,4)
\put(0,0){\small 1}
\put(5,0){\small 2}
\put(10,0){\small 3}
\put(15,0){\small 2}
\put(20,0){\small 1}
\put(10,8){\small 1}
\end{picture}
\
\buildrel - \over \rightarrow
\begin{picture}(22,12)(0,4)
\put(0,0){\small 1}
\put(5,0){\small 2}
\put(10,0){\small 3}
\put(15,0){\small 2}
\put(20,0){\small 1}
\put(10,8){\small 2}
\end{picture}
\
\buildrel + \over \rightarrow
\begin{picture}(22,12)(0,4)
\put(0,0){\small 1}
\put(5,0){\small 2}
\put(10,0){\small 3}
\put(15,0){\small 2}
\put(20,0){\small 1}
\put(10,8){\small 2}
\end{picture}
\\
&
\hskip137pt
\buildrel - \over \rightarrow
\cdots
\
\buildrel - \over \rightarrow
\begin{picture}(22,12)(0,4)
\put(0,0){\small 0}
\put(5,0){\small 0}
\put(10,0){\small 1}
\put(15,0){\small 0}
\put(20,0){\small 0}
\put(10,8){\small 0}
\end{picture}
\ ,
\\
\begin{picture}(22,12)(0,4)
\put(0,0){\small 0}
\put(5,0){\small 0}
\put(10,0){\small 0}
\put(15,0){\small 0}
\put(20,0){\small 0}
\put(10,8){\small 1}
\end{picture}
\
\buildrel + \over \rightarrow
\begin{picture}(22,12)(0,4)
\put(0,0){\small 0}
\put(5,0){\small 0}
\put(10,0){\small 1}
\put(15,0){\small 0}
\put(20,0){\small 0}
\put(10,8){\small 1}
\end{picture}
\
&
\buildrel - \over \rightarrow
\begin{picture}(22,12)(0,4)
\put(0,0){\small 0}
\put(5,0){\small 1}
\put(10,0){\small 1}
\put(15,0){\small 1}
\put(20,0){\small 0}
\put(10,8){\small 0}
\end{picture}
\
\buildrel + \over \rightarrow
\begin{picture}(22,12)(0,4)
\put(0,0){\small 1}
\put(5,0){\small 1}
\put(10,0){\small 1}
\put(15,0){\small 1}
\put(20,0){\small 1}
\put(10,8){\small 0}
\end{picture}
\
\buildrel - \over \rightarrow
\begin{picture}(22,12)(0,4)
\put(0,0){\small 1}
\put(5,0){\small 1}
\put(10,0){\small 1}
\put(15,0){\small 1}
\put(20,0){\small 1}
\put(10,8){\small 1}
\end{picture}
\
\buildrel + \over \rightarrow
\begin{picture}(22,12)(0,4)
\put(0,0){\small 0}
\put(5,0){\small 1}
\put(10,0){\small 2}
\put(15,0){\small 1}
\put(20,0){\small 0}
\put(10,8){\small 1}
\end{picture}
\
\buildrel - \over \rightarrow
\begin{picture}(22,12)(0,4)
\put(0,0){\small 0}
\put(5,0){\small 1}
\put(10,0){\small 2}
\put(15,0){\small 1}
\put(20,0){\small 0}
\put(10,8){\small 1}
\end{picture}
\\
&
\hskip137pt
\buildrel + \over \rightarrow
\cdots
\
\buildrel + \over \rightarrow
\begin{picture}(22,12)(0,4)
\put(0,0){\small 0}
\put(5,0){\small 0}
\put(10,0){\small 0}
\put(15,0){\small 0}
\put(20,0){\small 0}
\put(10,8){\small 1}
\end{picture}
\ .
\end{split}
\end{align}
As an example of the relation \eqref{eq:alpha2}, for $a=3$ and $k=3$,
we have
\begin{align}
\begin{picture}(22,12)(0,4)
\put(0,0){\small 1}
\put(5,0){\small 2}
\put(10,0){\small 3}
\put(15,0){\small 2}
\put(20,0){\small 1}
\put(10,8){\small 1}
\end{picture}
\
+
\begin{picture}(22,12)(0,4)
\put(0,0){\small 1}
\put(5,0){\small 1}
\put(10,0){\small 2}
\put(15,0){\small 1}
\put(20,0){\small 1}
\put(10,8){\small 1}
\end{picture}
\
=
\,
\begin{picture}(22,12)(0,4)
\put(0,0){\small 0}
\put(5,0){\small 1}
\put(10,0){\small 2}
\put(15,0){\small 1}
\put(20,0){\small 1}
\put(10,8){\small 1}
\end{picture}
\
+
\begin{picture}(22,12)(0,4)
\put(0,0){\small 1}
\put(5,0){\small 1}
\put(10,0){\small 2}
\put(15,0){\small 1}
\put(20,0){\small 0}
\put(10,8){\small 1}
\end{picture}
\
+
\begin{picture}(22,12)(0,4)
\put(0,0){\small 1}
\put(5,0){\small 1}
\put(10,0){\small 1}
\put(15,0){\small 1}
\put(20,0){\small 1}
\put(10,8){\small 0}
\end{picture}
\
.
\end{align}
\end{ex}

\section{Dynamics of tropical $Y$-systems}
\label{sec:tropical1}

We prove
Theorems \ref{thm:Yperiod1} and \ref{thm:YDI1}
simultaneously by the \emph{tropicalization method}.
Here, we continue to use the notation $[y_{a,a'}(u)]$ for the tropical $y$-variables
in Chapter \ref{ch:dilogarithm1}.

By applying the tropicalization \eqref{1eq:trophom1},
the $Y$-system \eqref{eq:yma3}
has the  form
\begin{align}
\label{eq:yma5}
[y_{a,a'}(s+1)][y_{a,a'}(s-1)]=
\frac
{\displaystyle
\prod_{b\in I:\, b\sim a}[1+y_{b,a'}(s)]}
{
\displaystyle
\prod_{b'\in I':\, b'\sim a'}[1+y_{a,b'}(s)^{-1}]
},
\end{align}
which we call the \emph{tropical $Y$-system}\index{tropical!$Y$-system}.
Even though it is in the same form as \eqref{eq:yma3},
it is drastically simplified, as explained below.
As we see in Proposition \ref{prop:tropF1},
each tropical $y$-variable
$[y_{a,a'}(s)]$ is a monomial in the initial $y$-variables
$\bfy$ whose powers are given by the corresponding $c$-vector
$\bfc_{a,a'}(s)$. Let us write it simply as
\begin{align}
[y_{a,a'}(s)]=y^{\bfc_{a,a'}(s)}.
\end{align}
Recall the sign-coherence property of $\bfc_{a,a'}(s)$ in Theorem \ref{thm:sign1}.
Let $\varepsilon_{a,a'}(s)$ be the tropical sign of $\bfc_{a,a'}(s)$.
Then,
\begin{align}
[1+y_{a,a'}(s)]
&=
\begin{cases}
1 & \varepsilon_{a,a'}(s)=1,\\
y^{\bfc_{a,a'}(s)} & \varepsilon_{a,a'}(s)=-1,
\end{cases}
\\
[1+y_{a,a'}(s)^{-1}]^{-1}
&=
\begin{cases}
y^{\bfc_{a,a'}(s)} & \varepsilon_{a,a'}(s)=1,\\
1 & \varepsilon_{a,a'}(s)=-1.
\end{cases}
\end{align}
Thus, the tropical $Y$-system \eqref{eq:yma3} is equivalent to the following
\emph{piecewise-linear} system
for $c$-vectors:
\begin{align}
\label{eq:yma6}
\bfc_{a,a'}(s+1)
+
\bfc_{a,a'}(s-1)
=
\sum_{
\scriptstyle
b\in I: \, b\sim a
\atop
\scriptstyle
\varepsilon_{b,a'}(s)=-1
} \bfc_{b,a'}(s)
+
\sum_{
\scriptstyle
b'\in I':\,  b'\sim a'
\atop
\scriptstyle
\varepsilon_{a,b'}(s)=1
} \bfc_{a,b'}(s),
\end{align}
where the sum is zero if it it empty.
It is piecewise-linear because it involves the signs of $c$-vectors.
Do not confuse the tropical sign $\varepsilon_{a,a'}(s)$ with the sign $\kappa_{a,a'}$
for the graph $X\times X'$. 
Observe the similarity to the relation \eqref{eq:alpha2}.

We 
will solve the tropical $Y$-system 
\eqref{eq:yma6} 
with the help of Proposition \ref{prop:alpha1}.
Recall that we concentrate on $c$-vectors in the even sector.
Let $ \bfe_{a,a'}\in \bbZ^{I\times I'}$ be the unit vector whose component is 1 only for
the index $(a,a')$ and 0 otherwise.

{\bf Step I:} The initial condition at $s=-1$, $0$.

 We have $y_{a,a'}(-1)=y_{a,a'}(0)^{-1}$ for $\kappa_{a,a}=-$.
Thus,  the initial conditions for \eqref{eq:yma6}
is given by
\begin{align}
\label{eq:cini1}
\bfc_{a,a'}(-1) &=- \bfe_{a,a'}\quad (\kappa_{a,a}=-),
\quad
\bfc_{a,a'}(0) =\bfe_{a,a'}\quad (\kappa_{a,a}=+).
\end{align}
Let us compare it with \eqref{eq:alpha3}.

{\bf Step II:} The positive region $0\leq s \leq h'-1$.

We determine $\bfc_{a,a'}(1)$ by \eqref{eq:yma6}
with the initial condition \eqref{eq:cini1}.
Since $\bfc_{a,a'}(0)$ is positive,  the first term
in the RHS of  \eqref{eq:yma6} vanishes.
Then, it has the same form as the recurrence relation \eqref{eq:alpha2}
of $\alpha(a;k)$ for the vertical diagram $X'$.
Let $ \alpha'_{a'}$ and $\alpha'(a';k)$ be the corresponding ones  in $X'$
to $\alpha_a$ and $\alpha(a;k)$ in $X$.
Then, we have
\begin{align}
\bfc_{a,a'}(1)=\sum_{ b'\in I'} c_{b'} \bfe_{a,b'},
\quad
\text{where $ \alpha'(a';1)=\sum_{ b'\in I'} c_{b'} \alpha'_{b'}$}.
\end{align}
By \eqref{eq:aa2}, $\bfc_{a,a'}(1)$ is positive.
So, the same formula holds for $\bfc_{a,a'}(2)$
by replacing $\alpha'(a';1)$ with $\alpha'(a';2)$.
Moreover, by Proposition \ref{prop:alpha1} (a), 
$\bfc_{a,a'}(2)$ is positive if $h'\geq 3$.
The same process continues until $s=h'-1$, where
\begin{align}
\label{eq:cini2}
\bfc_{a,a'}(h'-1)= \bfe_{a,\omega'(a')},
\quad
\bfc_{a,a'}(h')= - \bfe_{a,\omega'(a')}
\end{align}
by \eqref{eq:alpha1}.

{\bf Step III:} The negative region $h'\leq s \leq h+h-1$.

We determine $\bfc_{a,a'}(h'+1)$ by \eqref{eq:yma6}
and with the initial condition \eqref{eq:cini2}.
Since $\bfc_{a,a'}(h')$ is negative,  the second term
in the RHS of  \eqref{eq:yma6}  vanishes.
Then, it has the same form as the recurrence relation \eqref{eq:alpha2}
of $\alpha(a;k)$ for the horizontal diagram $X$.
Thus, we have
\begin{align}
\bfc_{a,a'}(h'+1)=- \sum_{ b\in I} c_{b} \bfe_{b,a'},
\quad
\text{where $ \alpha(a;1)=\sum_{ b\in I} c_{b} \alpha_{b}$}.
\end{align}
Again, the  process continues until $s=h'+h-1$, where
\begin{align}
\bfc_{a,a'}(h'+h-1)=-  \bfe_{\omega(a),\omega'(a')},
\quad
\bfc_{a,a'}(h'+h)= \bfe_{\omega(a),\omega'(a')}.
\end{align}

Let us collect the desired results obtained in the above process.
\begin{itemize}
\item
We have the  periodicity
\begin{align}
\bfc_{a,a'}(h'+h)=\bfc_{\omega(a),\omega'(a')}(0).
\end{align}
\item
The number of the positive and negative $c$-vectors 
$c_{a,a'}(s)$
 (in the even sector)
for $0\leq s\leq h'+h-1$ are counted as
\begin{align}
N_+=\frac{1}{2} h' rr',
\quad
N_-=\frac{1}{2} h rr'.
\end{align}
This is clear when $h'$ and $h$ are even.
When $h'$ is odd, for example, note that $r'$ is even.
Then,  the formula still holds.
\end{itemize}
This completes the proof of Theorems \ref{thm:Yperiod1} and \ref{thm:YDI1}.

\begin{ex}
Consider the case $(X, X')=(A_3,A_2)$ in Example  \ref{ex:QB1}.
We have $h=4$ and $h'=3$.
Let $Q$ be the quiver therein
with the following assignment of signs:
$$
Q=
\
\lower12pt
\hbox{
\begin{xy}
(0,10)*\cir<2pt>{},
(0,0)*{\bullet},
(10, 10)*{\bullet},
(10,0)*\cir<2pt>{},
(20,10)*\cir<2pt>{},
(20,0)*{\bullet},
(0,13)*{\text{\small $-$}},
(10,13)*{\text{\small +}},
(20,13)*{\text{\small $-$}},
(-4,0)*{\text{\small $-$}},
(-4,10)*{\text{\small $+$}},
\ar@{->} (2,10);(8,10) 
\ar@{->} (18,10);(12,10) 
\ar@{->} (8,0);(2,0) 
\ar@{->} (12,0);(18,0) 
\ar@{->} (0,2);(0,8) 
\ar@{->} (10,8);(10,2) 
\ar@{->} (20,2);(20,8) 
\end{xy}
}
$$
\vskip5pt

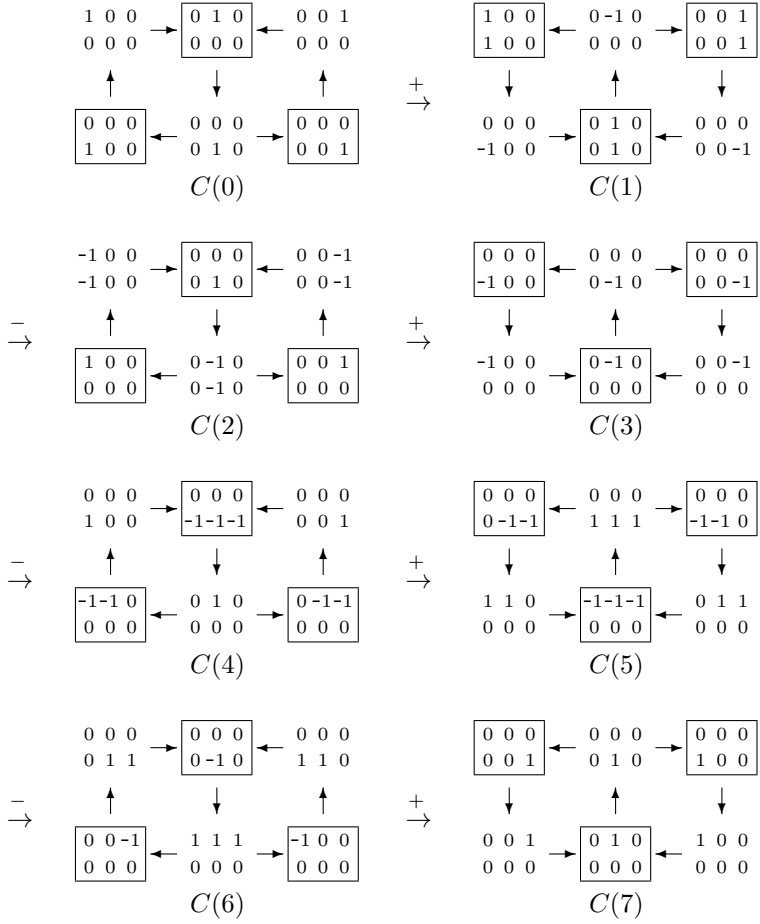
\begin{figure}[t]
\begin{center}
\begin{picture}(255,340)(-15,-10)
\put(0,270)
{
\put(0,48){$\scriptstyle 1$}
\put(8,48){$\scriptstyle 0$}
\put(16,48){$\scriptstyle  0$}
\put(0,38){$\scriptstyle 0$}
\put(8,38){$\scriptstyle 0$}
\put(16,38){$\scriptstyle  0$}
\dottedline{3}(5,35)(5,55)
\dottedline{3}(13,35)(13,55)
\put(10,20){\vector(0,1){10}}
\put(25,45){\vector(1,0){10}}
\put(37, 35){\framebox(26,20)[c]}
\put(40,48){$\scriptstyle 0$}
\put(48,48){$\scriptstyle 1$}
\put(56,48){$\scriptstyle 0$}
\put(40,38){$\scriptstyle 0$}
\put(48,38){$\scriptstyle 0$}
\put(56,38){$\scriptstyle 0$}
\dottedline{3}(45,35)(45,55)
\dottedline{3}(53,35)(53,55)
\put(50,30){\vector(0,-1){10}}
\put(75,45){\vector(-1,0){10}}
\put(80,48){$\scriptstyle 0$}
\put(88,48){$\scriptstyle 0$}
\put(96,48){$\scriptstyle 1$}
\put(80,38){$\scriptstyle 0$}
\put(88,38){$\scriptstyle 0$}
\put(96,38){$\scriptstyle 0$}
\dottedline{3}(85,35)(85,55)
\dottedline{3}(93,35)(93,55)
\put(90,20){\vector(0,1){10}}
\put(-3, -5){\framebox(26,20)[c]}
\put(0,8){$\scriptstyle 0$}
\put(8,8){$\scriptstyle 0$}
\put(16,8){$\scriptstyle  0$}
\put(0,-2){$\scriptstyle 1$}
\put(8,-2){$\scriptstyle 0$}
\put(16,-2){$\scriptstyle  0$}
\dottedline{3}(5,-5)(5,15)
\dottedline{3}(13,-5)(13,15)
\put(35,5){\vector(-1,0){10}}
\put(40,8){$\scriptstyle 0$}
\put(48,8){$\scriptstyle 0$}
\put(56,8){$\scriptstyle 0$}
\put(40,-2){$\scriptstyle 0$}
\put(48,-2){$\scriptstyle 1$}
\put(56,-2){$\scriptstyle 0$}
\dottedline{3}(45,-5)(45,15)
\dottedline{3}(53,-5)(53,15)
\put(65,5){\vector(1,0){10}}
\put(77, -5){\framebox(26,20)[c]}
\put(80,8){$\scriptstyle 0$}
\put(88,8){$\scriptstyle 0$}
\put(96,8){$\scriptstyle 0$}
\put(80,-2){$\scriptstyle 0$}
\put(88,-2){$\scriptstyle 0$}
\put(96,-2){$\scriptstyle 1$}
\dottedline{3}(85,-5)(85,15)
\dottedline{3}(93,-5)(93,15)
\put(40,-17){$C(0)$}
}
\put(150,270)
{
\put(-3, 35){\framebox(26,20)[c]}
\put(0,48){$\scriptstyle 1$}
\put(8,48){$\scriptstyle 0$}
\put(16,48){$\scriptstyle  0$}
\put(0,38){$\scriptstyle 1$}
\put(8,38){$\scriptstyle 0$}
\put(16,38){$\scriptstyle  0$}
\dottedline{3}(5,35)(5,55)
\dottedline{3}(13,35)(13,55)
\put(10,30){\vector(0,-1){10}}
\put(35,45){\vector(-1,0){10}}
\put(40,48){$\scriptstyle 0$}
\put(46,48){-$\scriptstyle 1$}
\put(56,48){$\scriptstyle 0$}
\put(40,38){$\scriptstyle 0$}
\put(48,38){$\scriptstyle 0$}
\put(56,38){$\scriptstyle 0$}
\dottedline{3}(45,35)(45,55)
\dottedline{3}(53,35)(53,55)
\put(50,20){\vector(0,1){10}}
\put(65,45){\vector(1,0){10}}
\put(77, 35){\framebox(26,20)[c]}
\put(80,48){$\scriptstyle 0$}
\put(88,48){$\scriptstyle 0$}
\put(96,48){$\scriptstyle 1$}
\put(80,38){$\scriptstyle 0$}
\put(88,38){$\scriptstyle 0$}
\put(96,38){$\scriptstyle 1$}
\dottedline{3}(85,35)(85,55)
\dottedline{3}(93,35)(93,55)
\put(90,30){\vector(0,-1){10}}
\put(0,8){$\scriptstyle 0$}
\put(8,8){$\scriptstyle 0$}
\put(16,8){$\scriptstyle  0$}
\put(-2,-2){-$\scriptstyle 1$}
\put(8,-2){$\scriptstyle 0$}
\put(16,-2){$\scriptstyle  0$}
\dottedline{3}(5,-5)(5,15)
\dottedline{3}(13,-5)(13,15)
\put(25,5){\vector(1,0){10}}
\put(37, -5){\framebox(26,20)[c]}
\put(40,8){$\scriptstyle 0$}
\put(48,8){$\scriptstyle 1$}
\put(56,8){$\scriptstyle 0$}
\put(40,-2){$\scriptstyle 0$}
\put(48,-2){$\scriptstyle 1$}
\put(56,-2){$\scriptstyle 0$}
\dottedline{3}(45,-5)(45,15)
\dottedline{3}(53,-5)(53,15)
\put(75,5){\vector(-1,0){10}}
\put(80,8){$\scriptstyle 0$}
\put(88,8){$\scriptstyle 0$}
\put(96,8){$\scriptstyle 0$}
\put(80,-2){$\scriptstyle 0$}
\put(88,-2){$\scriptstyle 0$}
\put(94,-2){-$\scriptstyle 1$}
\dottedline{3}(85,-5)(85,15)
\dottedline{3}(93,-5)(93,15)
\put(40,-17){$C(1)$}
\put(-29,15){$\rightarrow$}
\put(-28,23){$\scriptstyle +$}
}
\put(0,180)
{
\put(-2,48){-$\scriptstyle 1$}
\put(8,48){$\scriptstyle 0$}
\put(16,48){$\scriptstyle  0$}
\put(-2,38){-$\scriptstyle 1$}
\put(8,38){$\scriptstyle 0$}
\put(16,38){$\scriptstyle  0$}
\dottedline{3}(5,35)(5,55)
\dottedline{3}(13,35)(13,55)
\put(10,20){\vector(0,1){10}}
\put(25,45){\vector(1,0){10}}
\put(37, 35){\framebox(26,20)[c]}
\put(40,48){$\scriptstyle 0$}
\put(48,48){$\scriptstyle 0$}
\put(56,48){$\scriptstyle 0$}
\put(40,38){$\scriptstyle 0$}
\put(48,38){$\scriptstyle 1$}
\put(56,38){$\scriptstyle 0$}
\dottedline{3}(45,35)(45,55)
\dottedline{3}(53,35)(53,55)
\put(50,30){\vector(0,-1){10}}
\put(75,45){\vector(-1,0){10}}
\put(80,48){$\scriptstyle 0$}
\put(88,48){$\scriptstyle 0$}
\put(94,48){-$\scriptstyle 1$}
\put(80,38){$\scriptstyle 0$}
\put(88,38){$\scriptstyle 0$}
\put(94,38){-$\scriptstyle 1$}
\dottedline{3}(85,35)(85,55)
\dottedline{3}(93,35)(93,55)
\put(90,20){\vector(0,1){10}}
\put(-3, -5){\framebox(26,20)[c]}
\put(0,8){$\scriptstyle 1$}
\put(8,8){$\scriptstyle 0$}
\put(16,8){$\scriptstyle  0$}
\put(0,-2){$\scriptstyle 0$}
\put(8,-2){$\scriptstyle 0$}
\put(16,-2){$\scriptstyle  0$}
\dottedline{3}(5,-5)(5,15)
\dottedline{3}(13,-5)(13,15)
\put(35,5){\vector(-1,0){10}}
\put(40,8){$\scriptstyle 0$}
\put(46,8){-$\scriptstyle 1$}
\put(56,8){$\scriptstyle 0$}
\put(40,-2){$\scriptstyle 0$}
\put(46,-2){-$\scriptstyle 1$}
\put(56,-2){$\scriptstyle 0$}
\dottedline{3}(45,-5)(45,15)
\dottedline{3}(53,-5)(53,15)
\put(65,5){\vector(1,0){10}}
\put(77, -5){\framebox(26,20)[c]}
\put(80,8){$\scriptstyle 0$}
\put(88,8){$\scriptstyle 0$}
\put(96,8){$\scriptstyle 1$}
\put(80,-2){$\scriptstyle 0$}
\put(88,-2){$\scriptstyle 0$}
\put(96,-2){$\scriptstyle 0$}
\dottedline{3}(85,-5)(85,15)
\dottedline{3}(93,-5)(93,15)
\put(40,-17){$C(2)$}
\put(-29,15){$\rightarrow$}
\put(-28,23){$\scriptstyle -$}
}
\put(150,180)
{
\put(-3, 35){\framebox(26,20)[c]}
\put(0,48){$\scriptstyle 0$}
\put(8,48){$\scriptstyle 0$}
\put(16,48){$\scriptstyle  0$}
\put(-2,38){-$\scriptstyle 1$}
\put(8,38){$\scriptstyle 0$}
\put(16,38){$\scriptstyle  0$}
\dottedline{3}(0,45)(20,45)%
\put(10,30){\vector(0,-1){10}}
\put(35,45){\vector(-1,0){10}}
\put(40,48){$\scriptstyle 0$}
\put(48,48){$\scriptstyle 0$}
\put(56,48){$\scriptstyle 0$}
\put(40,38){$\scriptstyle 0$}
\put(46,38){-$\scriptstyle 1$}
\put(56,38){$\scriptstyle 0$}
\dottedline{3}(40,45)(60,45)
\put(50,20){\vector(0,1){10}}
\put(65,45){\vector(1,0){10}}
\put(77, 35){\framebox(26,20)[c]}
\put(80,48){$\scriptstyle 0$}
\put(88,48){$\scriptstyle 0$}
\put(96,48){$\scriptstyle 0$}
\put(80,38){$\scriptstyle 0$}
\put(88,38){$\scriptstyle 0$}
\put(94,38){-$\scriptstyle 1$}
\dottedline{3}(80,45)(100,45)
\put(90,30){\vector(0,-1){10}}
\put(-2,8){-$\scriptstyle 1$}
\put(8,8){$\scriptstyle 0$}
\put(16,8){$\scriptstyle  0$}
\put(0,-2){$\scriptstyle 0$}
\put(8,-2){$\scriptstyle 0$}
\put(16,-2){$\scriptstyle  0$}
\dottedline{3}(0,5)(20,5)
\put(25,5){\vector(1,0){10}}
\put(37, -5){\framebox(26,20)[c]}
\put(40,8){$\scriptstyle 0$}
\put(46,8){-$\scriptstyle 1$}
\put(56,8){$\scriptstyle 0$}
\put(40,-2){$\scriptstyle 0$}
\put(48,-2){$\scriptstyle 0$}
\put(56,-2){$\scriptstyle 0$}
\dottedline{3}(40,5)(60,5)
\put(75,5){\vector(-1,0){10}}
\put(80,8){$\scriptstyle 0$}
\put(88,8){$\scriptstyle 0$}
\put(94,8){-$\scriptstyle 1$}
\put(80,-2){$\scriptstyle 0$}
\put(88,-2){$\scriptstyle 0$}
\put(96,-2){$\scriptstyle 0$}
\dottedline{3}(80,5)(100,5)
\put(40,-17){$C(3)$}
\put(-29,15){$\rightarrow$}
\put(-28,23){$\scriptstyle +$}
}
\put(0,90)
{
\put(0,48){$\scriptstyle 0$}
\put(8,48){$\scriptstyle 0$}
\put(16,48){$\scriptstyle  0$}
\put(0,38){$\scriptstyle 1$}
\put(8,38){$\scriptstyle 0$}
\put(16,38){$\scriptstyle 0$}
\dottedline{3}(0,45)(20,45)%
\put(10,20){\vector(0,1){10}}
\put(25,45){\vector(1,0){10}}
\put(37, 35){\framebox(26,20)[c]}
\put(40,48){$\scriptstyle 0$}
\put(48,48){$\scriptstyle 0$}
\put(56,48){$\scriptstyle 0$}
\put(38,38){-$\scriptstyle 1$}
\put(46,38){-$\scriptstyle 1$}
\put(54,38){-$\scriptstyle 1$}
\dottedline{3}(40,45)(60,45)
\put(50,30){\vector(0,-1){10}}
\put(75,45){\vector(-1,0){10}}
\put(80,48){$\scriptstyle 0$}
\put(88,48){$\scriptstyle 0$}
\put(96,48){$\scriptstyle 0$}
\put(80,38){$\scriptstyle 0$}
\put(88,38){$\scriptstyle 0$}
\put(96,38){$\scriptstyle 1$}
\dottedline{3}(80,45)(100,45)
\put(90,20){\vector(0,1){10}}
\put(-3, -5){\framebox(26,20)[c]}
\put(-2,8){-$\scriptstyle 1$}
\put(6,8){-$\scriptstyle 1$}
\put(16,8){$\scriptstyle  0$}
\put(0,-2){$\scriptstyle 0$}
\put(8,-2){$\scriptstyle 0$}
\put(16,-2){$\scriptstyle  0$}
\dottedline{3}(0,5)(20,5)
\put(35,5){\vector(-1,0){10}}
\put(40,8){$\scriptstyle 0$}
\put(48,8){$\scriptstyle 1$}
\put(56,8){$\scriptstyle 0$}
\put(40,-2){$\scriptstyle 0$}
\put(48,-2){$\scriptstyle 0$}
\put(56,-2){$\scriptstyle 0$}
\dottedline{3}(40,5)(60,5)
\put(65,5){\vector(1,0){10}}
\put(77, -5){\framebox(26,20)[c]}
\put(80,8){$\scriptstyle 0$}
\put(86,8){-$\scriptstyle 1$}
\put(94,8){-$\scriptstyle 1$}
\put(80,-2){$\scriptstyle 0$}
\put(88,-2){$\scriptstyle 0$}
\put(96,-2){$\scriptstyle 0$}
\dottedline{3}(80,5)(100,5)%
\put(40,-17){$C(4)$}
\put(-29,15){$\rightarrow$}
\put(-28,23){$\scriptstyle -$}
}
\put(150,90)
{
\put(-3, 35){\framebox(26,20)[c]}
\put(0,48){$\scriptstyle 0$}
\put(8,48){$\scriptstyle 0$}
\put(16,48){$\scriptstyle  0$}
\put(0,38){$\scriptstyle 0$}
\put(6,38){-$\scriptstyle 1$}
\put(14,38){-$\scriptstyle  1$}
\dottedline{3}(0,45)(20,45)%
\put(10,30){\vector(0,-1){10}}
\put(35,45){\vector(-1,0){10}}
\put(40,48){$\scriptstyle 0$}
\put(48,48){$\scriptstyle 0$}
\put(56,48){$\scriptstyle 0$}
\put(40,38){$\scriptstyle 1$}
\put(48,38){$\scriptstyle 1$}
\put(56,38){$\scriptstyle 1$}
\dottedline{3}(40,45)(60,45)
\put(50,20){\vector(0,1){10}}
\put(65,45){\vector(1,0){10}}
\put(77, 35){\framebox(26,20)[c]}
\put(80,48){$\scriptstyle 0$}
\put(88,48){$\scriptstyle 0$}
\put(96,48){$\scriptstyle 0$}
\put(78,38){-$\scriptstyle 1$}
\put(86,38){-$\scriptstyle 1$}
\put(96,38){$\scriptstyle 0$}
\dottedline{3}(80,45)(100,45)
\put(90,30){\vector(0,-1){10}}
\put(0,8){$\scriptstyle 1$}
\put(8,8){$\scriptstyle 1$}
\put(16,8){$\scriptstyle  0$}
\put(0,-2){$\scriptstyle 0$}
\put(8,-2){$\scriptstyle 0$}
\put(16,-2){$\scriptstyle  0$}
\dottedline{3}(0,5)(20,5)
\put(25,5){\vector(1,0){10}}
\put(37, -5){\framebox(26,20)[c]}
\put(38,8){-$\scriptstyle 1$}
\put(46,8){-$\scriptstyle 1$}
\put(54,8){-$\scriptstyle 1$}
\put(40,-2){$\scriptstyle 0$}
\put(48,-2){$\scriptstyle 0$}
\put(56,-2){$\scriptstyle 0$}
\dottedline{3}(40,5)(60,5)
\put(75,5){\vector(-1,0){10}}
\put(80,8){$\scriptstyle 0$}
\put(88,8){$\scriptstyle 1$}
\put(96,8){$\scriptstyle 1$}
\put(80,-2){$\scriptstyle 0$}
\put(88,-2){$\scriptstyle 0$}
\put(96,-2){$\scriptstyle 0$}
\dottedline{3}(80,5)(100,5)
\put(40,-17){$C(5)$}
\put(-29,15){$\rightarrow$}
\put(-28,23){$\scriptstyle +$}
}
\put(0,0)
{
\put(0,48){$\scriptstyle 0$}
\put(8,48){$\scriptstyle 0$}
\put(16,48){$\scriptstyle  0$}
\put(0,38){$\scriptstyle 0$}
\put(8,38){$\scriptstyle 1$}
\put(16,38){$\scriptstyle 1$}
\dottedline{3}(0,45)(20,45)%
\put(10,20){\vector(0,1){10}}
\put(25,45){\vector(1,0){10}}
\put(37, 35){\framebox(26,20)[c]}
\put(40,48){$\scriptstyle 0$}
\put(48,48){$\scriptstyle 0$}
\put(56,48){$\scriptstyle 0$}
\put(40,38){$\scriptstyle 0$}
\put(46,38){-$\scriptstyle 1$}
\put(56,38){$\scriptstyle 0$}
\dottedline{3}(40,45)(60,45)
\put(50,30){\vector(0,-1){10}}
\put(75,45){\vector(-1,0){10}}
\put(80,48){$\scriptstyle 0$}
\put(88,48){$\scriptstyle 0$}
\put(96,48){$\scriptstyle 0$}
\put(80,38){$\scriptstyle 1$}
\put(88,38){$\scriptstyle 1$}
\put(96,38){$\scriptstyle 0$}
\dottedline{3}(80,45)(100,45)
\put(90,20){\vector(0,1){10}}
\put(-3, -5){\framebox(26,20)[c]}
\put(0,8){$\scriptstyle 0$}
\put(8,8){$\scriptstyle 0$}
\put(14,8){-$\scriptstyle  1$}
\put(0,-2){$\scriptstyle 0$}
\put(8,-2){$\scriptstyle 0$}
\put(16,-2){$\scriptstyle  0$}
\dottedline{3}(0,5)(20,5)
\put(35,5){\vector(-1,0){10}}
\put(40,8){$\scriptstyle 1$}
\put(48,8){$\scriptstyle 1$}
\put(56,8){$\scriptstyle 1$}
\put(40,-2){$\scriptstyle 0$}
\put(48,-2){$\scriptstyle 0$}
\put(56,-2){$\scriptstyle 0$}
\dottedline{3}(40,5)(60,5)
\put(65,5){\vector(1,0){10}}
\put(77, -5){\framebox(26,20)[c]}
\put(78,8){-$\scriptstyle 1$}
\put(88,8){$\scriptstyle 0$}
\put(96,8){$\scriptstyle 0$}
\put(80,-2){$\scriptstyle 0$}
\put(88,-2){$\scriptstyle 0$}
\put(96,-2){$\scriptstyle 0$}
\dottedline{3}(80,5)(100,5)%
\put(40,-17){$C(6)$}
\put(-29,15){$\rightarrow$}
\put(-28,23){$\scriptstyle -$}
}
\put(150,0)
{
\put(-3, 35){\framebox(26,20)[c]}
\put(0,48){$\scriptstyle 0$}
\put(8,48){$\scriptstyle 0$}
\put(16,48){$\scriptstyle  0$}
\put(0,38){$\scriptstyle 0$}
\put(8,38){$\scriptstyle 0$}
\put(16,38){$\scriptstyle  1$}
\dottedline{3}(5,35)(5,55)
\dottedline{3}(13,35)(13,55)
\put(10,30){\vector(0,-1){10}}
\put(35,45){\vector(-1,0){10}}
\put(40,48){$\scriptstyle 0$}
\put(48,48){$\scriptstyle 0$}
\put(56,48){$\scriptstyle 0$}
\put(40,38){$\scriptstyle 0$}
\put(48,38){$\scriptstyle 1$}
\put(56,38){$\scriptstyle 0$}
\dottedline{3}(45,35)(45,55)
\dottedline{3}(53,35)(53,55)
\put(50,20){\vector(0,1){10}}
\put(65,45){\vector(1,0){10}}
\put(77, 35){\framebox(26,20)[c]}
\put(80,48){$\scriptstyle 0$}
\put(88,48){$\scriptstyle 0$}
\put(96,48){$\scriptstyle 0$}
\put(80,38){$\scriptstyle 1$}
\put(88,38){$\scriptstyle 0$}
\put(96,38){$\scriptstyle 0$}
\dottedline{3}(85,35)(85,55)
\dottedline{3}(93,35)(93,55)
\put(90,30){\vector(0,-1){10}}
\put(0,8){$\scriptstyle 0$}
\put(8,8){$\scriptstyle 0$}
\put(16,8){$\scriptstyle  1$}
\put(0,-2){$\scriptstyle 0$}
\put(8,-2){$\scriptstyle 0$}
\put(16,-2){$\scriptstyle  0$}
\dottedline{3}(5,-5)(5,15)
\dottedline{3}(13,-5)(13,15)
\put(25,5){\vector(1,0){10}}
\put(37, -5){\framebox(26,20)[c]}
\put(40,8){$\scriptstyle 0$}
\put(48,8){$\scriptstyle 1$}
\put(56,8){$\scriptstyle 0$}
\put(40,-2){$\scriptstyle 0$}
\put(48,-2){$\scriptstyle 0$}
\put(56,-2){$\scriptstyle 0$}
\dottedline{3}(45,-5)(45,15)
\dottedline{3}(53,-5)(53,15)
\put(75,5){\vector(-1,0){10}}
\put(80,8){$\scriptstyle 1$}
\put(88,8){$\scriptstyle 0$}
\put(96,8){$\scriptstyle 0$}
\put(80,-2){$\scriptstyle 0$}
\put(88,-2){$\scriptstyle 0$}
\put(96,-2){$\scriptstyle 0$}
\dottedline{3}(85,-5)(85,15)
\dottedline{3}(93,-5)(93,15)
\put(40,-17){$C(7)$}
\put(-29,15){$\rightarrow$}
\put(-28,23){$\scriptstyle +$}
}
\end{picture}
\end{center}
\caption{Tropical $Y$-system for $(X,X')=(A_3,A_2)$,
where $h+h'=7$.
The framed ones are the $c$-vectors in the even sector.
}
\label{fig:higher}
\end{figure}

\noindent
We take $Q$ as the initial quiver for the corresponding $Y$-pattern.
We arrange each $c$-vector in   a matrix form with
horizontal and vertical indices in $I\times I'$.
For example, the matrix
\begin{align}
\begin{matrix}
0 & 0 & 0\\
0 & 1 & 1
\end{matrix}
\end{align}
represents the $c$-vector $\bfc=\bfe_{2,2}+\bfe_{3,2}$.
Here, we calculate the mutation of $c$-vectors
directly  by the mutation rule \eqref{2eq:Cmat1} 
in the ambient $Y$-pattern,
rather than by the tropical $Y$-system
\eqref{eq:yma6}.
The rule \eqref{2eq:Cmat1} is translated
into the following pictorial rule, which is very efficient:
\begin{align}
\bfc'_k&=-\bfc_k,
\\
\label{eq:ypic2}
\bfc'_i 
&=
\begin{cases}
\bfc_i + \bfc_k&
\begin{xy}
(0,0)*\cir<2pt>{},
(10,0)*{\bullet},
(0,-3)*{\text{\small $i$}},
(10,-3)*{\text{\small $k$}},
(22,0)*{\text{\small and $\varepsilon_k=-$}},
\ar@{<-} (8,0);(2,0) 
\end{xy},
\\
\bfc_i + \bfc_k
&
\begin{xy}
(0,0)*\cir<2pt>{},
(10,0)*{\bullet},
(0,-3)*{\text{\small $i$}},
(10,-3)*{\text{\small $k$}},
(22,0)*{\text{\small and $\varepsilon_k=+$}},
\ar@{->} (8,0);(2,0) 
\end{xy},
\\
\bfc_i
&
\text{othterwise}
\end{cases}
\quad
\text{($i\neq k$)}.
\end{align}
All $c$-vectors from $s=0$ to $s=7$ are presented in Figure
\ref{fig:higher}.
The desired periodicity is found at $s=h+h'=7$.
\end{ex}
As we see in the above example,
and  it is also clear from the analysis in the general case,
in the \emph{positive region} $0\leq s \leq h'-1$,
a nontrivial mutation occurs only ``in the  vertical direction''
by the coordination of tropical signs and arrows.
(This is indicated in the dotted vertical lines in Figure
\ref{fig:higher}.)
Similarly,
in the \emph{negative region} $h'\leq s \leq h'+h-1$,
a nontrivial mutation occurs only ``in the horizontal direction''.
We call this factorization into the vertical and horizontal directions the \emph{factorization property}\index{factorization property}
of the tropical $Y$-system \eqref{eq:yma5}.
Such a simplification does not occur in the nontropical $Y$-system \eqref{eq:yma3}.

In summary,
 we use the formulation of the $Y$-systems by $Y$-patterns, together with the tropicalization,
 to prove the periodicity and the DIs for $Y$-systems.
The tropicalization reveals that
 the Coxeter elements are behind the scene and govern the dynamics of the $Y$-systems.
 This is why the Coxeter number is relevant for the periodicity.
 
 \notes
 The cluster algebraic structure for the $Y$-systems  \eqref{eq:ysys1} of type $(X,A_1)$
was recognized first by \cite{Fomin03b} and formulated by $Y$-patterns  by \cite{Fomin07}.
It was generalized to the type $(X,X')$ by \cite{Keller08}
using the quiver $Q(X,X')$.

There is a ``twin" system to   $Y$-systems
called  \emph{$T$-systems}, which were also introduced and  studied together with $Y$-systems in B.C.~\cite{Kuniba94a}.
$T$-systems are satisfied by  $x$-variables for  the same sequence of mutations \eqref{eq:Yseq1}
\cite{DiFrancesco09a,Inoue10c}.

After the results by \cite{Gliozzi96,Frenkel95} in B.C.\   mentioned in
Chapter \ref{ch:prologue},
the periodicity of $Y$-systems was  proved by 
\cite{Fomin03b} for type $(X,A_1)$ by a cluster algebraic method.
The method consists of the tropicalization and the construction of $F$-polynomials (with the assistance of a computer).
The role of the Coxeter elements was also clarified.
The periodicity of $Y$-systems for type $(X,X')$
 was proved by 
\cite{Keller08, Keller10}
based on the \emph{categorification} of  cluster patterns by
triangulated categories.
The refinement to the half periodicity was given
in \cite{Inoue10c}, where the periodicity of
associated $T$-systems was also proved.
Using the ideas and results of \cite{Caracciolo99, Fomin03b, Fomin07, Chapoton05, Keller08},
the DIs for the $Y$-systems
of type $(X,X')$ were proved by \cite{Nakanishi09}.

By further developing the above methods and ideas,
together with the detropicalization by cluster categories by \cite{Plamondon10b},
the periodicity and DIs for more complicated $Y$-systems were
proved.
For the $Y$-systems of nonsimply-laced types 
they were proved by \cite{Inoue10a, Inoue10b}.
For the  sine-Gordon and reduced sine-Gordon $Y$-systems
they were proved
by \cite{Nakanishi10b} in a special case and  by \cite{Nakanishi12c}
in full generality.

Meanwhile, the periodicity of $Y$-systems of type $(A_r, A_{r'})$
was also proved  by \cite{Volkov07} with  the determinant solution
and by \cite{Szenes09} using the flatness of a graph connection.
Thus, the type $A$  is the special case where several ideas in mathematics intersect
as in Lie theory.

%

\part{Classical mechanical method}

\chapter{Fock-Goncharov decomposition and dilogarithm}
\label{ch:Fock1}

We introduce two new ideas for $Y$-patterns.
The first one is the decomposition of a mutation of $Y$-seeds into
tropical and nontropical parts based on the idea of 
Fock-Goncharov.
The second one is the mutation-compatible Poisson bracket
introduced by Gekhtman-Shapiro-Vainshtein.
By combining two ideas,  the Euler dilogarithm 
emerges as the Hamiltonian
for the nontropical part of each mutation.

\section{Fock-Goncharov decomposition}
\label{sec:Fock1}

Let us introduce the following notion.

\begin{defn}[Skew-symmetric decomposition]
For a skew-symmetriz\-able (integer) matrix $B=(b_{ij})_{i,j=1}^n$,
we consider a decomposition
\begin{align}
\label{eq:BDO1}
B=\Delta \Omega,
\end{align}
where $\Delta=\mathrm{diag}(\delta_1,\dots,\delta_n)$ is a diagonal matrix with positive integer diagonals
and $\Omega$ is a skew-symmetric rational matrix.
We call it a \emph{skew-symmetric decomposition}\index{skew-symmetric!decomposition} of $B$.
\end{defn}

For a skew-symmetrizer $D$ of the form \eqref{eq:Dd1} for $B$,
we set
\begin{align}
\Delta=D^{-1},
\quad
\Omega = DB.
\end{align}
Then, we have a skew-symmetric decomposition of $B$.
Conversely, for a skew-symmetric decomposition
of $B$, $D=\Delta^{-1}$ is a skew-symmetrizer of $B$
 of the form \eqref{eq:Dd1} such that $DB=\Omega$.
Thus, choosing a skew-symmetrizer of the form   \eqref{eq:Dd1}
is equivalent to choosing a skew-symmetric decomposition.

We start  by observing that there is some flexibility in writing the mutation
 \eqref{2eq:xmut1}--\eqref{2eq:bmut1}.
Namely, Let $\varepsilon\in \{1, -1\}$.
Then, by \eqref{1eq:pos1}, one can easily check that the RHSs of the following expressions
are independent of $\varepsilon$:
\begin{align}
\label{2eq:xmut5}
x'_i
&=
\begin{cases}
\displaystyle
x_k^{-1}\Biggl(\, \prod_{j=1}^n x_j^{[-\varepsilon b_{jk}]_+}
\Biggr)
\frac{
 1+\hat{y}_k^{\varepsilon}}
 {1\oplus y_k^{\varepsilon}}
 & i=k,
\\
x_i
&i\neq k,
\end{cases}
\\
\label{2eq:ymut5}
y'_i
&=
\begin{cases}
\displaystyle
y_k^{-1}
& i=k,
\\
y_i y_k^{[\varepsilon b_{ki}]_+} (1\oplus y_k^{\varepsilon})^{-b_{ki}}
&i\neq k,
\end{cases}
\\
\label{2eq:bmut5}
b'_{ij}&=
\begin{cases}
-b_{ij}
&
\text{$i=k$ or $j=k$,}
\\
b_{ij}+
b_{ik} [\varepsilon b_{kj}]_+
+
[-\varepsilon b_{ik}]_+b_{kj}
&
i,j\neq k.
\end{cases}
\end{align}
The case $\varepsilon=1$ yields
the original mutation formulas.
We call them the \emph{$\varepsilon$-expressions}\index{$\varepsilon$-expression}
of the mutation.
Similarly, the mutations of $C$- and $G$-matrices
in Section \ref{sec:separation1}
have the following  {$\varepsilon$-expressions}:
\begin{align}
 \label{2eq:cmutmat3}
 c_{ij;t'}&=
 \begin{cases}
 -c_{ik;t}
 &
 j= k,
 \\
 c_{ij;t} + c_{ik;t} [\varepsilon b_{kj;t}]_+
 + [- \varepsilon c_{ik;t}]_+ b_{kj;t}
 &
 j \neq k,
 \end{cases}
\\
 \label{2eq:gmut2}
 g_{ij;t'}&=
 \begin{cases}
 \displaystyle
 -g_{ik;t}
 + \sum_{\ell=1}^n g_{i\ell;t} [-\varepsilon b_{\ell k;t}]_+
 -  \sum_{\ell=1}^nb_{i\ell;t_0}  [-\varepsilon c_{\ell k;t}]_+ 
 &
 j= k,
 \\
 g_{ij;t}  &
 j \neq k.
 \end{cases}
 \end{align}
 To check  the $\varepsilon$-independence of \eqref{2eq:gmut2},
 we also use the formula \eqref{2eq:dual0}.

Now, we are back in the situation in Section \ref{sec:periodicity1}.
Namely,
we consider a cluster pattern $\bfSigma$ with \emph{free coefficients} at $t_0$.
Again, we concentrate on the $Y$-pattern only,
and we write  $\oplus$ in $\bbQ_{\rmsf}(\bfy)$ as $+$.
Consider a sequence of mutations therein
\begin{align}
\label{eq:mseq5}
&\Upsilon=\Upsilon(0) 
\
{\buildrel {k_0} \over \rightarrow}
\
\Upsilon(1) 
\
{\buildrel {k_1} \over \rightarrow}
\
\cdots
\
{\buildrel {k_{P-1}} \over \rightarrow}
\
\Upsilon(P),
\end{align}
where $\Upsilon(s)=(\bfy(s),B(s))$.
Here, we especially assume that
\emph{the $Y$-seed $\Upsilon(0)$ coincides with the initial 
$Y$-seed $\Upsilon_{t_0}=\Upsilon=(\bfy,B)$}.
The $\varepsilon$-expression of the mutation  $\mu_{k_s}$ for $y$-variables is given by
\begin{align}
\label{eq:ymut6}
y_i(s+1)
&=
\begin{cases}
\displaystyle
y_{k_s}(s)^{-1}
& i=k_s,
\\
y_i(s) y_{k_s}(s) ^{[\varepsilon b_{k_s i}(s)]_+} (1+ y_{k_s}(s)^{\varepsilon})^{-b_{k_s i}(s)}
&i\neq k_s.
\end{cases}
\end{align}

Originally, all $y$-variables $y_i(s)$ ($s=0$, \dots, $P$) belong to the common semifield $\bbQ_{\rmsf}(\bfy)$
for the initial $y$-variables $\bfy$,
and the formula \eqref{eq:ymut6} gives relations among them.
Here, we introduce an alternative veiwpoint.
For each $s=0$, \dots, $P$, we regard the $y$-variables $\bfy(s)=(y_1(s), \dots,y_n(s))$   as
formal variables and consider the  universal semifield 
$\bbP(s):=\bbQ_{\rmsf}(\bfy(s))$.
Then, we interpret the mutation \eqref{eq:ymut6} as a semifield isomorphism
($s=0$, \dots, $P-1$)
\begin{align}
\begin{matrix}
\label{eq:mu1}
\mu{(s)} \colon &\bbP(s+1)& \rightarrow & \bbP(s)
\\
&
y_i(s+1)
&
\mapsto
&
\begin{cases}
\displaystyle
y_{k_s}(s)^{-1}
& i=k_s,
\\
y_i(s) y_{k_s}(s) ^{[\varepsilon b_{k_s i}(s)]_+} (1+ y_{k_s}(s)^{\varepsilon})^{-b_{k_s i}(s)}
&i\neq k_s.
\end{cases}
\end{matrix}
\end{align}

Now taking  advantage of the $\varepsilon$-expression,
we set $\varepsilon$ therein to the tropical sign $\varepsilon_s:=\varepsilon_{k_s}(s)$
in Section \ref{sec:DI1}.
Accordingly,
we introduce  a semifield isomorphism  ($s=0$, \dots, $P-1$)
\begin{align}
\label{eq:tau1}
\begin{matrix}
\tau{(s)}\colon & \bbP(s+1) & \rightarrow & \bbP(s)
\\
&
y_i(s+1)
&
\mapsto
&
\begin{cases}
\displaystyle
y_{k_s}(s)^{-1}
& i=k_s,
\\
y_i(s) y_{k_s}(s) ^{[\varepsilon_s b_{k_s i}(s)]_+} 
&i\neq k_s
\end{cases}
\end{matrix}
\end{align}
and a semifield automorphism   ($s=0$, \dots, $P-1$)
\begin{align}
\label{eq:rho1}
\begin{matrix}
\rho{(s)} \colon & \bbP(s) & \rightarrow &  \bbP(s)
\\
&
y_i(s)
&
\mapsto
&
y_i(s) (1+ y_{k_s}(s)^{\varepsilon_s})^{-b_{k_s i}(s)}.
\end{matrix}
\end{align}
Then, we have a decomposition of the mutation $\mu{(s)}$
\begin{align}
\label{eq:FG1}
\mu{(s)}=
\rho{(s)}\circ \tau{(s)}.
\end{align}
We call it the \emph{Fock-Goncharov decomposition}\index{Fock-Goncharov decomposition!of mutation} $\mu{(s)}$.
Note that this decomposition fully depends on the choice of the initial vertex $t_0$,
because the tropical sign $\varepsilon_s$ depends on $t_0$.
Also, it crucially relies on the sign-coherence property in Theorem
\ref{thm:sign1}.

Let us clarify the meaning of this decomposition.
Let $C(s)$ be the  $C$-matrix for $\Upsilon(s)$,
and let $\bfc_i(s)$ be its $i$th $c$-vector.
We do the same specialization $\varepsilon=\varepsilon_s$ for the $C$-matrix mutation
\eqref{2eq:cmutmat3}.
By the sign-coherence property, we have
\begin{align}
[-\varepsilon_s c_{\ell k_s}(s)]_+ = 0
\quad
(1 \leq\ell \leq n).
\end{align}
Thus, the mutation \eqref{2eq:cmutmat3}  is simplified
as
\begin{align}
 \label{2eq:cmutmat4}
\bfc_i(s+1)&=
 \begin{cases}
-\bfc_{k_s}(s)
 &
 i= k_s,
 \\
\bfc_i(s) +  [\varepsilon_s b_{k_s i}(s)]_+\bfc_{k_s}(s)
 &
 i \neq k_s.
 \end{cases}
\end{align}
We recognize it as the log-version of  \eqref{eq:tau1}.
By this reason, we call $\tau{(s)}$ and $\rho{(s)}$
the \emph{tropical part}\index{tropical part!of mutation} and the \emph{nontropical part}\index{nontropical part!of mutation} of $\mu{(s)}$,
respectively.

Next,
we introduce  the  composite mutation
$(s=0,\, \dots,\, P-1)$
\begin{align}
\label{eq:mus01}
\mu(s;0)&:=
\mu(0)\circ \mu(1)\circ \cdots \circ \mu(s)\colon
\bbP(s+1) \rightarrow\bbP(0).
\end{align}
Let us define its Fock-Goncharov decomposition.

Consider   a semifield isomorphism ($s=0,\, \dots,\, P-1$)
\begin{align}
\label{eq:taus01}
\tau(s;0)&:=
\tau(0)\circ \tau(1)\circ \cdots \circ \tau(s)\colon
\bbP(s+1)\rightarrow \bbP(0).
\end{align}
The map $\tau{(s;0)}$ is identified with the \emph{tropical part}\index{tropical part!of composite mutation} of  $\mu{(s;0)}$
as follows.

\begin{prop}
[e.g., {\cite[Prop.~II.4.3]{Nakanishi22a}}]
\label{prop:tauy1}
The following formula holds:
\begin{align}
\label{eq:taus1}
\tau(s;0)(y_i(s+1))=y^{\bfc_i(s+1)}.
\end{align}
\end{prop}
\begin{proof}
We prove it by the induction on $s$.
For $s=0$,
by \eqref{2eq:cmutmat4}, we have
\begin{align}
\label{eq:tau01}
\begin{matrix}
\tau{(0)}: & \bbP(1)& \rightarrow & \bbP(0)
\\
&
y_i(1)
&
\mapsto
&
\begin{cases}
\displaystyle
y_{k_0}^{-1}=y^{\bfc_{k_0}(1)}
& i=k_0,
\\
y_i y_{k_0} ^{[\varepsilon_0 b_{k_0 i}(0)]_+} 
=y^{\bfc_{i}(1)}
&i\neq k_0.
\end{cases}
\end{matrix}
\end{align}
Similarly, assuming \eqref{eq:taus1} for $s-1$, we obtain
\begin{align*}
&\quad \
\tau{(s;0)}(y_i(s+1))
=  (\tau{(s-1;0)} \circ \tau(s))(y_i(s+1))
\\
&=
\begin{cases}
\displaystyle
y^{-\bfc_{k_s}(s)}
& i=k_s,
\\
y^{\bfc_{i}(s)} y^{\bfc_{k_s}(s)}{}^{[\varepsilon_s b_{k_s i}(s)]_+} 
&i\neq k_s
\end{cases}
\\
&= y^{\bfc_i(s+1)},
\end{align*}
where we used  \eqref{2eq:cmutmat4} again in the last equality.
\end{proof}

Below, we consider the {nontropical part} of  $\mu{(s;0)}$.
Let us fix a  skew-symmetric decomposition
\begin{align}
\label{eq:Omegas1}
B(s)=\Delta \Omega(s),
\end{align}
where $\Delta=\diag(\d_1,\dots,\d_n)$ is common for   $B(s)$ ($s=0$, \dots, $P$).
Let $\Omega=\Omega(0)$.
We introduce
a skew-symmetric bilinear form $\{\bfn,\bfn'\}_{\Omega}$
on $\bbZ^n$  by
\begin{align}
\label{eq:Omega2}
\{\bfn,\bfn'\}_{\Omega}:=
\bfn^T \Omega \bfn'.
\end{align}
Also, for each $c$-vector $\bfc_{i}(s)$ and its tropical sign $\varepsilon_i(s)$, we set
\begin{align}
\label{eq:c+1}
\bfc^+_{i}(s):= \varepsilon_i(s)  \bfc_{i}(s).
\end{align}
This is a positive vector, and we call it a \emph{$c^+$-vector}\index{$c^+$-vector}.
Then, we introduce  a semifield automorphism ($s=0,\,\dots,\, P-1$)
\begin{align}
\label{eq:fq1}
\begin{matrix}
\frakq{(s)} \colon & \bbP(0) & \rightarrow &  \bbP(0)
\\
& y^{\bfn}& \mapsto & 
y^{\bfn}(1+y^{ \bfc^+_{k_s}(s)})^{-
\{\delta_{k_s} \bfc_{k_s}(s), \bfn\}_{\Omega}}.
\end{matrix}
\end{align}
Note that $\frakq{(0)}=\rho(0)$.

\begin{prop}
[{\cite[Prop.II.4.5]{Nakanishi22a}}]
\label{prop:qs1}
(a).
The following formulas hold:
\begin{gather}
\label{eq:ccO1}
 \{\delta_i \bfc_{i}(s), \bfc_{j}(s)\}_{\Omega} = b_{ij}(s),
\\
\label{eq:q2}
\frakq{(s)}(y^{\bfc_i(s)})=
y^{\bfc_i(s)}(1+y^{ \bfc^+_{k_s}(s)})^{-
b_{k_s i}(s)}.
\end{gather}
\par
(b). The following commutative diagram holds $(s=1,\, \dots,\, P-1)$:
\begin{align}
\label{eq:cd1}
\raisebox{25pt}
{
\begin{xy}
(0,0)*+{\bbP(s)}="aa";
(25,0)*+{\bbP(0)}="ba";
(0,-15)*+{\bbP(s)}="ab";
(25,-15)*+{\bbP(0)}="bb";
{\ar "aa";"ba"};
{\ar "ab";"bb"};
{\ar "aa";"ab"};
{\ar "ba";"bb"};
(12,3)*+{\text{\small $\tau(s-1;0)$}};
(12,-12)*+{\text{\small $\tau(s-1;0)$}};
(-5,-7.5)*+{\text{\small $\rho(s)$}};
(30,-7.5)*+{\text{\small $\frakq(s)$}};
\end{xy}
}
\end{align}

\end{prop}
\begin{proof}
(a).
The equality \eqref{eq:ccO1} is written in the matrix form
\begin{align}
\Delta C(s)^{T}  \Omega C(s) = \Delta \Omega(s).
\end{align}
This is
 equivalent to the equality \eqref{eq:db1}.
The equality \eqref{eq:q2} follows from  \eqref{eq:fq1} and \eqref{eq:ccO1}.
(b).
By
Proposition \ref{prop:tauy1}, \eqref{eq:rho1}, and \eqref{eq:q2},
we have
\begin{align}
\begin{split}
y_i(s) 
&\buildrel \tau(s-1;0) \over {\mapsto}y^{\bfc_i(s)} 
\buildrel \frakq(s) \over {\mapsto}
y^{\bfc_i(s)}(1+y^{ \bfc^+_{k_s}(s)})^{-
b_{k_s i}(s)},
\\
 y_i(s)
&\buildrel \rho(s) \over {\mapsto}
{\rp y_i(s)}  (1+ y_{k_s}(s)^{\varepsilon_s})^{-b_{k_s i}(s)}
\buildrel \tau(s-1;0) \over {\mapsto}
y^{\bfc_i(s)}(1+y^{ \bfc^+_{k_s}(s)})^{-
b_{k_s i}(s)}
.
\end{split}
\end{align}
\end{proof}

All ingredients introduced so far are
 summarized in a  single commutative diagram as follows:
\begin{align}
\label{eq:com1}
\raisebox{70pt}
{
\begin{xy}
(0,0)*+{\bbP(P)}="aa";
(25,0)*+{\bbP(P-1)}="ba";
(25,-15)*+{\bbP(P-1)}="bb";
(52,-15)*+{\bbP(P-2)}="cb";
(50,-30)*+{\bbP(2)}="cc";
(75,-30)*+{\bbP(1)}="dc";
(75,-45)*+{\bbP(1)}="dd";
(100,0)*+{\bbP(0)}="ea";
(100,-15)*+{\bbP(0)}="eb";
(100,-30)*+{\bbP(0)}="ec";
(100,-45)*+{\bbP(0)}="ed";
(100,-60)*+{\bbP(0)}="ee";
(10,3)*+{\text{\small $\tau(P-1)$}};
(62,3)*+{\text{\small $\tau(P-2;0)$}};
(76,-12)*+{\text{\small $\tau(P-3;0)$}};
(38,-12)*+{\text{\small $\tau(P-2)$}};
(62,-27)*+{\text{\small $\tau(1)$}};
(88,-27)*+{\text{\small $\tau(0)$}};
(88,-42)*+{\text{\small $\tau(0)$}};
(2,-7.5)*+{\text{\small $\mu(P-1)$}};
(55,-37.5)*+{\text{\small $\mu(1)$}};
(80,-52.5)*+{\text{\small $\mu(0)$}};
(33,-7.5)*+{\text{\small $\rho(P-1)$}};
(80,-37.5)*+{\text{\small $\rho(1)$}};
(109,-7.5)*+{\text{\small $\frakq(P-1)$}};
(105,-37.5)*+{\text{\small $\frakq(1)$}};
(105,-52.5)*+{\text{\small $\frakq(0)$}};
(100,-22)*+{\vdots};
(50,-22)*+{\vdots};
(37.5,-22)*+{\ddots};
{\ar "aa";"ba"};
{\ar "ba";"bb"};
{\ar "aa";"bb"};
{\ar "ba";"ea"};
{\ar "bb";"cb"};
{\ar "cb";"eb"};
{\ar "cc";"dc"};
{\ar "cc";"dd"};
{\ar "ea";"eb"};
{\ar "dc";"dd"};
{\ar "dc";"ec"};
{\ar "ec";"ed"};
{\ar "dd";"ed"};
{\ar "ed";"ee"};
{\ar "dd";"ee"};
\end{xy}
}
\end{align}
Finally,
we introduce a semifield automorphism
$(s=0,\, \dots,\,P-1)$
\begin{align}
\label{eq:q01}
\frakq(s;0)&:=
\frakq(0)\circ \frakq(1)\circ \cdots \circ \frakq(s)\colon
\bbP(0)\rightarrow \bbP(0).
\end{align}
From the  diagram \eqref{eq:com1}, we immediately obtain
the  decomposition
\begin{align}
\label{eq:decom1}
\mu(s;0)= \frakq(s;0) \circ \tau(s;0).
\end{align}
We call it the \emph{Fock-Goncharov decomposition}\index{Fock-Goncharov decomposition!of  composite mutation} of the 
composite
mutation  $\mu(s;0)$,
where $\frakq{(s;0)}$ is the \emph{nontropical part}\index{nontropical part!of composite mutation} of  $\mu{(s;0)}$.

\section{Mutation-compatible Poisson bracket}
\label{sec:mutation1}

Let us recall the most  basic notion  in Poisson geometry
(e.g., \cite{Weinstein83}).

\begin{defn}[Poisson bracket]
Let $\calC(M)$ be the  commutative algebra of all smooth real-valued functions on
a smooth manifold $M$. 
A bilinear map $\{\cdot, \cdot \}\colon \calC(M)\times \calC(M) \rightarrow \calC(M)$ is a \emph{Poisson bracket}\index{Poisson!bracket} on $M$ (or $\calC(M)$) if the following
conditions are satisfied.
\begin{itemize}
\item skew-symmetry: $\{f, g \}=-\{g,f \}$.
\item Leibniz rule: $\{ f , gh \}=\{ f , g \} h +g \{ f, h \}.$
\item Jacobi identity: $\{ \{ f , g \}, h\} + \{ \{ g , h \}, f\} +\{ \{ h , f \}, g\} =0.$
\end{itemize}
Equipped with a Poisson bracket,
the  manifold $M$ is called a \emph{Poisson manifold}\index{Poisson!manifold},
and
the algebra $\calC(M)$ is called a \emph{Poisson algebra}\index{Poisson!algebra}.
\end{defn}

For simplicity,
consider a smooth manifold $M= \bbR^n$
with global coordinates
    $\bfy=(y_1,\dots,y_n)$.
 Then, $\calC(M)$ is identified with  the algebra $\calC(\bfy)$ of all smooth functions 
 of $\bfy$.
 The following fact is due to Lie.
 \begin{prop}
 [{e.g., \cite{Weinstein83}}]
 \label{prop:Poi1}
 (a).
 Let $\{\cdot, \cdot\}$ be any poisson bracket on $\calC(M)$.
 Then, the following formula holds:
 \begin{align}
\label{eq:Poi1}
\{f,g\}=\sum_{i,j=1}^n 
\frac{\partial f}{\partial y_i}
\frac{\partial g}{\partial y_j}
\{ y_i, y_j\}.
\end{align}
(b).
Conversely, for any  data $w_{ij} \in \calC(M)$ $(i,\,j=1,\,\dots,\,r)$ satisfying 
\begin{gather}
\label{eq:skewy1}
w_{ij}=-w_{ji},
\\
\label{eq:Jacobiw1}
\sum_{\ell=1}^n
\left(
\frac{\partial w_{ij}}{\partial y_{\ell}} w_{\ell k}
+
\frac{\partial w_{jk}}{\partial y_{\ell}} w_{\ell i}
+
\frac{\partial w_{ki}}{\partial y_{\ell}} w_{\ell j}
\right)=0,
\end{gather}
the bracket $\{\cdot,\cdot\}$ defined by $\{y_i,y_j\}:=w_{ij}$
and \eqref{eq:Poi1}
 is a Poisson bracket on 
$\calC(\bfy)$,
Moreover, \eqref{eq:Jacobiw1} is written as
\begin{align}
\label{eq:Jacobiy1}
 \{\{y_i,y_j\},y_k\}
 +
 \{\{y_j,y_k\},y_i\}
 +
\{\{y_k,y_i\},y_j\}
 =0.
\end{align}
 \end{prop}

 \begin{proof}
 For the reader's convenience, we present an outline of the proof.
 
 (a).
 For  polynomial functions $f$, $g$ in $\bfy$,
the formula  \eqref{eq:Poi1} is deduced from 
the Leibniz rule.
Then, it is uniquely extended to smooth functions
by the Weierstrass approximation theorem.

(b).
The only nontrivial condition to be checked is the Jacobi identity.
Let us use  the abbreviations
$\partial f/\partial y_i = f_i$, $\partial^2 f/\partial y_i \partial y_j= f_{ij}$.
Then, we have
\begin{align}
\begin{split}
&\quad \
\{ \{ f , g \}, h\}
=
\biggl\{
\sum_{i,j}
f_ig_j w_{ij},
h
\biggr\}
\\
&=
\sum_{i,j,k,\ell}
(f_i g_j)_k h_{\ell}  w_{ij} w_{k\ell} 
+
\sum_{i,j,k,\ell}
f_i g_j h_{\ell} (w_{ij})_k  w_{k\ell} 
\\
&=
\sum_{i,j,k,\ell}
(f_{ik} g_j h_{\ell}  - g_{ik} h_{j}  f_{\ell}  )w_{ij}w_{k\ell},
+
\sum_{i,j,k,\ell}
f_i g_j h_{\ell} (w_{ij})_k  w_{k\ell},
\end{split}
\end{align}
where we used the skew-symmetry \eqref{eq:skewy1}
in the last equality.
The Jacobi identity follows 
from the last expression and the condition \eqref{eq:Jacobiw1}.
\end{proof}

We are interested in the following \emph{quadratic}\index{quadratic Poisson brackets}  Poisson brackets.
\begin{lem}
 \label{lem:Poi1}
Let $\Omega=(\omega_{ij})_{i,j=1}^n$ be any skew-symmetric real matrix.
Define the data   $w_{ij}\in \calC(\bfy)$ $(i,\,j=1,\,\dots,\,n)$ by
\begin{align}
\label{eq:Poi2}
w_{ij}=\{y_i,y_j\}:=\omega_{ij} y_i y_j.
\end{align}
Then, they satisfy the conditions \eqref{eq:skewy1} and \eqref{eq:Jacobiy1}.
Thus, they define a Poisson bracket on $\calC(\bfy)$.
\end{lem}
\begin{proof}
The condition \eqref{eq:skewy1} is clear.
The condition \eqref{eq:Jacobiy1} is verified by the  cyclic expression
\begin{align}
\{\{y_i,y_j\},y_k\}
=(\omega_{ij}\omega_{ik}-\omega_{jk}\omega_{ji})y_iy_jy_k.
\end{align}
\end{proof}

The variables $\bfy$ in 
the Poisson bracket  \eqref{eq:Poi2} 
  are called \emph{log-canonical variables}\index{log-canonical variable} 
by  \cite{Gekhtman02} because of the following  property:
\begin{align}
\label{eq:Poi3}
\{\log y_i,\log y_j\}=\omega_{ij}.
\end{align}
(The name might be a little confusing because  $\log y_i$ are not canonical in the  terminology for classical mechanics.)

Now we come back to the situation in the previous section.
Let  $\calC_+(\bfy(s))$ be the algebra of all smooth functions 
 of $\bfy(s)$
defined on $\bbR_{>0}^n$
with coordinates $\bfy(s)$.
By regarding 
the semifield isomorphism \eqref{eq:mu1}--\eqref{eq:rho1}
as  \emph{changes of variables},
we have
the  algebra isomorphisms  (denoted by the same symbols)
\begin{align}
\label{eq:mu2}
\mu(s)\colon &\ \calC_+(\bfy(s+1)) \rightarrow  \calC_+(\bfy(s)),
\\
\label{eq:tau2}
\tau(s)\colon &\ \calC_+(\bfy(s+1)) \rightarrow  \calC_+(\bfy(s)),
\\
\label{eq:rho2}
\rho(s)\colon &\ \calC_+(\bfy(s)) \rightarrow  \calC_+(\bfy(s)).
\end{align}
Now, 
we consider a smooth manifold $M\simeq\bbR_{>0}^n$
equipped with $P+1$ global positive coordinates $\bfy(s)$
($s=0,\,\dots,\,P$) such that $\bfy(s+1)$ and $\bfy(s)$ are glued 
by the tropical transformation $\tau(s)$ (\emph{not by the ordinary mutation $\mu(s)$}).
The algebra
$\calC(M)$ is identified with  $\calC_+(\bfy(s))$
in  the coordinates $\bfy(s)$ for each $s$.

Let $\Omega(s)$ be the rational skew-symmetric matrix in \eqref{eq:Omegas1}.
By  \eqref{2eq:bmut5},
we have the mutation formula
\begin{align}
\label{2eq:omut1}
\omega_{ij}(s+1)&=
\begin{cases}
-\omega_{ij}(s)
&
\text{$i=k$ or $j=k$,}
\\
\omega_{ij}(s)+
 [\varepsilon b_{kj}(s)]_+ \omega_{ik}(s)
+
[\varepsilon b_{ki}(s)]_+\omega_{kj}(s)
&
i,\,j\neq k.
\end{cases}
\end{align}
Following Gekhtman-Shapiro-Vainshtein \cite{Gekhtman02},
we define a Poisson bracket on each $\calC_+(\bfy(s))$ ($s=0$, \dots, $P-1$) by
\begin{align}
\label{eq:Poi4}
\{y_i(s),y_j(s)\}_s:=\omega_{ij}(s) y_i (s) y_j(s),
\end{align}
so that $y_i(s)$ are log-canonical variables.
They
 are \emph{mutation-compatible}\index{mutation-compatible (Poisson bracket)} in the following sense.
\begin{thm}
[{\cite[Thm.~1.4]{Gekhtman02}}]
\label{thm:comp1}
The maps
$\mu(s)$, $\tau(s)$, and $\rho(s)$ in 
 \eqref{eq:mu2}--\eqref{eq:rho2} commute
with the Poisson brackets \eqref{eq:Poi4}.
Namely, we have
\begin{align}
\label{eq:mu3}
\{\mu(s)(f), \mu(s)(g)\}_s &= \mu(s)(\{f,g\}_{s+1})
\quad
(f,g\in \calC_+(\bfy(s+1))),
\\
\label{eq:tau3}
\{\tau(s)(f), \tau(s)(g)\}_s &= \tau(s)(\{f,g\}_{s+1})
\quad
(f,g\in \calC_+(\bfy(s+1))),
\\
\label{eq:rho3}
\{\rho(s)(f), \rho(s)(g)\}_s &= \rho(s)(\{f,g\}_{s})
\quad
(f,g\in \calC_+(\bfy(s))).
\end{align}
\end{thm}

\begin{proof}
Let us first prove \eqref{eq:rho3}.
By Proposition \ref{prop:Poi1} (a),
it is enough to prove it for
$f(\bfy)=y_i(s)$ and $g(\bfy)=y_j(s)$.
We have
\begin{align*}
&\quad \
\{\rho(s)(y_i(s)), \rho(s)(y_j(s))\}_{s}
\\
&=
\{
y_i(s)
(1+ y_{k_s}(s)^{\varepsilon_s})^{-b_{k_s i}(s)},
y_j(s)
(1+ y_{k_s}(s)^{\varepsilon_s})^{-b_{k_s j}(s)}
\}_s
\\
&=
(1+ y_{k_s}(s)^{\varepsilon_s})^{-b_{k_s i}(s)}
(1+ y_{k_s}(s)^{\varepsilon_s})^{-b_{k_s j}(s)}
\{
y_i(s),
y_j(s)
\}_s
\\
&\qquad
+
y_i(s)
(1+ y_{k_s}(s)^{\varepsilon_s})^{-b_{k_s j}(s)}
\{
(1+ y_{k_s}(s)^{\varepsilon_s})^{-b_{k_s i}(s)}
,
y_j(s)
\}_s
\\
&\qquad
+
(1+ y_{k_s}(s)^{\varepsilon_s})^{-b_{k_s i}(s)}
y_j(s)
\{
y_i(s),
(1+ y_{k_s}(s)^{\varepsilon_s})^{-b_{k_s j}(s)}
\}_s
\\
\overset {\eqref{eq:Poi4}} & { =}
\omega_{ij}(s) 
(1+ y_{k_s}(s)^{\varepsilon_s})^{-b_{k_s i}(s)}
(1+ y_{k_s}(s)^{\varepsilon_s})^{-b_{k_s j}(s)}
y_i(s)
y_j(s)
\\
&=
\omega_{ij}(s) \rho(s)(y_i(s))\rho(s)(y_j(s))
=\rho(s)(\{ y_i(s), y_j(s)\}_{s}).
\end{align*}
The equality \eqref{eq:tau3} is proved similarly.
For example,
for $i,\,j \neq k$,
we have
\begin{align*}
&\quad \
\{\tau(s)(y_i(s+1)), \tau(s)(y_j(s+1))\}_{s}
\\
&=
\{
y_i(s)
y_{k_s}(s) ^{[\varepsilon_s b_{k_s i}(s)]_+} ,
y_j(s)
y_{k_s}(s) ^{[\varepsilon_s b_{k_s i}(s)]_+} 
\}_s
\\
&=
y_{k_s}(s) ^{[\varepsilon_s b_{k_s i}(s)]_+} 
y_{k_s}(s) ^{[\varepsilon_s b_{k_s j}(s)]_+} 
\{
y_i(s),
y_j(s)
\}_s
\\
&\qquad
+
y_i(s)
y_{k_s}(s) ^{[\varepsilon_s b_{k_s j}(s)]_+} 
\{
y_{k_s}(s) ^{[\varepsilon_s b_{k_s i}(s)]_+} 
,
y_j(s)
\}_s
\\
&\qquad
+
y_{k_s}(s) ^{[\varepsilon_s b_{k_s i}(s)]_+} 
y_j(s)
\{
y_i(s),
y_{k_s}(s) ^{[\varepsilon_s b_{k_s j}(s)]_+} 
\}_s
\\
\overset {\eqref{eq:Poi4}} & { =}
(\omega_{ij}(s) 
+ [\varepsilon_s b_{k_s i}(s)]_+ \omega_{k_s j}(s)
+[\varepsilon_s b_{k_s j}(s)]_+ \omega_{i k_s}(s)
)
\\
&\qquad
\times
y_{k_s}(s) ^{[\varepsilon_s b_{k_s i}(s)]_+} 
y_{k_s}(s) ^{[\varepsilon_s b_{k_s j}(s)]_+} 
y_i(s)
y_j(s)
\\
\overset {\eqref{2eq:omut1}}
& {=}
\omega_{ij}(s+1) \tau(s)(y_i(s+1))\tau(s)(y_j(s+1))
\\
&=\tau(s)(\{ y_i(s+1), y_j(s+1)\}_{s+1}).
\end{align*}
Other cases are similar and easier.
Lastly, the equality  \eqref{eq:mu3} follows from 
\eqref{eq:tau3} and \eqref{eq:rho3}.
\end{proof}

\section{Euler dilogarithm as Hamiltonian}
\label{sec:Euler2}
Let $\rho(s)$ be the automorphism of $\calC_+(\bfy(s))$ in \eqref{eq:rho2}.
Fock-Goncharov \cite[\S1.3]{Fock07}
made the following remarkable observation (translated in the present context):
\emph{With the   Poisson bracket \eqref{eq:Poi4},
the automorphism $\rho(s)$ is described by
the time-one  flow 
for the Hamiltonian given by the Euler dilogarithm.}
Below, we explain the meaning of this statement.

In general,
for an algebra of functions $\calC(M)$ with a Poisson bracket
$\{\cdot, \cdot\}$,
let $f_t \in \calC$ ($t\in \bbR$) be a one-parameter family of functions,
where $t$ is the \emph{continuous time}.
We specify 
a  function $\scrH\in \calC(M)$ called the \emph{Hamiltonian}\index{Hamiltonian}
that
 governs the time development of $f_t$ by the following
\emph{equation of motion}\index{equation of motion}:
\begin{align}
\label{eq:em1}
\dot f_t
 =\{\scrH, f_t\},
\quad
\dot f_t
:=
\frac{d}{dt} f_t.
\end{align}

We apply this idea
to our Poisson bracket
\eqref{eq:Poi4}
 to describe the automorphism $\rho(s)$.
We choose the {Hamiltonian} $\scrH(s)\in \calC(\bfy(s))$
that is
given
by the Euler dilogarithm as follows:
\begin{align}
\label{eq:Hs1}
\scrH(s):={\varepsilon_s}{\d_{k_s}} \mathrm{Li}_2(-y_{k_s}(s)^{\varepsilon_s})
=
-
{\varepsilon_s}{\d_{k_s}} 
\int_0^{y_{k_s}(s)^{\varepsilon_s}}
\frac{\log(1+y)}{y}\, dy.
\end{align}
We have
\begin{align}
\label{eq:em2}
\begin{split}
\{ \scrH(s), y_{i}(s)\}_s
&=
{\varepsilon_s}{\d_{k_s}}\{ \mathrm{Li}_2(-y_{k_s}(s)^{\varepsilon_s}), 
y_i(s)
 \}_s
 \\
 &=
 -
 {\d_{k_s}}
 \frac{\log(1+y_{k_s}(s)^{\varepsilon_s})}{y_{k_s}(s)^{\varepsilon_s}}
 \omega_{k_s i} (s) y_{k_s}(s)^{\varepsilon_s} y_i(s)
 \\
 &=
  y_i(s) \log(1+y_{k_s}(s)^{\varepsilon_s})^{- b_{k_s i} (s)}.
  \end{split}
\end{align}
In particular, we have
\begin{align}
\{ \scrH(s), y_{k_s}(s)\}_s=0.
\end{align}
Thus, we  have a solution of the equation of motion \eqref{eq:em1}
\begin{align}
\label{eq:traj1}
f_t = 
  y_i(s) (1+y_{k_s}(s)^{\varepsilon_s})^{- b_{k_s i} (s) t }.
\end{align}
By setting $t=0$ and $t=1$,
we obtain
\begin{align}
\label{eq:timeone1}
f_0=  y_i(s),
\quad
f_1 = 
  y_i(s) (1+y_{k_s}(s)^{\varepsilon_s})^{- b_{k_s i} (s) }
  =\rho(s)(y_i(s)).
\end{align}
This explains the meaning of ``the automorphism $\rho(s)$ is described by the time-one  flow
of $\scrH(s)$''.

By the formula  \eqref{eq:taus1}
and the compatibility \eqref{eq:tau3},
the automorphism $\frakq(s)$ in \eqref{eq:fq1} is described by the
time-one flow of the Hamiltonian
\begin{align}
\label{eq:tauH1}
\scrH_0(s):=
\tau(s-1;0)
(
\scrH(s)
)
=
{\varepsilon_s}{\d_{k_s}}
 \mathrm{Li}_2(-y^{ \bfc^+_{k_s}(s)})
\end{align}
in the initial $y$-coordinates $\bfy=\bfy(0)$ with the Poisson bracket $\{\cdot, \cdot\}=
\{\cdot, \cdot\}_0$.

More generally,
for  any integer vector $\bfc=(c_i)\in \bbZ^n$ and any real number $a$,
we may consider a Hamiltonian
\begin{align}
\label{eq:Hc1}
\scrH_{\bfc,a}= a \mathrm{Li}_2(- y^{\bfc})
\end{align}
for the Poisson bracket $\{\cdot, \cdot\}=
\{\cdot, \cdot\}_0$.
\begin{prop}
\label{prop:Hc1}
The time-one flow of
the Hamiltonian $\scrH_{\bfc,a}$ in \eqref{eq:Hc1}
is given by
\begin{align}
\label{eq:tHc1}
y^{\bfn} \mapsto y^{\bfn} (1+y^{\bfc})^{-a \{\bfc,\bfn\}_{\Omega}},
\end{align}
where $\{\bfc,\bfn\}_{\Omega}$ is the one in \eqref{eq:Omega2}.
\end{prop}
\begin{proof}
In the same way as
\eqref{eq:em2},
we have
\begin{align}
\{\scrH_{\bfc,a}, y_i\}=
- a \biggl(\sum_{j=1}^n c_j  \omega_{ji}\biggr) y_i \log (1+y^{\bfc}).
\end{align}
Then, by the Leibniz rule, we have
\begin{align}
\{\scrH_{\bfc,a}, y^{\bfn}\}=
- a \{\bfc,\bfn\}_{\Omega} y^{\bfn} \log (1+y^{\bfc}).
\end{align}
Thus, we obtain the claim.
\end{proof}

In summary, \emph{the Euler dilogarithm is built into cluster algebra theory} as the Hamiltonian
for the nontropical part 
of each mutation in a  $Y$-pattern.
This gives an  \emph{intrinsic connection} between the dilogarithm and cluster algebra theory.

\notes
The $\varepsilon$-expressions  and the connection to the tropicalization
are gradually recognized
by \cite{Berenstein05,Nakanishi10c,Nakanishi11a,Nakanishi11c}.
The Fock-Goncharov decomposition was introduced
by \cite{Fock03,Fock07} for $\varepsilon=1$
in studying the cluster varieties and their quantizations.
The signed version first appeared in \cite{Nagao10, Keller11} in the study of the Donaldson-Thomas invariants.
The presentation here follows  \cite{Nakanishi22a}.
A parallel decomposition
is also available for  $x$-variables \cite{Fock03, Nakanishi22a}.

The mutation-compatible Poisson bracket was introduced by \cite{Gekhtman02}.
A more extensive treatment is given in \cite{Gekhtman10}.
For Poisson geometry, see the reviews  \cite{Weinstein83, Weinstein98} and references therein.

The appearance of the dilogarithm as the Hamiltonian
 naturally guided Fock-Goncharov \cite{Fock03,Fock07}
to quantize  cluster algebras using the  \emph{quantum dilogarithm},
which we treat in Part IV.

\chapter{Classical mechanical method for dilogarithm identities}
\label{ch:classical1}

We reformulate the Hamiltonian system for 
mutations with \emph{canonical coordinates}.
The Legendre transformation in classical mechanics
turns the \emph{Hamiltonian} given by the Euler dilogarithm
into the \emph{Lagrangian} given by the modified Rogers dilogarithm.
Based on this perspective, we give an alternative and independent proof of Theorem \ref{thm:DI1}.
This proof provides a more intrinsic understanding of the DIs for $Y$-patterns,
where the constancy of the dilogarithm sum is regarded as a variant of Noether's theorem.

\section{Hamiltonian formulation of mutations}
In  the previous chapter, we observed that the Euler dilogarithm is built into
a $Y$-pattern as the Hamiltonian \eqref{eq:Hs1} for the nontropical part
of each mutation.
We then ask the following natural questions:
\begin{itemize}
\item[Q1.] Is there any relation between this fact and the DIs in 
Theorem \ref{thm:DI1}?
\item[Q2.] If so, how and why does 
 the Euler dilogarithm
 turn into
the modified Rogers dilogarithm 
in these DIs?
\end{itemize}
It was shown by Gekhtman-Nakanishi-Rupel \cite{Gekhtman16}
that the above questions are naturally answered 
if we reformulate the   Hamiltonian system with
\emph{canonical coordinates}
and transform the Hamiltonian into the \emph{Lagrangian}
by the \emph{Legendre transformation} in classical mechanics.

To this end,
let us recall and make it precise
how  the sequence of mutations
\eqref{eq:mseq5} is formulated as a Hamiltonian system
with   \emph{continuous time}
based on the setting and results in  Chapter \ref{ch:Fock1}.
\begin{enumerate}
\item
We consider a smooth manifold $M\simeq\bbR_{>0}^n$  equipped with
$P+1$ global positive coordinates by $\bfy(s)$ ($s=0,\, \dots,\,P$)
that are glued by the tropical transformations
$\tau(s)$  in \eqref{eq:tau2} (not by the ordinary mutation $\mu(s)$).
We call $M$ the \emph{phase space}\index{phase space}.
\item
\label{en:bij1}
As usual,  the change of coordinates (variables) 
$\tau(s):\calC_+(s+1)\rightarrow \calC_+(s)$
induces the bijection $\tau^* (s):\bbR_{>0}^n \rightarrow \bbR_{>0}^n$  in the \emph{opposite direction}.
Namely, $\tau^* (s)$ maps a point $(a_1,\dots,a_n)$  in the $\bfy(s)$-coordinates
 to the point $(a'_1,\dots,a'_n)$  in the $\bfy(s+1)$-coordinates, where
\begin{align}
\label{eq:a1}
a'_i
=
\begin{cases}
\displaystyle
a_{k_s}^{-1}
& i=k_s,
\\
a_i a_{k_s}^{[\varepsilon_s b_{k_s i}(s)]_+} 
&i\neq k_s.
\end{cases}
\end{align}
The same applies to other maps $\mu(s)$,  $\rho(s)$, and $\frakq(s)$
and induces the maps $\mu^*(s)$,  $\rho^*(s)$, and $\frakq^*(s)$
all from $\bbR_{>0}^n$ to itself.

\item
The phase space $M$ is equipped with the mutation-compatible Poisson
bracket in \eqref{eq:Poi4}.
\item
We consider the time span $[0,P]$
for the continuous time  $t$.
In each time span $[s,s+1]$ ($s=0$, \dots, {\rp $P-1$}),
the Hamiltonian $\scrH(s)$  is given in the $\bfy(s)$-coordinates by
\begin{align}
\label{eq:Hs2}
\scrH(s)={\varepsilon_s}{\d_{k_s}} \mathrm{Li}_2(-y_{k_s}(s)^{\varepsilon_s})
\end{align}
as  in \eqref{eq:Hs1}.
In the initial $\bfy$-coordinates, the Hamiltonian is given by
\begin{align}
\label{eq:Hs5}
\scrH_0(s)=
{\varepsilon_s}{\d_{k_s}}
 \mathrm{Li}_2(-y^{ \bfc^+_{k_s}(s)})
\end{align}
\item
By \eqref{eq:traj1}, in each time span $[s,s+1]$,
the  trajectory (the Hamiltonian flow)  of a point in the phase space $M$
along  time $t$
is described in the $\bfy(s)$-coordinates as follows:
Let $\bfa_s=(a_1, \cdots, a_n)\in \bbR_{>0}^n$ be the position of the point at $t=s$.
Then,  the  position $\bfa_t=(a'_1, \cdots, a'_n)$ at 
$t\in [s,s+1]$
is given by
\begin{align}
\label{eq:traj2}
  a'_i = a_i (1+a_{k_s})^{- b_{k_s i} (s) (t-s) }.
\end{align}
In particular, $a'_{k_s}=a_{k_s}$ for any $t\in [s,s+1]$.

\item
For a trajectory in $M$ in the full time span $[0,P]$,
if  the initial position  in the initial coordinates $\bfy=\bfy(0)$ is $\bfa_0$,
 the final position  in the $\bfy(P)$-coordinates 
  is given by $\mu^*(P-1;0)(\bfa_0)$,
  while the final position 
 in the initial coordinates $\bfy$ 
  is given by $\frakq^*(P-1;0)(\bfa_0)$.
\end{enumerate}

See Figure \ref{fig:sche1} for the summary of the above formulation.

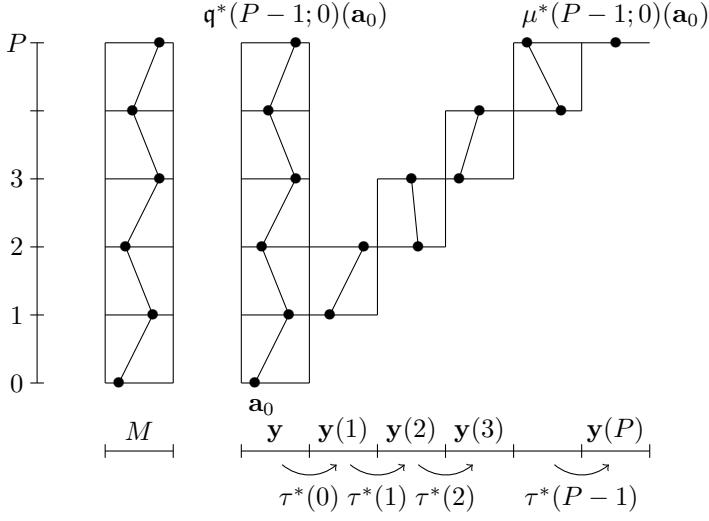
\begin{figure}
\begin{center}
\begin{tikzpicture}
[scale=0.9]
\draw (0,0)--(0,5);
\draw (-0.1,0)--(0.1,0);
\draw (-0.1,1)--(0.1,1);
\draw (-0.1,2)--(0.1,2);
\draw (-0.1,3)--(0.1,3);
\draw (-0.1,4)--(0.1,4);
\draw (-0.1,5)--(0.1,5);
\draw (1,0)--(2,0);
\draw (1,1)--(2,1);
\draw (1,2)--(2,2);
\draw (1,3)--(2,3);
\draw (1,4)--(2,4);
\draw (1,5)--(2,5);
\draw (1,0)--(1,5);
\draw (2,0)--(2,5);
\draw (1,-1)--(2,-1);
\draw (1,-1.1)--(1,-0.9);
\draw (2,-1.1)--(2,-0.9);
\draw (3,-1)--(9,-1);
\draw (3,-1.1)--(3,-0.9);
\draw (4,-1.1)--(4,-0.9);
\draw (5,-1.1)--(5,-0.9);
\draw (6,-1.1)--(6,-0.9);
\draw (7,-1.1)--(7,-0.9);
\draw (8,-1.1)--(8,-0.9);
\draw (9,-1.1)--(9,-0.9);
\draw (3,0)--(4,0);
\draw (3,1)--(5,1);
\draw (3,2)--(6,2);
\draw (3,3)--(4,3);
\draw (3,4)--(4,4);
\draw (3,5)--(4,5);
\draw (3,0)--(3,5);
\draw (4,0)--(4,5);
\draw (5,3)--(7,3);
\draw (6,4)--(8,4);
\draw (7,5)--(9,5);
\draw (5,1)--(5,3);
\draw (6,2)--(6,4);
\draw (7,3)--(7,5);
\draw (8,4)--(8,5);
\draw (1.2,0) node {$\bullet$};
\draw (1.7,1) node {$\bullet$};
\draw (1.3,2) node {$\bullet$};
\draw (1.8,3) node {$\bullet$};
\draw (1.4,4) node {$\bullet$};
\draw (1.8,5) node {$\bullet$};
\draw (1.2,0)--(1.7,1);
\draw (1.7,1)--(1.3,2);
\draw (1.3,2)--(1.8,3);
\draw (1.8,3)--(1.4,4);
\draw (1.4,4)--(1.8,5);
\draw (3.2,0) node {$\bullet$};
\draw (3.7,1) node {$\bullet$};
\draw (3.3,2) node {$\bullet$};
\draw (3.8,3) node {$\bullet$};
\draw (3.4,4) node {$\bullet$};
\draw (3.8,5) node {$\bullet$};
\draw (3.2,0)--(3.7,1);
\draw (3.7,1)--(3.3,2);
\draw (3.3,2)--(3.8,3);
\draw (3.8,3)--(3.4,4);
\draw (3.4,4)--(3.8,5);
\draw (4.3,1) node {$\bullet$};
\draw (4.8,2) node {$\bullet$};
\draw (5.6,2) node {$\bullet$};
\draw (5.5,3) node {$\bullet$};
\draw (6.2,3) node {$\bullet$};
\draw (6.5,4) node {$\bullet$};
\draw (7.7,4) node {$\bullet$};
\draw (7.2,5) node {$\bullet$};
\draw (8.5,5) node {$\bullet$};
\draw (4.3,1)--(4.8,2);
\draw (5.6,2)--(5.5,3);
\draw (6.2,3)--(6.5,4);
\draw (7.7,4)--(7.2,5);
\draw (-0.3,0) node {0};
\draw (-0.3,1) node {1};
\draw (-0.3,2) node {2};
\draw (-0.3,3) node {3};
\draw (-0.3,5) node {$P$};
\draw (1.5,-0.7) node {$M$};
\draw (3.5,-0.74) node {$\bfy$};
\draw (4.5,-0.7) node {$\bfy(1)$};
\draw (5.5,-0.7) node {$\bfy(2)$};
\draw (6.5,-0.7) node {$\bfy(3)$};
\draw (8.5,-0.7) node {$\bfy(P)$};
\draw (4,-1.7) node {{\rp $\tau^*(0)$}};
\draw (5,-1.7) node {{\rp $\tau^*(1)$}};
\draw (6,-1.7) node {{\rp $\tau^*(2)$}};
\draw (8,-1.7) node {$\tau^*(P-1)$};
\draw (3.3,-0.35) node {$\bfa_0$};
\draw (3.8,5.4) node {$\frakq^*(P-1;0)(\bfa_0)$};
\draw (8.5,5.4) node {$\mu^*(P-1;0)(\bfa_0)$};
\draw[->] (3.6,-1.2) to [bend right=40]   (4.4,-1.2);
\draw[->] (4.6,-1.2) to [bend right=40]   (5.4,-1.2);
\draw[->] (5.6,-1.2) to [bend right=40]   (6.4,-1.2);
\draw[->] (7.6,-1.2) to [bend right=40]   (8.4,-1.2);
\end{tikzpicture}
\end{center}
\vskip-10pt
\caption{Schematic diagram of Hamilton formulation of mutations.}
\label{fig:sche1}
\end{figure}

\section{Canonical coordinates}
\label{sec:canonical1}

We will ``embed'' the above Hamiltonian system into the one with
canonical coordinates.
Let us first concentrate on the initial coordinates $\bfy$, for simplicity.
Let $\bfu=(u_1,\dots,u_n)$ and $\bfp=(p_1,\dots,p_n)$
be variables.
We regard the pair $(\bfu, \bfp)$
 as global coordinates of   a manifold $\tilde M \simeq \bbR^{2n}$.
Let $\calC(\bfu,\bfp)$ be the set of all smooth functions
 of $(\bfu,\bfp)$ defined on $\bbR^{2n}$.
We consider the following standard Poisson bracket on $\tilde M$:
\begin{align}
\label{eq:Poi5}
\{f,g\}:=\sum_{i=1}^n
\left(
\frac{\partial f}{\partial p_i}
\frac{\partial g}{\partial u_i}
-
\frac{\partial g}{\partial p_i}
\frac{\partial f}{\partial u_i}
\right)
\quad
(f,g \in \calC(\bfu,\bfp)).
\end{align}
Thus, $(\bfu,\bfp)$ are \emph{canonical coordinates (variables)}\index{canonical!coordinates}\index{canonical!variables, \see{coordinates}}
in classical mechanics  (also called  \emph{Darboux coordinates}\index{Darboux coordinates, \see{canonical coordinates}} in Poisson geometry).
The following properties hold.
\begin{gather}
\label{eq:Poipu1}
\{p_i,u_j\}=\delta_{ij},
\quad
\{u_i,u_j\}=\{p_i,p_j\}=0,
\\
\{f,u_i\}= \frac{\partial f}{\partial p_i},
\quad
\{f,p_i\}= - \frac{\partial f}{\partial u_i}
\quad
(f\in \calC(\bfu,\bfp)).
\end{gather}
(Here and below, we use both
 the 
Kronecker delta $\delta_{ij}$
and the symbol   $\delta_i$ for  a component of $\Delta$.
We ask
the reader to distinguish them carefully.)

We set
\begin{align}
\label{eq:ty2}
\tilde{y}_i := \exp\left(\frac{1}{\sqrt{2}}(\d^{-1}_i p_i + w_i)\right),
\quad
w_i:=\sum_{j=1}^n b_{ji}  u_j.
\end{align}

\begin{prop}
[{\cite[Prop.~3.3]{Gekhtman16}}]
The following formula holds.
\begin{align}
\label{eq:tyty1}
\{\tilde{y}_i, \tilde{y}_j\}=\omega_{ij}\tilde{y}_i \tilde{y}_j.
\end{align}
\end{prop}
\begin{proof}
As noted in \eqref{eq:Poi3},
the claim is equivalent to
\begin{align}
\{\log \tilde{y}_i, \log \tilde{y}_j\}=\omega_{ij}.
\end{align}
Indeed, we have
\begin{align}
\begin{split}
\{\log \tilde{y}_i, \log \tilde{y}_j\}
&
=
\frac{1}{2} 
\{ \d^{-1}_i p_i, w_j\}
+
\frac{1}{2} 
\{ \d^{-1}_j p_j, w_i\}
\\
&
=\frac{1 }{2} \omega_{ij}
-
\frac{1 }{2} \omega_{ji}
=\omega_{ij}.
\end{split}
\end{align}
\end{proof}

The bracket \eqref{eq:tyty1} coincides with the one  \eqref{eq:Poi4} with $s=0$.
Moreover, the functions (variables) $\tilde  y_1$, \dots, $\tilde  y_n$ are algebraically independent.
Thus, we have an embedding of   the Poisson algebra
\begin{align}
\label{eq:eta1}
\begin{matrix}
\eta \colon & \calC_+(\bfy) & \rightarrow & \calC(\bfu,\bfp)
\\
&
y_i & \mapsto & \tilde{y}_i.
\end{matrix}
\end{align}

Let us concentrate on the time span $[0,1]$.
We consider the Hamiltonian\begin{align}
\label{eq:Hs3}
\scrH={\d_{k_0}} \mathrm{Li}_2(-\tilde y_{k_0})
\end{align}
corresponding to $\scrH(0)$ in \eqref{eq:Hs2}.
We have
\begin{align}
\{\scrH, \tilde y_{k_0}\}=0.
\end{align}
Therefore, the function $\tilde y_{k_0}$ is constant along the time development.
The equations of motion for the canonical coordinates are given by
\begin{align}
\dot u_i &=\{\scrH,u_i\}
=
\frac{\partial \scrH}{\partial p_i}
=-  \frac{1}{\sqrt{2}}  \delta_{k_0 i}  \log (1+ \tilde y_{k_0}),
\\
\dot p_i &=\{\scrH,p_i\}
= - \frac{\partial \scrH}{\partial u_i}
= - \frac{1}{\sqrt{2}} \d_i b_{k_0 i } \log (1+ \tilde y_{k_0} ),
\end{align}
where we used $b_{ik_0}\d_{k_0}=- b_{k_0i} \d_i$
in the second formula.
Thus, for $t\in [0,1]$, the functions $u_i$, $p_i$, $w_i$, $\d^{-1}_ip_i$ develop
linearly as
\begin{align}
\label{eq:u2}
 u_i (t)&
=  u_i -  \left(  \frac{1}{\sqrt{2}}  \delta_{k_0 i}  \log (1+ \tilde y_{k_0}) \right) t ,
\\
\label{eq:p2}
 p_i (t)&
=  p_i -  \left(  \frac{1}{\sqrt{2} } \d_i b_{k_0 i}  \log (1+ \tilde y_{k_0}) \right) t ,
\\
\label{eq:w2}
 w_i (t)&
=  w_i -  \left(  \frac{1}{\sqrt{2}}  b_{k_0 i}  \log (1+ \tilde y_{k_0}) \right) t ,
\\
\label{eq:dp2}
\d^{-1}_i  p_i (t)& = \d^{-1}_i p_i  - \left( \frac{1}{\sqrt{2}} b_{k_0 i } \log (1+ \tilde y_{k_0} ) \right) t,
\\
\label{eq:ty3}
{\rp \tilde y_i(t)}
&
=
\tilde y_i (1+ \tilde y_{k_0})^{-  b_{k_0 i} t}.
\end{align}
The last one agrees with the result \eqref{eq:traj1} as expected.
Therefore,
the time-one flow of the Hamiltonian $\scrH$ induces the following automorphism
\begin{align}
\label{eq:rho5}
\begin{matrix}
\rho(0) \colon & \calC (\bfu,\bfp) & \rightarrow &  \calC(\bfu,\bfp)
\\
&
u_i
&
\mapsto
&
\displaystyle
u_i
-   \frac{1}{\sqrt{2}}  \delta_{k_0 i}  \log (1+ \tilde y_{k_0}) ,
\\
&
p_i
&
\mapsto
&
\displaystyle
\rule{0pt}{18pt}
p_i
-   \frac{1}{\sqrt{2} }  \d_i b_{k_0 i}  \log (1+ \tilde y_{k_0}),
\\
&
w_i
&
\mapsto
&
\displaystyle
\rule{0pt}{18pt}
w_i
-   \frac{1}{\sqrt{2}}  b_{k_0 i}  \log (1+ \tilde y_{k_0}) ,
\\
&
\d^{-1}_i p_i
&
\mapsto
&
\displaystyle
\rule{0pt}{18pt}
\d^{-1}_i p_i
-   \frac{1}{\sqrt{2} }  b_{k_0 i}  \log (1+ \tilde y_{k_0}),
\\
&
\tilde y_i
&
\mapsto
&
\displaystyle
\tilde y_i (1+ \tilde y_{k_0})^{-  b_{k_0 i}}.
\end{matrix}
\end{align}
Again, the last one is compatible with the maps $\rho(0)=\frakq(0)$ in  \eqref{eq:rho2}.

\begin{prop}
The map $\rho(0)$ preserves the Poisson bracket in \eqref{eq:Poipu1}.
\end{prop}
\begin{proof}
This is a special case of a well-known fact that 
the Hamiltonian flow preserves the Poisson bracket (e.g., \cite[Pop.~3.3.4]{Abraham78}).
Alternatively,
one can directly check it by \eqref{eq:rho5}.
\end{proof}

A drawback of this formulation is that
 we have $2n$ variables $(\bfu,\bfp)$,
which are twice of the original $y$-variables $\bfy$.
To eliminate this redundancy,
let us introduce the subset $\tilde M_0$ of $\tilde M$
consisting of the points such that
 in the $(\bfu,\bfp)$-coordinates the following condition holds:
\begin{align}
\label{eq:dpw2}
p_i = \d_i w_i
\quad
(i=1,\, \dots, \, n).
\end{align}

\begin{lem}
\label{lem:small1}
The subset $\tilde M_0$ is preserved under the time-development by
the Hamiltonian $\scrH$ in \eqref{eq:Hs3}.
\end{lem}
\begin{proof}
This is clear from \eqref{eq:w2} and \eqref{eq:dp2}.
\end{proof}

We call $\tilde M_0$ the \emph{small phase space}\index{small phase space}\index{phase space!small ---}.
The algebra $\calC(\tilde M_0)$ of the smooth functions on $\tilde M_0$ is identified with
the quotient
\begin{align}
\calC(\bfu,\bfp)_0:= \calC(\bfu,\bfp)/ I_0,
\end{align}
where $I_0$ is the ideal of $\calC(\bfu,\bfp)$ consisting of the functions
vanishing on $\tilde M_0$.

This causes, however, a new problem.
Due to the definition \eqref{eq:ty2} of $\tilde y$,
the map
\begin{align}
\label{eq:eta2}
\begin{matrix}
\eta_0 \colon & \calC_+(\bfy) & \rightarrow & \calC(\bfu,\bfp)_0
\\
&
y_i & \mapsto & \tilde{y}_i.
\end{matrix}
\end{align}
is injective (and also bijective) if and only if the matrix $B$ is nonsingular.
Equivalently, $ \tilde{y}_i$'s are algebraically independent
in $\calC(\bfu,\bfp)_0$
if and only if the matrix $B$ is nonsingular.
For example, in the extreme case $B=0$,
$\tilde y_i=1$ for any $i$.
We will come back to this issue later.
At this moment
we do not require that $B$ is nonsingular.

\section{Tropical transformations of canonical coordinates}

The  canonical coordinates for the time span $[0,1]$
in the previous section can be generalized for each time span $[s,s+1]$ ($s=0,\,\dots,\,P-1$)
just by repeating the same construction.
So, we quickly summarize the notation and results.

For each $s=0,\, \dots,\, P-1$,
we introduce  variables $\bfu(s)=(u_i(s))_{i=1}^n$ and $\bfp(s)=(p_i(s))_{i=1}^n$,
and regard the pair $(\bfu(s),\bfp(s))$ as  global coordinates of $\tilde M \simeq \bbR^{2n}$.
The Poisson bracket $\{\cdot, \cdot\}_s$ on $\tilde M$ is defined
as in \eqref{eq:Poi5} by replacing $u_i$ and $p_i$ therein with $u_i(s)$ and $p_i(s)$.
We set
\begin{align}
\label{eq:tyw1}
\tilde{y}_i(s) := \exp\left(\frac{1}{\sqrt{2}}(\d^{-1}_i p_i (s)+ w_i(s))\right),
\quad
w_i(s):=\sum_{j=1}^n b_{ji}(s)  u_j(s).
\end{align}
Then, we have
\begin{align}
\{\tilde{y}_i(s), \tilde{y}_j(s)\}_s=\omega_{ij}(s)\tilde{y}_i (s)\tilde{y}_j(s),
\end{align}
which agrees with \eqref{eq:Poi4}.
So, we have an embedding
\begin{align}
\label{eq:emb1}
\begin{matrix}
\eta(s)\colon
&
\calC_+(\bfy(s)) & \rightarrow & \calC(\bfu(s),\bfp(s))
\\
&y_i(s) & \mapsto & \tilde{y}_i(s).
\end{matrix}
\end{align}
The Hamiltonian $\scrH(s)$ is given by
\begin{align}
\label{eq:Hs4}
\scrH(s)={\varepsilon_s}{\d_{k_s}} \mathrm{Li}_2(-\tilde y_{k_s}(s)^{\varepsilon_s}).
\end{align}
The equations of motion for the canonical coordinates are given as
\begin{align}
\label{eq:us1}
\dot u_i (s)&=
-  \frac{1}{\sqrt{2}}  \delta_{k_s i}  \log (1+ \tilde y_{k_s}(s)^{\varepsilon_s}),
\\
\label{eq:ps1}
\dot p_i(s) &
= - \frac{1}{\sqrt{2} }\d_i  b_{k_s i }(s) \log (1+ \tilde y_{k_s}(s)^{\varepsilon_s} ).
\end{align}
The time-one flow of the Hamiltonian $\scrH(s)$ induces the
 automorphism
\begin{align}
\notag
\rho(s) \colon \calC (\bfu(s),&\, \bfp(s))  \rightarrow   \calC(\bfu(s),\bfp(s))
\\
\label{eq:rhou1}
u_i(s)
&\mapsto
\displaystyle
u_i(s)
-   \frac{1}{\sqrt{2}}  \delta_{k_s i}  \log (1+ \tilde y_{k_s}(s)^{\varepsilon_s}) ,
\\
\label{eq:rhop1}
p_i(s)
&
\mapsto
\displaystyle
p_i(s)
-   \frac{1}{\sqrt{2} }  \d_i b_{k_s i}(s)  \log (1+ \tilde y_{k_s}(s)^{\varepsilon_s}),
\\
\label{eq:rhow1}
 w_i (s)&
\mapsto
\displaystyle
w_i(s) -  \left(  \frac{1}{\sqrt{2}}  b_{k_s i}(s)  \log (1+ \tilde y_{k_s}(s)^{\varepsilon_s}) \right),
\\
\label{eq:rhodp1}
\d^{-1}_i  p_i (s)& \mapsto 
\displaystyle
\d^{-1}_i p_i (s) - \left( \frac{1}{\sqrt{2}} b_{k_s i }(s) \log (1+ \tilde y_{k_s}(s)^{\varepsilon_s} ) \right),
\\
\label{eq:rhoy1}
\tilde y_i (s)& \mapsto 
\tilde y_i (s) 
 (1+ \tilde y_{k_s} (s)^{\varepsilon_s})^{-b_{k_s i } (s)}.
\end{align}
Moreover, it preserves the Poisson bracket $\{\cdot, \cdot\}_s$.

Now we glue
the canonical coordinates $(\bfu(s+1),\bfp(s+1))$ and
$(\bfu(s),\bfp(s))$
 by the following linear transformation,
which is the counterpart of the tropical transformation $\tau(s)$ in \eqref{eq:tau2}.
\begin{align}
\notag
\tau{(s)}\colon   \calC(\bfu(s+1),&\, \bfp(s+1))) \rightarrow \calC(\bfu(s),\bfp(s)),
\\
\label{eq:tauu1}
 u_i(s+1)
& \mapsto
\begin{cases}
\displaystyle
-u_{k_s}(s)
 + \sum_{j=1}^n [-\varepsilon_s b_{jk_s}(s)]_+ u_j(s)
& i=k_s,
\\
u_i (s)
&i\neq k_s,
\end{cases}
\\
\label{eq:taup1}
p_i(s+1)
& \mapsto
\begin{cases}
\displaystyle
-p_{k_s}(s)
& i=k_s,
\\
p_i (s)
+[- \varepsilon_s b_{ik_s}(s)]_+ p_{k_s}(s)
&i\neq k_s,
\end{cases}
\\
\label{eq:ws4}
w_i(s+1)
&\mapsto
\begin{cases}
\displaystyle
-w_i(s)
& i=k_s,
\\
w_i (s)
+ [\varepsilon_s b_{k_s i }(s)]_+ w_{k_s}(s)
&i\neq k_s,
\end{cases}
\\
\label{eq:dps4}
\d^{-1}_i p_i(s+1)
&\mapsto
\begin{cases}
\displaystyle
-\d^{-1}_{k_s} p_{k_s}(s)
& i=k_s,
\\
\d^{-1}_i p_i (s)
+[\varepsilon_s b_{k_s i }(s)]_+  \d^{-1}_{k_s} p_{k_s}(s)
&i\neq k_s,
\end{cases}
\\
\label{eq:ys4}
 \tilde y_i(s+1)
&\mapsto
\begin{cases}
\displaystyle
\tilde y_{k_s}(s)^{-1}
& i=k_s,
\\
\tilde y_i(s) \tilde y_{k_s}(s) ^{[\varepsilon_s b_{k_s i}(s)]_+} 
&i\neq k_s.
\end{cases}
\end{align}
By \eqref{eq:ys4},
we have a  commutative diagram:
\begin{align}
\raisebox{25pt}
{
\begin{xy}
(0,0)*+{ \calC(\bfy(s+1))}="aa";
(45,0)*+{\calC(\bfu(s+1),\bfp(s+1))}="ba";
(0,-15)*+{ \calC(\bfy(s))}="ab";
(45,-15)*+{ \calC(\bfu(s),\bfp(s))}="bb";
{\ar "aa";"ba"};
{\ar "ab";"bb"};
{\ar "aa";"ab"};
{\ar "ba";"bb"};
(18,3)*+{\text{\small $\eta(s+1)$}};
(18,-12)*+{\text{\small $\eta(s)$}};
(-5,-7.5)*+{\text{\small $\tau(s)$}};
(40,-7.5)*+{\text{\small $\tau(s)$}};
\end{xy}
}
\end{align}
\begin{prop}
[{\cite[Prop.~3.9]{Gekhtman16}}]
\label{prop:tauP1}
(a).
The  map $\tau(s)$ commutes
with the Poisson brackets.
Namely, we have
\begin{align}
\label{eq:tauup1}
\{\tau(s)(f), \tau(s)(g)\}_s &= \tau(s)(\{f,g\}_{s+1})
\quad
(f,\,g\in \calC(\bfu(s+1),\bfp(s+1))).
\end{align}
\par
(b). The following property holds:
\begin{align}
\label{eq:updual1}
\tau(s)
\left(
 \sum_{i=1}^n
u_i(s+1) p_i(s+1)
\right)
= \sum_{i=1}^n
u_i(s) p_i(s).
\end{align}
\end{prop}
\begin{proof}
(a).
It is enough to check it for $f=p_i(s+1)$ and $g=u_j(s+1)$.
So,  let us show
$\{\tau(s)(p_i(s+1)), \tau(s)(u_j(s+1))\}_s=\delta_{ij}$.
By \eqref{eq:tauu1} and \eqref{eq:taup1},
the map $\tau(s)$ acts on
$\bfu(s+1)$ and $\bfp(s+1)$ separately and linearly.
Let ${\rp M}$ and ${\rp N}$ be their matrix representations
with respect to the bases $\bfu(s+1)$, $\bfu(s)$ and $\bfp(s+1)$, $\bfp(s)$,
respectively.
Then, it is easy to check that ${\rp N^{T}=M=M^{-1}}$ holds.
Thus, the equality \eqref{eq:tauup1} holds.
(b). The same proof is applicable.
\end{proof}

 Define the mutation $\mu(s)$ by the composition
 \begin{align}
 \label{eq:mu7}
 \mu(s):=\rho(s)\circ\tau(s)\colon
 \calC(\bfu(s+1),\bfp(s+1))\rightarrow \calC(\bfu(s),\bfp(s)).
 \end{align}
 We also define 
 $\mu(s;0)$, $\tau(s;0)$,
 $\frakq(s)$, and $\frakq(s;0)$
  in a parallel way with the ones in Section \ref{sec:Fock1}.

Let us turn to the small phase space $\tilde M_0$ in the previous section.
Let $\tilde M_0(s)$ be set of the points in the phase space $\tilde M$
such that in the $(\bfu(s),\bfp(s))$-coordinates the following condition holds:
\begin{align}
\label{eq:dpw1}
 p_i(s) = \d_iw_i(s)
 \quad
 (i=1,\, \dots,\, n).
\end{align}

\begin{lem}
\label{lem:small2}
The subset $\tilde M_0(s)$ is preserved under the time-development by
the Hamiltonian $\scrH(s)$.
Moreover,
the set $\tilde M_0(s)$ is independent of $s$, namely,
$\tilde M_0(s)=\tilde M_0$.
\end{lem}
\begin{proof}
The first claim is proved in the same way as Lemma \ref{lem:small1}.
Also, by \eqref{eq:ws4} and \eqref{eq:dps4},
the condition \eqref{eq:dpw1} is preserved by $\tau(s)$.
\end{proof}

By  Lemma \ref{lem:small2},
we can restrict our Hamiltonian system on the small phase space  $\tilde M_0$
through the full time-span $[0,P]$ whenever necessary.

Again, we have a map
\begin{align}
\label{eq:eta3}
\begin{matrix}
\eta_0(s) \colon & \calC(\bfy(s)) & \rightarrow & \calC(\bfu(s),\bfp(s))_0:= \calC(\bfu(s),\bfp(s))/ I_0(s)
\\
&
y_i(s) & \mapsto & \tilde{y}_i(s),
\end{matrix}
\end{align}
where $I_0(s)$ is the ideal of $ \calC(\bfu(s),\bfp(s))$ consisting of the functions vanishing on $\tilde M_0$.
It is injective (also bijective) if and only if $B(s)$ is nonsingular.

\section{Modified Rogers dilogarithm as Lagrangian}
\label{sec:modified2}

Let us recall  basic facts on  the Legendre transformation
and the Lagrangian.
See, for example, \cite{Abraham78} for more information.

For simplicity, consider a Hamiltonian $\scrH$ on the phase space $M\simeq \bbR^{2n}$
with canonical coordinates  $(\bfu,\bfp)$ as global coordinates.
The space $M$ is  identified with the contangent bundle $\pi: T^*\bbR^n \rightarrow \bbR^n$
so that $\pi(\bfu,\bfp) \mapsto \bfu$.
Then, the Hamiltonian $\scrH$ induces the following fiber-preserving map:
\begin{align}
\label{eq:F1}
\begin{matrix}
F_{\scrH}\colon & T^*\bbR^n&\rightarrow&T^*\bbR^n
\\
& (\bfu,\bfp)& \mapsto &(\bfu, \dot \bfu),
&\quad
\dot u_i = \{\scrH, u_i\}.
\end{matrix}
\end{align}

\begin{defn}
A Hamiltonian $\scrH$ is \emph{regular}\index{regular!Hamiltonian} (resp.~\emph{singular}\index{singular!Hamiltonian}) if the map $F_{\scrH}$ is a diffeomorphism
(resp.~otherwise).
\end{defn}

Suppose that a Hamiltonian $\scrH$ is regular.
The \emph{Lagrangian}\index{Lagrangian} $\scrL$ for  $\scrH$ is defined by the \emph{Legendre transformation}\index{Legendre transformation}

\begin{align}
\label{eq:Lag1}
\scrL (\bfu, \dot{\bfu})= \sum_{i=1}^n \dot u_i p_i - \scrH(\bfu,\bfp).
\end{align}
Here and below the symbol $\scrL$ should not be confused with the Rogers dilogarithm $L(x)$.
The  RHS of \eqref{eq:Lag1} is a function of $(\bfu,\bfp)$.
However, by converting the map \eqref{eq:F1}, we may regard it as a function
 of $(\bfu,\dot\bfu)$.
The equations of motion for $\scrH$ are equivalent to
the following equations:
\begin{align}
\label{eq:EL1}
p_i &= \frac{\partial\scrL}{\partial \dot u_i},
\\
\label{eq:EL2}
\frac{d}{dt}\left(
\frac{\partial\scrL}{\partial \dot u_i}
\right)
&=
\frac{\partial \scrL}{\partial u_i}
\quad
\text{(the \emph{Euler-Lagrange equations}\index{Euler-Lagrange equations})}.
\end{align}
Moreover, the Lagrangian $\scrL$ is regular in the following sense
\cite[\S3.6]{Abraham78}.

\begin{defn}
A Lagrangian $\scrL$ is \emph{regular}\index{regular!Lagrangian} (resp.~singular\index{singular!Lagrangian}) if the map
\begin{align}
\label{eq:F2}
\begin{matrix}
F_{\scrL}\colon & T^*\bbR^n&\rightarrow&T^*\bbR^n
\\
& (\bfu,\dot \bfu)& \mapsto &(\bfu,  \bfp),
&\quad
\displaystyle
 p_i = \frac{\partial \scrL}{\partial \dot u_i}(\bfu,\dot\bfu).
\end{matrix}
\end{align}
is a diffeomorphism (resp.~otherwise).
\end{defn}

Conversely, starting from a regular Lagrangian $\scrL$, 
we obtain a regular Hamiltonian $\scrH$ by reversing the procedure.
Moreover, the two systems are equivalent again.

Coming back to the current situation,
for the Hamiltonian
$\scrH(s)$ in \eqref{eq:Hs4},
the map \eqref{eq:F1} is given by
\begin{gather}
F_{\scrH(s)}\colon  (\bfu(s),\bfp(s)) \mapsto  (\bfu(s), \dot \bfu(s)),
\\
\label{eq:ud1}
\dot u_i (s)=
\displaystyle
-  \frac{1}{\sqrt{2}}  \delta_{k_s i}  \log (1+ \tilde y_{k_s}(s)^{\varepsilon_s}).
\end{gather}
Unfortunately,
this is not surjective at all. Thus, the Hamiltonian $\scrH(s)$
is singular.
Nevertheless, let us apply the Legendre transformation to $\scrH(s)$
to see what we get. We have
\begin{align}
\label{eq:lag1}
\scrL(s) =
 \dot u_{k_s}(s) p_{k_s}(s)
- {\varepsilon_{s}}{\d_{k_s}} \mathrm{Li}_2(-\tilde y_{k_s}(s)^{\varepsilon_s}).
\end{align}
At this moment, this is a function of $(\bfu(s), \bfp(s))$.
Observe that $\scrL(s)$ depends on $\bfp(s)$ only through $p_{k_s}(s)$,
which is written by
\begin{align}
p_{k_s}(s)={\d_{k_s}}(\sqrt{2} \log \tilde y_{k_s}(s) - w_{k_s}(s)).
\end{align}
Also, we convert \eqref{eq:ud1} as 
\begin{align}
\label{eq:yu1}
  \tilde y_{k_s}(s)^{\varepsilon_s} = \exp(-\sqrt{2}\dot u_{k_s}(s))-1.
\end{align}
Thus, the Lagrangian \eqref{eq:lag1} is written in the form
\begin{gather}
\label{eq:lag2}
\scrL(s)=
{\d_{k_s}}
 \dot u_{k_s}(s)  (\sqrt{2} \log \tilde y_{k_s} (s)- w_{k_s}(s))
- {\varepsilon_{s}}{\d_{k_s}} \mathrm{Li}_2(-\tilde y_{k_s}(s)^{\varepsilon_s}),
\\
w_i(s)=\sum_{j=1}^n b_{ji}(s) u_j(s),
\quad
  \tilde y_{k_s}(s)^{\varepsilon_s} = \exp(-\sqrt{2}\dot u_{k_s}(s))-1,
\end{gather}
which is  a function of $(\bfu(s), \dot\bfu(s))$.
The Lagrangian $\scrL(s)$ is singular, because
$\partial \scrL(s)/\partial \dot u_i(s)=0$ for $i\neq k_s$.

Since both $\scrH(s)$ and $\scrL(s)$ are singular,
we  expect that the equations 
\eqref{eq:EL1} and \eqref{eq:EL2}
for $\scrL(s)$ are \emph{not} equivalent to
the equations of motion  for $\scrH(s)$,
which are \eqref{eq:us1}
 and \eqref{eq:ps1}.
Indeed, we have the following equations.

\begin{lem}
[{\cite[Prop.~4.3]{Gekhtman16}}]
\label{lem:EL1}
For the Lagrangian  $\scrL(s)$ in \eqref{eq:lag2},
the equations \eqref{eq:EL1} and \eqref{eq:EL2} are explicitly written as
follows:
For $i=k_s$,
\begin{align}
\label{eq:EL3}
p_{k_s}(s)= {\d_{k_s}}
  (\sqrt{2} \log \tilde y_{k_s} (s)- w_{k_s}(s)),
\quad
\dot p_{k_s}(s) = 0,
\end{align}
and, for $i\neq k_s$,
\begin{align}
\label{eq:EL4}
p_i(s)=0,
\quad
(1+\tilde y_{k_s}(s)^{\varepsilon_s})^{b_{k_s i }(s)}
=1.
\end{align}
\end{lem}
\begin{proof}
They follow from the following calculations:
\begin{align}
\begin{split}
 \frac{\partial\scrL(s)}{\partial \dot u_{k_s}(s)}
& =
 {\d_{k_s}}
  (\sqrt{2} \log \tilde y_{k_s} (s)- w_{k_s}(s))
  \\
  &\qquad
  -2
   {\varepsilon_{k_s}}{\d_{k_s}}
   \frac{\dot u_{k_s}(s)}{\tilde y_{k_s}(s)}
   -\sqrt{2}
     {\varepsilon_{k_s}}{\d_{k_s}}
     \frac{\log(1+\tilde y_{k_s}(s)^{\varepsilon_s})}{\tilde y_{k_s}(s)^{\varepsilon_s}}
     \\
 \overset {\eqref{eq:yu1}}
     & {=}
 {\d_{k_s}}
  (\sqrt{2} \log \tilde y_{k_s}(s) - w_{k_s}(s)),
  \end{split}
  \\
  \frac{\partial\scrL(s)}{\partial  u_{k_s}}
& = 0,
\end{align}
and, for $i\neq k_s$,
\begin{align}
   \frac{\partial\scrL(s)}{\partial \dot u_{i}(s)}
& =
0 ,
\\
\label{eq:Lu1}
   \frac{\partial\scrL(s)}{\partial  u_{i}(s)}
& =
-
{\d_{k_s}} b_{ik_s  }(s)
\dot u_{k_s}(s)
=
-
\frac{\d_i}{\sqrt{2}} b_{k_s i }(s)
\log
(1+\tilde y_{k_s}(s)^{\varepsilon_s})
.
\end{align}
\end{proof}

We observe that the  equations  \eqref{eq:EL3} coincide with
\eqref{eq:tyw1} and 
\eqref{eq:ps1} for $i=k_s$.
On the other hand, the  equations  \eqref{eq:EL4} impose
 unwanted constraints.
 Therefore, \emph{they are false} for our dynamical system.

Next, let us  restrict the system in the small phase space $\tilde M_0$
and set $p_{k_s}(s)=\d_{k_s} w_{k_s}(s)$ in \eqref{eq:EL3}.
Then, we have
\begin{align}
\label{eq:wy1}
w_{k_s}(s)=\frac{1}{\sqrt{2}}\log \tilde y_{k_s}(s).
\end{align}
By \eqref{eq:yu1} and \eqref{eq:wy1},
the function \eqref{eq:lag2} is rewritten
entirely as a function of $\tilde y_{k_s}(s)$ (therefore, as a function of $\dot u_{k_s}(s)$), as
\begin{align}
\begin{split}
\label{eq:lag3}
\scrL(s)&=
-
\frac{1}{2 }\varepsilon_s \d_{k_s}
 \log (1+ \tilde y_{k_s}(s)^{\varepsilon_s}) \log \tilde y_{k_s}(s)^{\varepsilon_s}
- {\varepsilon_{s}}{\d_{k_s}} \mathrm{Li}_2(-\tilde y_{k_s}(s)^{\varepsilon_s})
\\
&= {\varepsilon_s}{\d_{k_s}} \tilde L(\tilde y_{k_s}(s)^{\varepsilon_s}).
\end{split}
\end{align}
We see that the modified Rogers dilogarithm emerges as
the (false) Lagrangian.

Having  obtained  the exact function
appearing in the DIs  in Theorem \ref{thm:DI1},
our strategy for the rest of the chapter is as follows: We
continue to rely on the equations of motion for the Hamiltonian $\scrH(s)$
for our dynamical system.
On the other hand, we also use the function $\scrL(s)$ and its intrinsic property to 
prove Theorem \ref{thm:DI1}.

\section{Periodicity conditions}
\label{sec:periodicity2}

Let us return to the sequence  \eqref{eq:mseq5} for a $Y$-pattern.
Our next task is to clarify the relation between the periodicity of $y$-variables therein
and the one for the dynamical system with canonical coordinates $(\bfu,\bfp)$  in the previous section.

Before that, we resolve the issue concerning the nonsingularity of $B$.
For the initial matrix $B$ of \eqref{eq:mseq5} with a skew-symmetrizer $D$, we introduce an $2n\times 2n$ matrix 
\begin{align}
\label{eq:pext1}
\overline B
=
\begin{pmatrix}
B & -I\\
I & O\\
\end{pmatrix}
\end{align}
called the \emph{principal extension of $B$}\index{principal!extension (of exchange matrix)}.
The matrix $\overline B$ is skew-symmetrizable 
with a skew-symmetrizer $D\oplus D$.
Moreover, $\overline B$ is nonsingular whether $B$ is singular or not.
Starting from the initial extended $Y$-seed $\overline \Upsilon(0)=\overline \Upsilon
(\overline\bfy,\overline B)$ with $\overline \bfy=(y_1,\dots,y_{2n})$,
we have a lift of the sequence \eqref{eq:mseq5} in the extended $Y$-pattern of rank $2n$,
\begin{align}
\label{eq:mseq6}
&\overline\Upsilon(0) 
\
{\buildrel {k_0} \over \rightarrow}
\
\overline\Upsilon(1) 
\
{\buildrel {k_1} \over \rightarrow}
\
\cdots
\
{\buildrel {k_{P-1}} \over \rightarrow}
\
\overline\Upsilon(P),
\end{align}
where $k_s \in \{1,\, \dots, \, n\}$.

\begin{prop}
[{\cite[Thm.~4.3]{Nakanishi10c}}]
\label{prop:principal1}
Let $\nu\in S_n$ be a permutation of the first $n$ indices $1$, \dots, $n$ of 
extended $Y$-seeds in \eqref{eq:mseq6}.
Then, the sequence \eqref{eq:mseq6} is $\nu$-periodic 
if and only if  the sequence \eqref{eq:mseq5} is $\nu$-periodic.
\end{prop}
\begin{proof}
Let $\overline C(s)$ be the $C$-matrix for the
extended $Y$-seed $\overline \Upsilon (s)$.
Then, the following fact is known (\cite[Thm.~4.2]{Fujiwara18}):
\begin{align}
\label{eq:oC1}
\overline C(s)
=
\begin{pmatrix}
C(s) & O\\
O& I\\
\end{pmatrix}
.
\end{align}
(The proof requires the sign-coherence property in Theorem \ref{thm:sign1}.
Thus, it is considered as an advanced result.)
Thus, $\overline C(P)=\nu \overline C(0)$ holds if and only if
$ C(P)=\nu  C(0)$ holds.
Therefore, by Theorem \ref{1thm:synchro1},
 $\overline \Upsilon(P)=\nu  \overline\Upsilon(0)$ holds
 if and only if  $ \Upsilon(P)=\nu  \Upsilon(0)$ holds.
\end{proof}

Observe that the DI for  the sequence \eqref{eq:mseq5}
 coincides with the one for the extended sequence \eqref{eq:mseq6}
 due to \eqref{eq:oC1}.
Thus,
thanks to Proposition \ref{prop:principal1},
to prove the DI for $B$ in Theorem \ref{thm:DI1}, we may prove the one for $\overline B$.
Therefore, 
in proving Theorem \ref{thm:DI1},
we may assume that $B$ is nonsingular 
 without loss of generality.

Now,
let us formulate the periodicity condition for
the sequence \eqref{eq:mseq5} in the current setting.
In a parallel way with \eqref{eq:mu1},
for any permutation $\nu\in S_n$, 
we define a semifield automorphism
(denoted by the same symbol)
\begin{align}
\begin{matrix}
\label{eq:sigma1}
\nu\colon &\bbP(P)& \rightarrow & \bbP=\bbP(0)
\\
&
y_i(P)
&
\mapsto
&
\displaystyle
y_{\nu^{-1}(i)}.
\end{matrix}
\end{align}
Then, the $\nu$-periodicity for the sequence
 \eqref{eq:mseq5} is rephrased as the equality
 \begin{align}
 \label{eq:mu6}
\mu(P-1;0) =\nu
 \quad \text{for}\ 
 \bbP(P) \rightarrow  \bbP.
 \end{align}
 The $\nu$-periodicity of  \eqref{eq:mseq5}
 also
 implies the same periodicity for tropical $y$-variables.
 Thus, by \eqref{eq:taus1}, we also have the following equality:
 \begin{align}
 \label{eq:taui6}
\tau(P-1;0) =\nu
\quad
\text{for}\ 
 \bbP(P) \rightarrow \bbP.
 \end{align}
 Therefore, by the decomposition \eqref{eq:decom1},
 we have
 \begin{align}
 \label{eq:qPid1}
 \frakq(P-1;0)=\rmid
 \quad
 \text{for}\ 
  \bbP \rightarrow \bbP.
 \end{align}
 Note that $\nu$ disappeared
 in this equality.
 This further induces the equality
 \begin{align}
  \label{eq:rho6}
 \frakq(P-1;0)=\rmid
 \quad
\text{for}\ 
  \calC(\bfy) \rightarrow \calC(\bfy).
 \end{align}
In particular, any trajectory $\alpha=(\alpha_t)_{t\in [0,P]}$ in the phase space $ M$ in the full time span $[0,P]$
 is periodic. (See Figure \ref{fig:sche1}.). Namely, we have
 \begin{align}
  \alpha_0=\alpha_P.
  \end{align}

Let us formulate a parallel result for the canonical coordinates $(\bfu,\bfp)$.
For any permutation $\nu\in S_n$, define
\begin{align}
\begin{matrix}
\label{eq:sigma2}
\nu\colon &\bfC(\bfu(P),\bfp(P))& \rightarrow & \bfC(\bfu,\bfp)=\bfC(\bfu(0),\bfp(0))
\\
&
u_i(P)
&
\mapsto
&
\displaystyle
u_{\nu^{-1}(i)},
\\
&
p_i(P)
&
\mapsto
&
\displaystyle
p_{\nu^{-1}(i)}.
\end{matrix}
\end{align}
Remarkably, the same periodicity still holds in the phase space $\tilde M$,
which is larger than $M$, if $B$ is nonsingular.
\begin{prop}
[{\cite[Prop.~5.11]{Gekhtman16}}]
\label{prop:periodup1}
Suppose that
  the sequence
 \eqref{eq:mseq5}
is $\nu$-periodic
and
 the initial exchange matrix $B$ is nonsingular.
Then, the following periodicities hold:
 \begin{align}
 \label{eq:periodup1}
 \mu(P-1;0)& =\nu
\quad\text{for}\ 
 \calC(\bfu(P),\bfp(P)) \rightarrow \calC(\bfu,\bfp),
 \\
   \label{eq:tau7}
 \tau(P-1;0)& =\nu
\quad \text{for}\ 
 \calC(\bfu(P),\bfp(P)) \rightarrow \calC(\bfu,\bfp),
 \\
  \label{eq:q7}
 \frakq(P-1;0)&=\rmid 
\quad \text{for}\ 
 \calC(\bfu,\bfp) \rightarrow \calC(\bfu,\bfp).
 \end{align}
In particular, any trajectory $\alpha=(\alpha_t)_{t\in [0,P]}$ in the phase space $\tilde M$ in the full time span $[0,P]$
 is periodic; namely,
 \begin{align}
  \alpha_0=\alpha_P.
  \end{align}
\end{prop}
\begin{proof}
The variables $y_i(s)$ in $\calC(\bfy(s))$ and $\tilde y_i(s)$  in $\calC(\bfu(s),\bfp(s))$
obey the same transformations under $\rho(s)$ and $\tau(s)$.
Thus, they also obey the same transformations under $\frakq(s)$ and $\frakq(s;0)$.
Therefore, the periodicities \eqref{eq:mu6}--\eqref{eq:rho6} imply the periodicities
for $\tilde y$-variables
\begin{align}
\label{eq:mu8}
\mu(P-1;0): &\  \tilde y_i(P) \mapsto \tilde y_{\nu^{-1}(i)},
\\
\label{eq:tau8}
\tau(P-1;0):&\  \tilde y_i(P) \mapsto \tilde y_{\nu^{-1}(i)},
\\
\label{eq:q8}
\frakq(P-1;0): &\ \tilde y_i \mapsto \tilde y_i.
\end{align}
Let us show that they imply the same periodicities for $\bfu(s)$ and $\bfp(s)$ as well.
We set $z_i(s)=\log  \tilde y_i(s)$.
Then, 
by \eqref{eq:rhodp1}, \eqref{eq:rhoy1}, \eqref{eq:dps4}, \eqref{eq:ys4},
we see that
$w_i(s)$, 
$\sqrt{2} \d_i^{-1} p_i(s)$,  and $z_i(s)$ obey the same affine transformations under $\rho(s)$ and $\tau(s)$.
It follows that $\bfw(s)$ and $\bfp(s)$ obey the same periodicities 
in \eqref{eq:mu8}--\eqref{eq:q8}
for  $\tilde \bfy(s)$.
Moreover,
since $B$ is nonsingular, one can invert 
the relation between $\bfu(s)$ and $\bfw(s)$ in \eqref{eq:tyw1}.
Thus, $\bfu(s)$ also obeys the same periodicity.
\end{proof}

\section{Classical mechanical proof of DIs (I): In local picture}
\label{sec:alternative1}
Let us give an alternative proof of
Theorem \ref{thm:DI1}
based on the classical mechanical formulation presented in this chapter.
Actually, we will give two alternative proofs, and this is the first one.

We prove the DI  \eqref{eq:DI3} 
with variables $\tilde \bfy$.
However, variables $\tilde \bfy$ are algebraically independent on $\tilde M_0$
only when the matrix $B$ is nonsingular.
So, to restore the DI in \eqref{eq:DI3} 
from the DI  with variables $\tilde \bfy$,
we need to assume that $B$ is nonsingular.
On the other hand, we have already argued in the last section
that this assumption does not lose generality.
So, from now on, \emph{we assume that the initial matrix $B$ is nonsingular}.
Then, under this assumption,
the map \eqref{eq:eta3} is injective
and we may identify $\tilde \bfy(s)$ and $ \bfy(s)$.

Suppose that
the mutation sequence  \eqref{eq:mseq5}
is $\nu$-periodic.
Let us consider a trajectory $\alpha=(\alpha_t)_{t\in [0,P]}$ in the small phase space $\tilde M_0$.
By Proposition \ref{prop:periodup1},
 $\alpha$ is periodic. Namely, 
 \begin{align}
 \label{eq:alphap1}
 \alpha_0=\alpha_P
 \end{align}
 holds.
 The main idea of the proof of Theorem \ref{thm:DI1} is
 to consider the following  integral of the Lagrangian $\scrL$ given by \eqref{eq:lag3}
  along  the trajectory $\alpha$:
 \begin{align}
 \label{eq:action1}
 S[\alpha]:=
 \int_{0}^P \scrL(\alpha_t) \, dt.
 \end{align}
We call \eqref{eq:action1} the \emph{action integral}\index{action integral} along $\alpha$ following the convention in classical mechanics.
Recall that $ \tilde y_{k_s}$ is constant in the time span $[s,s+1]$.
Thus,
 by \eqref{eq:lag3},
the action integral is reduced to a sum of the modified Rogers dilogarithm
\begin{align}
\label{eq:DI5}
 S[\alpha]=
\sum_{s=0}^{P-1}
{\varepsilon_s}{\d_{k_s}} \tilde L(\tilde y_{k_s}(s)^{\varepsilon_s}).
\end{align}
Observe that
this is exactly the LHS of the DI  \eqref{eq:DI3}.

Once the constancy of \eqref{eq:DI5} is shown,
the constant term is 0 by considering the tropical limit
$\bfy \rightarrow \bfzero$ as we did when proving Theorem \ref{thm:DI1}.
Thus, we are left to prove the constancy of  \eqref{eq:DI5}.
We formulate it as the invariance of the action integral as follows.
\begin{thm}
[{\cite[Thm.~6.3]{Gekhtman16}}]
\label{thm:action1}
Under the periodicity \eqref{eq:alphap1},
the action integral  $S[\alpha]$ in \eqref{eq:action1} is independent of
a trajectory $\alpha$ in $\tilde M_0$.
\end{thm}

The rest of the section is devoted to proving  Theorem \ref{thm:action1}.
We prove it by the standard \emph{variational calculus} in classical mechanics.
Consider an infinitesimal variation $\alpha + \delta \alpha$
of a trajectory $\alpha$,
where we assume that $\alpha + \delta \alpha$ is also a trajectory
in the small phase space $\tilde M_0$.
Let us concentrate on the time span $[s,s+1]$.
Let $(\overline \bfu(t), \overline \bfp(t))$ ($t\in [s,s+1]$)
be the point corresponding to $\alpha_t$   in the $(\bfu(s),\bfp(s))$-coordinates.
Due to the constraint \eqref{eq:dpw1},
 we have the variations
\begin{align}
\label{eq:uinf1}
\overline u_i (t)&\rightarrow \overline  u_i (t) + \delta \overline u_i(t),
\\
\d^{-1}_i \overline p_i (t)&\rightarrow \d^{-1}_i  \overline p_i (t) +  \sum_{j=1}^n b_{ji}(s) \delta 
\overline u_j(t).
\end{align}
Let
 \begin{align}
 \label{eq:action2}
 S[\alpha;s]:=
 \int_{s}^{s+1} \scrL(s) (\alpha_t) \, dt.
 \end{align}
 
The variation of $ S[\alpha;s]$ is evaluated by the ``boundary terms'' as follows.
 \begin{lem}
 [{\cite[Lemma.~6.4]{Gekhtman16}}]
 \label{lem:dS1}
 (a).
 In  the $(\bfu(s),\bfp(s))$-coordinates, we have
 \begin{align}
 \label{eq:dS1}
 \delta  S[\alpha;s]
 = R[\alpha;s+1] - R[\alpha;s],
 \end{align}
 where
 \begin{align}
 \label{eq:Ba1}
  R[\alpha;t]:=\sum_{i=1}^n
 \overline p_i (t) \delta \overline u_i(t).
 \end{align}
 (b). The above quantity $ R[\alpha;t]$  is  independent of the choice of
  coordinates $(\bfu(s),\bfp(s))$ for $\tilde M_0$.
 \end{lem}
 \begin{proof}
 (a).
We regard $ \scrL(\alpha_t) $ as
a function of $(\overline \bfu(t),  \dot{\overline\bfu}(t))$
as explained in Section \ref{sec:modified2}.
  We first recall that, by \eqref{eq:us1},
  we have
 \begin{align}
 \label{eq:sdu1}
\dot{\overline u}_i (t)&=0
\quad
(i\neq k_s).
 \end{align}
 This  implies that,
 \begin{align}
 \label{eq:du1}
\delta \dot{\overline u}_i  (t)&=0
\quad
(i\neq k_s).
 \end{align}
Thus, we decompose the variation $ \delta  S[\alpha;s]$ as
\begin{gather}
 \delta  S[\alpha;s]
 =\delta S_1 + \delta S_2,
 \end{gather}
 where
 \begin{align}
 \label{eq:dS11}
 \delta S_1
 :=&\ \int_s^{s+1} 
 \left\{
 \frac{\partial \scrL(\alpha_t) }{\partial \overline u_{k_s}(t) }
 \delta \overline u_{k_s}(t)
 +
  \frac{\partial \scrL(\alpha_t) }{\partial \dot{\overline u}_{k_s}(t) }
 \delta \dot{\overline u}_{k_s}(t)
 \right\}dt,
 \\
 \label{eq:dS21}
 \delta S_2
 :=&\ 
 \sum_{ \scriptstyle i=1 \atop  \scriptstyle
 i\neq k_s}^n
 \int_s^{s+1} 
 \left\{
 \frac{\partial \scrL(\alpha_t) }{\partial \overline u_{i}(t) }
 \delta \overline u_{i}(t)
 \right\} dt.
\end{align}
As we see in Lemma \ref{lem:EL1}, 
the equations \eqref{eq:EL1} and  \eqref{eq:EL2}
for $i=k_s$ are reliable.
By using them, we have
\begin{align}
\label{eq:dS3}
\begin{split}
 \delta S_1
& =\int_s^{s+1} 
 \left\{
\dot{\overline p}_{k_s}(t)
 \delta \overline u_{k_s}(t)
 +
\overline p_{k_s}(t)
 \delta \dot{\overline u}_{k_s}(t)
 \right\} dt
 \\
 & =\int_s^{s+1} 
 \frac{d}{dt}
 \left\{
{\overline p}_{k_s}(t)
 \delta \overline u_{k_s}(t)
 \right\} dt
 \\
 & =
\overline p_{k_s}(s+1) \delta \overline u_{k_s}(s+1)-
\overline p_{k_s}(s) \delta \overline u_{k_s}(s).
\end{split}
 \end{align}
To evaluate $ \delta S_2$,
we have to be careful because
the equations \eqref{eq:EL1} and  \eqref{eq:EL2}
for $i\neq k_s$
 are false.
By \eqref{eq:Lu1}, we have
 \begin{align}
 \delta S_2
 =
 \sum_{ \scriptstyle i=1 \atop  \scriptstyle
 i\neq k_s}^n
 \int_s^{s+1} 
 \left\{
 -
 \frac{1}{\sqrt{2}} \d_{i} b_{k_s i }(s)
\log
(1+\overline y_{k_s}(s)^{\varepsilon_s})
 \delta \overline u_{i}(t)
 \right\} dt,
\end{align}
where
\begin{align}
\overline{y}_i(t) 
= \exp\left(\sqrt{2} \delta_i^{-1}\overline p_i(t)\right)
= \exp\left(\sqrt{2} \overline w_i(t)\right),
\quad
\overline w_i(t):=\sum_{j=1}^n b_{ji}(s)  \overline u_j(t).
\end{align}
in $\tilde M_0$.
By  \eqref{eq:du1}, the integrand is constant on $[s,s+1]$. So, we have
\begin{align}
 \delta S_2
 =
 \sum_{ \scriptstyle i=1 \atop  \scriptstyle
 i\neq k_s}^n
 \left(
 -
 \frac{1}{\sqrt{2}} \d_{i} b_{k_s i }(s)
\log
(1+\overline y_{k_s}(s)^{\varepsilon_s})
 \delta \overline u_{i}(s)
 \right).
\end{align}
Moreover, 
 we have
\begin{align}
\label{eq:dS2}
\begin{split}
\delta S_2
\overset {\eqref{eq:rhop1}}
& {=}
  \sum_{ \scriptstyle i=1 \atop  \scriptstyle
 i\neq k_s}^n
 \left\{
\overline p_{i}(s+1) -
\overline p_{i}(s) \right\}\delta \overline u_{i}(s)
\\
\overset {\eqref{eq:du1}}
 & =
  \sum_{ \scriptstyle i=1 \atop  \scriptstyle
 i\neq k_s}^n
 \left\{
\overline p_{i}(s+1) \delta \overline u_{i}(s+1)-
\overline p_{i}(s) \delta \overline u_{i}(s)
\right\}.
\end{split}
\end{align}
By summing up \eqref{eq:dS3} and \eqref{eq:dS2},
we obtain \eqref{eq:dS1}.

(b).
Since $\tau^*(s)$ acts on $\overline u_i(t)$ as a linear transformation,
it also acts on
 the variation $\delta \overline u_i(t)$ 
in the same way.
Then,
the claim is proved in the same way as
Proposition \ref{prop:tauP1} (b).
\end{proof}

Let us finish the proof of
 Theorem \ref{thm:action1}.
By Lemma \ref{lem:dS1} (a) and (b),
we have
\begin{align}
\label{eq:dSa1}
\begin{split}
\delta S[\alpha]&=\sum_{s=0}^{P-1} 
(R[\alpha;s+1] - R[\alpha;s])
\\
&=
R[\alpha;P]- R[\alpha;0].
\end{split}
\end{align}
By Lemma \ref{lem:dS1} (b) again,
we may evaluate both $R[\alpha;P]$ and  $R[\alpha;0]$
in the initial coordinates $(\bfu,\bfp)$.
Finally and most importantly, by the periodicity \eqref{eq:alphap1} of $\alpha$,
we have
 \begin{align}
 R[\alpha;P]
 =
  R[\alpha;0].
 \end{align}
 Therefore, $\delta S[\alpha]=0$.
 This completes the proof of 
  Theorem \ref{thm:action1} and also
  Theorem \ref{thm:DI1}.

The  method of evaluating  $\delta S[\alpha;s]$ in the proof of Lemma \ref{lem:dS1} is
essentially the same one for  celebrated \emph{Noether's theorem}\index{Noether's theorem} (e.g., \cite[Thm.~1.3]{Takhtajan08}) in classical mechanics.
Thus, we may regard 
  Theorem \ref{thm:action1}
  as a variant of Noether's theorem.
  This gives us a conceptual understanding of the constancy of the DI in Theorem \ref{thm:DI1}.

\section{Classical mechanical proof of DIs (II): In global picture}
\label{sec:initial1}

The proof of   Theorem \ref{thm:DI1} in the previous section
uses the canonical coordinates $(\bfu(s),\bfp(s))$ for each time span $[s,s+1]$.
Here we present  the second  proof of Theorem \ref{thm:DI1}
based on the same classical mechanical method, but
 using only the initial canonical coordinates.
 The proof is less intrinsic than the first one.
 On the other hand, it is applicable in a more general context.
 The main reason to present this proof is that
 we use it later
 in Chapter \ref{ch:DICSD1}.
 
 Again, we assume that $B$ is nonsingular,
 and we identify $\tilde \bfy$ and $\bfy$
 by the map \eqref{eq:eta2}.
First, we consider a little more general Hamiltonian  (see \eqref{eq:Hc1})
\begin{align}
\scrH_{\bfc,a}= a \mathrm{Li}_2(- \tilde y^{\bfc})
\end{align}
for the time span $[0,1]$
and a trajectory
 $\alpha=(\alpha_t)_{t\in [0,1]}$  in  $\tilde M_0$ by $\scrH_{\bfc,a}$,
 where $\alpha_0=(\overline \bfu, \overline \bfp)$ and 
 $\alpha_{1}=(\overline \bfu', \overline \bfp')$.
 Let
\begin{align}
\label{eq:ypw1}
\overline{y}_i := \exp\left({\sqrt{2}}\d^{-1}_i \overline p_i\right)=\exp\left(\sqrt{2}\overline w_i\right),
\quad
\overline w_i:=\sum_{j=1}^n b_{ji}  \overline u_j.
\end{align}
Then, we have
\begin{align}
\label{eq:ou1}
\overline u_i' &=
\overline u_i - \frac{a}{\sqrt{2}} c_i \d^{-1}_i \log (1+\overline y^{\bfc}),
\\
\label{eq:op1}
\overline p_i' &=
\overline p_i + \frac{a}{\sqrt{2}} 
\biggl(\sum_{j=1}^n c_j b_{ij}\biggr) \log (1+\overline y^{\bfc}),
\\
\label{eq:oy1}
\log \overline y_i' &=
\log \overline y_i - a
\biggl(\sum_{j=1}^n c_j \omega_{ji}\biggr) \log (1+\overline y^{\bfc}).
\end{align}
Consider the corresponding Lagrangian
\begin{align}
\scrL_{\bfc,a}= a \tilde {L} ( \tilde y^{\bfc}),
\end{align}
and   the action integral along
the above trajectory
 \begin{align}
 \label{eq:action4}
 S_{\bfc,a}[\alpha]:=
 \int_0^1 \scrL_{\bfc,a} (\alpha_t) \, dt.
 \end{align}
The integrand is constant because
$\{\scrH_{\bfc,a},  \tilde y^{\bfc}\}=0$.
Thus, we have
 \begin{align}
 \label{eq:action5}
 S_{\bfc,a}[\alpha]= \scrL_{\bfc,a} (\alpha_0).
 \end{align}
Then, we consider an infinitesimal variation
 $\alpha + \delta \alpha$
of a trajectory $\alpha$ in $\tilde M_0$.
 
The following result formally coincides with  Lemma \ref{lem:dS1}.
However, it is more general than 
 Lemma \ref{lem:dS1} because
 it is applicable to any  integer vector $\bfc$ other than the $c$-vectors $\bfc^+_{k_s}(s)$.

\begin{prop}
[{\cite[Lemma 3.5]{Nakanishi21d}}]
\label{prop:dL1}
The infinitesimal variation of $ S_{\bfc,a}[\alpha]$ in  \eqref{eq:action5} is given by
\begin{align}
\label{eq:LPu1}
\delta  S_{\bfc,a}[\alpha]=
\sum_{i=1}^n \overline p'_i \delta \overline u'_i
-
\sum_{i=1}^n \overline p_i \delta \overline u_i.
\end{align}
\end{prop}

\begin{proof}
Here, we compute and compare both sides explicitly.
Thanks to \eqref{eq:ypw1} and \eqref{eq:L4},
the LHS   of \eqref{eq:LPu1} is given by
\begin{align}
\label{eq:dld1}
\delta  \scrL_{\bfc,a} (\alpha_0)=
\frac{a}{\sqrt{2}} 
\biggl(\, \sum_{i,j=1}^n c_j b_{ij} \delta \overline u_i\biggr)
\biggl(
\log(1+\overline y^{\bfc})
-
\frac{  \overline y^{\bfc} \log \overline y^{\bfc}}
{1+\overline y^{\bfc}}
\biggr).
\end{align}
Let us calculate the RHS  of \eqref{eq:LPu1}. 
By \eqref{eq:ou1}, we have
\begin{align}
\label{eq:ou2}
\delta \overline u_i' &=
\delta \overline u_i - \frac{a}{\sqrt{2}} c_i \d^{-1}_i 
\biggl(\, \sum_{j,k=1}^n c_j b_{kj} \delta \overline u_k\biggr)
 \frac{ \overline y^{\bfc}}{1+\overline y^{\bfc}}.
\end{align}
Let us write \eqref{eq:op1} and \eqref{eq:ou2} as
$
\overline p_i' =
\overline p_i + A_i
$
and
$
\delta \overline u_i' =
\delta \overline u_i - B_i
$.
Then, the RHS of \eqref{eq:LPu1} is given by
\begin{align}
\label{eq:dud1}
\sum_{i=1}^n  A_i  \delta \overline u_i
-
\sum_{i=1}^n    \overline p_i B_i 
-
\sum_{i=1}^n  A_i  B_i.
\end{align}
They are calculated as
\begin{align}
\sum_{i=1}^n  A_i  \delta \overline u_i
&=
\frac{a}{\sqrt{2}}
\biggl(\, \sum_{i, j=1}^n c_j b_{ij}  \delta \overline u_i \biggr) \log (1+\overline y^{\bfc}),
\\
\sum_{i=1}^n    \overline p_i B_i 
&=
\frac{a}{\sqrt{2}}
\biggl(\, \sum_{i, j=1}^n c_j b_{ij}  \delta \overline u_i \biggr) 
\frac{  \overline y^{\bfc} \log \overline y^{\bfc}}
{1+\overline y^{\bfc}},
\\
\sum_{i=1}^n  A_i  B_i
&=
\frac{a^2}{{2}}
\{\bfc,\bfc\}_{\Omega}
\biggl(\, \sum_{i, j=1}^n c_j b_{ij}  \delta \overline u_i \biggr) 
\frac{\overline y^{\bfc} \log (1+ \overline y^{\bfc})}{1+\overline y^{\bfc}}
=0,
\end{align}
where the last equality is due to the skew-symmetry of $\{\cdot,\cdot\}_{\Omega}$
in \eqref{eq:Omega2}.
Then, we see that \eqref{eq:dud1} coincides with \eqref{eq:dld1}.
\end{proof}

The rest of the proof of  Theorem \ref{thm:action1} is the same as before.
Assume that that $B$ is nonsingular, and 
suppose that $\bar \bfy_{t=P}=\bar \bfy_{t=0}$.
Then, by \eqref{eq:ou1}--\eqref{eq:oy1},
we have
$\bar \bfu_{t=P}=\bar \bfu_{t=0}$ and $\bar \bfp_{t=P}=\bar \bfp_{t=0}$.
Therefore, for the action integral
 \begin{align}
 \label{eq:action3}
 S[a]:=
 \int_{0}^P \scrL(\alpha_t) \, dt
=
\sum_{s=0}^{P-1}
{\varepsilon_s}{\d_{k_s}} \tilde L(y^{ \bfc^+_{k_s}(s)}),
\end{align}
we have
\begin{align}
\label{eq:dSa2}
\begin{split}
\delta S[\alpha]&=\sum_{s=0}^{P-1} 
(R[\alpha;s]_+ - R[\alpha;s]_-)
=
R[\alpha;P]- R[\alpha;0]=0,
\end{split}
\end{align}
where we used Proposition \ref{prop:dL1} for the first equality.

\notes
The content of this section is mostly based on \cite{Gekhtman16}
except for Section \ref{sec:initial1}, which appeared in \cite{Nakanishi21d}.
The idea  of representing $y$-variables with canonical  coordinates (variables) originates 
in the work of \cite{Fock07,Fock07b,Kashaev11},
where the quantum version was considered.
The appearance of the modified Rogers dilogarithm as
the Lagrangian was expected by the semi-classical calculation
of quantum dilogarithm identities by \cite{Kashaev11}.

\part{Scattering diagram  method}

\chapter{Algebraic formulation of dilogarithm and mutations}
\label{ch:algebraic1}

Motivated by the cluster scattering diagram formalism in the next chapter,
we define a certain infinite group for each skew-symmetric matrix.
We focus on certain distinguished elements in the group called the \emph{dilogarithm elements}.
They are regarded as an algebraic version of the dilogarithm and mutations,
and
they satisfy the \emph{pentagon relation} corresponding to the pentagon periodicity
of the $Y$-pattern of type $A_2$.
We present several interesting relations generated by the pentagon relation, including infinite ones, 

\section{Lie algebra associated with skew-symmetric form}
\label{sec:Lie1}

Let $\Omega=(\omega_{ij})_{i,j=1}^n$
be any skew-symmetric rational matrix.
Let us fix a lattice  $N= \bbZ^n$ of rank $n$.
Let $\{\cdot, \cdot\}_{\Omega}$ be the skew-symmetric rational bilinear form on $N$ defined by
\eqref{eq:Omega2}.
Let 
\begin{align}
\label{eq:N+1}
N^+
=
 \bbZ^n_{\geq 0} \setminus \{\bfzero\}
\end{align}
be the set of the positive vectors of $N$.
The \emph{degree}\index{degree (of vector)} $\deg(\bfn)$ of $\bfn=(n_i)\in N^+$ is defined by
\begin{align}
\deg(\bfn):=\sum_{i=1}^n n_i.
\end{align}
For any integer $\ell>0$,
we write
\begin{align}
(N^+)^{\leq \ell}:=&\{\bfn \in N^+ \mid \deg(\bfn)\leq \ell\},
\\
(N^+)^{> \ell}:=&\{\bfn \in N^+ \mid \deg(\bfn)> \ell\}.
\end{align}
We say that $\bfn\in N^+$ is \emph{primitive}\index{primitive (vector)}
if there is no pair $j\in \bbZ_{>1}$
and  $\bfn'\in N^{+}$
such that $\bfn=j\bfn'$.
The set of all primitive vectors are denoted by
$N^{+}_{\rmpr}$.

Following Gross-Siebert \cite{Gross07},
Kontsevich-Soibelman \cite{Kontsevich08, Kontsevich13},
and
Gross-Hacking-Keel-Kontsevich (GHKK) \cite{Gross14},
we introduce an $N^+$-graded Lie algebra  over $\bbQ$,
\begin{align}
\frakg=\frakg_{\Omega}
=\bigoplus_{\bfn\in N^+}
\bbQ X_{\bfn}
\end{align}
with generators
$X_{\bfn}$ ($\bfn\in N^+$) and relations
\begin{align}
\label{eq:Xcom1}
[X_{\bfn}, X_{\bfn'}]=\{\bfn,\bfn'\}_{\Omega} X_{\bfn+\bfn'}.
\end{align}
Clearly, the above Lie bracket is skew-symmetric.
The Jacobi identity is also clear  from   the following cyclic expression:
\begin{align}
\label{eq:XJac1}
\begin{split}
[X_{\bfn_1}, [X_{\bfn_2},X_{\bfn_3}]]
&=
\{\bfn_2,\bfn_3\}_{\Omega} [X_{\bfn_1},X_{\bfn_2+\bfn_3}]
\\
&=
(\{\bfn_1,\bfn_2\}_{\Omega} \{\bfn_2,\bfn_3\}_{\Omega} -
\{\bfn_2,\bfn_3\}_{\Omega} \{\bfn_3,\bfn_1\}_{\Omega} )X_{\bfn_1+\bfn_2+\bfn_3}.
\end{split}
\end{align}
We consider the completion of $\frakg$ by $\deg$.
Namely, for any  integer $\ell>0$,
consider a Lie algebra ideal and its quotient
\begin{align}
\label{eq:gl1}
\frakg^{>\ell}& = \bigoplus_{\bfn\in (N^+)^{>\ell}} \bbQ X_{\bfn},
\\
\label{eq:gl2}
\frakg^{\leq \ell} &= \frakg / \frakg^{>\ell}.
\end{align}
Note that $\frakg^{\leq \ell}$ is nilpotent.
Then, we take the inverse limit
\begin{align}
\label{eq:hg1}
\widehat \frakg =\widehat {\frakg}_{\Omega}:=
\lim_{\longleftarrow} \frakg^{\leq \ell}.
\end{align}
Thus, an element of $\widehat \frakg $ is given by a formal infinite sum
\begin{align}
\sum_{\bfn
\in N^+}
c_{\bfn} X_{\bfn}
\quad
(c_{\bfn}\in \bbQ).
\end{align}

For each $\bfn\in N^+_{\rmpr}$, 
let $\widehat \frakg_{\bfn}^{\parallel}$
denote the commutative Lie subalgebra of $\widehat \frakg$ spanned by
elements
$\sum_{j=1}^{\infty} c_{j\bfn} X_{j\bfn}$.

\section{Exponential group}
\label{sec:exponential1}
For the Lie algebra $\widehat{g}$ in \eqref{eq:hg1},
we define the associated  exponential group
\begin{align}
\label{eq:exG1}
G=G_{\Omega}:=\exp(\widehat \frakg),
\end{align}
where $\exp$ is a formal exponential map.
Namely, as a set, $G$ is identified with $\widehat \frakg$
with $\exp(X)\leftrightarrow X$,
and its product is defined by the \emph{Baker-Campbell-Hausdorff (BCH) formula}\index{Baker-Campbell-Hausdorff (BCH) formula}
(e.g., \cite[Ch.2, \S6]{Bourbaki89}, \cite{Bonfiglioli12}),
\begin{align}
\begin{split}
\label{eq:BCH1}
&\ \exp(X) \exp(Y)
\\
=&\ \exp \biggl(X+Y+\frac{1}{2}[X,Y]
+\frac{1}{12}[X,[X,Y]]
-\frac{1}{12}[Y,[X,Y]]
+\cdots
\biggr)
,
\end{split}
\end{align}
where we do not need the explicit form of the higher-order
brackets.
The infinite sum in the RHS converges in $\widehat\frakg$ due to the nilpotency of $\frakg^{\leq \ell}$ for any $\ell$.
Originally, the BCH formula is a product formula
for the exponentials of noncommutative variables $X$ and $Y$,
where $[X,Y]:=XY - YX$.
So, the associativity of the product is guaranteed. See \cite[Thm.~5.42]{Bonfiglioli12} for a detailed treatment.

Alternatively, 
for the Lie algebra ${g}^{\leq \ell}$ in \eqref{eq:gl2},
we define
 the associated  exponential group
\begin{align}
G^{\leq \ell}=\exp( \frakg^{\leq\ell})
\end{align}
in the same way.
Then, $G$ is identified with  the inverse limit
\begin{align}
G=\lim_{\longleftarrow} G^{\leq \ell}.
\end{align}
For the canonical projection
\begin{align}
\label{eq:cp1}
\pi_{\ell}: G \rightarrow G^{\leq \ell},
\end{align}
we set $G^{>\ell}:=\mathrm{Ker}\, \pi_{\ell}$.

For each $\bfn\in N^+_{\rmpr}$, 
let $ G_{\bfn}^{\parallel}$ be the  exponential group of
$\widehat \frakg_{\bfn}^{\parallel}$.
It is an abelian subgroup of $G$.
We call it the \emph{parallel subgroup}\index{parallel subgroup} of $\bfn$.

An \emph{infinite} product in $G$ 
makes sense if it is reduced to a finite product by applying $\pi_{\ell}$ for any $\ell$.
For any element $g=\exp(X)$, its rational power $g^c$ ($c\in \bbQ$) is ambiguously defined 
by $g^c:=\exp(cX)$.

\section{$y$-representation}
\label{sec:y-representation1}

We introduce an important representation of the group $G=G_{\Omega}$ in the previous section.
Let $\bfy=(y_1,\dots,y_n)$ be an $n$-tuple of variables
and let $\bbQ[[\bfy]]$ be the formal power series ring of $\bfy$
over $\bbQ$.
Namely, any element of $\bbQ[[\bfy]]$
is written as a formal infinite sum
\begin{align}
\sum_{\bfn\in \bbZ^n_{\geq 0}} c_{\bfn} y^{\bfn}
\quad
(c_{\bfn} \in \bbQ),
\quad
y^{\bfe_i}=y_i.
\end{align}
The algebra $\bbQ[[\bfy]]$ is naturally equipped with the Poisson bracket
defined by
\begin{align}
\label{eq:yy1}
\{ y^{\bfn}, y^{\bfn'}\}=\{\bfn,\bfn'\}_{\Omega} y^{\bfn+\bfn'},
\end{align}
where $\{\bfn,\bfn'\}_{\Omega}$ is the same skew-symmetric rational bilinear form in \eqref{eq:Xcom1}.
The Leibniz rule  is confirmed by
\begin{align}
\label{eq:Leibnitz2}
\begin{split}
\{ y^{\bfn}, y^{\bfn_1}\} y^{\bfn_2}
+y^{\bfn_1}\{ y^{\bfn}, y^{\bfn_2}\} 
&=\{\bfn,\bfn_1\}_{\Omega} y^{\bfn+\bfn_1}y^{\bfn_2}
+y^{\bfn_1}\{\bfn,\bfn_2\}_{\Omega} y^{\bfn+\bfn_2}
\\
&=\{\bfn,\bfn_1+\bfn_2\}_{\Omega} y^{\bfn+\bfn_1+\bfn_2},
\end{split}
\end{align}
while the Jacobi identity is verified by the following cyclic expression:
\begin{align}
\label{eq:yJac1}
\{ y^{\bfn_1}, \{ y^{\bfn_2}, y^{\bfn_3}
\}\}
&=
(\{\bfn_1,\bfn_2\}_{\Omega}\{\bfn_2,\bfn_3\}_{\Omega}
-\{\bfn_2,\bfn_3\}_{\Omega}\{\bfn_3,\bfn_1\}_{\Omega})y^{\bfn_1+\bfn_2+\bfn_3}.
\end{align}
Compare \eqref{eq:yJac1} with \eqref{eq:XJac1}.
As a special case of \eqref{eq:yy1}, we have
\begin{align}
\label{eq:yy2}
\{ y_i, y_j\}=\omega_{ij} y_iy_j.
\end{align}
Thus, the Poisson bracket on  $\bbQ[[\bfy]]$ here is
essentially the same as the quadratic Poisson bracket \eqref{eq:Poi2} on $\calC(\bfy)$
  under the identification of the matrix $\Omega$
in both situations.

We define an endomorphism $\tilde X_{\bfn}$ ($\bfn\in N^+$) of $\bbQ[[\bfy]]$ by
\begin{align}
\label{eq:tildeX1}
\tilde X_{\bfn}(y^{\bfn'}):=\{y^\bfn, y^{\bfn'}\}= \{\bfn,\bfn'\}_{\Omega} y^{\bfn+\bfn'}
.
\end{align}
This is a \emph{derivation}\index{derivation}; namely,  we have
\begin{align}
\tilde X_{\bfn}(y^{\bfn_1}y^{\bfn_2})
=
\tilde X_{\bfn}(y^{\bfn_1})y^{\bfn_2}
+
y^{\bfn_1}\tilde X_{\bfn}(y^{\bfn_2})
\end{align}
due to  \eqref{eq:Leibnitz2}.
Moreover, it preserves the Poisson bracket; namely, we have
\begin{align}
\tilde X_{\bfn} (\{y^{\bfn_1},y^{\bfn_2})\})
=
 \{ \tilde X_{\bfn}(y^{\bfn_1}),y^{\bfn_2}\}
+
\{y^{\bfn_1},\tilde X_{\bfn}(y^{\bfn_2})\}
\end{align}
due to the Jacobi identity.

Let $\mathrm{Der}(\bbQ[[\bfy]])$ be 
the \emph{derivation Lie algebra}\index{derivation!Lie algebra} of $\bbQ[[\bfy]]$,
i.e.,
the Lie algebra  consisting of all derivations of $\bbQ[[\bfy]]$
\cite[\S I.2]{Jacobson79}.

\begin{prop}
[Cf. {\cite[\S 1.1]{Gross14}}]
\label{prop:gaction1}
The  linear map
\begin{align}
\begin{matrix}
{\rp \rho_y}: &
\widehat g & \rightarrow
& \mathrm{Der}(\bbQ[[\bfy]])
\\
&
\displaystyle
\sum_{\bfn\in N_+} c_{\bfn}
 X_{\bfn}
&
\mapsto
&
\displaystyle
\sum_{\bfn\in N_+} c_{\bfn} 
\tilde X_{\bfn}
\end{matrix}
\end{align}
yields a Lie algebra homomorphism.
Moreover, if $\Omega$ is nonsingular, it is injective.
\end{prop}
\begin{proof}
Thanks to the Jacobi identity of the Poisson bracket, we have
\begin{align*}
[\tilde X_{\bfn_1}, \tilde X_{\bfn_2}](y^{\bfn})
&
=\{y^{\bfn_1}, \{y^{\bfn_2},y^{\bfn}\}\}
- \{y^{\bfn_2}, \{y^{\bfn_1},y^{\bfn}\}\}
\\
&
=
\{\{y^{\bfn_1}, y^{\bfn_2}\},y^{\bfn}\}
\\
&
= \{\bfn_1,\bfn_2\}_{\Omega} \{y^{\bfn_1+\bfn_2}, y^{\bfn}\}
\\
&
= \{\bfn_1,\bfn_2\}_{\Omega} \tilde X_{\bfn_1+\bfn_2}(y^{\bfn}).
\end{align*}
Thus, it is a Lie algebra homomorphism.
If $\Omega$ is nonsingular,
$\tilde X_{\bfn}$ ($\bfn\in N^+$) are linearly independent by \eqref{eq:tildeX1}.
Therefore, ${\rp \rho_y}$ is injective.
\end{proof}

For any $X \in\widehat\frakg$,
let $\tilde{X}:={\rho}_{y}(X)$.
We define 
\begin{align}
\label{eq:exp1}
\mathrm{Exp}(\tilde{X}):=\sum_{k=0}^{\infty} \frac{1}{k!}\tilde{X}^k
\in \mathrm{GL}(\bbQ[[\bfy]]).
\end{align}
Since $\tilde{X}$ is a derivation, $\mathrm{Exp}(\tilde{X})$
is an algebra automorphism of $\bbQ[[\bfy]]$
\cite[\S I.2]{Jacobson79}.
Then, we have the  following group homomorphism,
where we abuse the symbol ${\rp \rho_y}$, for simplicity.
\begin{prop}
[Cf. {\cite[\S 1.1]{Gross14}}]
\label{prop:gaction2}
We have a  group homomorphism
\begin{align}
\label{eq:Xn2}
\begin{matrix}
{\rp \rho_y}:&  G &\rightarrow & \mathrm{Aut}(\bbQ[[\bfy]]) \\
& \exp(X) &  \mapsto &\mathrm{Exp} (\tilde{X}).
\end{matrix}
\end{align}
Moreover, if $\Omega$ is nonsingular,
it is injective.
\end{prop}
\begin{proof}
The map $\rho_y$ is a group homomorphism,
because the exponential \eqref{eq:exp1} also  satisfies the 
BCH formula.
The injectivity of ${\rp \rho_y}$ follows from the one in
 Proposition  \ref{prop:gaction1}.
\end{proof}

As we see in the next section,
the representation ${\rp \rho_y}$ is
 closely related to the mutations of $y$-variables.
So, we call it the \emph{$y$-representation}\index{$y$-representation} of $G$.

\section{Dilogarithm elements and mutations}
\label{sec:dilogelements1}

Now we introduce the main player of Part III.

\begin{defn}[Dilogarithm element]
\label{3defn:diloge1}
For each $\bfn\in N^+$,
we define
\begin{align}
\label{3eq:gei1}
\Psi[\bfn]:=\exp
\Biggl(\,
\sum_{j=1}^{\infty} \frac{(-1)^{j+1}}{j^2} X_{j \bfn}
\Biggr)
\in 
G_{\bfn_0}^{\parallel},
\end{align}
where $\bfn_0\in N^+_{\rmpr}$ is the  one such that $\bfn=h \bfn_0$ for some  $h\in \bbZ_{>0}$.
We call $\Psi[\bfn]$ the \emph{dilogarithm element}\index{dilogarithm!element}
 for $\bfn$.
\end{defn}

The above definition is closely related to the following formula.

\begin{prop}[{Cf.\ \cite[Lemma 1.3]{Gross14}}]
\label{prop:gn1}
Under the $y$-represent\-ation ${\rp \rho_y}$, $\Psi[\bfn]$ acts on $\bbQ[[\bfy]]$ as
\begin{align}
\label{eq:gei2y}
\Psi[\bfn](y^{\bfn'})=y^{\bfn'} (1+y^{\bfn})^{\{ \bfn, \bfn'\}_{\Omega}}
\quad (\bfn'\in  \bbZ_{\geq 0}^n).
\end{align}
\end{prop}
\begin{proof}
By \eqref{eq:tildeX1}, we have
\begin{align*}
\label{eq:dilogc1}
\Psi[\bfn](y^{\bfn'})
&=
y^{\bfn'}
\exp
\biggl(
\,
\sum_{j=1}^{\infty}  \frac{(-1)^{j+1}}{j^2} y^{j \bfn}\{ j\bfn, \bfn' \}_{\Omega}
\biggr)
\\
&=
y^{\bfn'}
\exp
\biggl(\,
 \sum_{j=1}^{\infty}   \frac{(-1)^{j+1}}{j}  y^{j\bfn}
\biggr)
^{\{ \bfn, \bfn' \}_{\Omega}}
\\
&=y^{\bfn'} (1+y^\bfn)^{\{ \bfn, \bfn'\}_{\Omega}}.
\end{align*}
\end{proof}

We may compare the above calculation with the following one (see also \eqref{eq:dLx1}):
\begin{align}
\begin{split}
x\frac{d}{dx}
\biggl(
-\mathrm{Li}_2(-x)
\biggr)
&=
x\frac{d}{dx}
\biggl(
\sum_{j=1}^{\infty} {\frac{(-1)^{j+1}}{j^2}}x^j
\biggr)
\\
&=
\sum_{j=1}^{\infty} {\frac{(-1)^{j+1}}{j}}x^j
=
\log (1+x).
\end{split}
\end{align}
Thus, the dilogarithm elements are regarded as a certain algebraic formulation
of (the exponential of) the Euler dilogarithm.

Next, let us compare the formula \eqref{eq:gei2y} with
the automorphism
$\frakq(s)$ in \eqref{eq:fq1}.
Under the identification of the matrix $\Omega$ in  both situations,
we have
\begin{align}
\label{eq:dm1}
\frakq{(s)}=
{\rp \rho_y}(
\Psi[\bfc^+_{k_s}(s)]^{-\varepsilon_{s}\d_{k_s}}),
\end{align}
where we regard $\frakq(s)$ as acting on
$\bbQ[[\bfy]]$.
More generally, 
 the time-one flow   \eqref{eq:tHc1}
of the Hamiltonian $\scrH_{\bfn,a}= a \mathrm{Li}_2(-y^{\bfn})$
is identified with the action of $\Psi[\bfn]^{-a}$.

In summary,
mutations are described by
the action of dilogarithm elements.
This is also regarded as an algebraic reformulation
of the Hamiltonian picture of mutations,
where the  time-one flow of the Hamiltonian  is replaced with the
 action of dilogarithm elements.
The identification \eqref{eq:dm1}  serves as a key to relate a cluster pattern/$Y$-pattern and a forthcoming \emph{cluster scattering diagram}.

There are  several advantages of 
 this algebraic formulation as follows:
 \begin{enumerate}
 \item
 A composition of mutations is expressed
 as a product of dilogarithm elements in $G$.
 \item
 In particular, a period of mutations is expressed as a
 relation among dilogarithm elements in the group $G$.
 \item
 The infinite product of dilogarithm elements makes sense in $G$.
 This enables us to go beyond mutations in a cluster pattern,
 where only finite compositions of mutations are allowed.
\item
The group $G$ controls mutations locally and globally.
 \item
 The quantum case will be treated in a parallel way.
 \end{enumerate}

\section{DI for  dilogarithm elements}

The periodicity condition \eqref{eq:sigmap2} implies a relation of dilogarithm elements in the group $G$.

\begin{thm}
[DI for dilogarithm elements]
\label{thm:DIE1}
 \index{dilogarithm identity!for  dilogarithm element}
Suppose that
  the  sequence of  mutations
  \eqref{eq:mseq2} is  \emph{$\nu$-periodic}.
  Then, the following relation holds in $G$.
\begin{align}
\label{eq:DIE2}
\Psi[\bfc^+_{k_{P-1}}({P-1})]^{\varepsilon_{P-1}\delta_{k_{P-1}}}
\cdots
\Psi[\bfc^+_{k_0}(0)]^{\varepsilon_0 \delta_{k_0}}
=\rmid.
\end{align}
\end{thm}
\begin{proof}
Case 1.
Assume $\Upsilon(0)=\Upsilon$.
Then, we are in the situation of \eqref{eq:mseq5}.
By \eqref{eq:rho6}, we have
 \begin{align}
 \label{eq:qPid2}
\frakq(P-1;0)=\rmid.
\end{align}
Thus,
thanks to the identification \eqref{eq:dm1},
\begin{align}
g=
\Psi[\bfc^+_{k_0}(0)]^{- \varepsilon_0 \delta_{k_0}}
\cdots
\Psi[\bfc^+_{k_{P-1}}({P-1})]^{- \varepsilon_{P-1}\delta_{k_{P-1}}}
\end{align}
acts trivially on $\bbZ[[\bfy]]$.
Assume that $B$ is nonsingular.
Then, by Proposition \ref{prop:gaction2},
$g$ is the identity element in $G$.
Taking its inverse, we obtain \eqref{eq:DIE2}.
If $B$ is singular,
let  $\overline B$  be the principal extension in \eqref{eq:pext1}.
 Accordingly, the matrix $\Omega$ is extended as
 \begin{align}
\label{eq:pextO1}
\overline \Omega
=
\begin{pmatrix}
\Omega & - \Delta^{-1}\\
\Delta^{-1} & O\\
\end{pmatrix}
\end{align}
By Proposition \ref{prop:principal1} and \eqref{eq:oC1},
we obtain the relation \eqref{eq:DIE2} in the group $G_{\overline\Omega}$.
Since $G=G_{\Omega}$ is a subgroup of $G_{\overline\Omega}$, we obtain the relation
\eqref{eq:DIE2} in $G$.
\par
Case 2. Assume $\Upsilon(0)\neq \Upsilon$.
By \eqref{eq:comp1}, we extend the sequence \eqref{eq:mseq2}
as 
\begin{align}
\label{eq:extmut1}
&\Upsilon
\
{\buildrel {k'_0} \over \rightarrow}
\
\Upsilon'(1)
\
{\buildrel {k'_1} \over \rightarrow}
\
\cdots
\
{\buildrel {k'_r} \over \rightarrow}
\
\Upsilon(0) 
\
{\buildrel {k_0} \over \rightarrow}
\
\cdots
\
{\buildrel {k_{P-1}} \over \rightarrow}
\
\Upsilon(P)
=\nu\Upsilon(0)
\
{\buildrel {\nu(k'_r)} \over \rightarrow}
\cdots
\
{\buildrel {\nu(k'_0)} \over \rightarrow}
\
\nu\Upsilon.
\end{align}
Then, the extended sequence belongs to Case 1.
Note that the $k_t$th $c$-vector of $\Upsilon'(t)$ coincides with
 the $\nu(k_t)$th $c$-vector of $\nu\Upsilon'(t)$.
Moreover, by \eqref{2eq:ck1},
it is the inverse of the $\nu(k_t)$th $c$-vector of $\nu\Upsilon'(t+1)$.
Thus,
the corresponding relation \eqref{eq:DIE2} has the form
\begin{align}
\label{eq:hph1}
 h^{-1} ( \Psi[\bfc^+_{k_{P-1}}({P-1})]^{\varepsilon_{P-1}\delta_{k_{P-1}}}
\cdots
\Psi[\bfc^+_{k_0}(0)]^{\varepsilon_0 \delta_{k_0}} )h
=\rmid
\end{align}
where $h\in G$ is the contribution of the mutations from $\Upsilon$ to $\Upsilon(0)$.
Therefore, we have \eqref{eq:DIE2}.
\end{proof}

We regard the relation \eqref{eq:DIE2} as an algebraic version of the dilogarithm identity
\eqref{eq:DI3}.

\begin{ex}
\label{ex:DIA23}
Let us consider the pentagon periodicity of type $A_2$.
We choose a skew-symmetric decomposition of the initial exchange matrix $B$ in \eqref{eq:BA2}
as
\begin{align}
\label{eq:Omega1}
\Delta=I,
\quad
\Omega=
\begin{pmatrix}
0 & -1 
\\
1 & 0
\end{pmatrix}.
\end{align}
The $c$-vectors are found in \eqref{eq:ys3}.
Then, the relation \eqref{eq:DIE2} is written as
\begin{align}
\Psi[\bfe_2]^{-1}
\Psi[\bfe_1+\bfe_2]^{-1}
\Psi[\bfe_1]^{-1}
\Psi[\bfe_2]
\Psi[\bfe_1]
=\rmid.
\end{align}
Equivalently, we have
\begin{align}
\label{eq:pent9}
\Psi[\bfe_2]
\Psi[\bfe_1]
=
\Psi[\bfe_1]
\Psi[\bfe_1+\bfe_2]
\Psi[\bfe_2].
\end{align}
We call it the \emph{pentagon relation}\index{pentagon!relation}.
\end{ex}

\section{Pentagon relation}

The pentagon relation in Example \ref{ex:DIA23} holds in a more general situation.
\begin{prop}
\label{3prop:pent1}
Let $\bfn_1,\bfn_2\in N^+$.
The following relations hold in $G$.
\par
(a) (Commutative relation\index{commutative relation}). If $\{\bfn_2,\bfn_1\}_{\Omega}=0$,
then, for any $c_1,c_2\in \bbQ$, we have
\begin{align}
\label{eq:Pcom1}
\Psi[\bfn_2]^{c_2} \Psi[ \bfn_1]^{c_1} =\Psi[  \bfn_1 ]^{c_1}  \Psi[ \bfn_2]^{c_2}.
\end{align}

(b) (Pentagon relation\index{pentagon!relation} \cite[Ex.~1.14]{Gross14}, \cite[Prop.~III.1.14]{Nakanishi22a}). 
If $\{\bfn_2,\bfn_1\}_{\Omega}=c\neq 0$,
we have
\begin{align}
\label{eq:pent8}
\Psi[\bfn_2 ]^{1/c} \Psi[ \bfn_1]^{1/c}=
\Psi[  \bfn_1 ]^{1/c} \Psi[\bfn_1+\bfn_2]^{1/c} \Psi[ \bfn_2]^{1/c}.
\end{align}
\end{prop}
\begin{proof}
(a).
This is clear from \eqref{eq:Xcom1}.
(b).
Let us consider the rank 2 sublattice $N'$ of $N$ generated by $\bfn_1$ and $\bfn_2$.
By the assumption $\{\bfn_2, \bfn_1 \}_{\Omega}\neq 0$, the form $\{\cdot, \cdot \}_{\Omega}$  restricted on $N'$ is nondegenerate.
Let $G'$ be the subgroup of $G$ corresponding to $N'$.
Since the relation \eqref{eq:pent8} involves only elements in $G'$,
one can prove it
by the $y$-representation of $G'$, which is faithful.
The LHS is given by
\begin{align*}
&\quad\ \Psi[\bfn_2]^{1/c} \Psi[ \bfn_1]^{1/c} (y^{\bfn})\\
&= \Psi[\bfn_2]^{1/c}(y^{\bfn} (1+y^{\bfn_1})^{\{ \bfn_1, \bfn\}_{\Omega}/c})
\\
&= y^{\bfn}  (1+y^{\bfn_2})^{ \{  \bfn_2, \bfn\}_{\Omega}/c}
 (1+y^{\bfn_1}  (1+y^{\bfn_2})
)^{\{ \bfn_1, \bfn\}_{\Omega}/c}.
\end{align*}
The RHS is given by
\begin{align*}
&\quad\ \Psi[ \bfn_1 ]^{1/c} \Psi[ \bfn_1+ \bfn_2]^{1/c} \Psi[\bfn_2]^{1/c}(y^{\bfn})\\
&=  \Psi[\bfn_1 ]^{1/c} \Psi[ \bfn_1+\bfn_2]^{1/c} (y^{\bfn} (1+y^{\bfn_2})^{\{ \bfn_2, \bfn\}_{\Omega}/c})
\\
&=  \Psi[ \bfn_1 ]^{1/c}( y^{\bfn}  (1+y^{\bfn_1+\bfn_2})^{\{  \bfn_1+\bfn_2, \bfn\}_{\Omega}/c}\\
&\qquad \times
 (1+y^{\bfn_2}  (1+y^{\bfn_1+\bfn_2})^{-1}
)^{\{   \bfn_2, \bfn\}_{\Omega}/c})
\\
&=  \Psi[ \bfn_1 ]^{1/c} (y^{\bfn}  (1+y^{\bfn_1+\bfn_2})^{\{  \bfn_1, \bfn\}_{\Omega}/c}
 (1+y^{\bfn_1+\bfn_2} +y^{\bfn_2})^{\{   \bfn_2, \bfn\}_{\Omega}/c})
\\
&=  y^{\bfn} (1+y^{\bfn_1})^{\{ \bfn_1, \bfn\}_{\Omega}/c}
  (1+y^{\bfn_1+\bfn_2} (1+y^{\bfn_1})^{-1})^{\{  \bfn_1, \bfn\}_{\Omega}/c}\\
&\qquad \times
 (1+y^{\bfn_1+\bfn_2}(1+y^{\bfn_1})^{-1}+y^{\bfn_2}
 (1+y^{\bfn_1})^{-1}
)^{\{   \bfn_2, \bfn\}_{\Omega}/c}
\\
&=  y^{\bfn}
  ( 1+y^{\bfn_1}+y^{\bfn_1+\bfn_2})^{\{  \bfn_1, \bfn\}_{\Omega}/c}
 (1+y^{\bfn_2})^{\{   \bfn_2, \bfn\}_{\Omega}/c}.
\end{align*}
Thus, two expressions agree.
\end{proof}

\begin{ex}
\label{ex:pentpent1}
For the matrix $\Omega$ in \eqref{eq:Omega1},
we have $\{\bfe_2,\bfe_1\}_{\Omega}=\omega_{21}=1$.
Then, we recover \eqref{eq:pent9} by the specialization of the relation  \eqref{eq:pent8}.

\end{ex}

\section{Ordering problem}
\label{sec:ordering1}

Let us demonstrate that the pentagon relation
\eqref{eq:pent8} generates
 several interesting relations in $G$, including infinite ones.
 
 Concentrate on the rank 2 case,
 and let $\Omega$ be the one in \eqref{eq:Omega1}.
We consider a possibly infinite product of powers of dilogarithm elements $\Psi[\bfn]^c$
($\bfn\in N^+$, $c\in \bbQ$). Such a product is well-defined in $G$ if, for each degree $\ell$,
there are only  finitely many terms in the product such that $\deg \bfn\leq \ell$.
We say that
such
a (possibly infinite) product is \emph{ordered}\index{ordered product} (resp. \emph{anti-ordered}\index{anti-ordered product})
if, for any pair $\Psi[\bfn']^{c'}$ and $\Psi[\bfn]^c$ in the product, where the former appears left to the latter,
$\{\bfn',\bfn\}_{\Omega}\leq 0$
 (resp. $\{\bfn',\bfn\}_{\Omega}\geq 0$)
holds.
For example, in the relation \eqref{eq:pent9},
the LHS is anti-ordered, while the RHS is ordered.

Then, the following problem arises.

\begin{prob}[Ordering Problem]
\label{prob:order1}\index{ordering problem}
Is it possible to order a given anti-ordered product
to an ordered one  only by applying the commutative relation \eqref{eq:Pcom1} and the
 pentagon relation  \eqref{eq:pent8} repeatedly (possibly infinitely many times)?
\end{prob}
 
The answer crucially depends on the exponents in the product. 
For example, consider the anti-ordered product $\Psi[\bfe_2]^{1/2} \Psi[ \bfe_1]$.
Then, it is easy to see that it cannot be ordered
by the pentagon relation \eqref{eq:pent9}
 due to  the fractional power of $\Psi[\bfe_2]$.

The following fact holds.

\begin{prop}[{\cite[Prop.~III.5.4]{Nakanishi22a}}]
\label{prop:order1}
For any positive integers $\delta_1$ and $\delta_2$,
the anti-ordered product $\Psi[\bfe_2]^{\delta_2} \Psi[\bfe_1]^{\delta_1}$
can be transformed into a possibly infinite ordered product
only by applying the commutative relation \eqref{eq:Pcom1} and the
 pentagon relation  \eqref{eq:pent8} possibly infinitely many times.
\end{prop}

This is a special case of a more general result in  the forthcoming
Proposition \ref{3prop:rank21} (Ordering Lemma).
Later,
the integers $\delta_1$ and $\delta_1$ will be identified with the diagonal elements 
of the matrix $\Delta=\diag(\delta_1,\delta_2)$ in the decomposition \eqref{eq:BDO1} of $B$.
Here we present some examples, which reveal a surprisingly rich structure of the ordered product obtained
in Proposition \ref{prop:order1}.
Below,
for $\bfn=(n_1,n_2)$,
we write $\Psi[\bfn]^c$ as $[n_1,n_2]^c$ (or column-wise).
By \eqref{eq:Omega1},
we have $\{\bfn',\bfn\}_{\Omega}=n_1n'_2 - n_2 n'_1$.
Thus, 
the condition $\{\bfn',\bfn\}_{\Omega}\leq 0$ is equivalent to
$n'_1/n'_2 \geq n_1/n_2$,
where $n_1/n_2=\infty$ if $n_2=0$.

\begin{ex}
[Finite type: $\delta_1\delta_2\leq 3$]
\label{ex:ordfinite1}
\
\par
(a) Type $A_2$. Let $(\delta_1,\delta_2)=(1,1)$.
This is already covered in \eqref{eq:pent9}.
In the current notation, it is written as
\begin{align}
\label{eq:order2}
\begin{split}
\begin{bmatrix}
0\\
1
\end{bmatrix}
\begin{bmatrix}
1\\
0
\end{bmatrix}
&=
\begin{bmatrix}
1\\
0
\end{bmatrix}
\begin{bmatrix}
1\\
1
\end{bmatrix}
\begin{bmatrix}
0\\
1
\end{bmatrix}
.
\end{split}
\end{align}
Observe that the vectors appearing in the RHS are
identified with the positive roots $\alpha_1$, $\alpha_1+\alpha_2$, $\alpha_2$
of the root system of type $A_2$. 
See Figure \ref{1fig:root1}.
\par
(b) Type $B_2$. 
Here we consider the case
$(\delta_1,\delta_2)=(1,2)$.
By applying
the pentagon relation \eqref{eq:pent8}  with $c=1$ repeatedly for adjacent pairs
$\bfn'$, $\bfn$
with $\{\bfn', \bfn\}_{\Omega}=1$, we have
\begin{align}
\label{eq:order3}
\begin{split}
\begin{bmatrix}
0\\
1
\end{bmatrix}
^2
\begin{bmatrix}
1\\
0
\end{bmatrix}
&=
\begin{bmatrix}
0\\
1
\end{bmatrix}
\begin{bmatrix}
1\\
0
\end{bmatrix}
\begin{bmatrix}
1\\
1
\end{bmatrix}
\begin{bmatrix}
0\\
1
\end{bmatrix}
=
\begin{bmatrix}
1\\
0
\end{bmatrix}
\begin{bmatrix}
1\\
1
\end{bmatrix}
\begin{bmatrix}
0\\
1
\end{bmatrix}
\begin{bmatrix}
1\\
1
\end{bmatrix}
\begin{bmatrix}
0\\
1
\end{bmatrix}
\\
&=
\begin{bmatrix}
1\\
0
\end{bmatrix}
\begin{bmatrix}
1\\
1
\end{bmatrix}
^
2
\begin{bmatrix}
1\\
2
\end{bmatrix}
\begin{bmatrix}
0\\
1
\end{bmatrix}
^2
.
\end{split}
\end{align}
This relation coincides with the equality \eqref{eq:DIE2}
for the periodicity of the free $Y$-pattern of type $B_2$ in \eqref{eq:B2seq1}.
The vectors appearing in the RHS are
identified with the positive roots $\alpha_1$, $\alpha_1+\alpha_2$, 
 $\alpha_1+2 \alpha_2$, $\alpha_2$
of the root system of type $B_2$.
Moreover,  the exponent is 2 if and only if the corresponding root is a short one.

\par
(c) Type $G_2$. Let $(\delta_1,\delta_2)=(1,3)$.
In the same way as above,
we have
\begin{align}
\label{eq:order5}
\begin{split}
\begin{bmatrix}
0\\
1
\end{bmatrix}
^3
\begin{bmatrix}
1\\
0
\end{bmatrix}
&=
\begin{bmatrix}
0\\
1
\end{bmatrix}
\biggl(
\begin{bmatrix}
1\\
0
\end{bmatrix}
\begin{bmatrix}
1\\
1
\end{bmatrix}
^
2
\begin{bmatrix}
1\\
2
\end{bmatrix}
\begin{bmatrix}
0\\
1
\end{bmatrix}
^2
\biggr)
\\
&=
\begin{bmatrix}
1\\
0
\end{bmatrix}
\begin{bmatrix}
1\\
1
\end{bmatrix}
^2
\begin{bmatrix}
1\\
2
\end{bmatrix}
\begin{bmatrix}
1\\
1
\end{bmatrix}
\begin{bmatrix}
1\\
2
\end{bmatrix}
^2
\begin{bmatrix}
1\\
3
\end{bmatrix}
\begin{bmatrix}
0\\
1
\end{bmatrix}
^3
\\
&=
\begin{bmatrix}
1\\
0
\end{bmatrix}
\begin{bmatrix}
1\\
1
\end{bmatrix}
^3
\begin{bmatrix}
2\\
3
\end{bmatrix}
\begin{bmatrix}
1\\
2
\end{bmatrix}
^3
\begin{bmatrix}
1\\
3
\end{bmatrix}
\begin{bmatrix}
0\\
1
\end{bmatrix}
^3
.
\end{split}
\end{align}
This relation coincides with the equality \eqref{eq:DIE2}
for the periodicity of the free $Y$-pattern of type $G_2$ in \eqref{eq:G2seq1}.
The vectors appearing in the RHS are
identified with the positive roots $\alpha_1$, $\alpha_1+\alpha_2$, 
 $2\alpha_1+3 \alpha_2$,  $\alpha_1+2 \alpha_2$,  $\alpha_1+3 \alpha_2$,
 $\alpha_2$
of the root system of type $G_2$.
Moreover, the exponent is 3 if and only if the corresponding root is a short one.
\end{ex}

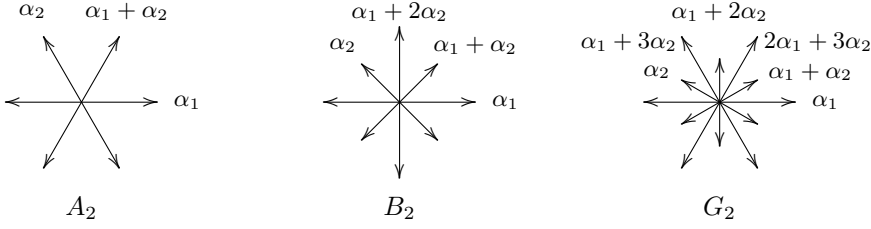
\begin{figure}
\centering
\leavevmode
\begin{xy}
(14,0)*{\mbox{\small$\alpha_1$}},
(6,12)*{\mbox{\small$\alpha_1+\alpha_2$}},
(-6.5,12)*{\mbox{\small$\alpha_2$}},
(0,-14)*{A_2}
\ar@{->} (0,0);(10,0) %
\ar@{->} (0,0);(5,8.66) %
\ar@{->} (0,0);(-5,8.66) %
\ar@{->} (0,0);(-10,0) %
\ar@{->} (0,0);(-5,-8.66) %
\ar@{->} (0,0);(5,-8.66) %
\end{xy}
\hskip43pt
\begin{xy}
(14,0)*{\mbox{\small$\alpha_1$}},
(0,12)*{\mbox{\small$\alpha_1+2\alpha_2$}},
(10,7.5)*{\mbox{\small$\alpha_1+\alpha_2$}},
(-7.5,7.5)*{\mbox{\small$\alpha_2$}},
(0,-14)*{B_2}
\ar@{->} (0,0);(10,0) %
\ar@{->} (0,0);(5,5) %
\ar@{->} (0,0);(0,10) %
\ar@{->} (0,0);(-5, 5) %
\ar@{->} (0,0);(-10,0) %
\ar@{->} (0,0);(-5,-5) %
\ar@{->} (0,0);(0,-10) %
\ar@{->} (0,0);(5,-5) %
\end{xy}
\qquad
\begin{xy}
(14,0)*{\mbox{\small$\alpha_1$}},
(-12,8)*{\mbox{\small$\alpha_1+3\alpha_2$}},
(13,8)*{\mbox{\small$2\alpha_1+3\alpha_2$}},
(0,12)*{\mbox{\small$\alpha_1+2\alpha_2$}},
(12,4)*{\mbox{\small$\alpha_1+ \alpha_2$}},
(-8.3,4)*{\mbox{\small$\alpha_2$}},
(0,-14)*{G_2}
\ar@{->} (0,0);(10,0) %
\ar@{->} (0,0);(5,8.66) %
\ar@{->} (0,0);(-5,8.66) %
\ar@{->} (0,0);(-10,0) %
\ar@{->} (0,0);(-5,-8.66) %
\ar@{->} (0,0);(5,-8.66) %
\ar@{->} (0,0);(5,2.89) %
\ar@{->} (0,0);(5,-2.89) %
\ar@{->} (0,0);(-5,2.89) %
\ar@{->} (0,0);(-5,-2.89) %
\ar@{->} (0,0);(0,5.77) %
\ar@{->} (0,0);(0,-5.77) %
\end{xy}
\caption{Rank 2 root systems of finite type.}
\label{1fig:root1}
\end{figure}

\begin{ex}
[Affine type: $\delta_1\delta_2=4$]
\label{ex:ordaffine1}
\ \par
(a) Type $A_1^{(1)}$. Let $(\delta_1,\delta_2)=(2,2)$.
By using the pentagon relation \eqref{eq:pent8}
with $c=1$,
we obtain the equality
\begin{align}
\label{3eq:com2}
\begin{bmatrix}
0\\
1
\end{bmatrix}
^2
\begin{bmatrix}
1\\
0
\end{bmatrix}
^2
&=
\begin{bmatrix}
1\\
0
\end{bmatrix}
\begin{bmatrix}
1\\
1
\end{bmatrix}
^2
\begin{bmatrix}
1\\
2
\end{bmatrix}
\begin{bmatrix}
1\\
0
\end{bmatrix}
\begin{bmatrix}
1\\
1
\end{bmatrix}
^2
\begin{bmatrix}
1\\
2
\end{bmatrix}
\begin{bmatrix}
0\\
1
\end{bmatrix}
^2.
\end{align}
The expression in the RHS
is not yet ordered. 
To make it ordered, we need to interchange
$[1,2]$ and $[1,0]$ in the middle,
where
$\{ \bfn',\bfn\}_{\Omega}=2$.
As the lowest approximation, we consider modulo
$G^{>3}$.
Then, $[1,2]$ and $[1,0]$ commute, and we have,
modulo $G^{>3}$,
\begin{align}
\begin{bmatrix}
0\\
1
\end{bmatrix}
^2
\begin{bmatrix}
1\\
0
\end{bmatrix}
^2
&
\equiv
\begin{bmatrix}
1\\
0
\end{bmatrix}
^2
\begin{bmatrix}
2\\
1
\end{bmatrix}
^2
\begin{bmatrix}
1\\
1
\end{bmatrix}
^4
\begin{bmatrix}
1\\
2
\end{bmatrix}
^2
\begin{bmatrix}
0\\
1
\end{bmatrix}
^2.
\end{align}

To proceed to a higher degree, we
now apply the pentagon relation  \eqref{eq:pent8}
with $c=2$ to the pair $[1,2]^{1/2}$ and $[1,0]^{1/2}$
in \eqref{3eq:com2},
which are obtained by factorizing the elements $[1,2]$ and $[1,0]$.
To clarify the structure among fractional powers,
we use the temporal notation $[n_1,n_2]_2:=[n_1,n_2]^{1/2}$.
Then, we have
\begin{align}
\label{3eq:com3}
\begin{bmatrix}
1\\
2
\end{bmatrix}
\begin{bmatrix}
1\\
0
\end{bmatrix}
=
\begin{bmatrix}
1\\
2
\end{bmatrix}
_2^2
\begin{bmatrix}
1\\
0
\end{bmatrix}
_2^2
&=
\begin{bmatrix}
1\\
0
\end{bmatrix}
_2
\begin{bmatrix}
2\\
2
\end{bmatrix}
_2
^2
\begin{bmatrix}
3\\
4
\end{bmatrix}
_2
\begin{bmatrix}
1\\
0
\end{bmatrix}
_2
\begin{bmatrix}
2\\
2
\end{bmatrix}
_2
^2
\begin{bmatrix}
3\\
4
\end{bmatrix}
_2
\begin{bmatrix}
1\\
2
\end{bmatrix}
_2
^2.
\end{align}
This is parallel with \eqref{3eq:com2}.
As the next approximation, we consider modulo
$G^{>7}$.
Then, $[3,4]$ and $[1,0]$ commute, and we have,
modulo $G^{>7}$,
\begin{align}
\label{3eq:fact1}
\begin{bmatrix}
1\\
2
\end{bmatrix}
\begin{bmatrix}
1\\
0
\end{bmatrix}
&
\equiv
\begin{bmatrix}
1\\
0
\end{bmatrix}
_2
^2
\begin{bmatrix}
3\\
2
\end{bmatrix}
_2
^2
\begin{bmatrix}
2\\
2
\end{bmatrix}
_2
^4
\begin{bmatrix}
3\\
4
\end{bmatrix}
_2
^2
\begin{bmatrix}
1\\
2
\end{bmatrix}
_2
^2
=
\begin{bmatrix}
1\\
0
\end{bmatrix}
\begin{bmatrix}
3\\
2
\end{bmatrix}
\begin{bmatrix}
2\\
2
\end{bmatrix}
^2
\begin{bmatrix}
3\\
4
\end{bmatrix}
\begin{bmatrix}
1\\
2
\end{bmatrix}
.
\end{align}
Then, we plug it into
\eqref{3eq:com2}
and apply the pentagon relation with $c=1$ again,
and
we have,
modulo $G^{>7}$,
\begin{align}
\label{3eq:a11mod7}
\begin{bmatrix}
0\\
1
\end{bmatrix}
^2
\begin{bmatrix}
1\\
0
\end{bmatrix}
^2
&
\equiv
\begin{bmatrix}
1\\
0
\end{bmatrix}
^2
\begin{bmatrix}
2\\
1
\end{bmatrix}
^2
\begin{bmatrix}
3\\
2
\end{bmatrix}
^2
\begin{bmatrix}
4\\
3
\end{bmatrix}
^2
\begin{bmatrix}
1\\
1
\end{bmatrix}
^4
\begin{bmatrix}
2\\
2
\end{bmatrix}
^2
\begin{bmatrix}
3\\
4
\end{bmatrix}
^2
\begin{bmatrix}
2\\
3
\end{bmatrix}
^2
\begin{bmatrix}
1\\
2
\end{bmatrix}
^2
\begin{bmatrix}
0\\
1
\end{bmatrix}
^2.
\end{align}
By repeating this procedure modulo  $G^{>2^{\l}-1}$,
the relations converge to 
\begin{align}
\label{3eq:a115}
\begin{bmatrix}
0\\
1
\end{bmatrix}
^2
\begin{bmatrix}
1\\
0
\end{bmatrix}
^2
&
=
\begin{bmatrix}
1\\
0
\end{bmatrix}
^2
\begin{bmatrix}
2\\
1
\end{bmatrix}
^2
\begin{bmatrix}
3\\
2
\end{bmatrix}
^2
\cdots
\prod_{j=0}^{\infty}
\begin{bmatrix}
2^j\\
2^j
\end{bmatrix}
^{2^{2-j}}
\cdots
\begin{bmatrix}
2\\
3
\end{bmatrix}
^2
\begin{bmatrix}
1\\
2
\end{bmatrix}
^2
\begin{bmatrix}
0\\
1
\end{bmatrix}
^2,
\end{align}
where an infinite product emerges in the RHS.
Also,  fractional powers appear due to the factorization in the procedure.
 A complete proof of the relation \eqref{3eq:a115} by the pentagon relation was given by Matsushita
  \cite{Matsushita21}.

\smallskip

(b) Type $A_2^{(2)}$. Let $(\delta_1,\delta_2)=(1,4)$.
The situation is in parallel with $A_1^{(1)}$ though a little more complicated.
The resulting relation is as follows:
\begin{align}
\label{eq:a227}
\begin{split}
&
\begin{bmatrix}
0\\
1
\end{bmatrix}
^4
\begin{bmatrix}
1\\
0
\end{bmatrix}
=
\begin{bmatrix}
1\\
0
\end{bmatrix}
\begin{bmatrix}
1\\
1
\end{bmatrix}
^4
\begin{bmatrix}
3\\
4
\end{bmatrix}
\begin{bmatrix}
2\\
3
\end{bmatrix}
^4
\begin{bmatrix}
5\\
8
\end{bmatrix}
\begin{bmatrix}
3\\
5
\end{bmatrix}
^4
\cdots
\\
&
\hskip60pt
\times
\begin{bmatrix}
1\\
2
\end{bmatrix}
^6
\prod_{j=1}^{\infty}
\begin{bmatrix}
 2^j\\
 2^{j+1}
\end{bmatrix}
^
{2^{2-j}}
\cdots
\begin{bmatrix}
5\\
12
\end{bmatrix}
\begin{bmatrix}
2\\
5
\end{bmatrix}
^4
\begin{bmatrix}
3\\
8
\end{bmatrix}
\begin{bmatrix}
1\\
3
\end{bmatrix}
^4
\begin{bmatrix}
1\\
4
\end{bmatrix}
\begin{bmatrix}
0\\
1
\end{bmatrix}
^4
.
\end{split}
\end{align}
Again,
a complete proof of the equality \eqref{eq:a227} by the pentagon relation was given by \cite{Matsushita21}.
\end{ex}

\begin{ex}
[Non-affine infinite type: $\delta_1\delta_2\geq 5$]
\label{ex:ordnonaffine1}

The explicit expression of the ordered product is not yet known,
but it is quite intriguing.
There is a program to compute the ordered product
written for SageMath 
\cite{Sage94}
whose source code is available
in \cite[III. Appendix A]{Nakanishi22a}.
The program can calculate the ordered product 
up to a given degree; however, 
in practice, once some exponent exceeds 1,000,
it is difficult to continue calculation due to the limitation of time.
As an example, for $(\delta_1,\delta_2)=(1,5)$,
the product has the following form modulo $G^{>16}$:
\begin{align}
\label{eq:a5116}
\begin{split}
\begin{bmatrix}
0\\
1
\end{bmatrix}
^5
\begin{bmatrix}
1\\
0
\end{bmatrix}
&\equiv
\begin{bmatrix}
1\\
0
\end{bmatrix}
\begin{bmatrix}
1\\
1
\end{bmatrix}
^5
\begin{bmatrix}
4\\
5
\end{bmatrix}
\begin{bmatrix}
3\\
4
\end{bmatrix}
^5
\biggl\{
\begin{bmatrix}
5\\
7
\end{bmatrix}
^{10}
\begin{bmatrix}
2\\
3
\end{bmatrix}
^{10}
\begin{bmatrix}
4\\
6
\end{bmatrix}
^{10}
\begin{bmatrix}
6\\
9
\end{bmatrix}
^{15}
\begin{bmatrix}
5\\
8
\end{bmatrix}
^{85}
\begin{bmatrix}
3\\
5
\end{bmatrix}
^{27}
\\
&\quad\times
\begin{bmatrix}
6\\
10
\end{bmatrix}
^{302}
\begin{bmatrix}
4\\
7
\end{bmatrix}
^{85}
\begin{bmatrix}
5\\
9
\end{bmatrix}
^{295}
\begin{bmatrix}
1\\
2
\end{bmatrix}
^{10}
\begin{bmatrix}
2\\
4
\end{bmatrix}
^{10}
\begin{bmatrix}
3\\
6
\end{bmatrix}
^{15}
\begin{bmatrix}
4\\
8
\end{bmatrix}
^{45}
\begin{bmatrix}
5\\
10
\end{bmatrix}
^{130}
\\
&\quad\times
\begin{bmatrix}
5\\
11
\end{bmatrix}
^{1095}
\begin{bmatrix}
4\\
9
\end{bmatrix}
^{295}
\begin{bmatrix}
3\\
7
\end{bmatrix}
^{85}
\begin{bmatrix}
2\\
5
\end{bmatrix}
^{27}
\begin{bmatrix}
4\\
10
\end{bmatrix}
^{302}
\begin{bmatrix}
3\\
8
\end{bmatrix}
^{85}
\begin{bmatrix}
4\\
11
\end{bmatrix}
^{295}
\begin{bmatrix}
1\\
3
\end{bmatrix}
^{10}
\\
&\quad\times
\begin{bmatrix}
2\\
6
\end{bmatrix}
^{10}
\begin{bmatrix}
3\\
9
\end{bmatrix}
^{15}
\begin{bmatrix}
4\\
12
\end{bmatrix}
^{45}
\begin{bmatrix}
3\\
10
\end{bmatrix}
^{27}
\begin{bmatrix}
2\\
7
\end{bmatrix}
^{10}
\biggr\}
\begin{bmatrix}
3\\
11
\end{bmatrix}
\begin{bmatrix}
1\\
4
\end{bmatrix}
^5
\begin{bmatrix}
1\\
5
\end{bmatrix}
\begin{bmatrix}
0\\
1
\end{bmatrix}
^5
.
\end{split}
\end{align}
The behavior of
the terms outside the bracket are completely determined by combining  the results by 
GHKK \cite{Gross14} and Reading
\cite{Reading12}.
Namely,
they are
\begin{align}
[\bfc_0]
[\bfc_1]
^5
[\bfc_2]
[\bfc_3]
^5
\cdots
\{\cdots\}
\cdots
[\bfc'_3]
[\bfc'_2]
^5
[\bfc'_1]
[\bfc'_0]
^5,
\end{align}
where
\begin{align}
\bfc_0
&=\begin{pmatrix}
1\\
0
\end{pmatrix}
,
\
\bfc_1
=\begin{pmatrix}
1\\
1
\end{pmatrix}
,
\
\bfc_2
=\begin{pmatrix}
4\\
5
\end{pmatrix}
,
\
\bfc_3
=\begin{pmatrix}
3\\
4
\end{pmatrix}
,
\
\bfc_4
=\begin{pmatrix}
11\\
15
\end{pmatrix}
,
\,
\cdots,
\\
\bfc'_0
&=\begin{pmatrix}
0\\
1
\end{pmatrix}
,
\
\bfc'_1
=\begin{pmatrix}
1\\
5
\end{pmatrix}
,
\
\bfc'_2
=\begin{pmatrix}
1\\
4
\end{pmatrix}
,
\
\bfc'_3
=\begin{pmatrix}
4\\
15
\end{pmatrix}
,
\
\bfc'_4
=\begin{pmatrix}
3\\
11
\end{pmatrix}
,
\,
\cdots.
\end{align}
These vectors are uniquely determined from the recurrence relations
\begin{align}
\bfc_i=
\begin{cases}
5\bfc_{i-1} - \bfc_{i-2}
& \text{$i$: even},
\\
\bfc_{i-1} - \bfc_{i-2}
& \text{$i$: odd},
\end{cases}
\quad
\bfc'_i=
\begin{cases}
\bfc'_{i-1} - \bfc'_{i-2}
& \text{$i$: even},
\\
5 \bfc'_{i-1} - \bfc'_{i-2}
& \text{$i$: odd}.
\end{cases}
\end{align}
The ratios of  the components of these vectors have the limits
\begin{align}
\bfc_i&=\begin{pmatrix}
n_{1;i}\\
n_{2;i}
\end{pmatrix}
,
\quad
\lim_{i\rightarrow \infty}
\frac{n_{1;i}}{n_{2;i}}=
\frac{2}{5-\sqrt{5}}
=0.723\dots
,
\\
\bfc'_i&=\begin{pmatrix}
n'_{1;i}\\
n'_{2;i}
\end{pmatrix}
,
\quad
\lim_{i\rightarrow \infty}
\frac{n'_{1;i}}{n'_{2;i}}=
\frac{2}{5+\sqrt{5}}
=0.276\dots
.
\end{align}
In contrast, inside the bracket
in \eqref{eq:a5116},
every $\bfn\in N^+$ with $\deg(\bfn)\leq 16$
such that
\begin{align}
\label{eq:nn1}
\frac{2}{5+\sqrt{5}}<
\frac{n_1}{n_2}
<
\frac{2}{5-\sqrt{5}}
\end{align}
appears.
It is recently shown by Gr\"afnitz and Luo \cite{Grafnitz23} that this is true for any $\bfn$
satisfying \eqref{eq:nn1}.
However, the explicit description of the exponents is known only partially (e.g., \cite{Akagi23}).
\end{ex}

The meaning of the relations in the above examples  will be  clarified in the next chapter
in view of cluster scattering diagrams.

\notes
The class of Lie algebras and their exponential groups
in this chapter
were introduced
in \cite{Gross07, Kontsevich08, Kontsevich13}
in more generality in the study of  mirror symmetry and wall-crossing structures
for the Donaldson-Thomas invariants.
The  group $G$ here appeared  in \cite{Gross14}
in the disguised form of the $x$-representation.
The dilogarithm elements and the pentagon relation 
appeared in \cite{Kontsevich08, Kontsevich13} in a slightly different
setup, and also
  in  \cite{Gross14}
in the disguised form of the $x$-representation.
Here we follow
the presentation in the companion monograph  \cite{Nakanishi22a}.
Theorem \ref{thm:DIE1} is new here, but
the rank 2 examples of DIs appeared in
 \cite{Nakanishi22a}.
The ordering problem is closely related to the wall-crossing formula in \cite{Kontsevich08}.

\chapter{Quick course on cluster scattering diagrams}
\label{ch:CSD1}

Following GHKK \cite{Gross14},
we introduce an  algebraic and geometrical object called a \emph{cluster scattering diagram
(CSD)}.
For any cluster pattern, the fan constructed from its $G$-matrices are embedded in the corresponding CSD.
Therefore, it is regarded as an extension of
a cluster pattern.
It turns out that the dilogarithm elements and the pentagon relation play fundamental roles in CSDs.
Our presentation is minimal without proofs.
The proofs and details are found in
 the companion monograph
 \cite{Nakanishi22a}.

\section{Scattering diagrams}

Continuing from the previous chapter, 
we fix a given skew-symmetric rational matrix $\Omega$.
Let $G=G_{\Omega}$ be the group
constructed in Section \ref{sec:exponential1}.

Following \cite{Gross14}, we introduce several notions related to {scattering diagrams}.
For the lattice $N=\bbZ^n$,
let
\begin{align}
M:=\mathrm{Hom}(N,\bbZ), \quad M_{\bbR}:=M\otimes \bbR.
\end{align}
Let $\langle  \bfn, m\rangle$ be the canonical paring  $N\times M\rightarrow \bbZ$
or its extension $N\times M_{\bbR} \rightarrow \bbR$, depending on the context.
For each $\bfn\in N$, $\bfn\neq \bfzero$,
let
\begin{align}
 \bfn^{\perp}:=\{z \in M_{\bbR} \mid \langle \bfn, z\rangle =0 \}
 \end{align}
 be the hyperplane in $M_{\bbR}$ which is orthogonal to $\bfn$.
For any $m_1,\, \dots,\, m_k\in M$, the positive span
\begin{align}
\sigma(m_1,\dots,m_k):=\biggl\{ \sum_{i=1}^k a_i m_i \, \bigg\vert \, a_i\in \bbR_{\geq 0}\biggr\}
\subset M_{\bbR}
\end{align}
is called a \emph{convex rational cone}\index{convex rational cone, \see{cone}}. We simple call it a \emph{cone}\index{cone}.
We also set $\sigma(\emptyset):=\{0\}$, which is also a cone.
A cone $\sigma$ is \emph{strongly convex}\index{strongly convex cone} if 
$\sigma \cap  (-\sigma) = \{0\}$.

\begin{defn}[Wall/Scattering diagram]
Let $G=G_{\Omega}$ be the group as above.
\par
(a). A \emph{wall}\index{wall} $\bfw=(\frakd, g)_{\bfn}$ is a triplet such that
$\bfn\in N^+_{\rmpr}$,  a cone  $\frakd \subset \bfn^{\perp}$  of dimension $n-1$,
and $g\in G_{\bfn}^{\parallel}$.
We call $\bfn$, $\frakd$, $g$ the \emph{normal vector}\index{normal vector}, the \emph{support}\index{support (of wall)}, and the \emph{wall element}\index{wall element}
of $\bfw$, respectively.

(b). A  \emph{scattering diagram}\index{scattering diagram} $\frakD$  is a collection of walls
 $\{ \bfw_{\lambda}=(\frakd_{\lambda}, g_{\lambda})_{\bfn_{\lambda}}\}_{\lambda\in \Lambda}$
with a (possibly infinite) index set $\Lambda$
satisfying  the following \emph{finiteness condition}\index{finiteness condition}:
\begin{itemize}
\item
For any  integer $\ell\geq 1$, there are only finitely many walls such that 
$\pi_{\ell}(g_{\lambda})\neq \rmid$, where $\pi_{\ell}:G \rightarrow G^{\leq \ell}$ is
the canonical projection \eqref{eq:cp1}.
\end{itemize}
The group $G$ is called the \emph{structure group}\index{structure group} of $\frakD$.

(c). For each integer $\ell\geq 1$, the subcollection
 $\frakD_{\ell}$ consisting of the walls with $\pi_{\ell}(g_{\lambda})\neq \rmid$
 is called the \emph{reduction}\index{reduction!of scattering diagram} of $\frakD$ at  degree $\ell$.
 By the finiteness condition, there are only finitely may walls of  $\frakD_{\ell}$.
\end{defn}

We give more related notions.

\begin{defn}[Support/Singular locus]
For a scattering diagram $\frakD=\{ \bfw_{\lambda}= (\frakd_{\lambda}, g_{{\lambda}})_{\bfn_{\lambda}}\}_{\lambda\in \Lambda}$,
the \emph{support}\index{support!of scattering diagram} and the \emph{singular locus}\index{singular locus (of scattering diagram)}  of $\frakD$ are defined by
\begin{align}
\mathrm{Supp}(\frakD)&:=\bigcup_{\lambda\in \Lambda} \frakd_{\lambda},\\ 
\mathrm{Sing}(\frakD)&:=\bigcup_{\lambda\in \Lambda} \partial\frakd_{\lambda} 
\cup
\bigcup_{ \scriptstyle\lambda,\, \lambda'\in \Lambda \atop  \scriptstyle\dim 
\frakd_{\lambda}\cap \frakd_{\lambda'}\,=\,n-2
} \frakd_{\lambda}\cap \frakd_{\lambda'}.
\end{align}
\end{defn}

\begin{defn}[Admissible curve]
\label{3defn:adm1}
A curve $\gamma:[0,1]\rightarrow M_{\bbR}$ is
\emph{admissible}\index{admissible (curve)} for a scattering diagram $\frakD$
if it satisfies the following conditions:
\begin{itemize}
\item[(1).]
The endpoints of $\gamma$ are in $M_{\bbR}\setminus \mathrm{Supp}(\frakD)$.
\item[(2).]
It is a smooth curve and intersects $\mathrm{Supp}(\frakD)$ transversally.
\item[(3).]
$\gamma$ does not intersect  $\mathrm{Sing}(\frakD)$.
\end{itemize}
\end{defn}

\begin{defn}[Path-ordered product]
\label{3defn:pop1}
Let $\frakD$ be any scattering diagram.
For any admissible curve $\gamma$ for $\frakD$,
we define an element $\frakp_{\gamma,\frakD}\in G$ as follows:
For each integer $\l>0$, let  $\frakD_\l$ be the reduction of $\frakD$ at 
degree $\l$.
Suppose that
$\gamma=\gamma(t)$ crosses   walls $\bfw_i=(\frakd_i, g_{i})_{\bfn_i}$
($i=1$, \dots, $s$)
of $\frakD_\l$
in this order at $t=t_i$ with 
\begin{align}
0<t_1\leq t_2\leq \cdots \leq t_s<1.
\end{align}
Since $\gamma(t_i)\notin \mathrm{Sing}(\frakD)$,
 when $\gamma$ crosses multiple walls  at a time,
these walls have a common normal vector.
We define the \emph{intersection sign}\index{intersection sign}  $\epsilon_i$ ($i=1$, \dots, $s$)  by 
\begin{align}
\label{3eq:factor1}
\epsilon_i
=
\begin{cases}
1 & \langle n_{i}, \gamma'(t_i)\rangle<0,\\
-1 & \langle n_{i}, \gamma' (t_i)\rangle>0,\\
\end{cases}
\end{align}
where  $\gamma'(t_i)$ is the velocity vector of $\gamma(t)$ at $t_i$.
Then, we define
\begin{align}
\frakp_{\gamma,\frakD_\l}:=&\
g_{s}^{\epsilon_s}
\cdots
g_{1}^{\epsilon_1}
, \quad
\frakp_{\gamma,\frakD}:=
\lim_{\l \rightarrow\infty}\frakp_{\gamma,\frakD_\l}
\in G.
\end{align}
We call $\frakp_{\gamma,\frakD}$
the \emph{path-ordered product}\index{path-ordered product} (of wall elements in $\frakD$)
along $\gamma$.
\end{defn}

Note that $\frakp_{\gamma,\frakD}$ only depends on the homotopy class of $\gamma$
in $M_{\bbR}\setminus \mathrm{Sing}(\frakD)$.

\begin{defn}[Equivalence]
Two scattering diagrams $\frakD$ and $\frakD'$ with  a common  structure group $G$ are
\emph{equivalent}\index{equivalent (for scattering diagram)} if, for any curve $\gamma$ that is admissible for both $\frakD$ and $\frakD'$, the equality 
$\frakp_{\gamma,\frakD}=\frakp_{\gamma,\frakD'}$ holds.
\end{defn}

For a given scattering diagram $\frakD$, one can obtain infinitely many scattering diagrams
that are equivalent to $\frakD$ by joining and splitting the supports 
and  wall elements and also by adding and removing the walls with the trivial wall element 
$g=\rmid$ (trivial walls)  possibly infinitely many times.

Finally, we introduce
a crucial notion for scattering diagrams.
\begin{defn}[Consistency]
A scattering diagram $\frakD$ is \emph{consistent}\index{consistent (for scattering diagrams)} if for any admissible curve $\gamma$
for $\frakD$,
the associated path-ordered product $\frakp_{\gamma,\frakD}$  depends
only on
the endpoints of $\gamma$.
In other words, $\frakD$ is consistent if
\begin{align}
\label{3eq:pgfi1}
\frakp_{\gamma, \frakD}=\rmid
\end{align}
for any admissible loop (i.e., closed admissible curve) $\gamma$ for $\frakD$.
\end{defn}

Let
\begin{align}
\label{eq:C+1}
\begin{split}
\calC^+ &:= \{z\in M_{\bbR} \mid \langle \bfe_i, z \rangle\geq 0 \ \text{for any $i$}\},
\\
\calC^- &:= \{z\in M_{\bbR} \mid \langle \bfe_i, z \rangle\leq 0\ \text{for any $i$}\}.
\end{split}
\end{align}
It is easy to see that
the support of a scattering diagram $\frakD$ intersects $\calC^{\pm}$
only in the boundary of $\calC^{\pm}$.
For a given consistent scattering diagram $\frakD$,
let $\gamma_{+-}$ be any admissible path with the initial point
in $\mathrm{Int}(\calC^+)$ and the final point
in $\mathrm{Int}(\calC^-)$.
Then,
 thanks to the consistency of $\frakD$,
  the unique element $g(\frakD):=\frakp_{\gamma_{+-},\frakD} \in G$ is assigned
to $\frakD$.

The following fact reveals the intrinsic nature of consistent scattering diagrams.

\begin{thm}
[{\cite[Thm.~2.1.6]{Kontsevich13}, \cite[Thm.~1.17]{Gross14}}]
\label{thm:bij1}
The assignment $\frakD \mapsto g(\frakD)$ gives a one-to-one correspondence between
the consistent scattering diagrams up to equivalence and
the elements in $G$.
\end{thm}

\begin{defn}
We say that a scattering diagram $\frakD$ is with \emph{minimal support}\index{minimal support}
if its support $\mathrm{Supp}(\frakD)$ is  the minimal set
among the supports of all equivalent scattering diagrams to $\frakD$.
\end{defn}

The proof of Theorem \ref{thm:bij1} also ensures the following fact.
\begin{prop}
[e.g., {\cite[Prop.~III.2.25]{Nakanishi22a}}]
For any consistent scattering diagram $\frakD$,
there exists a (not unique) scattering diagram with minimal support
that is equivalent to $\frakD$.
\end{prop}

\section{Cluster scattering diagrams}

Next, we will define a unique consistent scattering diagram $\frakD(B)$ (up to equivalence and isomorphism) associated  with
a skew-symmetrizable integer matrix $B$ called a \emph{cluster scattering diagram}.

First, we choose an arbitrary skew-symmetric decomposition \eqref{eq:BDO1} of $B$.
We then construct  the group $G=G_{\Omega}$ for  the matrix $\Omega$ as in Section \ref{sec:Lie1}.

Let us introduce an abelian group homomorphism
\begin{align}
\label{eq:p*3}
\begin{matrix}
p^*\colon & N  & \rightarrow &M_{\bbR}
\\
& \bfn & \mapsto & \{\cdot, \bfn\}_{\Omega}.
\end{matrix}
\end{align}
In particular,
\begin{align}
\langle \bfn, p^*(\bfn)\rangle = \{\bfn,\bfn\}_{\Omega}=0.
\end{align}
Thus, 
\begin{align}
\label{eq:p*2}
p^*(\bfn)\in \bfn^{\perp}
\end{align}
holds for any $\bfn\neq \bfzero$.

\begin{defn}
A wall $\bfw=(\frakd, g)_{\bfn}$ is said to be \emph{incoming}\index{incoming (wall)} (resp. \emph{outgoing}\index{outgoing (wall)})
if $p^*(\bfn)\in \frakd$ (resp. otherwise).
See figure \ref{fig:inout1}.
\end{defn}

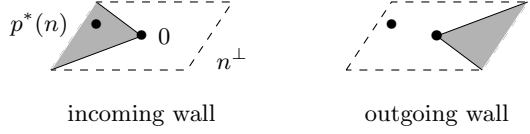
\begin{figure}
\begin{center}
\begin{tikzpicture}[scale=0.3]
\draw[dashed](0,0)--(2,3)--(8,3)--(6,0)--(0,0);
\filldraw [fill=black!30, draw=white] (4,1.5) -- (0,0) -- (2,3) -- (4,1.5);
\draw(4,1.5)--(0,0);
\draw(4,1.5)--(2,3);
\node at (4,1.5) {\small $\bullet$};
\node at (5,1.5) {\small $0$};
\node at (8,0.5) {\small $n^{\perp}$};
\node at (2,2) {\small $\bullet$};
\node at (-0.5,2) {\small $p^*(n)$};
\node at (4,-2) {\small incoming wall};
\end{tikzpicture}
\hskip30pt
\begin{tikzpicture}[scale=0.3]
\draw[dashed](0,0)--(2,3)--(8,3)--(6,0)--(0,0);
\filldraw [fill=black!30, draw=white] (4,1.5) -- (6,0) -- (8,3) -- (4,1.5);
\draw(4,1.5)--(6,0);
\draw(4,1.5)--(8,3);
\node at (4,1.5) {$\small\bullet$};
\node at (2,2) {\small $\bullet$};
\node at (4,-2) {\small outgoing wall};
\end{tikzpicture}
\end{center}
\vskip-10pt
\caption{Incoming and outgoing walls.}
\label{fig:inout1}
\end{figure}

For example, a wall $\bfw=(\bfn^{\perp}, g)_{\bfn}$ is incoming for any $\bfn\in N^+_{\rmpr}$
by \eqref{eq:p*2}.

So far, we only use the data $\Omega$ in the decomposition 
\eqref{eq:BDO1}.
The following definition involves $\Delta$ and $\Omega$ together.

\begin{defn}
\label{defn:CSD1}
For a skew-symmetrizable matrix $B$ with a skew-symmetric decomposition
\eqref{eq:BDO1},
a \emph{cluster scattering diagram} (CSD)\index{cluster!scattering diagram (CSD)} for $B$
is a consistent scattering diagram $\frakD$ 
satisfying the following properties:
\begin{itemize}
\item
The structure group of
$\frakD$ is $G_{\Omega}$.
\item
The incoming walls $\bfw_i:=(\bfe_i^{\perp}, \Psi[\bfe_i]^{\delta_i})_{\bfe_i}$ ($i=1,\dots, n$) belong to $\frakD$,
where  $\delta_i$ is the $i$th diagonal entry of $\Delta$.
\item
All other walls of $\frakD$ are outgoing.
\end{itemize}
\end{defn}

Now we present the first fundamental result on CSDs.
\begin{thm}
[{\cite[Thm.~1.21]{Gross14}}]
\label{thm:CSD1}
For any skew-symmetrizable matrix $B$ with a skew-symmetric decomposition
\eqref{eq:BDO1},
there exists a unique  CSD $\frakD(B)$ up to equivalence.

\end{thm}

The decomposition \eqref{eq:BDO1} is not unique at all.
For example, one may rescale $\Delta'=\lambda^{-1}\Delta$ and $\Omega'=\lambda\Omega$ 
such that  the entries of $\Delta'$ are integers.
However,  $G_{\Omega}$ and $G_{\Omega'}$ are isomorphic by the correspondence
$g\mapsto g'=g^{\lambda}$.
So, the CSDs for the two decompositions are naturally identified under this isomorphism.
Moreover, this is generalized to any decomposition of $B$  \cite[Prop.~III.1.23]{Nakanishi22a}.
Therefore, a CSD  essentially depends only on $B$.

The matrix $\Delta$ also determines the data
\begin{align}
\label{eq:Bcirc1}
N^{\circ}:=
\bigoplus_{i=1}^n \bbZ  \delta_i \bfe_i,
\quad
M^{\circ}:=\mathrm{Hom}(N^{\circ},\bbZ),
\end{align}
so that $N^{\circ}\subset N$ and $M\subset M^{\circ}\subset M_{\bbR}$.
For the standard basis $\bfe_1$, \dots, $\bfe_n$ of $N$,
let $e^*_1$, \dots, $e^*_n$ be the dual basis of $M$.
Then, we have a basis of $\delta_1 \bfe_1$, \dots, $\delta_n \bfe_n$ of $N^{\circ}$
and its dual basis $f_1=e^*_1/\delta_1 $, \dots, $f_n=e^*_n/\delta_n$ of $M^{\circ}$.
We identify $M_{\bbR} \simeq \bbR^n$ with $f_i \mapsto \bfe_i$, not with $e^*_i \mapsto \bfe_i$.
Under the identification,
the canonical pairing
$\langle \bfn, z \rangle\colon
 N\times M_{\bbR}\rightarrow \bbR$
   is given 
   by 
\begin{align}
\label{3eq:can1}
\langle \bfn, \bfz \rangle
=
\bfn^T
\Delta^{-1}
\bfz.
\end{align}
We also have
\begin{align}
\label{eq:p*1}
p^*(\bfe_j)(\delta_i \bfe_i)=
\{\delta_i \bfe_i, \bfe_j\}_{\Omega}=
\delta_i \omega_{ij}=b_{ij}.
\end{align}
This means that  the representation matrix of the map 
$p^*$ in \eqref{eq:p*3} with respect the basis
$\bfe_1$, \dots, $\bfe_n\in N$ and $f_1$, \dots, $f_n\in M^{\circ}\subset M_{\bbR}$ is given by $B$.
Thus,  $p^*(\bfn)\in M_{\bbR}$ is represented by the vector $B\bfn\in \bbR^n$.

\section{Rank 2 CSDs: finite type}
\label{sec:rank1}

The proof of Theorem \ref{thm:CSD1}
in {\cite{Gross14}}
 is based on certain abstract construction.
 For the rank 2 case, however, it is possible to describe CSDs more explicitly,
thanks to the results in Section \ref{sec:ordering1}.
 Here we concentrate on the finite type to illustrate the idea of a CSD.
 
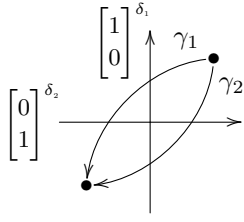
\begin{figure}
\centering
\leavevmode
\begin{xy}
0;/r1.2mm/:,
(4,9)*{\gamma_1};
(9,4)*{\gamma_2};
(7,7)*+{\bullet};(-7,-7)*+{\bullet};
(7,7)*+{};(-7,-7)*+{}
   **\crv{(0,6.7)&(-6.7,0)}
          ?>*\dir{>};
       (7,7)*+{};(-7,-7)*+{}
       **\crv{(6.7,0)&(0,-6.7)}
       ?>*\dir{>};
(-3, 9)*{\text{\small $\begin{bmatrix}1\\0\end{bmatrix}^{\delta_1}$}},
(-13, 0)*{\text{\small $\begin{bmatrix}0\\1\end{bmatrix}^{\delta_2}$}},
(0,0)="A",
\ar "A"+(0,0); "A"+(10,0)
\ar "A"+(0,0); "A"+(0,10)
\ar@{-} "A"+(0,0); "A"+(-10,0)
\ar@{-} "A"+(0,0); "A"+(0,-10)
\end{xy}
\caption{CSD of type $A_1\times A_1$.}
\label{3fig:scat1}
\end{figure}

Let us first consider the simplest case.
\begin{ex}
[Type $A_1 \times A_1$]
\label{ex:CSDA1A1}
When $B=O$, we have a decomposition
\eqref{eq:BDO1} with $\Omega=O$ and arbitrary $\Delta$.
By  the commutativity \eqref{eq:Pcom1}, we have
\begin{align}
\label{3eq:pent0}
\Psi[\bfe_2 ]^{\delta_2} \Psi[ \bfe_1]^{\delta_1}=
\Psi[  \bfe_1 ]^{\delta_1}  \Psi[ \bfe_2]^{\delta_2}.
\end{align}
This relation is naturally interpreted as
the (unique) consistency condition for a scattering diagram of rank 2 in Figure \ref{3fig:scat1}.
Namely, it consists of two incoming walls
\begin{align}
(\bfe_1^{\perp}, \Psi[\bfe_1]^{\delta_1})_{\bfe_1},
\quad
(\bfe_2^{\perp}, \Psi[\bfe_2]^{\delta_2})_{\bfe_2}.
\end{align}
The LHS of the relation \eqref{3eq:pent0}
is the path-ordered product $\frakp_{\gamma_1, \frakD}$
along $\gamma_1$,
while the RHS is the one $\frakp_{\gamma_2, \frakD}$ along $\gamma_2$.
Thus, the relation \eqref{3eq:pent0} is exactly the consistency relation
$\frakp_{\gamma_1, \frakD}=\frakp_{\gamma_1, \frakD}$.
Thus, this is a CSD for $B$.
\end{ex}

Next, let us assume  
\begin{align}
\label{eq:B3}
B=
\begin{pmatrix}
0 & -\delta_1  \\
\delta_2  & 0
\end{pmatrix}
\end{align}
for some positive integers $\delta_1$ and $\delta_2$.
We consider a skew-symmetric decomposition of $B$ with
\begin{align}
\Delta=
\begin{pmatrix}
\delta_1  & 0 \\
0  & \delta_2
\end{pmatrix},
\quad
\Omega=
\begin{pmatrix}
0 & -1  \\
1  & 0
\end{pmatrix}.
\end{align}
The advantage of this decomposition is that one can
work with the common structure group $G_{\Omega}$
for any   matrix $B$ in  \eqref{eq:B3}.
Note that the matrix $\Omega$ is the same one
in \eqref{eq:Omega1}.
So, we can use the results in Section \ref{sec:ordering1}.
Also,  by  \eqref{eq:B3},
$B\bfn$ ($\bfn \in N^+_{\rmpr}$) is in 
 the closure of the \emph{second}
quadrant excluding the origin $\bfzero$ in $\bbR^2$.
Thus, a wall $(\frakd, g)_{\bfn}$ is incoming if and only if 
the support $\frakd$ intersects with 
 the closure of the {second}
quadrant excluding the origin $\bfzero$.

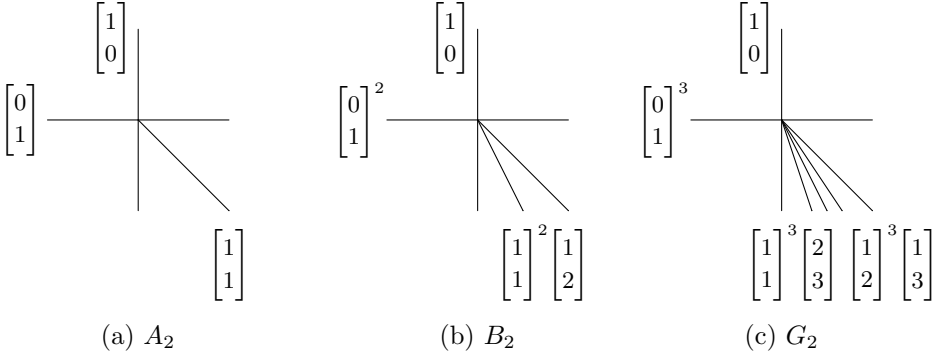
\begin{figure}
\centering
\leavevmode
\begin{xy}
0;/r1.2mm/:,
(0,-24)*{\text{(a) $A_2$}},
(-3, 9)*{\text{\small $\begin{bmatrix}1\\0\end{bmatrix}$}},
(-13, 0)*{\text{\small $\begin{bmatrix}0\\1\end{bmatrix}$}},
(10, -16)*{\text{\small $\begin{bmatrix}1\\1\end{bmatrix}$}},
(0,0)="A",
\ar@{-} "A"+(0,0); "A"+(10,0)
\ar@{-} "A"+(0,0); "A"+(0,10)
\ar@{-} "A"+(0,0); "A"+(-10,0)
\ar@{-} "A"+(0,0); "A"+(0,-10)
\ar@{-} "A"+(0,0); "A"+(10,-10)
\end{xy}
\hskip10pt
\
\hskip15pt
\begin{xy}
0;/r1.2mm/:,
(0,-24)*{\text{(b) $B_2$}},
(-3, 9)*{\text{\small $\begin{bmatrix}1\\0\end{bmatrix}$}},
(-13, 0)*{\text{\small $\begin{bmatrix}0\\1\end{bmatrix}^2$}},
(5, -15.75)*{\text{\small $\begin{bmatrix}1\\1\end{bmatrix}^2$}},
(10, -16)*{\text{\small $\begin{bmatrix}1\\2\end{bmatrix}$}},
(0,0)="A"
\ar@{-} "A"+(0,0); "A"+(10,0)
\ar@{-} "A"+(0,0); "A"+(0,10)
\ar@{-} "A"+(0,0); "A"+(-10,0)
\ar@{-} "A"+(0,0); "A"+(0,-10)
\ar@{-} "A"+(0,0); "A"+(5,-10)
\ar@{-} "A"+(0,0); "A"+(10,-10)
\end{xy}
\hskip15pt
\begin{xy}
0;/r1.2mm/:,
(0,-24)*{\text{(c) $G_2$}},
(-3, 9)*{\text{\small $\begin{bmatrix}1\\0\end{bmatrix}$}},
(-13, 0)*{\text{\small $\begin{bmatrix}0\\1\end{bmatrix}^3$}},
(-1, -15.75)*{\text{\small $\begin{bmatrix}1\\1\end{bmatrix}^3$}},
(4, -16)*{\text{\small $\begin{bmatrix}2\\3\end{bmatrix}$}},
(10, -15.75)*{\text{\small $\begin{bmatrix}1\\2\end{bmatrix}^3$}},
(15, -16)*{\text{\small $\begin{bmatrix}1\\3\end{bmatrix}$}},
(0,0)="A"
\ar@{-} "A"+(0,0); "A"+(10,0)
\ar@{-} "A"+(0,0); "A"+(0,10)
\ar@{-} "A"+(0,0); "A"+(-10,0)
\ar@{-} "A"+(0,0); "A"+(0,-10)
\ar@{-} "A"+(0,0); "A"+(3.33,-10)
\ar@{-} "A"+(0,0); "A"+(5,-10)
\ar@{-} "A"+(0,0); "A"+(6.66,-10)
\ar@{-} "A"+(0,0); "A"+(10,-10)
\end{xy}
\caption{Rank 2 CSDs of finite type.}
\label{3fig:scat2}
\end{figure}
\begin{ex}
[Finite type: $\delta_1\delta_2\leq 3$]
\label{ex:CSDfin1}
We use the results in Example \ref{ex:ordfinite1}.
\ \par
(a) Type $A_2$. Let $(\delta_1,\delta_2)=(1,1)$.
The relation \eqref{eq:order2} is  interpreted as
the consistency relation for
the  scattering diagram  in Figure \ref{3fig:scat2} (a).
It consists of three walls
\begin{align}
\label{eq:wallsA2}
(\bfe_1^{\perp}, \Psi[\bfe_1])_{\bfe_1},
\quad
(\bfe_2^{\perp}, \Psi[\bfe_2])_{\bfe_2},
\quad
(\bbR_{\geq 0} (1,-1), \Psi[(1,1)])_{(1,1)},
\end{align}
where
\begin{align}
\bbR_{\geq 0} (1,-1)\subset (1,1)^{\perp}
\end{align}
by \eqref{3eq:can1}.
This ray is contained in the closure of the fourth quadrant.
Thus, it is outgoing.
Therefore, this is a CSD for $B$.

\par
(b) Type $B_2$. 
Let $(\delta_1,\delta_2)=(1,2)$.
The relation \eqref{eq:order3} is  interpreted as
the consistency relation for
the  scattering diagram in Figure \ref{3fig:scat2} (b),
which consists of four walls
\begin{align}
\label{eq:wallsB2}
\begin{split}
&(\bfe_1^{\perp}, \Psi[\bfe_1])_{\bfe_1},
\quad
(\bfe_2^{\perp}, \Psi[\bfe_2]^2)_{\bfe_2},
\\
&(\bbR_{\geq 0} (1,-2), \Psi[(1,1)]^2)_{(1,1)},
\quad
(\bbR_{\geq 0} (1,-1), \Psi[(1,2)])_{(1,2)}.
\end{split}
\end{align}
This is a CSD for $B$.

\par
(c) Type $G_2$. Let $(\delta_1, \delta_2)=(1,3)$.
The relation  \eqref{eq:order5} is  interpreted as
the consistency relation for
the  scattering diagram in Figure \ref{3fig:scat2} (c),
which consists of six walls
\begin{align}
\begin{split}
&(\bfe_1^{\perp}, \Psi[\bfe_1])_{\bfe_1},
\quad
(\bfe_2^{\perp}, \Psi[\bfe_2]^3)_{\bfe_2},
\\
&(\bbR_{\geq 0} (1,-3), \Psi[(1,1)]^3)_{(1,1)},
\quad
(\bbR_{\geq 0} (1,-2), \Psi[(2,3)])_{(2,3)},
\\
&(\bbR_{\geq 0} (2,-3), \Psi[(1,2)]^3)_{(1,2)},
\quad
(\bbR_{\geq 0} (1,-1), \Psi[(1,3)])_{(1,3)}.
\end{split}
\end{align}
This is a CSD for $B$.
\end{ex}

We see that
the above CSDs are constructed only with dilogarithm elements and the pentagon relation.

\section{$G$-fan and CSD}
\label{sec:Gfan1}

Let us explain the  relation between
a CSD $\frakD(B)$ and
 a cluster pattern/$Y$-pattern
 with
the initial  exchange matrix $B$.

Recall  that the $G$-pattern
$\bfG=\bfG^{t_0}=\{ G_t\}_{ t\in \bbT_n}$ for
the initial exchange matrix $B$
in Definition \ref{defn:Gmat1}
encodes all essential information of the corresponding 
cluster pattern/$Y$-pattern
by Theorem \ref{1thm:synchro1}.
One can geometrically represent each $G$-matrix $G_t$ by
a \emph{$G$-cone}\index{$G$-cone} in $\bbR^n$
defined by
\begin{align}
\label{eq:Gcone1}
\sigma(G_t):=
\sigma
( \bfg_{1;t},
\dots,
\bfg_{n;t}
).
\end{align}
Due to the unimodularity \eqref{eq:uni1},
the dimension of $\sigma(G_t)$ is $n$.
Note that the cone  $\sigma(G_t)$ is invariant under the
action of a permutation $\nu$ on a $G$-matrix $G_t$
in \eqref{2eq:sigmag1}. 
 Each cone generated by a subset of 
$ \bfg_{1;t}$, \dots,
$\bfg_{n;t}$
is called a \emph{face}\index{face (of cone)} of  $\sigma(G_t)$,
including the one $\sigma(\emptyset)=\{\bfzero \}$.
In particular, 
let
\begin{align}
\label{eq:Gcone2}
\sigma_i(G_t):=
\sigma
( \bfg_{1;t},
\dots,
\check\bfg_{i;t}
\dots,
\bfg_{n;t}
),
\end{align}
which is  a face of dimension $n-1$.
Then,
by \eqref{2eq:gmut1},
 if $t$, $t'\in \bbT_n$ are $k$-adjacent, the $G$-cones
$\sigma(G_t)$ and $\sigma(G_{t'})$ intersect each other in their common face
$\sigma_k(G_t)=\sigma_k(G_{t'})$.

\begin{defn}
A set of cones $\calF$ is a \emph{fan}\index{fan}
if the intersection $\sigma\cap \sigma'$ of any pair $\sigma$, $\sigma'\in \calF$
belongs to $\calF$.
\end{defn}

The following property is a consequence
of the sign-coherence in Theorem \ref{thm:sign1}.

\begin{prop}
\cite[Thm.~8.7]{Reading12}
\label{prop:Gfan1}
The set $\calG(B)$ consisting of all $G$-cones and their faces
is a fan in $\bbR^n$.
\end{prop}

We call the above fan the \emph{$G$-fan}\index{$G$-fan} for a skew-symmetrizable matrix $B$
or the corresponding cluster pattern/$Y$-pattern.
The \emph{support}\index{support!of $G$-fan} of the $G$-fan $\calG(B)$ is defined by
\begin{align}
|\calG(B)|:=
\bigcup_{t\in \bbT^n}
\sigma(G_t).
\end{align}
We say that the $G$-fan $\calG(B)$
 is \emph{complete}\index{complete (for $G$-fan)} if $|\calG(B)|=\bbR^n$.

The following fact is known.
\begin{prop}
[{\cite[Thm.~10.6]{Reading12}, \cite[Thm.~4.2]{Nakanishi24}}]
\label{prop:Gfinite1}
The $G$-fan $\calG(B)$ is complete if and only if the
corresponding cluster pattern/$Y$-pattern is of finite type.
\end{prop}

Let $\calG_{n-1}(B)$ be the set 
 of all cones in $\calG(B)$ of dimension $n-1$.
 Let $|\calG_{n-1}(B)|$ be the union of the supports of all cones in $\calG_{n-1}(B)$.

Here is the second fundamental result on CSDs.

\begin{thm}
[{\cite[Construction 1.30]{Gross14}, \cite[Cor.~4.4]{Reading17}}]
\label{thm:CSD2}
Let $\frakD(B)$ be a CSD for $B$  with minimal support.
Let $\frakS(B)\subset \bbR^n$ be
the image of the  $\mathrm{Supp}(\frakD(B))\subset M_{\bbR}$
under the  identification $M_{\bbR}\simeq \bbR^n$ with $f_i\mapsto \bfe_i$.
Then, we have the inclusion
\begin{align}
\label{eq:FanG1}
|\calG_{n-1}(B)| \subset \frakS(B).
\end{align}
Also, the interior $\mathrm{Int}(\sigma(G_t))$ of  each $G$-cone 
dones not intersect $\frakS(B)$.
Moreover, 
for a cluster pattern/$Y$-pattern of finite type,
we have the equality
\begin{align}
\label{eq:FG2}
|\calG_{n-1}(B)| = \frakS(B).
\end{align}
\end{thm}

Therefore,
 all essential information of a cluster pattern/$Y$-pattern
 is contained in the corresponding CSD.
 
\begin{ex}
Let us check the equality \eqref{eq:FG2}
in Examples \ref{ex:CSDA1A1} and \ref{ex:CSDfin1}.
For type $A_1\times A_1$, there are four $G$-matrices
\begin{align}
\begin{pmatrix}
1 & 0  \\
0  & 1
\end{pmatrix},
\quad
\begin{pmatrix}
-1 & 0  \\
0  & 1
\end{pmatrix},
\quad
\begin{pmatrix}
-1 & 0  \\
0  & -1
\end{pmatrix},
\quad
\begin{pmatrix}
1 & 0  \\
0  & -1
\end{pmatrix}.
\end{align}
Comparing it with 
Figure \ref{3fig:scat1},
we see that the equality \eqref{eq:FG2} holds.
For type $A_2$, as we see in Example \ref{ex:typeA23},
there are five $G$-matrices (up to the action of the permutation $\tau_{12}$) 
\begin{align}
\begin{pmatrix}
1 & 0  \\
0  & 1
\end{pmatrix},
\quad
\begin{pmatrix}
-1 & 0  \\
0  & 1
\end{pmatrix},
\quad
\begin{pmatrix}
-1 & 0  \\
0  & -1
\end{pmatrix},
\quad
\begin{pmatrix}
1 & 0  \\
-1  & -1
\end{pmatrix},
\quad
\begin{pmatrix}
1 & 1 \\
-1  & 0
\end{pmatrix}.
\end{align}
Comparing it with 
Figure \ref{3fig:scat2} (a),
we see that the equality \eqref{eq:FG2} holds.
Other cases are similar
and left as exercises for the reader.
\end{ex}

\section{Rank 2 CSDs: infinite type}
Let us   present  examples of rank 2 CSDs of  \emph{infinite type}
using  the results in Section \ref{sec:ordering1}.
It turns out that the infinite type is more interesting because
CSDs contain some additional information
outside the $G$-fan.

\begin{ex}
[Affine type: $\delta_1\delta_2=4$]
\label{ex:CSDaffine1}

\ \par
(a). Type $A_1^{(1)}$. Let $(\delta_1,\delta_2)=(2,2)$.
The relation \eqref{3eq:a115}
is translated into the following consistent scattering diagram:
\begin{gather}
\label{3eq:a111}
(\bfe_1^{\perp}, \Psi[\bfe_1]^2)_{\bfe_1},
\quad
(\bfe_2^{\perp}, \Psi[\bfe_2]^2)_{\bfe_2},
\\
\label{3eq:a112}
(\bbR_{\geq 0} (p,-p-1), \Psi[(p+1,p)]^2)_{(p+1,p)}
\quad
(p\in \bbZ_{>0}),
\\
\label{3eq:a113}
(\bbR_{\geq 0} (p+1,-p), \Psi[(p,p+1)]^2)_{(p,p+1)}
\quad
(p\in \bbZ_{>0}),
\\
\label{3eq:a114}
(\bbR_{\geq 0} (1,-1), \prod_{j=0}^{\infty} \Psi[2^j(1,1)]^{2^{2-j}})_{(1,1)}.
\end{gather}
Equivalently, one can split the wall in \eqref{3eq:a114} 
into infinitely many walls
\begin{align}
\label{3eq:a116}
(\bbR_{\geq 0} (1,-1),\Psi[2^j(1,1)]^{2^{2-j}})_{(1,1)}
\quad
(j=0,\,1,\,\dots).
\end{align}
See Figure \ref{3fig:scat3} (a).
All walls except for the ones in \eqref{3eq:a111}
are outgoing. 
Thus, it is a CSD for $B$.
The equality \eqref{eq:FG2} does not hold anymore.
Indeed, the complement of $|\calG(B)|$ in  $\bbR^2$
is given by the ray $\bbR_{>0}(1,-1)$, which is depicted as
a thick ray in the figure.
As we see  in \eqref{3eq:a116},  there is some rich structure of walls in the ray.

\smallskip

(b). Type $A_2^{(2)}$. Let $(\delta_1,\delta_2)=(1,4)$.
The relation \eqref{eq:a227}
is translated into the following consistent scattering diagram:
\begin{gather}
\label{3eq:a221}
(\bfe_1^{\perp}, \Psi[\bfe_1])_{\bfe_1},
\quad
(\bfe_2^{\perp}, \Psi[\bfe_2]^4)_{\bfe_2},
\\
\label{3eq:a222}
(\bbR_{\geq 0} (p,-2p-1), \Psi[(2p+1,4p)])_{(2p+1,4p)}
\quad
(p\in \bbZ_{>0}),
\\
\label{3eq:a223}
(\bbR_{\geq 0} (2p-1,-4p), \Psi[(p,2p-1)]^4)_{(p,2p-1)}
\quad
(p\in \bbZ_{>0}),
\\
\label{3eq:a224}
(\bbR_{\geq 0} (p,-2p+1), \Psi[(2p-1,4p)])_{(2p-1,4p)}
\quad
(p\in \bbZ_{>0}),
\\
\label{3eq:a225}
(\bbR_{\geq 0} (2p+1,-4p), \Psi[(p,2p+1)]^4)_{(p,2p+1)}
\quad
(p\in \bbZ_{>0}),
\\
\label{3eq:a226}
\begin{split}
&\quad\ (\bbR_{\geq 0} (1,-2), \Psi[(1,2)]^6\prod_{j=1}^{\infty} \Psi[2^j (1,2)]^{2^{2-j}})_{(1,2)}.
\end{split}
\end{gather}
See Figure \ref{3fig:scat3} (b).
Again, it is a CSD  for $B$.
The complement of $|\calG(B)|$ in  $\bbR^2$
is given by the ray $\bbR_{>0}(1,-2)$.

\end{ex}

\begin{figure}
\begin{center}
\begin{tikzpicture}[scale=1.4]
\draw(0,0)--(1,0);
\draw(0,0)--(0,1);
\draw(0,0)--(-1,0);
\draw(0,0)--(0,-1);
\draw(0,0)--(0.5,-1);
\draw(0,0)--(0.66,-1);
\draw(0,0)--(0.75,-1);
\draw(0,0)--(0.8,-1);
\draw [very thick] (0,0)--(1,-1);
\draw(0,0)--(1,-0.5);
\draw(0,0)--(1,-0.66);
\draw(0,0)--(1,-0.75);
\draw(0,0)--(1,-0.8);
%
 \node at (0,-1.3){(a) $A_1^{(1)}$};
   \node at (0.29,-1.6){\phantom{$(b,c)=(5,1)$}};
 \end{tikzpicture}
\hskip30pt
\begin{tikzpicture}[scale=1.4]
\draw(0,0)--(1,0);
\draw(0,0)--(0,1);
\draw(0,0)--(-1,0);
\draw(0,0)--(0,-1);
\draw(0,0)--(0.25,-1);
\draw(0,0)--(0.33,-1);
\draw(0,0)--(0.375,-1);
\draw(0,0)--(0.4,-1);
\draw(0,0)--(0.625,-1);
\draw(0,0)--(0.66,-1);
\draw(0,0)--(0.75,-1);
\draw(0,0)--(1,-1);
\draw [very thick] (0,0)--(0.5,-1);
%
 \node at (0,-1.3){(b) $A_2^{(2)}$};
 \node at (0.29,-1.6){\phantom{$(a,b)=(5,1)$}};
 \end{tikzpicture}
\hskip10pt
\begin{tikzpicture}[scale=1.4]
\filldraw [fill=black!10, draw=white] (0,0) -- (0.276,-1) -- (0.723,-1) -- (0,0);
\draw(0,0)--(1,0);
\draw(0,0)--(0,1);
\draw(0,0)--(-1,0);
\draw(0,0)--(0,-1);
\draw(0,0)--(0.25,-1);
\draw(0,0)--(0.266,-1);
\draw(0,0)--(0.272,-1);
\draw(0,0)--(0.276,-1);
\draw(0,0)--(0.723,-1);
\draw(0,0)--(0.727,-1);
\draw(0,0)--(0.733,-1);
\draw(0,0)--(0.75,-1);
\draw(0,0)--(1,-1);
%
 \node at (0,-1.3){(c) Non-affine infinite type};
 \node at (0.2,-1.6){$(a,b)=(5,1)$};
 \end{tikzpicture}
 \end{center}
 \vskip-15pt
 \caption{Rank 2 CSDs of infinite type.
 The shaded area in (c) is called the Badlands.}
\label{3fig:scat3}
\end{figure}
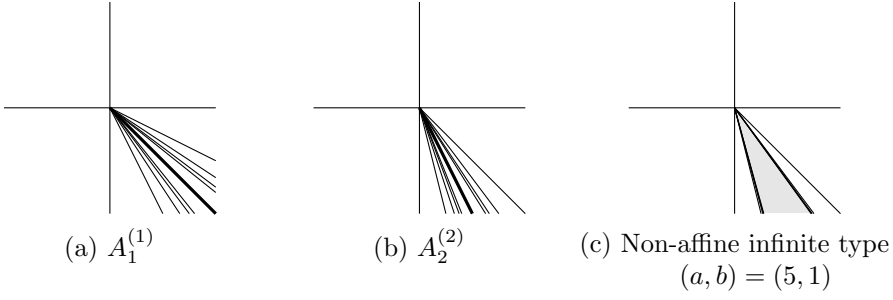

The remaining case is most intriguing,
as already seen in Example \ref{ex:ordnonaffine1},
though the entire structure of a CSD is not  known yet.
\begin{ex}
[Non-affine infinite type: $\delta_1\delta_2>4$]
We generalize and translate the results and facts in Example \ref{ex:ordnonaffine1}
into the CSD language.
Let 
$\sigma(\bfv_+, \bfv_-)$ be the irrational
cone spanned
by two
vectors 
 \begin{align}
 \label{3eq:ds1}
 \bfv_{\pm}
 =
 \begin{pmatrix}
\delta_1\delta_2\pm\sqrt{\delta_1\delta_2(\delta_1\delta_2-4)} \\
- 2\delta_2\\
\end{pmatrix}.
\end{align}
The complement of $|\calG(B)|$ in $\bbR^2$ 
is $\sigma(\bfv_+, \bfv_-)\setminus \{0\}$,
which is 
 informally called the \emph{Badlands}\index{Badlands}.
For example, in the case $(\delta_1,\delta_2)=(1,5)$,
 the  cone $\sigma(\bfv, \bfv')$ is depicted as a shaded region in Figure \ref{3fig:scat3} (c).
  As  in Example \ref{ex:ordnonaffine1},
every rational ray
 in the  cone $\sigma(\bfv, \bfv')$ appears
 as a nontrivial wall of $\frakD(B)$  \cite{Grafnitz23}.  However, their wall elements are not known yet
 except for the limited case.
\end{ex}

\section{Positive realization of CSDs and pentagon relation}
\label{sec:positive1}

To conclude this quick course,
 let us explain that any CSD of any rank can be  
constructed only by the dilogarithm elements and the pentagon relation.

Recall that, for a given skew-symmetric decomposition \eqref{eq:BDO1} of $B$,
a sublattice $N^{\circ}$ of $N$ is defined by \eqref{eq:Bcirc1}.

\begin{defn}
\label{defn:norm1}
For each $\bfn\in N^+$, the \emph{normalization factor}\index{normalization factor} $\delta(\bfn)$ of $\bfn$
is the smallest positive rational number such that
$\delta(\bfn) \bfn \in N^{\circ}$.
\end{defn}

For example,  for $\bfn=\bfe_i$, we have $\delta(\bfe_i)=\delta_i$.
Also, for any $h\in \bbZ_{>0}$ and $\bfn\in N^+$, we have
$\delta(h\bfn)=\delta(\bfn)/h$.

Now we present the third fundamental result on CSDs.
\begin{thm}
[{\cite[Theorem~1.13]{Gross14}}]
\label{3thm:pos1}
Let $\frakD(B)$ be a CSD for $B$.
Then, there is a consistent scattering diagram $\frakD$
that is equivalent to  $\frakD(B)$
such that
the wall element of any wall
of $\frakD$
has the following form
\begin{align}
\label{3eq:gpos1}
g=\Psi[h\bfn]^{s \delta (h\bfn) }
\quad
(\bfn\in N_{\rmpr}^+;\,   s,\, h \in \bbZ_{>0}).
\end{align}
\end{thm}

We call a  CSD with the property \eqref{3eq:gpos1} a \emph{positive realization}\index{positive!realization}
due to the positivity of the exponents $s$.
Such a CSD is always with minimal support due to the positivity of $s$.

Below we present the construction of the positive realization of a  CSD by the pentagon relation.
Let us temporarily concentrate on the rank 2 case.
The following is a more general answer to Problem \ref{prob:order1}
than  Proposition \ref{prop:order1}.

\begin{prop}[{Ordering Lemma
\cite[Prop.~III.5.4]{Nakanishi22a}}]
\label{prop:ordering1}
 \index{Ordering Lemma}
\label{3prop:rank21}
For the rank 2 case,
let 
\begin{align}
\label{3eq:init1}
C^{\mathrm{in}}=
\Psi[h'_j \bfn'_j]^{s'_j  \delta(h'_j \bfn'_j)}
\cdots
\Psi[h'_1\bfn'_1]^{s'_1\delta(h'_1 \bfn'_1)}
\quad
(\bfn'_i\in N_{\rmpr}^+;\, \ s'_i,\, h'_i \in \bbZ_{>0})
\end{align}
be any finite anti-ordered product.
 Then, 
 by applying the  relations in Proposition \ref{3prop:pent1}
 possibly infinitely many times,
  $C^{\mathrm{in}}$
  is transformed into a (possibly infinite) ordered product $C^{\mathrm{out}}$ of 
factors  of the same form
\begin{align}
\label{3eq:gpos2}
\Psi[h_i \bfn_i]^{s_i \delta(h_i \bfn_i) }
\quad
(\bfn_i\in N_{\rmpr}^+;\, \ s_i,\, h_i \in \bbZ_{>0}).
\end{align}
\end{prop}

Note the positivity of exponents $s_i$.
The proof of
Proposition \ref{3prop:rank21} is based on an explicit algorithm (Ordering Algorithm) 
\cite[Algorithm III.5.7]{Nakanishi22a},
which generalizes the procedure in the examples in 
Section \ref{sec:ordering1}.
The positivity is guaranteed because the pentagon relation involves only positive
exponents throughout.

Now we go back to the general rank case.
We need the following notion from \cite{Gross14}.
\begin{defn}[Joint]
Let $\frakD$ be a scattering diagram.
\par
(a). 
For any pair of walls $\bfw_i=(\frakd_i,g_{i})_{\bfn_i}$ ($i=1,2$), the intersection
of their supports $\frakj=\frakd_1\cap \frakd_2$
is called a \emph{joint}\index{joint} of $\frakD$ if $\frakj$
is a cone of dimension $n-2$. 
\par
(b) For each joint $\frakj$, let
\begin{align}
\label{3eq:Njn1}
N_{\frakj}:=
\{ \bfn\in N \mid \langle \bfn, z\rangle=0
\
( z\in \frakj) \},
\end{align}
which is a rank 2 sublattice of $N$.
Then,
a joint $\frakj$ is \emph{parallel}\index{parallel (joint)}
(resp.\ \emph{perpendicular}\index{perpendicular (joint)}) if 
$\{\bfn_1,\bfn_2\}=0$
 (resp. otherwise).
\end{defn}

By the topological reason, the consistency condition \eqref{3eq:pgfi1}
is reduced to the one for loops around joints.

The construction of a positive realization in  Theorem \ref{3thm:pos1} is given as follows. (Here we present  the modified version in \cite{Nakanishi22a} with the help of Proposition \ref{3prop:rank21}.)
\begin{const}
[{\cite[Construction C.1, Appendix C.3]{Gross14},  \cite[Construction III.5.14]{Nakanishi22a}}]
\label{const:csd1}
The construction proceeds inductively on
the degree $\ell$.
\par
\begin{enumerate}
\item
We start with the
scattering diagram $\frakD_1$ consisting of
only incoming walls
$\bfw_i:=(\bfe_i^{\perp}, \Psi[\bfe_i]^{\delta_i})_{\bfe_i}$ ($i=1$, \dots, $n$).

\item
For each perpendicular joint $\frakj:=\bfe_i^{\perp}\cap \bfe_j^{\perp}$ ($i\neq j$)
of $\frakD_1$
with $c=\{\bfe_j,\bfe_i\}>0$, we apply Proposition \ref{3prop:rank21} for
$\Psi[\bfe_i]^{\delta_i}\Psi[\bfe_j]^{\delta_j}$.
In the ordered product, we have the degree 2 element
$\Psi[\bfe_i+\bfe_j]^{c \delta_i\delta_j}$.
Accordingly, we add a new outgoing wall 
\begin{align}
\label{eq:neww1}
(\sigma(\frakj, -p^*(\bfe_i+\bfe_j)),
\Psi[\bfe_i+\bfe_j]^{c \delta_i\delta_j}
)_{\bfe_i+\bfe_j},
\end{align}
where $\sigma(\frakj, -p^*(\bfe_i+\bfe_j))$ is the cone spanned by
$\frakj$ and $-p^*(\bfe_i+\bfe_j)$.
\item
Suppose that the procedure is completed up to degree $\ell$
and we have  the  scattering diagram $\frakD_{\ell}$.
By splitting walls if necessary,
we may assume that any pair of distinct joints of $\frakD_{\ell}$ intersect each other only in their boundaries.
For each perpendicular joint $\frakj$, we set
$N^+_{\frakj}=N_{\frakj}\cap N^+$ and 
$N^+_{\frakj,\rmpr}=N_{\frakj}\cap N^+_{\rmpr}$.
Then, there is a unique $\tilde \bfe_1, \tilde \bfe_2\in N^+_{\frakj,\rmpr}$
such that
\begin{align}
N^+_{\frakj}\subset \bbQ_{\geq 0}  \tilde \bfe_1 + \bbQ_{\geq 0}  \tilde \bfe_2,
\quad
\{ \tilde \bfe_2, \tilde \bfe_1\}>0.
\end{align}
There are walls of $\frakD_{\ell}$ containing $\frakj$ and 
intersecting the ``second quadrant''  with respect to $\tilde \bfe_1$ and $\tilde \bfe_2$.
 We take
 the anti-ordered product of wall elements of
those walls
and
apply Proposition \ref{3prop:rank21}.
Then, for each element $ \Psi[h\bfn]^{s\delta(h\bfn)}$
with $\deg(h\bfn)=\ell+1$
in the ordered product,
we add a new outgoing wall
\begin{align}
(\sigma(\frakj, -p^*(\bfn)),
\Psi[h\bfn]^{s \delta(h\bfn)}
)_{\bfn}
\end{align}
 in the ``fourth quadrant''.
\item
The desired positive realization for $\frakD(B)$
in Theorem \ref{3thm:pos1}
 is obtained as the limit of $\frakD_{\ell}$.
\end{enumerate}
\end{const}

 The consistency around each perpendicular joint
is  satisfied by the construction. On the other hand, the consistency around each  parallel joint
is not obvious; however,  it
is  guaranteed 
thanks to  Theorem \ref{thm:CSD1}.

By construction, the consistency relation around each perpendicular joint
is  reduced to a trivial one (i.e., $g=g$) by applying the pentagon relation
in the reverse way to the ordereing.
On the other hand, the one around each parallel joint
is  reduced to a trivial one by applying the commutative relation
\eqref{eq:Pcom1}.
Thus, as a by-product,
the above construction 
also reveals the fourth fundamental fact
on the structure of a CSD.
\begin{thm}
[{\cite[Thm.~III.5.17]{Nakanishi22a}
}]
\label{3thm:struct1}
Let $\frakD(B)$ be a CSD for $B$.
For any admissible loop $\gamma$ for $\frakD(B)$,
the consistency relation  $\frakp_{\gamma, \frakD(B)}=\rmid$ is reduced to a trivial one
by applying the  relations in Proposition \ref{3prop:pent1}
possibly infinitely many times.
\end{thm}

\notes

Scattering diagrams were introduced by \cite{Kontsevich06} and \cite{Gross07}
in the study of homological mirror symmetry, 
and developed further in \cite{Kontsevich08,Gross09, Carl10, Gross11, Kontsevich13}.
The connection to cluster algebras was established
   by \cite{Gross14},
  where CSDs were introduced.
  With the scattering diagram method,
    several important conjectures on   cluster algebras
          were proved,
          including Theorems \ref{thm:sign1} and \ref{thm:Lpos1}.
 Here, we follow
 the presentation of CSDs
 in the companion monograph \cite{Nakanishi22a},
where the roles of the dilogarithm elements and the pentagon relation are highlighted.
 Proofs of all results in this chapter 
are conveniently found in   \cite{Nakanishi22a} as specified below. (We omit the ones already cited in the text.)
\begin{itemize}
\item
Theorem \ref{thm:bij1}:
See  \cite[Thm.~III.2.20]{Nakanishi22a}.
\item
Theorem \ref{thm:CSD1}:
See  \cite[Thm.~III.3.12]{Nakanishi22a}.
\item
Proposition \ref{prop:Gfan1}:
See  \cite[Thm.~II.2.17]{Nakanishi22a}.
\item
Theorem \ref{thm:CSD2}:
See  \cite[Thm.~II.3.3 (a)]{Nakanishi22a}.
\item
Theorem \ref{3thm:pos1}:
See  \cite[Thm.~III.5.2]{Nakanishi22a}.

\end{itemize}

\chapter{Dilogarithm identities in cluster scattering diagrams}

\label{ch:DICSD1}

We derive the DI associated with a loop in a CSD
by combining the results on CSDs and the classical dynamical method in Part II.
This generalizes the DI associated with a period of a $Y$-pattern
to the ones with infinite sum.
Thanks to the structure theorem for CSDs, such a DI
is always reduced to the trivial one by applying the pentagon identity
possibly infinitely many times.
Thus, it is \emph{infinitely reducible}.
Moreover, this method provides yet another proof of the  DIs in Theorem \ref{thm:DI1}.

\section{$y$-variables in CSD}
\label{sec:yvariables1}

To formulate DIs in a CSD, we first need to extend the notion of $y$-variables
in a $Y$-pattern to the corresponding CSD.
From now on,  we assume that a CSD $\frakD(B)$ is a positive realization in
Theorem \ref{3thm:pos1}. In particular, $\frakD(B)$ is with minimal support.

We continue to identify 
$M_{\bbR}\simeq \bbR^n$ with $f_i=e^*_i/\delta_i \mapsto \bfe_i$
as in Section \ref{sec:Gfan1}.
Note that, under the identification,
the lattice $M^{\circ} \subset M_{\bbR}$ corresponds to $\bbZ^n\subset  \bbR^n$.
Through this identification, we regard a $c$-vector $\bfc_{i;t}$ and a $g$-vector $\bfg_{i;t}$  as
elements of $N$ and $M^{\circ}$, respectively.
Then,
the canonical pairing of them is given by
\begin{align}
\label{eq:can1}
\langle \bfc_{i;t}, \bfg_{j;t}\rangle
=( \bfc_{i;t})^T \Delta^{-1} \bfg_{j;t}
\end{align}
by  \eqref{3eq:can1}.
Let
$\bfc_{i;t}^+=\varepsilon_{i;t}\bfc_{i;t}$ be the $c^+$-vector in  \eqref{eq:c+1}.
Also, let 
$\sigma_i(G_t)$
be the cone of dimension $n-1$ in \eqref{eq:Gcone2}.

\begin{lem}[{\cite[\S II.5.2]{Nakanishi22a}}]
\label{lem:Gcone1}
The following facts hold.
\par
(a).  $\bfc_{i;t}^+$ is normal to 
$\sigma_i(G_t)$
with respect to the canonical pairing \eqref{eq:can1}.
\par
(b). $\bfc_{i;t}^+$ is primitive.
\par
(c). 
$\delta(\bfc_{i;t}^+)=\delta_i$.
\end{lem}
\begin{proof}
(a). 
By the duality \eqref{eq:dual3},
the canonical pairing \eqref{eq:can1}
is written as
\begin{align}
\label{eq:dual4}
\langle \delta_i \bfc_{i;t}, \bfg_{j;t}\rangle = \delta_{ij}.
\end{align}
\par
(b).
Since the $C$-matrix $C_t$ is unimodular by \eqref{eq:uni1},
 $\bfc_{1;t}$, \dots,  $\bfc_{n;t}$ are a basis of $N$.
 Thus, $\bfc_{i;t}$ is primitive.
 \par
(c).
Since the $G$-matrix $G_t$ is unimodular by \eqref{eq:uni1},
 $\bfg_{1;t}$, \dots,  $\bfg_{n;t}$ are a basis of $M^{\circ}$.
Then,
 by the duality \eqref{eq:dual4} again, we see that $\delta_i \bfc_{i;t}^+\in N^{\circ}$. 
 Moreover, it also tells that $\delta_i$ is the minimal one
such that $\delta_i \bfc_{i;t}^+\in N^{\circ}$.
Therefore, $\delta(\bfc_{i;t}^+)=\delta_i$.
\end{proof}

Let $\sigma_i(G_t)$ be the cone in \eqref{eq:Gcone2}.
Recall that it is contained in $\mathrm{Supp}(\frakD(B))$
by 
Theorem \ref{thm:CSD2}.
This means that,
if $z\in\sigma_i(G_t)$,
then there is a wall $\bfw=(\frakd, g)_{\bfn}$ of $\frakD(B)$
such that $z\in \frakd$.

\begin{lem}[{\cite[Prop.~III.6.24]{Nakanishi22a}}]
\label{lem:GC1}
Let $z\in M_{\bbR}$ be a point in $\sigma_i(G_t)$
with $z\not\in \mathrm{Sing}(\frakD(B))$.
Suppose that $\bfw=(\frakd, g)_{\bfn}$ is a wall of $\frakD(B)$
such that $z\in \frakd$.
Then, we have
\begin{align}
\bfn=\bfc_{i;t}^+,
\quad
g=\Psi[\bfc_{i;t}^+]^{\delta_i}.
\end{align}
Moreover, there is no other wall $\bfw'=(\frakd', g')_{\bfn}$ of $\frakD(B)$
such that $z\in  \frakd' $.
\end{lem}

We omit the proof.
See \cite[\S III.6.6]{Nakanishi22a} for details.

In Definition \ref{3defn:adm1} of an admissible curve $\gamma$,
we weaken Condition (1) so that
only the ending point $\gamma(1)$ is prohibited from being in $\mathrm{Supp}(\frakD)$.
Namely, the starting point $\gamma(0)$ is now allowed to be in $\mathrm{Supp}(\frakD)$
(but not allowed to be in $\mathrm{Sing(\frakD)}$ by Condition (3)).
Such a curve $\gamma:[0,1]\rightarrow M_{\bbR}$ is said to be
\emph{weakly admissible}\index{weakly admissible (curve)} for a scattering diagram $\frakD$.
The definition of the path-ordered product $\frakp_{\gamma,\frakD}$ in 
Definition \ref{3defn:pop1} is applied to a weakly admissible curve $\gamma$.
In particular,  the contribution of the wall at the starting point $\gamma(0)$ is \emph{not} counted in $\frakp_{\gamma,\frakD}$.

Let us extend the notion of $y$-variables
and  $c^+$-vectors in  \eqref{eq:c+1}
for a CSD.
\begin{defn}[$y$-variable/$c^+$-vector for CSD]
\label{defn:yvar1}
Let $\bfw=(\frakd, g)_{\bfn}$ be any wall of $\frakD(B)$
with $g=\Psi[h\bfn]^{s \delta(h\bfn)}$ as in  \eqref{3eq:gpos1}.
Let $z\in \frakd$
and $z\not\in \mathrm{Sing}(\frakD(B))$.
Let $\calC^+$ be the positive orthant in \eqref{eq:C+1}.
Let $\gamma_{z}$ be any weakly admissible curve in $\frakD(B)$
 from $z$ to any point in $\mathrm{Int}(\calC^+)$.
Then, we define a \emph{$y$-variable $y_{z}[h\bfn]$ at $z$ with  the $c^+$-vector}\index{$y$-variable!in CSD}
\index{$c^+$-vector!in CSD} $h\bfn$ by
\begin{align}
\label{eq:y1}
y_{z}[h\bfn]:=\frakp_{\gamma_{z},\frakD(B)}(y^{h\bfn}) \in \bbQ[[\bfy]],
\end{align}
where the path-ordered product $\frakp_{\gamma_{z},\frakD(B)}\in G$ acts on  $ y^{h\bfn}$
under the $y$-represent\-ation $\rho_y$ in Proposition \ref{prop:gaction2}.
(Even though $y_{z}[h\bfn]$ only depends on $z$  as we see below, to represent its $c^+$-vector $h\bfn$ explicitly is useful in our application.)
\end{defn}

\begin{rem}
Let $\gamma_{z}$ and $\gamma'_{z}$ be  weakly admissible curves.
Then, the equality $\frakp_{\gamma_{z},\frakD(B)}=\frakp_{\gamma'_{z},\frakD(B)}$
does \emph{not} hold in general because the contribution of the wall element
$\Psi[h\bfn]^{s \delta(h\bfn)}$ at $z$ is not counted.
Nevertheless, the RHS of  \eqref{eq:y1} does not depend on the choice of $\gamma_{z}$
because $\Psi[h\bfn]^{s \delta(h\bfn)}$ acts on $y^{h\bfn}$ trivially.
\end{rem}

The following fact justifies the above definition.
\begin{prop}
[{\cite[Remark 3.2 (a)]{Nakanishi21d}}]
\label{prop:yCSD1}
In  \eqref{eq:y1},
 if the
point $z$
  lies in $\sigma_i(G_t)$,
  then we have
  \begin{align}
  \label{eq:yy3}
  y_{z}[\bfc^+_{i;t}]= y_{i;t}^{\varepsilon_{i;t}},
  \end{align}
where the RHS is regarded as an element of $\bbQ[[\bfy]]$ by the mutations.
\end{prop}

\begin{proof}
For clarity, the proof is separated into three steps.

(i). 
Let $t_0$, $t_1$, \dots, $t_P=t$ be a  sequence in $\bbT_n$
such that $t_0$ is the initial vertex of $\bbT_n$ and $t_s$ and $t_{s+1}$ are adjacent.
  For each $s=0$, \dots, $P-1$, 
let $k_s\in \{1,\, \dots,\,  n\}$ be the one such that 
  $t_s$ and $t_{s+1}$ are $k_s$-adjacent.
We take a weakly admissible curve $\gamma_z$ in \eqref{eq:y1}
so that, starting form $z\in \sigma_i(G_{t_P})$,
 it passes through the $G$-cones
 $\sigma(G_{t_P})$,  $\sigma(G_{t_{P-1}})$, \dots, $\sigma(G_{t_0})=\calC^+$ in this order.
(To be precise, if $k_{P-1}=i$,
 the curve $\gamma_z$ may stay   in  $\sigma(G_{t_P})$ only at the starting point $\gamma_z(0)$
 depending on the choice of $\gamma_z$;
 otherwise, it stays in $\sigma(G_{t_P})$ for a while. 
See Figure \ref{fig:gamma1}.)
Let us take into account Lemma \ref{lem:GC1}.
Also,  if $k_{P-1}=i$, then  $\bfc^+_{k_{P-1};t_{P-1}}=\bfc^+_{i;t}$.
Thus, $\Psi[\bfc^+_{k_{P-1};t_{P-1}}]$ acts on $y^{\bfc^+_{i;t}}$ trivially.
  Then, in any situation, we have
  \begin{align}
 \label{eq:popg1}
   y_{z}[\bfc^+_{i;t}]
   =
   \frakp_{\gamma_{z},\frakD(B)}(y^{\bfc^+_{i;t}})
  =\Psi[\bfc_{k_0;t_0}^+]^{\epsilon_0\delta_{k_0}}\cdots \Psi[\bfc_{k_{P-1};t_{P-1}}^+]^{\epsilon_{P-1}\delta_{k_{P-1}}}
  (y^{\bfc^+_{i;t}}),
\end{align}
where $\epsilon_s$ is the intersection sign \eqref{3eq:factor1}
of  $\gamma_{z}$ at the wall $\sigma_{k_s}(G_s)=\sigma_{k_s}(G_{s+1})$.

 \begin{figure}
 \begin{center}
\begin{tikzpicture}[scale=1.5]
\coordinate (A) at (-0.1,0);
\coordinate (B) at (0,0.7);
\coordinate (C) at (-0.6,0.5);
\coordinate (D) at (-0.5,-0.3);
\coordinate (E) at (0.2,-0.6);
\coordinate (F) at (0.5,0);
\coordinate (G) at (1,-0.5);
\coordinate (H) at (1.2,0.3);
\coordinate (I) at (0.6,0.6);
\coordinate (J) at (0.7,1.1);
\coordinate (K) at (-0.2,1.2);
\filldraw [fill=black!10, draw=white] (E) -- (F) -- (G) -- (E);
\draw (A) --(B);
\draw (A) --(C);
\draw (B) --(C);
\draw (E) --(F);
\draw (E) --(G);
\draw (F) --(G);
\draw (F) --(H);
\draw (G) --(H);
\draw (F) --(I);
\draw (H) --(I);
\draw (H) --(J);
\draw (I) --(J);
\draw (I) --(B);
\draw (J) --(B);
\draw (B) --(K);
\draw (J) --(K);
\draw (C) --(K);
\draw (1.1,-0.2) node [right] {$\sigma(G_{t(P-1)})$};
\draw (0.2,-0.8) node [right] {$\sigma(G_{t(P)})$};
\draw (-0.65,0.1) node [right] {$\calC^+$};
\draw (0.6,-0.05) node [right] {$\gamma$};
\draw (0.6,-0.4) node [right] {$z$};
\draw (0.3,-1.2) node  {(a)  $k_{P-1}=i$};
\fill (0.76,-0.25) circle (1.2pt);
\fill  (-0.39,0.4) circle (1.2pt);
\draw (-0.39,0.4)arc[radius=0.7, start angle=170, end angle=-45];
\draw (0.9, -0.08)--(0.8,-0.15);
\draw (0.9, -0.08)--(0.88,-0.2);
\end{tikzpicture}
\hskip20pt
\begin{tikzpicture}[scale=1.5]
\coordinate (A) at (-0.1,0);
\coordinate (B) at (0,0.7);
\coordinate (C) at (-0.6,0.5);
\coordinate (D) at (-0.5,-0.3);
\coordinate (E) at (0.2,-0.6);
\coordinate (F) at (0.5,0);
\coordinate (G) at (1,-0.5);
\coordinate (H) at (1.2,0.3);
\coordinate (I) at (0.6,0.6);
\coordinate (J) at (0.7,1.1);
\coordinate (K) at (-0.2,1.2);
\filldraw [fill=black!10, draw=white] (E) -- (F) -- (G) -- (E);
\draw (A) --(B);
\draw (A) --(C);
\draw (B) --(C);
\draw (E) --(F);
\draw (E) --(G);
\draw (F) --(G);
\draw (F) --(H);
\draw (G) --(H);
\draw (F) --(I);
\draw (H) --(I);
\draw (H) --(J);
\draw (I) --(J);
\draw (I) --(B);
\draw (J) --(B);
\draw (B) --(K);
\draw (J) --(K);
\draw (C) --(K);
\draw (1.1,-0.2) node [right] {$\sigma(G_{t(P-1)})$};
\draw (0.2,-0.8) node [right] {$\sigma(G_{t(P)})$};
\draw (-0.65,0.1) node [right] {$\calC^+$};
\draw (0.6,-0.05) node [right] {$\gamma$};
\draw (0,-0.4) node [right] {$z$};
\draw (0.3,-1.2) node  {(b)  $k_{P-1}\neq i$};
\fill (0.3,-0.4) circle (1.2pt);
\fill  (-0.39,0.4) circle (1.2pt);
\draw (-0.39,0.4)arc[radius=0.7, start angle=170, end angle=-90];
\draw (0.9, -0.08)--(0.8,-0.15);
\draw (0.9, -0.08)--(0.88,-0.2);
\end{tikzpicture}
\end{center}
\vskip-15pt
\caption{Admissible curve $\gamma_z$.
The shaded region represents $\sigma(G_{t(P)})$.
}
\label{fig:gamma1}
\end{figure}
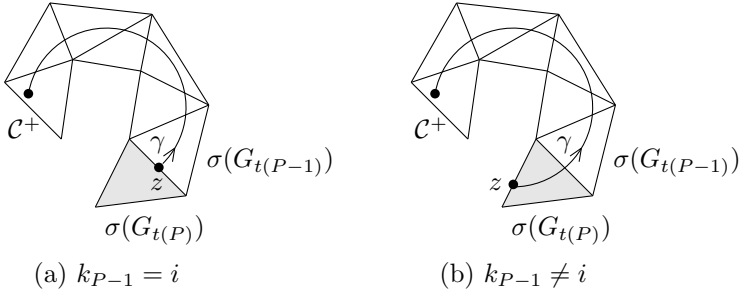

(ii). 
Consider the the mutation sequence \eqref{eq:mseq5} corresponding to the above sequence $k_0$, \dots, $k_{P-1}$.
Let us find the relation between the intersection sign $\epsilon_s$ and the tropical sign $\varepsilon_{k_s;t_s}$.
The duality \eqref{eq:dual4}
 implies that,
when the curve $\gamma_{z}$ crosses 
the wall $\sigma_{k_s}(G_s)=\sigma_{k_s}(G_{s+1})$
from $\sigma(G_{s+1})$ to $\sigma(G_{s})$,
its derivative $\gamma'_z$ satisfies
\begin{align}
\langle\bfc_{k_s;t_s},  \gamma'_z \rangle>0.
\end{align}
Since the  vector $\bfn_i$ in \eqref{3eq:factor1}
is given by $\bfc_{k_s;t_s}^+:=\varepsilon_{k_s;t_s} \bfc_{k_s;t_s}$,
the sign of 
$\langle \bfn_i,  \gamma'_z \rangle$
coincides with the tropical sign $\varepsilon_{k_s;t_s}$.
Thus, by \eqref{3eq:factor1}, we conclude that
\begin{align}
\epsilon_s=-\varepsilon_{k_s;t_s}.
\end{align}
Putting it into \eqref{eq:popg1}, we have
 \begin{align}
 \label{eq:popg2}
   y_{z}[\bfc^+_{i;t}]
  =\Psi[\bfc_{k_0;t_0}^+]^{-\varepsilon_{k_0;t_0}\delta_{k_0}}\cdots
   \Psi[\bfc_{k_{P-1};t_{P-1}}^+]^{-\varepsilon_{k_{P-1};t_{P-1}}\delta_{k_{P-1}}}  (y^{\bfc^+_{i;t}}).
\end{align}

(iii). 
Let us identify $y_{i;t_s}$ with $y_i(s)$ in the sequence of mutations in   \eqref{eq:mseq5}
and switch to the notation therein.
In the formalism of \eqref{eq:mu1},
the equality \eqref{eq:popg2}  is written as
 \begin{align*}
 y_z[\bfc^+_{i}(P)]
   \overset {\eqref{eq:dm1}}
& {=}
( \frakq(0)\circ \cdots
   \circ \frakq(P-1))
   (y^{\bfc^+_{i}(P)})
  \\
  \overset {\eqref{eq:q01}}
 & {=} \frakq(P-1;0)(y^{\bfc^+_{i}(P)})
 \\
 \overset {\eqref{eq:q2}}
& {=} (\frakq(P-1;0)\circ\tau(P-1;0))(y_i(P)^{\varepsilon_{i}(P)})
\quad
\\
\overset {\eqref{eq:decom1}}
& {=}\mu(P-1;0)(y_i(P)^{\varepsilon_{i}(P)})
).
 \end{align*}
This is the desired result.
\end{proof}

The above proof  also tells
that a sequence of   mutations of $y$-variables in a $Y$-pattern
is expressed by a path-ordered product in the corresponding CSD.
By the same token,
the periodicity condition of   a $Y$-pattern
is expressed by the consistency relation along a loop in the CSD.

\begin{ex}
\label{ex:periodloop1}  
Suppose that the sequence of mutations of $Y$-seeds in \eqref{eq:mseq2} is $\nu$-periodic; namely, $\Upsilon(P)=\nu \Upsilon (0)$.
Here we do not assume $\Upsilon(0)$ is the initial $Y$-seed $\Upsilon$.
By  Theorem \ref{1thm:synchro1}, we have $G(P)=\nu G(0)$;
hence, $\sigma(G(P))=\sigma(G(0))$.
Correspondingly, we consider an admissible loop $\gamma$ with the base point $w=\gamma(0)=\gamma(1)$ in $\frakD(B)$ such that
\begin{itemize}
   \item
   it is entirely contained in the support of  the $G$-fan $\calG(B)$,
   \item
 the base point $w$ is in the internal of $\sigma(G(0))=\sigma(G(P)) $,
   \item
   it crosses the $G$-cones $\sigma(G(P-1))$, \dots, $\sigma(G(0))=\sigma(G(P))$
   in this order.
\end{itemize}
See Figure \ref{fig:sconst1}.
We have the consistency relation along the loop $\gamma$
\begin{align}
\label{eq:gamma1}
   \frakp_{\gamma,\frakD(B)}=\rmid.
\end{align}
This relation coincides with the relation \eqref{eq:DIE2}.
\end{ex}

 \begin{figure}
 \begin{center}
\begin{tikzpicture}[scale=1.5]
\coordinate (A) at (0,0);
\coordinate (B) at (0,0.7);
\coordinate (C) at (-0.6,0.5);
\coordinate (D) at (-0.5,-0.3);
\coordinate (E) at (0.2,-0.6);
\coordinate (F) at (0.5,0);
\coordinate (G) at (1,-0.5);
\coordinate (H) at (1.2,0.3);
\coordinate (I) at (0.6,0.6);
\coordinate (J) at (0.7,1.1);
\coordinate (K) at (-0.2,1.2);
\coordinate (L) at (1.7,-0.5);
\coordinate (M) at (2.1,0.2);
\coordinate (N) at (2.5,-0.5);
\coordinate (O) at (2.9,0.2);
\draw (A) --(B);
\draw (A) --(C);
\draw (B) --(C);
\draw (C) --(D);
\draw (A) --(D);
\draw (A) --(E);
\draw (D) --(E);
\draw (A) --(F);
\draw (E) --(F);
\draw (E) --(G);
\draw (F) --(G);
\draw (F) --(H);
\draw (G) --(H);
\draw (F) --(I);
\draw (H) --(I);
\draw (H) --(J);
\draw (I) --(J);
\draw (I) --(B);
\draw (J) --(B);
\draw (B) --(K);
\draw (J) --(K);
\draw (C) --(K);
 \fill[black!10] (F)--(H)--(G)--cycle;
\draw (0.55,-0.05) node [right] {$w$};
\draw (0.05,-0.2) node [right] {$\gamma$};
\fill (0.89,-0.1) circle (1.2pt);
\draw (0.3,0.28)circle[radius=0.7];
\draw (0.48, -0.4)--(0.6,-0.3);
\draw (0.48, -0.4)--(0.63,-0.4);
\end{tikzpicture}
\end{center}
\vskip-10pt
\caption{Admissible loop corresponding to a period of a $Y$-pattern.
The shaded region represents $\sigma(G(0))=\sigma(G(P))$.
}
\label{fig:sconst1}
\end{figure}
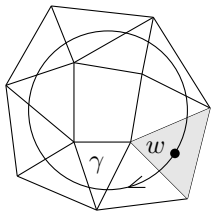

\section{DI associated with a loop in a CSD}
\label{sec:DICSD}
Now we extend our leitmotif, Theorem \ref{thm:DI1}, for an arbitrary loop in a CSD.
Since such a DI involves an infinite sum in general,
it is natural to regard it as an identity for formal power series in the initial variables $\bfy$.
However, as an additional issue,
the modified Rogers dilogarithm
$\tilde L(x)$ in \eqref{eq:L5} is not analytic at $x=0$. Instead, it has the following  
 expansion
around $x=0$ with $\log$ factor,
\begin{align}
\label{eq:Lexp1}
\tilde L(x)
&=
 \sum_{j=1}^{\infty} \frac{(-1)^{j+1} }{j^2}x^j 
-\frac{1}{2}\log x \sum_{j=1}^{\infty} \frac{(-1)^{j+1} }{j}x^j .
\end{align}
Accordingly,
for  variables $\bfy=(y_1, \dots, y_n)$,
we consider the following formal sum
\begin{align}
f(\bfy) +\sum_{i=1}^n (\log y_i) g_i(\bfy)
\quad
(f(\bfy), g_i(\bfy)\in \bbQ[[\bfy]]),
\end{align}
and call it a \emph{formal Puiseux series in  {$\bfy$} with $\log$ factor}\index{formal Puiseux series with $\log$ factor}.
For each integer $\ell\geq 1$,
let $\calF^{> \ell}$ be the set of all formal power series  $f(\bfy)\in  \bbQ[[\bfy]]$
such that  all coefficients of  $f(\bfy)$  vanish up to the total degree $\ell$.
Let
\begin{align}
\calF^{> \ell}_{\log}:=\calF^{> \ell} + \sum_{i=1}^n (\log y_i) \calF^{> \ell}.
\end{align}

Let $\gamma$ be any admissible loop in $\frakD(B)$
with the base point $w$.
Let us fix  $\ell\geq 1$.
Let $\frakD_{\ell}(B)$ the {reduction} of $\frakD(B)$ at 
degree $\ell$.
Suppose that $\gamma$ intersects walls  of $\frakD_{\ell}(B)$
at points $z_0$, \dots, $z_{P-1}$ in this order.
There might be multiple walls with a common normal vector intersected by $\gamma$ at a time.
We distinguish them by allowing the multiplicity $z_a=z_{a+1}=\cdots =z_{a+k}$.
Then, the walls crossed at $z_0$, \dots, $z_{P-1}$ are parametrized as $\bfw_0$, \dots, $\bfw_{P-1}$.
By our assumption,
each wall $\bfw_a$ has the form
 \begin{align}
 \bfw_a = (\frakd_a, \Psi[h_a \bfn_a]^{s_a \delta (h_a \bfn_a)})_{\bfn_a}.
 \end{align}
By the consistency of $\frakD(B)$,
we have $\frakp_{\gamma, \frakD(B)}=\rmid$.
Thus, we have
\begin{align}
\label{eq:gseq1}
\begin{split}
\frakp_{\gamma, \frakD_{\ell}(B)} & =
\Psi[h_{P-1} \bfn_{P-1}]^{\epsilon_ P s_{P-1} \delta (h_{P-1} \bfn_{P-1})}\cdots
\Psi[h_0 \bfn_0]^{\epsilon_0 s_0 \delta (h_0 \bfn_0)}
\equiv \rmid
\\
&
\hskip200pt
\mod G^{> \ell},
\end{split}
\end{align}
where $\epsilon_a$ is the intersection sign.
Let us turn it into an identity for the modified Rogers dilogarithm $\tilde{L}$
that is parallel with the one in Theorem \ref{thm:DI1}.

\begin{thm}
[{\cite{Nakanishi21d}}]
\label{thm:DI0}
In the situation  \eqref{eq:gseq1},
the following identity holds
as
a formal Puiseux series in  {$\bfy$} with $\log$ factor:
\begin{align}
\label{eq:sDI1}
\sum_{a=0}^{P-1}
\epsilon_a s_a
\delta(h_a \bfn_a)
\tilde L( y_{z_a}[h_a\bfn_a])
\equiv 0
\mod 
\calF^{> \ell}_{\log}.
\end{align}
\end{thm}
\begin{proof}
We give a proof based on the classical mechanical method
in Chapter \ref{ch:classical1} (especially, Proposition \ref{prop:dL1}).
Let us consider an admissible curve $\gamma_w$ from the base point $w$ of $\gamma$
to a point in $\mathrm{Int}({\calC}^+)$.
Then, we concatenate $\gamma_w^{-1}$, $\gamma$, and $\gamma_w$ in this order
to have an admissible loop $\gamma'$ with the base point in $\mathrm{Int}({\calC}^+)$.
Then, the contribution to the LHS of \eqref{eq:sDI1} from $\gamma_w^{-1}$ and the one from $\gamma_w$ cancel because the intersection signs therein are opposite.
Therefore,  we  assume that the base point $w$ of $\gamma$
is in $\mathrm{Int}({\calC}^+)$ without loss of generality.
We separate the proof into two cases.

(a). Assume that $B$ is nonsingular.
We use a parallel notation in Section \ref{sec:initial1}.
We consider the Hamiltonian system on the small phase space $\tilde M_0$
for the time span $[0,P]$.
The Hamiltonian for the time span $[a,a+1]$ ($a=0$, \dots,  $P-1$) 
is given by
\begin{align}
\scrH_a=\epsilon_a s_a
\delta(h_a \bfn_a) \mathrm{Li}_2(- y^{h_a \bfn_a}).
\end{align}
Accordingly, 
for a trajectory $\alpha=(\overline\bfu(t), \overline\bfp(t))_{t\in [0,P]}$ 
in the small phase space $\tilde M_0$,
we consider the quantity
\begin{gather}
\label{eq:Salpha1}
S[\alpha]=\sum_{a=0}^{P-1} \scrL_a,
\quad
\scrL_a=\epsilon_a s_a
\delta(h_a \bfn_a) \tilde{L}(\overline y(a)^{h_a \bfn_a}),
\\
\overline y_i(a)=\exp\left(\sqrt{2}\d^{-1}_i \overline p_i(a) \right)=
\exp\biggl(\sqrt{2} \sum_{j=1}^n b_{ji} \overline u_j(a)\biggr),
\end{gather}
which is the counterpart of the ``action integral'' in \eqref{eq:DI5}.
Recall that
the time-one flow of the Hamiltonian $\scrH_{\bfn,a}=a \mathrm{Li}_2(-y^{\bfn})$
coincides with the action of $\Psi[\bfn]^{-a}$ on the function space as explained in Section \ref{sec:dilogelements1}.
For the $i$th coordinate function  ${y}_i$ at time 0,
let ${y}'_i$ be the one at time $P$
after the above time development.
Then, it is described as a function of $\bfy$ as
\begin{align}
\label{eq:ypy1}
y'_i&=\frakp_{\gamma^{-1}, \frakD_{\ell}(B)}(y_i),
\\
\frakp_{\gamma^{-1}, \frakD_{\ell}(B)}&=
\Psi[h_0 \bfn_0]^{-\epsilon_0 s_0 \delta (h_0 \bfn_0)}
\cdots
\Psi[h_{P-1} \bfn_{P-1}]^{-\epsilon_{P-1} s_{P-1} \delta (h_{P-1} \bfn_{P-1})}.
\end{align}
Note that the dilogarithm elements act in the opposite order 
along the time development on
the coordinate functions.

For simplicity, suppose that the equality
 \eqref{eq:gseq1} is the exact one
  \emph{without} modulo $G^{> \ell}$.
  So, by \eqref{eq:ypy1}, 
 we have
\begin{align}
\label{eq:ypy2}
y'_i=y_i.
\end{align}
Then, we repeat the argument in Section \ref{sec:initial1}.
Namely, by \eqref{eq:ou1}--\eqref{eq:oy1},
we have $u'_i = u_i$ and $p'_i=p_i$.
It follows from Proposition \ref{prop:dL1} that
\begin{align}
\delta S[\alpha]=
\sum_{a=0}^{P-1}
\delta  \calL_{\bfc,a}(\alpha_a)=
\sum_{i=1}^n \overline p'_i \delta \overline u'_i
-
\sum_{i=1}^n \overline p_i \delta \overline u_i=0
\end{align}
under any infinitesimal deformation $\alpha+\delta \alpha$ in $\tilde M_0$.
 This implies that $S[\alpha]$ is constant.
 Moreover,
  the constant term is evaluated as zero in the limit $\overline \bfy
  =\overline \bfy(0) \rightarrow \bfzero$
  because, in \eqref{eq:Salpha1},
  $\overline y_i(a) \rightarrow 0$ and $h_a\bfn_a$ is positive.
By replacing $\overline y(a)^{h_a \bfn_a}$ with $y_{z_a}[h_a\bfn_a]$,
    $S[\alpha]$ can be viewed as a 
a {formal Puiseux series in  {$\bfy=\bfy(0)$} with $\log$ factor}.
Thus, we obtain the desired formula \eqref{eq:sDI1}.

Now we consider the general case in \eqref{eq:gseq1}.
The equality \eqref{eq:ypy2} is replaced with
\begin{align}
\label{eq:syy1}
y'_{i}=
y_{i} (1+ h_i(\bfy))
\end{align}
for some $h_i(\bfy)\in \calF^{>\ell}$ due to the action of elements in $G^{> \ell}$.
Then, again
by Proposition \ref{prop:dL1},
under any  infinitesimal variation $\alpha+\delta \alpha$,
we have
\begin{gather}
\delta S[\alpha]
=\sum_{i=1}^n ( \overline p'_i \delta \overline u'_i -  \overline p_i \delta  \overline u_i),
\\
\label{eq:pp2}
\d^{-1}_i \overline p_i=\sum_{j=1}^n b_{ji} \overline u_j=\frac{1}{\sqrt{2}} \log  \overline y_i,
\quad
\d^{-1}_i\overline p'_i=\sum_{j=1}^n b_{ji} \overline u'_j=\frac{1}{\sqrt{2}} \log \overline y'_i,
\end{gather}
where we set $\overline p_i=\overline p_i(0)$ and $\overline p'_i=\overline p_i(P)$, etc.
From now on, we view $\overline p_i$, $\overline p'_i$, $\overline u_i$, $\overline u'_i$
as formal Puiseux  series in  {$\bfy=\bfy(0)$} with $\log$ factor,
and we omit bars.
By \eqref{eq:syy1} and \eqref{eq:pp2}, we can write them as
\begin{align}
p'_i-p_i=f_i
\quad
u'_i-u_i=g_i
\quad
(f_i,\, g_i
\in \calF^{>\ell}).
\end{align}
Thus,
we have
\begin{align*}
\begin{split}
\quad\ \delta S[\alpha]
&=\sum_{i=1}^n \biggl((p_i + f_i) \biggl(\delta u_i + \sum_{j=1}^n \frac{\partial g_i }{\partial y_j} 
\frac{\partial y_j }{\partial u_i}\delta u_i
\biggr) -  p_i \delta u_i\biggr)
\\
&=\sum_{i=1}^n ((\log y_i) \tilde f_i + \tilde g_i)\delta u_i
\quad
(\tilde f_i,\, \tilde g_i\in \calF^{>\ell})
\\
&=\sum_{i=1}^n ((\log y_i)\hat f_i + \hat g_i) \delta y_i
\quad
(\hat f_i,\, \hat g_i\in \calF^{>\ell-1}).
\end{split}
\end{align*}
It follows that
\begin{align}
\frac{\partial S[\alpha]}{\partial y_i }  \in \calF_{\log}^{>\ell-1}.
\end{align}
Therefore, we have
\begin{align}
 S[\alpha] \in \calF_{\log}^{>\ell},
 \end{align}
 where the constant term is evaluated as $0$ as before.
 This is the desired formula \eqref{eq:sDI1}.

 (b). When $B$ is singular,
 as  discussed in Section \ref{sec:alternative1}, the problem is that
those $\bfy$ represented in $\eqref{eq:ty2}$
are not algebraically independent on the small phase space $\tilde M_0$.
So,
 let us take
 the principal extension $\overline{B}$ in \eqref{eq:pext1}, which is nonsingular.
 Let $\overline \Omega$ be the one in \eqref{eq:pextO1}.
 Meanwhile, let $G=G_{\Omega}$ be unchanged,
 so that the CSD $\frakD(B)$ is also unchanged.
 We extend $\bfy=(y_1,\dots,y_n)$ to $\tilde \bfy=(y_1,\dots,y_{2n})$.
 Then,  we extend the $y$-representation of $G$  on $\bbQ[[\bfy]]$
 in \eqref{eq:tildeX1}  to 
 the \emph{principal $y$-representation}\index{principal!$y$-representation} on $\bbQ[[\tilde \bfy]]$ by
 \begin{align}
 \label{eq:py1}
\tilde X_{\bfn}(y^{(\bfn',\bfm')}):=\{(\bfn,\bfzero),(\bfn',\bfm')\}_{\overline\Omega} \,y^{(\bfn+\bfn',\bfm')}
.
\end{align}
 Also, we introduce canonical variables 
 $\tilde \bfu=(u_1,\dots,u_{2n})$ and  $\tilde \bfp=(p_1,\dots,p_{2n})$
 and represent $\tilde \bfy$ in the same way as \eqref{eq:ty2},
 where the matrix $B$ is replaced with $\overline B$.
 Then, $\bfy=(y_1,\dots,y_{n})$ are  now algebraically independent 
 on the small phase space for $(\tilde \bfu, \tilde \bfp)$.
 After this, the proof of (a) is applicable.
\end{proof}

Now we take the limit $\ell\rightarrow \infty$.
We parametrize the intersection of $\gamma$ and
the walls in $\frakD(B)$ as $z_i$ ($i \in J$) by a countable and totally ordered set $J$
so that $\gamma$ crosses $z_i$ earlier than $z_j$ only if $i<j$.
Then, the consistent relation along $\gamma$ is presented by an infinite  product
in the increasing order in $J$ from right to left as follows:
\begin{align}
\label{eq:sgseq2}
\frakp_{\gamma, \frakD(B)}=
\prod_{i\in J}^{\leftarrow}
\Psi[h_i \bfn_i]^{\epsilon_ i s_i \delta(h_i \bfn_i)}
= \rmid.
\end{align}

\begin{thm}
[{\cite{Nakanishi21d}}]
\label{thm:sDI1}
For any admissible loop $\gamma$ with the consistent relation
\eqref{eq:sgseq2},
the following identity holds
as a formal Puiseux series in  {$\bfy$} with $\log$ factor:
\begin{align}
\label{eq:sDI2}
\sum_{i\in J}
\epsilon_i s_i
\delta(h_i \bfn_i)
\tilde L( y_{z_i}[h_i\bfn_i])
=0.
\end{align}
\end{thm}
\begin{proof}
This is  obtained by  taking the limit $\ell\rightarrow \infty$
of \eqref{eq:sDI1}.
\end{proof}
We call the identity \eqref{eq:sDI2}, together with its reduction \eqref{eq:sDI1} at degree $\ell$,
the \emph{DI associated with a loop $\gamma$ in $\frakD(B)$}\index{dilogarithm identity!for loop in CSD}.
When the loop $\gamma$ is contained in the support of $G$-fan $\calG(B)$
as in Example \ref{ex:periodloop1},
it is regarded
as a $\nu$-period of a $Y$-pattern.
Then, the sum in the identity is finite, and the identity coincides with the one
in Theorem \ref{thm:DI1}.

\begin{ex}
[Type  $A_1^{(1)}$]
For the consistency relation \eqref{3eq:a115},
the associated  DI \eqref{eq:DI1} is written as follows:
\begin{align}
\label{eq:DIa11}
&\tilde L(y_2(1+y_1)^2)+\tilde L(y_1)
=
\tilde L(y_1(1+y_2(1+y_1)^2)^{-2})
+
\Lambda
+
\tilde L(y_2)
,
\\
\begin{split}
&\Lambda= 
\tilde L (y[(2,1)]) 
+ \tilde L (y[(3,2)])
+
\cdots
\\
&\qquad
+
 \sum_{j=0}^{\infty} 2^{1-j} \tilde L (y[2^j(1,1)])
+
\cdots
+\tilde L (y[(2,3)]) 
+ \tilde L (y[(1,2)]).
\end{split}
\end{align}
In the sum $\Lambda$, $y_{z_i}[h_i\bfn_i]$ in \eqref{eq:sDI1} is unambiguously parametrized by $h_i \bfn_i$, so that $z_i$ is omitted.
Also,  the common factor 2 is omitted.
Since the LHS is a finite sum, the RHS converges
as a function of $\bfy\in \bbR^2_{> 0}$.
Unfortunately, we do not know the explicit expressions of $y[\bfn]$ in the RHS in general.
\end{ex}

\section{Alternative proof of DIs and infinite reducibility}
\label{sec:construction1}

Let us show that
the reduction of the consistency relations in
 Theorem \ref{3thm:struct1} implies  a parallel reduction of the corresponding DIs in Theorem \ref{thm:sDI1}
 (including the ones in Theorem \ref{thm:DI1}).
 This also provides an alternative proof of Theorem  \ref{thm:sDI1} and
 also Theomem \ref{thm:DI1}.
 
First, we reformulate the pentagon identity 
\eqref{eq:pent7}
into the form that is closer to 
the  pentagon relation \eqref{eq:pent8}.
We continue to work on the situation of the previous section.

\begin{lem}
[Pentagon identity {\cite{Nakanishi21d}}]
\label{lem:pents1}
Let $\bfn_1$, $\bfn_2\in N^+$.
 If $\{\bfn_2,\bfn_1\}_{\Omega}=c\neq 0$,
the following   identity as functions of $\bfy\in \bbZ_{>0}^n$ holds:
\begin{align}
\label{eq:pentL2}
\begin{split}
& \quad\ \frac{1}{c}
 \Psi[\bfn_1 ]^{-1/c} (\tilde L ( y^{\bfn_2})) + 
\frac{1}{c}
\tilde L (y^{\bfn_1})
\\
& =
\frac{1}{c} \Psi[\bfn_2 ]^{-1/c} \Psi[\bfn_1+\bfn_2]^{-1/c}( \tilde L (y^{\bfn_1}))
+
\frac{1}{c}
 \Psi[\bfn_2 ]^{-1/c} (\tilde L(y^{\bfn_1+\bfn_2}))
+
\frac{1}{c}
\tilde L(y^{\bfn_2}).
\end{split}
\end{align}
\end{lem}
\begin{proof}
We have
\begin{align*}
\Psi[\bfn_1 ]^{-1/c} (y^{\bfn_2})
&= y^{\bfn_2}  (1+y^{\bfn_1}),
\\
\Psi[\bfn_2 ]^{-1/c} (y^{\bfn_1+\bfn_2})
&= y^{\bfn_1+\bfn_2}  (1+y^{\bfn_2})^{-1},
\\
\begin{split}
\Psi[\bfn_2 ]^{-1/c} \Psi[\bfn_1+\bfn_2]^{-1/c} (y^{\bfn_1})
&=  \Psi[ \bfn_2 ]^{-1/c} (y^{\bfn_1}  (1+y^{\bfn_1+\bfn_2})^{-1})
\\
&=  y^{\bfn_1} (1+y^{\bfn_2})^{-1}
  (1+y^{\bfn_1+\bfn_2} (1+y^{\bfn_2})^{-1})^{-1}\\
&=  y^{\bfn}
  ( 1+y^{\bfn_2}+y^{\bfn_1+\bfn_2})^{-1}.
  \end{split}
  \notag
\end{align*}
Then, the pentagon identity \eqref{eq:pent7}
is written in the form
 \eqref{eq:pentL2}
by setting $y_1=y^{\bfn_1}$ and $y_2=y^{\bfn_2}$.
\end{proof}

The overall factor $1/c$ was put so that  the identity \eqref{eq:pentL2}  exactly matches  (the log form of) the
 pentagon relation \eqref{eq:pent8}.
Based on the correspondence between \eqref{eq:pentL2} and \eqref{eq:pent8},
we  obtain a parallel result with  Theorem \ref{3thm:struct1}.
\begin{thm}
[{\cite{Nakanishi21d}}]
\label{thm:DI2}
For any admissible loop $\gamma$ for $\frakD(B)$,
the associated DI in \eqref{eq:sDI2} 
is constructed from and 
reduced to a trivial one
by applying the pentagon identity in \eqref{eq:pent7}
possibly infinitely many times.
\end{thm}

\begin{figure}
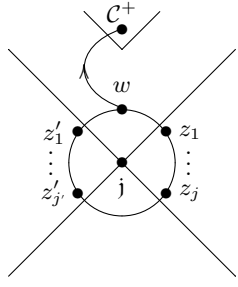

\centering
\leavevmode
\xy
(0,20)*{\text{\small $\calC^+$}};
(-9,4)*{\text{\small $z'_1$}};
(-9,-4)*{\text{\small $z'_{j'}$}};
(9,4)*{\text{\small $z_1$}};
(9,-4)*{\text{\small $z_j$}};
(0,10)*{\text{\small $w$}};
(8.8,1)*{\text{\small $\vdots$}};
(-9.4,1)*{\text{\small $\vdots$}};
(0,-3)*{\text{\small $\frakj$}};
(0,0)*+{\bullet};
(0,7)*+{\bullet};
(0,17.5)*+{\bullet};
(5.85,4.1)*+{\bullet};
(-5.85,4.1)*+{\bullet};
(-5.85,-4.1)*+{\bullet};
(5.85,-4.1)*+{\bullet};
(0,0)*\xycircle<20pt>{};
(1,18)*+{};(1,7)*+{}
   **\crv{(-7,16)&(-7,9)}
          ?>*\dir{};
\ar@{>}(-5,12);(-5,13)
\ar@{-} (-15,-15); (15,15)
\ar@{-} (15,-15); (-15,15)
\ar@{-} (0,15); (-5,20)
\ar@{-} (0,15); (5,20)
\endxy
\caption{Configuration for the consistency relation in
\eqref{eq:sconst2}.
}
\label{fig:sconst2}
\end{figure}

\begin{proof}
Fix $\ell>0$, and consider the reduction $\frakD_{\ell}(B)$ at $\ell$.
We may concentrate on a sufficiently small loop $\gamma$ around a  joint $\frakj$
with the base point $w$.
If the joint $\frakj$ is parallel, 
the consistency relation is reduced to a trivial one
by applying the commutative relation \eqref{eq:Pcom1}.
Accordingly, the corresponding DI is also trivial.
So, let us assume that the joint $\frakj$ is perpendicular.
The consistency relation around $\frakj$  has the form
\begin{align}
\label{eq:sconst2}
\begin{split}
&\quad \ \Psi[h'_{j'} \bfn'_{j'}]^{s'_{j'}  \delta (h'_{j'} \bfn'_{j'})}
\cdots
\Psi[h'_1\bfn'_1]^{s'_1\delta (h'_1 \bfn'_1)}
\\
&\equiv
\Psi[h_j \bfn_j]^{s_j  \delta (h_j \bfn_j)}
\cdots \Psi[h_1\bfn_1]^{s_1\delta (h_1 \bfn_1)}
\mod
G^{>{\l}},
\end{split}
\end{align}
where the LHS and RHS are the anti-ordered and ordered products
in Construction \ref{const:csd1} (3), respectively.
The corresponding DI   \eqref{eq:sDI2}  has the form (see Figure \ref{fig:sconst2})
\begin{align}
\label{eq:dd1}
\sum_{i=1}^{j'} s'_i \delta (h'_i \bfn'_i) \tilde L (y_{z'_i}[h'_i \bfn'_i])
\equiv
\sum_{i=1}^{j} s_i \delta (h_i \bfn_i) \tilde L (y_{z_i}[h_i \bfn_i])
\mod 
\calF^{> \ell}_{\log}.
\end{align}
In view of Figure \ref{fig:sconst2},
this is equivalent to
\begin{align}
\label{eq:dd2}
\begin{split}
&\quad \ 
\sum_{i=1}^{j'} 
 s'_i \delta (h'_i \bfn'_i)
\Psi[h'_1 \bfn'_1 ]^{- s'_1\delta (h'_1 \bfn'_1)}
\cdots
 \Psi[h'_{i-1} \bfn'_{i-1}]^{- s'_{i-1}  \delta (h'_{i-1} \bfn'_{i-1})}
 (
 \tilde L (y^{h'_i \bfn'_i})
 )
 \\
&\equiv
\sum_{i=1}^{j} 
s_i \delta (h_i \bfn_i)
\Psi[h_1 \bfn_1 ]^{- s_1\delta (h_1 \bfn_1)}
\cdots
 \Psi[h_{i-1} \bfn_{i-1}]^{- s_{i-1}  \delta (h_{i-1} \bfn_{i-1})}
 (
 \tilde L (y^{h_i \bfn_i})
 )
 \\
&
 \hskip240pt
 \mod \calF^{> \ell}_{\log}.
 \end{split}
\end{align}
Along the procedure, say, $\calP$ of obtaining \eqref{3eq:gpos2} from \eqref{3eq:init1} by
 Proposition \ref{prop:ordering1}
(more precisely, by \cite[Algorithm III.5.7]{Nakanishi22a}),
we do the following procedure:
\begin{enumerate}
\item As the initial input, we set $S(y)$ to be the LHS of \eqref{eq:dd2}.
\item If the pentagon relation \eqref{eq:pent8} is applied to the adjacent pair
$\Psi[h' \bfn']^{1/c}$ and $\Psi[t \bfn]^{1/c}$ 
in the procedure $\calP$,
we apply the pentagon identity \eqref{eq:pentL2}
to the corresponding terms in $S(y)$.
Simultaneously, we  apply  the pentagon relation \eqref{eq:pent8} itself
to the corresponding  pair
$\Psi[h' \bfn']^{1/c}$ and $\Psi[t \bfn]^{1/c}$ 
in $S(y)$.
(This does not change $S(y)$ as a formal Puiseux series in $\bfy$ with $\log$ factor.)
\item If the commutative relation modulo $G^{>\ell}$
 is applied
to the adjacent pair
$\Psi[h' \bfn']^{c'}$ and $\Psi[t \bfn]^c$
in the procedure $\calP$,
 we also apply  it
to the corresponding pair
$\Psi[h' \bfn']^{c'}$ and $\Psi[t \bfn]^c$ 
in $S(y)$.
(This occurs for the relation \eqref{eq:Pcom1}
or  the truncation of the relation \eqref{eq:pent8} modulo $G^{>\ell}$.
In the former case, this does not change $S(y)$,
while in the latter case the result equals to $S(y)$  modulo $\calF^{> \ell}_{\log}$.)
\end{enumerate}
Then, thanks to the correspondence between 
  \eqref{eq:pentL2}
and
 \eqref{eq:pent8},
 the final result is the RHS of \eqref{eq:dd2}.
 Thus, 
 the identity
\eqref{eq:dd2} is obtained from a trivial one $S(y)=S(y)$.
Also, by reversing the procedure,
 the equality  \eqref{eq:dd2} is reduced to a trivial one.
\end{proof}

Note that the above proof of Theorem \ref{thm:DI2} is completely formal and relies only on 
the results on the CSDs in Section \ref{sec:positive1}
and the correspondence between the pentagon identity
  \eqref{eq:pentL2}
and the pentagon relation
 \eqref{eq:pent8}.
 Therefore, it provides an alternative  proof of Theorem \ref{thm:sDI1},
 which is
 independent of the classical mechanical method in the previous section.
In particular, it gives yet another proof of Theorem \ref{thm:DI1}.

\begin{ex}[Type $B_2$]
\label{ex:B2}
Let us demonstrate how the procedure in the proof of Theorem \ref{thm:DI0}  works in practice.
Consider the consistency relation \eqref{eq:order3} for type $B_2$.
Here, we present again all steps of ordering by the pentagon relation therein,
where the pentagon relation is applied to pairs in the parentheses.
\begin{align*}
\begin{split}
\begin{bmatrix}
0\\
1
\end{bmatrix}
\left(
\begin{bmatrix}
0\\
1
\end{bmatrix}
\begin{bmatrix}
1\\
0
\end{bmatrix}
\right)
&=
\left(
\begin{bmatrix}
0\\
1
\end{bmatrix}
\begin{bmatrix}
1\\
0
\end{bmatrix}
\right)
\begin{bmatrix}
1\\
1
\end{bmatrix}
\begin{bmatrix}
0\\
1
\end{bmatrix}
\\
&=
\begin{bmatrix}
1\\
0
\end{bmatrix}
\begin{bmatrix}
1\\
1
\end{bmatrix}
\left(
\begin{bmatrix}
0\\
1
\end{bmatrix}
\begin{bmatrix}
1\\
1
\end{bmatrix}
\right)
\begin{bmatrix}
0\\
1
\end{bmatrix}
\\
&=
\begin{bmatrix}
1\\
0
\end{bmatrix}
\begin{bmatrix}
1\\
1
\end{bmatrix}
\begin{bmatrix}
1\\
1
\end{bmatrix}
\begin{bmatrix}
1\\
2
\end{bmatrix}
\begin{bmatrix}
0\\
1
\end{bmatrix}
\begin{bmatrix}
0\\
1
\end{bmatrix}
.
\end{split}
\end{align*}
Accordingly, we apply the procedure in the proof of Theorem \ref{thm:DI2} as follows:
\begin{align*}
&\
\uwave{
\Psi[\bfe_1]^{-1}\Psi[\bfe_2]^{-1}
}
\tilde L(y^{\bfe_2})
+
\uline{
\Psi[\bfe_1]^{-1}\tilde L (y^{\bfe_2})
}
+
\uline{
\tilde L
(y^{\bfe_1})
}
\notag
\\
=&\
\Psi[\bfe_2]^{-1} \Psi[(1,1)]^{-1} 
\uline{
\Psi[\bfe_1]^{-1}\tilde L(y^{\bfe_2})
}
+
\Psi[\bfe_2]^{-1} \Psi[(1,1)]^{-1}
\uline{
 \tilde L (y^{\bfe_1})
}
\notag
\\
&\quad
+
\Psi[\bfe_2]^{-1}\tilde L (y^{(1,1)})
+
\tilde L
(y^{\bfe_2})
\notag
\\
=&\
\Psi[\bfe_2]^{-1}
\uwave{
 \Psi[(1,1)]^{-1}  \Psi[\bfe_2]^{-1}
 }
  \Psi[(1,1)]^{-1}  \tilde L(y^{\bfe_1})
\notag
\\
&\quad
+
\Psi[\bfe_2]^{-1} 
\uwave{
\Psi[(1,1)]^{-1}  \Psi[\bfe_2]^{-1} 
}
  \tilde L(y^{(1,1)})
\\
&\quad
+
\Psi[\bfe_2]^{-1}
\uline{
 \Psi[(1,1)]^{-1} \tilde L (y^{\bfe_2})
 }
+
\Psi[\bfe_2]^{-1}
\uline{
\tilde L (y^{(1,1)})
}
+
\tilde L
(y^{\bfe_2})
\notag
\\
=&\
\Psi[\bfe_2]^{-1}  \Psi[\bfe_2]^{-1}   \Psi[(1,2)]^{-1} \Psi[(1,1)]^{-1} \Psi[(1,1)]^{-1}  \tilde L(y^{\bfe_1})
\notag
\\
&\quad
+
  \Psi[\bfe_2]^{-1}  \Psi[\bfe_2]^{-1}   \Psi[(1,2)]^{-1} \Psi[(1,1)]^{-1}  \tilde L(y^{(1,1)})
  \notag
\\
&\quad
+
   \Psi[\bfe_2]^{-1}   \Psi[\bfe_2]^{-1}   \Psi[(1,2)]^{-1}  \tilde L (y^{(1,1)})
 \notag
\\
&\quad
+
\Psi[\bfe_2]^{-1}  \Psi[\bfe_2]^{-1}  \tilde L (y^{(1,2)})
+
\Psi[\bfe_2]^{-1}\tilde L (y^{\bfe_2})
+
\tilde L
(y^{\bfe_2}).
\notag
\end{align*}
Here, the pentagon identity
\eqref{eq:pentL2} is applied to the pairs of terms with underline,
while the  pentagon relation  \eqref{eq:pent8} is applied to
the terms with wavy underline.
Thus, we obtain the associated DI
in the form \eqref{eq:dd2}.
\end{ex}

Let us specialize Theorem \ref{thm:DI0} to the case of a period of a $Y$-pattern.

\begin{cor}
[{\cite{Nakanishi21d}}]
\label{cor:DI2}
For any $\nu$-period of a $Y$-pattern,
the following facts hold.
\par
(a).
The corresponding consistency relation $\frakp_{\gamma, \frakD(B)}=\rmid$
in the CSD $\frakD(B)$
 is reduced to a trivial one
by applying the  commutative and pentagon relations in Proposition \ref{3prop:pent1}
possibly infinitely many times.
\par
(b).
The associated DI
in Theorem \ref{thm:DI1}
is
reduced to a trivial one by applying the pentagon identity  \eqref{eq:pent7}
possibly infinitely many times.
\end{cor}

Now we come back to the
reduction problem in Section \ref{sec:reduction1}
for the DI associated with a loop in a CSD.
(All DIs in this text fall into this class.)
As shown in Theorem \ref{thm:DI0},
it is reduced to the trivial one by applying the pentagon identity possibly infinitely many times.
We call this property \emph{infinite reducibility}\index{infinite reducibility}.
Of course, infinitely many operations are necessary for DIs with an infinite product.
However, our method only proves the infinite reducibility even for DIs with a finite product.
This is a somewhat unexpected answer to the problem.

\begin{rem}
(1)
Corollary \ref{cor:DI2} (b)  does not deny the finite reducibility 
for a DI with a finite sum
as we see in Example \ref{ex:B2}.
Also,
a finite reduction may occur differently from the procedure in Theorem \ref{thm:DI2}.
So,  the finite reducibility problem for the above DIs with a finite product is still open.

\par
(2)
Corollary \ref{cor:DI2} (a) does {not} imply that 
any $\nu$-period of a $Y$-pattern is reduced to the trivial one
by  the square and pentagon periodicities \emph{in the $Y$-pattern itself}.
{\rp Here, the square periodicity is the relation $\mu_i\mu_j\mu_i\mu_j=\rmid$ for
commuative mutations $\mu_i$ and $\mu_j$.}
Again,  Example \ref{ex:B2} is such a case because there is
only the hexagon periodicity
in the $Y$-pattern.
Therefore, the reducibility of a period in a $Y$-pattern and the one in a CSD are \emph{different} problems.
Some other examples of the periodicity that are not reducible in a $Y$-pattern  are known \cite{Fomin08, Kim16}.
\end{rem}

\notes
The whole content of this chapter is due to \cite{Nakanishi21d}.

\part{\text{Quantum dilogarithm identities}}

\chapter{Quantum dilogarithm}

A natural $q$-deformation of the dilogarithm
called the \emph{quantum dilogarithm} 
was introduced by Faddeev, Kashaev, Volkov \cite{Faddeev93, Faddeev94}
and modified by Fock-Goncharov \cite{Fock03}.
The function satisfies the pentagon identity for $q$-commutative variables,
which is regarded as the counterpart of the pentagon identity for the dilogarithm.

\section{$q$-exponential}

Let $q$ and $x$ be variables.
Let $j$ be a nonnegative integer.
We introduce the \emph{$q$-Pochhammer symbol}
\begin{gather}
(x;q)_{\infty}=\prod_{k=0}^{\infty} (1-xq^k),
\quad
(q)_j=\prod_{k=1}^{j} (1-q^k),
\quad
(q)_0=1,
\end{gather}
and the \emph{$q$-number}\index{$q$-number} ($a\in \bbZ$)
\begin{align}
\label{eq:qnum1}
[a]_q=\frac{q^a-q^{-a}}{q-q^{-1}},
\quad
\lim_{q\rightarrow 1} [a]_q=a.
\end{align}

We define
the \emph{$q$-exponential}\index{$q$-exponential} $e_q(x)$,
first as a formal power series in $x$ with coefficients in $\bbQ(q)$,
by the following  three equivalent expressions:
\begin{align}
\label{eq:qe1}
e_q(x)
&=\sum_{j=0}^{\infty}
\frac{1}{(q)_j} x^j
=
\sum_{j=0}^{\infty}
\frac{1}{ (1-q)\cdots (1-q^j)}
x^j
\\
\label{eq:qe2}
&=
\frac{1}{(x;q)_{\infty}}
=\prod_{k=0}^{\infty}(1-q^{k}x)^{-1}
\\
\label{eq:qe3}
&=
 \exp\biggl(
\sum_{j=1}^{\infty}
\frac{1 }{j (1 - q^{j})} x^j
\biggr).
\end{align}
Here, we follow the notation in \cite{Kirillov95}.
Observe the similarity between the third expression \eqref{eq:qe3}
and (the exponential of) the Euler dilogarithm $\mathrm{Li}_2(x)$ in \eqref{eq:Li1}.

The above three expressions are equivalent due to the following characterization of $e_q(x)$.
\begin{prop}
\label{prop:qech1}
(a). 
The  properties
\begin{alignat}{3}
\label{eq:qeini1}
e_q(0)&=1
\quad
\text{(initial condition)},
\\
\label{eq:qerec1}
e_q(q x) &=(1-x) e_q(x)
\quad
\text{($q$-difference relation)}
\end{alignat}
uniquely   determine a formal power series $e_q(x)$ in $x$  with coefficients in $\bbQ(q)$.
\par
(b). All expressions \eqref{eq:qe1}, \eqref{eq:qe2},
and \eqref{eq:qe3} satisfy
the conditions  \eqref{eq:qeini1} and \eqref{eq:qerec1}.
\end{prop}
\begin{proof}
(a). 
We set
\begin{align}
e_q(x)=\sum_{n=0}^{\infty} f_n(q) x^n
\quad
(f_n(q)\in \bbQ(q)),
\end{align}
where $f_0(q)=1$ by \eqref{eq:qeini1}.
Then, by \eqref{eq:qerec1}, we have the recurrence relation
\begin{align}
q^{n} f_n(q)=f_n(q)  - f_{n-1}(q)
\quad
(n\geq 1).
\end{align}
This determines $f_n(q)$ uniquely as
\begin{align}
\label{eq:fn1}
f_n(q)
=\frac{1}{1-q^{n}}
f_{n-1}(q)
=\frac{1}{(1-q)\cdots (1-q^{n})}.
\end{align}

(b).
We have already seen in \eqref{eq:fn1}
 that the expression \eqref{eq:qe1}
satisfies the conditions in (a).
Thus,
it is enough to check the relation \eqref{eq:qerec1}
for the expressions \eqref{eq:qe2} and \eqref{eq:qe3}.
For \eqref{eq:qe2}, we see it as
\begin{align}
\prod_{k=0}^{\infty}(1- q^{k} qx)^{-1}
=
(1-x)
\prod_{k=0}^{\infty}(1- q^{k} x)^{-1}.
\end{align}
For \eqref{eq:qe3}, we see it as
\begin{align}
\begin{split}
 \exp\biggl(
\sum_{j=1}^{\infty}
\frac{1 }{j (1 - q^{j})}  q^j x^j
\biggr)
&=
 \exp\biggl(
 -
\sum_{j=1}^{\infty}
\frac{1 }{j} x^j
\biggr)
 \exp\biggl(
\sum_{j=1}^{\infty}
\frac{1 }{j (1 - q^{j})} x^j
\biggr)
\\
&=
(1-x)
 \exp\biggl(
\sum_{j=1}^{\infty}
\frac{1 }{j (1 - q^{j})} x^j
\biggr).
\end{split}
\end{align}
\end{proof}

Next, we specialize $q\in \bbC^{\times}$ such that $q$ is not a root of 1.
Then,  every expression for $e_q(x)$ in \eqref{eq:qe1}--\eqref{eq:qe3} makes sense
as a formal power series in $x$ with coefficients in $\bbC$.
Furthermore, suppose that $|q|<1$.
Then, its convergence radius is 1 by \eqref{eq:qe1}.
Thus, we may also regard $e_q(x)$ as a complex function of $x\in \bbC$ with  $|x|<1$;
which is  analytically continued to the meromorphic function in \eqref{eq:qe2}.

We may consider the complex parameter $q$ ($|q|<1$) as a ``quantization
parameter'' so that the limit  $q\rightarrow 1^-$ is the \emph{classical limit}\index{classical limit}.
Equivalently,
we introduce the \emph{Planck constant}\index{Planck constant}
$\hbar \in \bbC$ with $\mathrm{Im}\, \hbar >0$ by
\begin{align}
q=\exp(  i \hbar),
\quad i = \sqrt{-1}.
\end{align}
Then, the limit $\hbar \rightarrow 0^{+i}$ is the classical limit,
where $0^{+i}$ indicates that $\hbar$ approaches to $0$ with $\mathrm{Im}\, \hbar >0$
in the complex plane.

\begin{prop}[{\cite[Eq.~(1.5)]{Faddeev94}, \cite[Cor.~10]{Kirillov95}}]
The following asymptotic  expansion holds in the limit  $\hbar \rightarrow 0^{+i}$:
\begin{align}
\label{eq:asyme1}
e_q(x) = (1-x)^{-1/2} 
\exp\biggl(
\frac{- \mathrm{Li}_2(x)}{i  \hbar}
\biggr) (1+ \mathcal{O}(\hbar)).
\end{align}
\end{prop}
\begin{proof}
Let $q^{1/2}=\exp(i \hbar/2)$.
By the expression \eqref{eq:qe3}, we have
\begin{align*}
\begin{split}
\log e_q(x) 
&=
\sum_{j=1}^{\infty}
\frac{1 }{j (1 - q^{j})} x^j
\\
&=
\sum_{j=1}^{\infty}
\frac{1}{j} x^j
-
\sum_{j=1}^{\infty}
\frac{q^{j/2}}{j (q^{j/2} - q^{-j/2})} x^j
\\
&=
-\log(1-x)
-
\frac{1}{q^{1/2}-q^{-1/2}}
\sum_{j=1}^{\infty}
\frac{q^{j/2}}{j [j]_{q^{1/2}}} x^j.
\end{split}
\end{align*}
The following expansion formulas hold:
\begin{align}
q^{j/2}&=1 + \frac{j}{2} i\hbar + \calO(\hbar^2),
\\
[j]_{q^{1/2} }&=j + \calO(\hbar^2),
\\
\frac{1}{q^{1/2}-q^{-1/2}}&
=\frac{1}{i\hbar + \calO(\hbar^3)}
=\frac{1}{i\hbar} + \calO(\hbar).
\end{align}
Thus, we have
\begin{align}
\begin{split}
\log e_q(x) &=
-\log(1-x)
-
\frac{1}{i\hbar}
\mathrm{Li}_2(x)
+
\frac{1}{2}\log(1-x)
+\calO(\hbar).
\end{split}
\end{align}
This gives the desired result \eqref{eq:asyme1}.
\end{proof}

The function  $1/e_q(x)$  was called the \emph{quantum dilogarithm}
in \cite{Faddeev94}.
Here, we reserve the name for its variant introduced in the next section.

\section{Quantum dilogarithm}

Following Fock-Goncharov \cite{Fock03},
we define
the \emph{quantum dilogarithm}\index{quantum dilogarithm} $\Psi_q(x)$ by
\begin{align}
\label{eq:qd0}
\Psi_q(x) = e_{q^2}(-qx).
\end{align}
Thus, it is essentially the $q$-exponential in the previous section with a slightly different parametrization.
As we will see, this function fits our approach to the quantum dilogarithm identities.

In parallel with \eqref{eq:qe1}--\eqref{eq:qe3}, we have the following 
three equivalent expressions:
\begin{align}
\label{eq:qd1}
\Psi_q(x)
&=\sum_{j=0}^{\infty}
\frac{(-q)^j}{(q^2)_j} x^j
=
\sum_{j=0}^{\infty}
\frac{1}{(q-q^{-1})^j}
\frac{1}{[1]_q \cdots [n]_q} x^j,
\\
\label{eq:qd2}
&=
\frac{1}{(-qx;q^2)_{\infty}}
=\prod_{k=0}^{\infty}(1+q^{2k+1}x)^{-1}
\\
\label{eq:qd3}
&= \exp\biggl(
\sum_{j=1}^{\infty}
\frac{(-1)^{j+1}}{j (q^j - q^{-j})} x^j
\biggr)
=\exp\biggl(
\frac{1}{q-q^{-1}}
\sum_{j=1}^{\infty}
\frac{(-1)^{j+1}}{j [j]_q} x^j
\biggr).
\end{align}
As we see, the expressions \eqref{eq:qd1} and \eqref{eq:qd3}
are
more ``balanced''  than the $q$-exponential $e_q(x)$ under the symmetry $q \leftrightarrow q^{-1}$. 
Again, we may view $\Psi_q(x)$ as a formal power series in $x$
with coefficients in $\bbQ(q)$,
or a meromorphic function of $x$  for a given $0<|q|<1$,
depending on the context.

Proposition \ref{prop:qech1}  is restated for $\Psi_q(x)$ as follows.
(The second relation in \eqref{eq:qdrec1} is 
redundant. However, it is included here for later convenience.)

\begin{prop}
As a formal power series in $x$
with coefficients in $\bbQ(q)$, 
the function $\Psi_q(x)$
is characterized by the conditions
\begin{gather}
\label{eq:qdini1}
\Psi_q(0)=1,
\\
\label{eq:qdrec1}
\Psi_q(q^2 x) =(1+q x) \Psi_q(x),
\quad
\Psi_q(q^{-2} x) =(1+q^{-1} x)^{-1} \Psi_q(x).
\end{gather}
\end{prop}
\begin{proof}
The first relation in \eqref{eq:qdrec1} is obtained from \eqref{eq:qerec1} as
\begin{align}
\begin{split}
\Psi_q(q^2 x)
&=
e_{q^2}(-q^3 x)
=
e_{q^2}(q^2(-q x))
=(1+q^2 x)
e_{q^2}(-qx)
\\
&=(1+q^2 x) \Psi_q(x).
\end{split}
\end{align}
The second relation in \eqref{eq:qdrec1} is equivalent to the first one,
and it is obtained from the first one by setting $x$ to $q^{-2}x$ therein. 
\end{proof}

The asymptotic property of $\Psi_q(x)$ is simpler than \eqref{eq:asyme1},
and more directly related with the Euler dilogarithm.
\begin{prop}[{\cite[Eq.~(54)]{Fock03}}]
\label{prop:Psiasym1}
The following asymptotic  expansion holds in the limit  $\hbar \rightarrow 0^{+i}$:
\begin{align}
\label{eq:asymd1}
\Psi_q(x) = 
\exp\biggl(
\frac{- \mathrm{Li}_2(-x)}{2i  \hbar}
\biggr) (1+ \mathcal{O}(\hbar)).
\end{align}
\end{prop}
\begin{proof}
By \eqref{eq:qd3},
\begin{align}
\label{eq:asymd2}
\log \Psi_q(x) 
&=
\frac{1}{q-q^{-1}}
\sum_{j=1}^{\infty}
\frac{(-1)^{j+1}}{j [j]_q} x^j
\end{align}
The following expansion formulas hold:
\begin{align}
[j]_{q}&=j + \calO(\hbar^2),
\\
\frac{1}{q-q^{-1}}&
=\frac{1}{2i\hbar + \calO(\hbar^3)}
=\frac{1}{2i\hbar} + \calO(\hbar).
\end{align}
Thus, we have
\begin{align}
\begin{split}
\log \Psi_q(x) &=
-
\frac{1}{2i\hbar}
\mathrm{Li}_2(-x)
+\calO(\hbar).
\end{split}
\end{align}
This gives the desired result \eqref{eq:asymd1}.
\end{proof}

\section{Pentagon relation}

Let us introduce the \emph{$q$-binomial coefficient}\index{$q$-binomial!coefficient} for $n, k\in \bbZ_{\geq 0}$ ($n\geq k$),
\begin{align}
\label{eq:qb1}
\genfrac[]{0pt}{0}{n}{k}_q=\frac{ (q)_n}{(q)_k (q)_{n-k}}.
\end{align}
For example,
\begin{align}
\genfrac[]{0pt}{0}{2}{1}_q &= \frac{(1-q)(1-q^2)}{1-q}=1+q,
\\
\genfrac[]{0pt}{0}{3}{1}_q &= \frac{(1-q)(1-q^2)(1-q^3)}{(1-q)(1-q)(1-q^2)}=1+q+q^2.
\end{align}

\begin{lem}
The following recurrence relation holds ($1\leq k \leq n-1$):
\begin{align}
\label{eq:qbrec1}
\genfrac[]{0pt}{0}{n-1}{k}_q
+
q^{n-k}
\genfrac[]{0pt}{0}{n-1}{k-1}_q
=\genfrac[]{0pt}{0}{n}{k}_q.
\end{align}
In particular, $\genfrac[]{0pt}{1}{n}{k}_q$ is a polynomial in $q$.
\end{lem}
\begin{proof}
We have
\begin{align}
(\text{LHS})=
\frac{ (q)_{n-1}}{(q)_{k} (q)_{n-k}} \{(1-q^{n-k}) + q^{n-k}(1-q^{k})\}
=(\text{RHS}).
\end{align}
\end{proof}

For a pair of noncommutative variables  $a$, $b$
and a commutative (scalar) variable $q$,
we consider the  relation (the \emph{$q$-commutative relation}\index{$q$-commutative relation})
\begin{align}
\label{eq:qcom1}
 ab=qba.
\end{align}

We have a $q$-analog of the binomial theorem. 

\begin{lem}[$q$-binomial theorem\index{$q$-binomial!theorem} \cite{Schutzenberger53}] Under the relation \eqref{eq:qcom1},
we have
\begin{align}
\label{eq:qbin1}
(a+b)^n = \sum_{k=0}^n \genfrac[]{0pt}{0}{n}{k}_q
b^k a^{n-k} .
\end{align}
\end{lem}
\begin{proof}
We show it by the induction on $n$ as follows.
\begin{align*}
(a+b)^n &= \sum_{k=0}^{n-1} \genfrac[]{0pt}{0}{n-1}{k}_q  b^k a^{n-1-k} (a+b) 
\\
&= \sum_{k=0}^{n-1} 
\genfrac[]{0pt}{0}{n-1}{k}_q
b^k a^{n-k}  
+
 \sum_{k=0}^{n-1} 
 q^{n-1-k}
  \genfrac[]{0pt}{0}{n-1}{k}_q
b^{k+1} a^{n-1-k} 
\\
&= \sum_{k=0}^{n-1} 
\genfrac[]{0pt}{0}{n-1}{k}_q
b^k a^{n-k}  
+
 \sum_{k=1}^{n} 
  q^{n-k} \genfrac[]{0pt}{0}{n-1}{k-1}_q
b^{k} a^{n-k}
\\
\overset {\eqref{eq:qbrec1}}
&{=} \sum_{k=0}^n 
 \genfrac[]{0pt}{0}{n}{k}_q
 b^k a^{n-k}.
\end{align*}
\end{proof}

We regard $e_q(x)$ as a formal power series in $x$ as in \eqref{eq:qe1}.
Then, $e_q(a)$ and $e_q(b)$ make sense in a formal (noncommutative) power series in $a$ and $b$
with coefficients in $\bbQ(q)$.
By \eqref{eq:qcom1}, we have
\begin{align}
e_q(b)e_q(a)\neq  e_q(a)e_q(b).
\end{align}

The product $e_q(b)e_q(a)$ is a ``normal'' one.
\begin{lem}[\cite{Schutzenberger53}] 
\label{lem:ee1}
Under the relation \eqref{eq:qcom1},
we have
\begin{align}
\label{eq:qee1}
e_q(b)e_q(a)=e_q(a+b).
\end{align}
\end{lem}
\begin{proof}
\begin{align*}
\text{(RHS)}&=
\sum_{n=0}^{\infty}
\frac{1}{(q)_n} (a+b)^n
\\
\overset{\eqref{eq:qbin1}}
&{=}
\sum_{n=0}^{\infty}
\frac{1}{(q)_n}
\sum_{k=0}^n \genfrac[]{0pt}{0}{n}{k}_q
b^k a^{n-k} 
\\
&=
\sum_{n=0}^{\infty}
\sum_{k=0}^n 
\frac{1}{(q)_k}
\frac{1}{(q)_{n-k}}
b^k a^{n-k} 
=\text{(LHS)}.
\end{align*}
\end{proof}

It turns out the opposite product
 $e_q(a)e_q(b)$
 contains a hidden rich structure,
 which was discovered by \cite{Faddeev93, Faddeev94}.
 Here, we follow the presentation by \cite{Kirillov95}.
 
 \begin{lem}[{\cite[Lemma 9]{Kirillov95}}]
 \label{lem:ee2}
 Under the relation \eqref{eq:qcom1},
we have
\begin{align}
\label{eq:qeid1}
e_q(a)e_q(b)&=e_q(b-ba)e_q(a).
\end{align}
\end{lem}
\begin{proof}
We regard both sides as a formal power series $\varphi(b)$ in $b$
and check that both sides satisfy the same initial condition $\varphi(0)=e_q(a)$
and the $q$-difference relation
\begin{align}
\varphi(qb)=(1-b+ba)\varphi(b),
\end{align}
which determine $\varphi(b)$  uniquely.
First, observe that
\begin{align}
\label{eq:eee1}
e_q(qb)e_q(b)^{-1} &
\overset {\eqref{eq:qerec1}}
{=}1-b,
\end{align}
and
\begin{align}
\label{eq:eee2}
\begin{split}
e_q(a) b e_q(a)^{-1}
\overset {\eqref{eq:qcom1}}
&{=}
be_q(qa)  e_q(a)^{-1}
\\
\overset {\eqref{eq:qerec1}}
 &{=}
b(1-a).
\end{split}
\end{align}
Thus, we have
\begin{align}
\begin{split}
e_q(a)e_q(qb)e_q(b)^{-1}e_q(a)^{-1}
\overset {\eqref{eq:eee1}}
&{=} e_q(a)(1-b)e_q(a)^{-1}
\\
\overset {\eqref{eq:eee2}}
& {=} 1- b + ba.
\end{split}
\end{align}
Meanwhile, we also have
\begin{align}
\begin{split}
e_q(q(b-ba))e_q(a)e_q(a)^{-1}e_q(b-ba)^{-1}
&= e_q(q(b-ba))e_q(b-ba)^{-1}
\\
\overset {\eqref{eq:eee1}}
&{=} 1- b + ba.
\end{split}
\end{align}
Therefore,  the equality  \eqref{eq:qeid1} holds.
\end{proof}

By combining Lemmas \ref{lem:ee1}
 and \ref{lem:ee2}, we obtain the following result.
 \begin{thm}[{\cite[\S3]{Faddeev93}, \cite[\S2]{Faddeev94}, \cite[Lemma 9]{Kirillov95}}]
 \label{thm:qepent1}
  Under the relation \eqref{eq:qcom1},
we have
\begin{align}
\label{eq:qepent1}
e_q(a)e_q(b)&=e_q(b)e_q(-ba)e_q(a)
\quad
\text{(pentagon relation)\index{pentagon!relation (quantum)}}
\\
\label{eq:qepent2}
&=e_q(b-ba+a).
\end{align}
\end{thm}
\begin{proof}
We note the following $q$-commutative relations
\begin{align}
a (b-ba) &= q (b-ba)a,
\quad
a (-ba) = q (-ba)a,
\quad
 (-ba)b = q b (-ba).
\end{align}
Then, we have
\begin{align}
e_q(a)e_q(b)
\overset {\eqref{eq:qeid1}}
{=} 
e_q(b-ba)e_q(a)
\overset {\eqref{eq:qee1}}
{=}e_q(b-ba + a).
\end{align}
On the other hand, we have
\begin{align}
e_q(b)e_q(-ba)e_q(a)
\overset {\eqref{eq:qee1}}
{=} e_q(b)e_q(-ba + a)
\overset {\eqref{eq:qee1}}
{=} e_q(b-ba +a).
\end{align}
Thus, we obtain the desired relations.
\end{proof}

The relation \eqref{eq:qepent1} is due to \cite[\S2]{Faddeev94},
while the relation \eqref{eq:qepent2} is due to \cite[\S3]{Faddeev93}.

For a pair of noncommutative variables  $u$, $v$
and a commutative (scalar) variable $q$,
we consider the $q$-commutative relation 
\begin{align}
\label{eq:qcom2}
uv =q^2vu.
\end{align}
Then, the pentagon relation \eqref{eq:qepent1} is translated to the ones
for $\Psi_q(x)$ as follows.
\begin{thm}[Pentagon relation {\cite[\S3]{Faddeev93}, \cite[\S2]{Faddeev94}}]
\label{thm:qdpent1}\index{pentagon!relation (quantum)}
 Under the relation \eqref{eq:qcom2},
we have
\begin{align}
\label{eq:qdpent1}
\Psi_q(u)\Psi_q(v)&=\Psi_q(v)\Psi_q(qvu)\Psi_q(u).
\end{align}
\end{thm}
\begin{proof}
For $u$ and $v$ satisfying \eqref{eq:qcom2},
we set $a=-qu$ and $b=-qv$, so that they satisfy the $q$-commutative relation
\begin{align}
ab= q^2ba
\end{align}
instead of \eqref{eq:qcom1}.
Then, \eqref{eq:qepent1} reads
\begin{align}
e_{q^2}(-qu)e_{q^2}(-qv)&=e_{q^2}(-qv)e_{q^2}(-q^2 vu)e_{q^2}(-qu).
\end{align}
By \eqref{eq:qd0}, it is written as \eqref{eq:qdpent1}.
\end{proof} 

We will give an alternative proof of the identity in Example \ref{ex:QDIA24}.

\notes

The $q$-exponential $e_q(x)$ and its functional relations naturally appeared 
in the study of the lattice Virasoro algebra in \cite{Faddeev93}.
Subsequently, the pentagon identity \eqref{eq:qepent1} was discovered by \cite{Faddeev94}.
Moreover, the pentagon identity of the classical dilogarithm
was recovered in the limit $q \rightarrow 1$ of \eqref{eq:qepent1}.
This is why the function $1/e_q(x)$ was regarded as the quantum dilogarithm.
The version $\Psi_q(x)$ was introduced in studying the quantization of cluster algebras by \cite{Fock03}.
Here, we mainly follow the presentation in \cite{Kirillov95}.

\chapter{Quantum mutations and quantum dilogarithm}

Cluster patterns and $Y$-patterns admit natural
quantizations \cite{Berenstein05,Fock07}.
The quantum mutations and the quantum dilogarithm are related in the same way
as the classical mutations and the dilogarithm.
We also have a natural quantization of the dilogarithm elements and
the pentagon relations.

\section{Quantization of $Y$-patterns}

The notion of \emph{quantum cluster algebras} was developed in two directions:
\begin{itemize}
\item
The quantization of $x$-variables of geometric type
\cite{Berenstein05b, Tran09}.
\item
The quantization of free $y$-variables
\cite{Fock03, Fock07}.
\end{itemize}
What is common in them is that $x$- and $y$-variables are replaced with
\emph{$q$-commutative} variables.
They were introduced and studied more or less independently.
Nevertheless, they are closely related to each other naturally.

Let us start with the quantization of $y$-variables, which are simpler and
more directly related to quantum dilogarithm identities.
For a given free $Y$-pattern $\bfUpsilon=\{(\bfy_t,B_t)\}_{t\in \bbT_n}$, 
let 
\begin{align}
\label{eq:BDO3}
B_t= \Delta\Omega_t
\end{align}
be a skew-symmetric decomposition in \eqref{eq:BDO1},
where $\Delta=D^{-1}$ is a common positive integer diagonal matrix for all $B_t$.
Let $q$ be a variable, and let
\begin{align}
\label{eq:qi1}
q_i :=q^{1/\delta_i} = q^{d_i}.
\end{align}
The quantization of  $\bfUpsilon$ consists
of the following two modifications \cite{Fock03}:
\begin{itemize}
\item
\emph{Quantum $y$-variables}\index{quantum!$y$-variable}.
For each $Y$-seed $(\bfy_t, B_t)$,
replace
the $y$-variables $\bfy_t=(y_{1;t},\dots,y_{n;t})$ with noncommutative variables
$\bfY_t=(Y_{1;t}, \dots, Y_{n;t})$.
In particular, the initial variables 
$\bfY=\bfY_{t_0}=(Y_{1}, \dots, Y_{n})$
 obey the $q$-commutative relation
\begin{align}
\label{eq:Ycom1}
Y_{i} Y_{j} = q^{2 \omega_{ij} } Y_{j} Y_{i}
= q_i ^{2 b_{ij} } Y_{j} Y_{i},
\quad \Omega=\Omega_{t_0},
\quad B=B_{t_0}.
\end{align}
\item
\emph{Quantum mutation}\index{quantum!mutation (of $y$-variable)}.
For $t,\, t'\in \bbT_n$ which are $k$-adjacent,
replace the mutation \eqref{2eq:ymut1} by 
\begin{align}
\label{eq:qymut1}
Y_{i;t'}
&=
\begin{cases}
\displaystyle
Y_{k;t}^{-1}
& i=k,
\\
\displaystyle
q^{ \omega_{ki;t}[ b_{ki;t}]_+}
Y_{i;t} Y_{k;t}^{[ b_{ki;t}]_+}
\\
\quad\times
\prod_{u=1}^{|b_{ki;t}|}
 (1+ 
 q_k^{\sgn(b_{ki;t})(2u-1)}
 Y_{k;t})^{-\sgn(b_{ki;t})}
&i\neq k.
\end{cases}
\end{align}
\end{itemize}
We call
the resulting pattern 
$\Upsilon_q=\{(\bfY_t,B_t)\}_{t\in \bbT_n}$
a \emph{quantum $Y$-pattern}\index{quantum!$Y$-pattern}.
By the specialization $q=1$, it is reduced to the (classical) $Y$-pattern $\Upsilon$.

The following basic example demonstrates the legitimacy
of the above definition.

\begin{ex}[Type $A_2$]
\label{ex:QDIA21}
We use the same convention in Example \ref{ex:typeA23}.
We choose $\Delta=I$ so that
\begin{align}
B = \Omega
=
\begin{pmatrix}
0 & -1
\\
1 & 0
\end{pmatrix}
.
\end{align}
Thus, we have the relation
\begin{align}
Y_1 Y_2 = q^{-2} Y_2 Y_1.
\end{align}
It is equivalent to the relations
\begin{align}
Y_1^{-1}Y_2= q^2Y_2 Y_1^{-1},
\quad
Y_1 Y_2^{-1}=q^2 Y_2^{-1} Y_1 ,
\quad
Y_1^{-1} Y_2^{-1} = q^{-2} Y_2^{-1} Y_1^{-1}.
\end{align}
The mutations of  quantum $Y$-variables,
which are parallel with \eqref{eq:A2mut1}--\eqref{eq:A2mut2},
are calculated as follows:
\begin{align}
  &
  \begin{cases}
 Y_{1;1}=Y_1^{-1},\\ 
   Y_{2;1}=Y_2 (1+ q^{-1}Y_1),
 \end{cases}
 \\
  &
  \begin{cases}
 Y_{1;2}=Y_1^{-1}(1+ q^{-1} Y_2+ Y_1Y_2),\\ 
   Y_{2;2}= Y_2^{-1}(1+ qY_1)^{-1},
 \end{cases}
 \\
  &
  \begin{cases}
 Y_{1;3}=Y_1(1+ qY_2+ q^{2}Y_1Y_2)^{-1},\\ 
   Y_{2;3}=q Y_1^{-1}Y_2^{-1} (1+ q^{-1} Y_2),
 \end{cases}
 \\
  &
  \begin{cases}
 Y_{1;4}=Y_2^{-1},\\ 
   Y_{2;4}=q Y_1Y_2 (1+ qY_2)^{-1},
 \end{cases}
  \\
  &
  \begin{cases}
 Y_{1;5}=Y_2,\\ 
   Y_{2;5}=Y_1.
 \end{cases}
 \end{align}
 In particular, the pentagon periodicity is preserved.
\end{ex}

\section{Fock-Goncharov decomposition}
\label{sec:Fock2}
Let us exhibit the structure of the quantum mutation \eqref{eq:qymut1}
more transparently.
This is done by the Fock-Goncharov decomposition in parallel with the classical case in Section \ref{sec:Fock1}.
From now on, we switch to the picture in a sequence of quantum mutations
starting from the initial seed $\Upsilon_q=(\bfY,B)$,
\begin{align}
\label{eq:qmseq1}
&\Upsilon_q=\Upsilon_q(0) 
\
{\buildrel {k_0} \over \rightarrow}
\
\Upsilon_q(1) 
\
{\buildrel {k_1} \over \rightarrow}
\
\cdots
\
{\buildrel {k_{P-1}} \over \rightarrow}
\
\Upsilon_q(P),
\end{align}
where $\Upsilon_q(s)=(\bfY(s),B(s))$.
Let
\begin{align}
\label{eq:BDO4}
B(s)=\Delta \Omega(s)
\end{align}
be the decomposition corresponding to \eqref{eq:BDO3}.

First, we present the counterpart of the $\varepsilon$-expression \eqref{eq:ymut6}
for the quantum mutation,
where the case $\varepsilon=1$ corresponds to 
the second case of \eqref{eq:qymut1}.

\begin{lem}
[{\cite[Lemma 4.9]{Keller11}}]
\label{lem:qeexp1}
The following expression is independent of $\varepsilon\in \{1, -1\}$:
\begin{align}
\label{eq:eY1}
\begin{split}
&
q^{ \omega_{ki}(s)[ \varepsilon b_{ki}(s)]_+}
Y_{i}(s) Y_{k}(s)^{[ \varepsilon b_{ki}(s)]_+}
\\
& \quad \times
\prod_{u=1}^{|b_{ki}(s)|}
 (1+ 
 q_k^{\varepsilon \sgn(b_{ki}(s))(2u-1)}
 Y_{k}(s)^{\varepsilon})^{-\sgn(b_{ki}(s))}.
 \end{split}
\end{align}
\end{lem}
\begin{proof}
We have
\begin{align*}
\begin{split}
&\ 
q^{\omega_{ki}(s)[  b_{ki}(s)]_+}
Y_{i}(s) Y_{k}(s)^{[  b_{ki}(s)]_+}
\\
&\quad
\times
\prod_{u=1}^{|b_{ki}(s)|}
 (1+ 
 q_k^{  \sgn(b_{ki}(s))(2u-1)}
 Y_{k}(s))^{-\sgn(b_{ki}(s))}
 \\
 = & \ 
q^{\omega_{ki}(s)[  b_{ki}(s)]_+}
Y_{i}(s) Y_{k}(s)^{[  b_{ki}(s)]_+}
q_k^{- b_{ki}(s) b_{ki}(s)}
Y_{k}(s)^{-b_{ki}(s)}
\\
&\quad
\times
\prod_{u=1}^{|b_{ki}(s)|}
 (
 q_k^{-\sgn(b_{ki}(s))(2u-1)}
 Y_{k}(s)^{-1} +1
 )^{-\sgn(b_{ki}(s))}
 \\
= &\ 
q^{\omega_{ki}(s)[ -  b_{ki}(s)]_+}
Y_{i}(s) Y_{k}(s)^{[ - b_{ki}(s)]_+}
\\
&\quad
\times
\prod_{u=1}^{|b_{ki}(s)|}
 (1+ 
 q_k^{ -\sgn(b_{ki}(s))(2u-1)}
 Y_{k}(s)^{-1})^{-\sgn(b_{ki}(s))},
 \end{split}
\end{align*}
where in the last equality
we used \eqref{1eq:pos1} and \eqref{eq:qi1}.
\end{proof}

In the spirit of the classical case in Section \ref{sec:Fock1},
for each $s=0$, \dots, $P$,
we regard the $Y$-variables $Y_i(s)$ 
($i=1$, \dots, $n$) as  noncommutative variables
obeying the $q$-commutative relation
\begin{align}
\label{eq:Ycom2}
Y_{i}(s) Y_{j}(s) = q^{2 \omega_{ij}(s) } Y_{j}(s) Y_{i}(s).
\end{align}

Let $\delta_0$ be the least common multiple of
$\delta_1$, \dots, $\delta_n$.
Let $N=\bbZ^n $ as in Section \ref{sec:Lie1}.
For each $s=0$, \dots, $P$,
let
$\{\bfn,\bfn'\}_{\Omega(s)}$
be the skew-symmetric bilinear form on $N$
defined in the same way as \eqref{eq:Omega2}.
Note that
\begin{align}
\label{eq:bilin1}
\{\bfn,\bfn'\}_{\Omega(s)}\in (1/\delta_0) \bbZ.
\end{align}
We introduce the noncommutative algebra
$\calA_{\Omega(s)}$
with generators $Y(s)^{\bfn}$ ($\bfn\in \bbZ_{\geq 0}^n$)
over $\bbQ(q^{1/\delta_0})$
obeying the relations
\begin{align}
\label{eq:Ycom3}
Y(s)^{\bfn }Y(s)^{\bfn'}=q^{\{\bfn,\bfn'\}_{\Omega(s)}} Y(s)^{\bfn+\bfn' }
= q^{2 \{\bfn,\bfn'\}_{\Omega(s)}}
Y(s)^{\bfn' } Y(s)^{\bfn}.
\end{align}
The product \eqref{eq:Ycom3} is associative because
\begin{align}
\begin{split}
&\
(Y(s)^{\bfn_1}Y(s)^{\bfn_2})Y(s)^{\bfn_3}
=
Y(s)^{\bfn_1}(Y(s)^{\bfn_2}Y(s)^{\bfn_3})
\\
=&\
q^{\{\bfn_1,\bfn_2\}_{\Omega(s)}+
\{\bfn_2,\bfn_3\}_{\Omega(s)}+
\{\bfn_1,\bfn_3\}_{\Omega(s)}}
Y(s)^{\bfn_1+\bfn_2+\bfn_3}.
\end{split}
\end{align}
By setting
$Y(s)^{\bfe_i}=Y_i(s)$, we recover the relation \eqref{eq:Ycom2}.
Moreover, $\calA_{\Omega(s)}$ is  generated 
by $Y_i(s)$ ($i=1$, \dots, $n$).

We say that a (noncommutative) domain $R$, i.e., an associative ring
with unit having no zero divisor other than 0, is an \emph{Ore domain}\index{Ore domain}
if it satisfies the (left) Ore condition:
\begin{align}
a R \cap b R \neq \{0\}
\quad
(a, b\in R\setminus \{0\}).
\end{align}
Any Ore domain is embedded into the 
skew-field of all (right) fractions $ab^{-1}$ with
$a,\, b \in R$, and $b\neq 0$
(e.g., \cite[\S11]{Berenstein05b}).

\begin{lem}
[{e.g., \cite[\S11]{Berenstein05b}}]
\label{lem:Ore1}
The algebra $\calA_{\Omega(s)}$ is an {Ore domain}.
\end{lem}
\begin{proof}
Let $R=\calA_{\Omega(s)}$. Let $R_m$ be the subspace 
spanned by $Y(s)^{\bfn}$ with $\bfn=(n_i)$ such that
$0\leq n_i \leq  m$.
The dimension of
$R_m$ is $(m+1)^n$,
which is a polynomial in $m$.
Suppose that $a R \cap b R = \{0\}$ for some $a,\, b\in R\setminus \{0\}$.
Let $p$ be an integer such that $a,\, b\in R_p$.
Since $aR_m \cap bR_m=\{0\}$,
we have
\begin{align}
\dim R_{m+p}\geq  \dim aR_m + \dim bR_m \geq 2 \dim R_m.
\end{align}
Thus, $\dim R_{m+kp}\geq 2^{k} \dim R_m$.
This contradicts the above linear growth property of $R_m$.
\end{proof}

Let $\calF_{\Omega(s)}$ be the
 skew-field of all fractions
 of $\calA_{\Omega(s)}$.
  In parallel with the mutation $\mu(s)$ in \eqref{eq:mu1},
 we formulate the quantum mutation as a skew-field
 isomorphism
 \begin{align}
\label{eq:qmu1}
 &\mu_q{(s)}\colon \calF_{\Omega(s+1)} \rightarrow \calF_{\Omega(s)}
 \\
\notag
&Y_{i}(s+1)
 \mapsto 
\begin{cases}
\displaystyle
Y_{k_s}(s)^{-1}
& i=k_s,
\\
\displaystyle
Y(s)^{e_i(s) + [\varepsilon b_{k_si}(s)]_+ e_{k_s}(s)}
\\
\displaystyle
\times
\prod_{u=1}^{|b_{k_si}(s)|}
 (1+ 
 q_{k_s}^{\varepsilon \sgn(b_{k_si}(s))(2u-1)}
 Y_{k_s}(s)^{\varepsilon})^{-\sgn(b_{k_si}(s))}
&i\neq k_s,
\end{cases}
\end{align}
where the isomorphism property will be clarified below.
The expression is independent of $\varepsilon \in \{1, -1\}$
thanks to Lemma \ref{lem:qeexp1}.

First, we consider the tropical part corresponding to the map $\tau{(s)}$ in
\eqref{eq:tau1}.
We set $\varepsilon =\varepsilon_s$ in \eqref{eq:qmu1},
where $\varepsilon_s:=\varepsilon_{k_s}(s)$ is the tropical sign.
In parallel with \eqref{eq:tau1},
we introduce  a skew-field isomorphism
\begin{align}
\label{eq:qtau1}
\begin{matrix}
\tau_q{(s)}\colon &\calF_{\Omega(s+1)}& \rightarrow & \calF_{\Omega(s)}
\\
&
Y(s+1)^{\bfn}
&
\mapsto
&
Y(s)^{T(s)(\bfn)},
\end{matrix}
\end{align}
where $T(s)$ is a linear isomorphism defined by
\begin{align}
\label{eq:qTau1}
\begin{matrix}
T{(s)}\colon& \bbZ^n & \rightarrow & \bbZ^n
\\
&
\bfe_i 
&
\mapsto
&
\begin{cases}
\displaystyle
- \bfe_{k_s}
& i=k_s,
\\
\bfe_i + [\varepsilon_s b_{k_s i}(s)]_+ \bfe_k
&i\neq k_s.
\end{cases}
\end{matrix}
\end{align}
Under the specialization $q=1$, the map $\tau_q{(s)}$
 is reduced to the (classical) tropical
transformation $\tau(s)$.
Moreover,  both $\tau_q{(s)}$ and  $\tau{(s)}$ are essentially
the exponential form of
 the common linear map $T{(s)}$.

\begin{lem}
The map $\tau_q(s)$ is a skew-field isomorphism.
\end{lem}
\begin{proof}
By the linearity of $T(s)$,
the desired condition is equivalent to the equality
\begin{align}
\{\bfe_i, \bfe_j\}_{\Omega(s+1)}
= \{T(s)(\bfe_i), T(s)(\bfe_j)\}_{\Omega(s)}.
\end{align}
This is the mutation formula of $\Omega(s)$ in \eqref{2eq:omut1}.
\end{proof}

Next, 
in parallel with \eqref{eq:rho1}, we introduce
 a skew-field automorphism
 \begin{gather}
\label{eq:qrho1}
 \begin{matrix}
\rho_q{(s)} \colon & \calF_{\Omega(s)} & \rightarrow &  \calF_{\Omega(s)}
\\
& Y(s)^{\bfn}& \mapsto & 
\displaystyle
Y(s)^{\bfn}
\prod_{u=1}^{|\alpha|}
 (1+ 
 q_{k_s}^{\varepsilon_s \sgn(\alpha)(2u-1)}
 Y_{k_s}(s)^{\varepsilon_s})^{-\sgn(\alpha)},
\end{matrix}
\end{gather}
where $\alpha=\{ \delta_{k_s} \bfe_{k_s}, \bfn\}_{\Omega(s)}\in \bbZ$.
Note that, for $\bfn = \bfe_i$, we have $\alpha=b_{k_s i}(s)$.

\begin{lem}
\label{lem:rqauto1}
The  map $\rho_q{(s)}$ is  
a skew-field automorphism.
\end{lem}
\begin{proof}
Let us check that
\begin{align}
\label{eq:qqY2}
\rho_q{(s)}(Y(s)^{\bfn})
\rho_q{(s)}(Y(s)^{\bfn'})
=q^{\{\bfn,\bfn'\}_{\Omega(s)}}
\rho_q{(s)}(Y(s)^{\bfn + \bfn'}).
\end{align}
Let $\alpha=\{\delta_{k_s} \bfe_{k_s}, \bfn\}_{\Omega(s)}$ and
$\alpha'=\{\delta_{k_s}\bfe_{k_s}, \bfn'\}_{\Omega(s)}$.
Then, 
\begin{align*}
\begin{split}
&\
\rho_q{(s)}(Y(s)^{\bfn})
\rho_q{(s)}(Y(s)^{\bfn'})
\\
=&\ 
Y(s)^{\bfn}
\prod_{u=1}^{|\alpha|}
 (1+ 
 q_{k_s}^{\varepsilon_s \sgn(\alpha)(2u-1)}
 Y_{k_s}(s)^{\varepsilon_s})^{-\sgn(\alpha)}
\\
&\ \quad
\times
Y(s)^{\bfn'}
\prod_{u=1}^{|\alpha'|}
 (1+ 
 q_{k_s}^{\varepsilon_s \sgn(\alpha')(2u-1)}
 Y_{k_s}(s)^{\varepsilon_s})^{-\sgn(\alpha')},
\\
=&\ 
q^{\{\bfn,\bfn'\}_{\Omega(s)}}
Y(s)^{\bfn + \bfn'}
\prod_{u=1}^{|\alpha|}
 (1+ 
 q_{k_s}^{\varepsilon_s \sgn(\alpha)(2u-1)}
 q_{k_s}^{2 \varepsilon_s \alpha'}
 Y_{k_s}(s)^{\varepsilon_s})^{-\sgn(\alpha)}
\\
&\ \quad
\times
\prod_{u=1}^{|\alpha'|}
 (1+ 
 q_{k_s}^{\varepsilon_s \sgn(\alpha')(2u-1)}
 Y_{k_s}(s)^{\varepsilon_s})^{-\sgn(\alpha')}.
\end{split}
\end{align*}
The last expression is equal to the RHS of \eqref{eq:qqY2}.
\end{proof}

By Lemma \ref{lem:qeexp1}, we have the \emph{Fock-Goncharov decomposition}\index{Fock-Goncharov decomposition!of quantum mutation} of the quantum mutation $\mu_q{(s)}$
\begin{align}
\label{eq:qFG1}
\mu_q{(s)}=
\rho_q{(s)}\circ \tau_q{(s)}.
\end{align}
In particular, $\mu_q{(s)}$ is a skew-field isomorphism.

In parallel with \eqref{eq:taus01}, we introduce the composition
\begin{align}
\label{eq:qtaus01}
\tau_q(s;0)&:=
\tau_q(0)\circ \tau_q(1)\circ \cdots \circ \tau_q(s):
\calF_{\Omega(s+1)}\rightarrow \calF_{\Omega(0)}.
\end{align}

We have an analog of Proposition \ref{prop:tauy1}.
\begin{prop}
\label{prop:qtauy1}
The following formula holds:
\begin{align}
\label{eq:qtaus1}
\tau_q(s;0)(Y_i(s+1))=Y^{\bfc_i(s+1)}.
\end{align}
\end{prop}
\begin{proof}
The proof of Proposition \ref{prop:tauy1} is applicable.
\end{proof}

Finally, 
in parallel with \eqref{eq:fq1}, 
 we introduce  a skew-field automorphism
\begin{align}
\label{eq:qfq1}
\begin{matrix}
\frakq_q{(s)} \colon & \calF_{\Omega} & \rightarrow &  \calF_{\Omega}
\\
& Y^{\bfn}& \mapsto & 
\displaystyle
Y^{\bfn}
\prod_{u=1}^{|\alpha|}
 (1+ 
 q_{k_s}^{\varepsilon_s \sgn(\alpha)(2u-1)}
Y^{\bfc^+_{k_s}(s)})^{-\sgn(\alpha)},
\end{matrix}
\end{align}
where $\alpha=\{\delta_{k_s} \bfc_{k_s}(s), \bfn\}_{\Omega}\in \bbZ$.
Note that $\frakq_q{(0)}=\rho_q(0)$.

We have an analog of 
Proposition \ref{prop:qs1}.
\begin{prop}
\label{prop:qcom1}
The following commutative diagram holds:
\begin{align}
\label{eq:qcd1}
\raisebox{25pt}
{
\begin{xy}
(0,0)*+{\calF_{\Omega(s)}}="aa";
(25,0)*+{\calF_{\Omega}}="ba";
(0,-15)*+{\calF_{\Omega(s)}}="ab";
(25,-15)*+{\calF_{\Omega}}="bb";
{\ar "aa";"ba"};
{\ar "ab";"bb"};
{\ar "aa";"ab"};
{\ar "ba";"bb"};
(12,3)*+{\text{\small $\tau_q(s-1;0)$}};
(12,-12)*+{\text{\small $\tau_q(s-1;0)$}};
(-5,-7.5)*+{\text{\small $\rho_q(s)$}};
(30,-7.5)*+{\text{\small $\frakq_q(s)$}};
\end{xy}
}
\end{align}
In particular, $\frakq_q(s)$ is a skew-field automorphism.
\end{prop}
\begin{proof}
The proof is the same as Proposition \ref{prop:qs1}.
\end{proof}

We introduce  the  composite mutations
\begin{align}
\label{eq:qmus01}
\mu_q(s;0)&:=
\mu_q(0)\circ \mu_q(1)\circ \cdots \circ \mu_q(s)\colon
\calF_{\Omega(s+1)} \rightarrow \calF_{\Omega},
\\
\label{eq:qq01}
\frakq_q(s;0)&:=
\frakq_q(0)\circ \frakq(1)\circ \cdots \circ \frakq_q(s)\colon
\calF_{\Omega}\rightarrow \calF_{\Omega}.
\end{align}
In parallel with \eqref{eq:decom1}, we have 
the  \emph{Fock-Goncharov decomposition}\index{Fock-Goncharov decomposition!of composite quantum mutation} of $\mu(s;0)$:
\begin{align}
\label{eq:qdecom1}
\mu_q(s;0)= \frakq_q(s;0) \circ \tau_q(s;0).
\end{align}

Let $\overline \calA_{\Omega} $ be the completion of 
the noncommutative algebra $\calA_{\Omega}$
with respect to $\deg{\bfn}$.
We see from  \eqref{eq:qfq1} that
$\frakq_q(s)$ acts also on $\overline \calA_{\Omega}$.
Thus, one can
regard $\frakq_q(s)$ as an automorphism of $\overline \calA_{\Omega}$, whenever preferred.

\section{Quantum mutations and quantum dilogarithm}

As first clarified by Fock and Goncharov \cite{Fock03},
the quantum mutation $\rho_q(s)$ is closely related
to  the quantum dilogarithm $\Psi_q(x)$ in \eqref{eq:qd0}.
For $\varepsilon\in \{1, -1\}$,
the \emph{adjoint action}\index{adjoint action} of $\Psi_{q_k}(Y_k(s)^{\varepsilon})$ on $\calF_{\Omega(s)}$
is defined by
\begin{align}
\begin{matrix}
\rmAd[\Psi_{q_k}(Y_k(s)^{\varepsilon})]
:
& \calF_{\Omega(s)} &\rightarrow & \calF_{\Omega(s)}
\\
& Y(s)^{\bfn} & \mapsto & \Psi_{q_k}(Y_k(s)^{\varepsilon}) Y(s)^{\bfn} \Psi_{q_k} (Y_k(s)^{\varepsilon})^{-1}.
\end{matrix}
\end{align}
Even though $ \Psi_{q_k}(Y_k (s)^{\varepsilon})$ itself does not belong to $\calF_{\Omega(s)}$,
the  map is well defined as shown below;
therefore, it yields a skew-field automorphism of $\calF_{\Omega(s)}$.

\begin{lem}
[{\cite[Lemma 3.4]{Fock03}}]
\label{lem:AdPsi1}
For any $\varepsilon\in \{1, -1\}$, we have
\begin{align}
\label{eq:AdPsi1}
\begin{split}
&\
\rmAd[\Psi_{q_k}(Y_{k}(s)^{\varepsilon})^{\varepsilon}](Y(s)^{\bfn})
=
Y(s)^{\bfn}
\prod_{u=1}^{|\alpha|}
 (1+ 
 q_k^{\varepsilon \sgn(\alpha)(2u-1)}
 Y_{k}(s)^{\varepsilon})^{\sgn(\alpha)},
 \end{split}
 \end{align}
 where $\alpha=\{\delta_ke_k(s), \bfn\}_{\Omega(s)}\in \bbZ$.
\end{lem}
\begin{proof}
By \eqref{eq:Ycom3}, we have
\begin{align}
 Y_{k}(s)^{\varepsilon} Y(s)^{\bfn}
 = q_k ^{  2  \varepsilon \alpha }  Y(s)^{\bfn} Y_{k}(s)^{\varepsilon} .
\end{align}
It follows that
\begin{align}
\label{eq:YPsiY1}
\Psi_{q_k}(Y_{k}(s)^{\varepsilon})^{\varepsilon} 
Y(s)^{\bfn}
 =
 Y(s)^{\bfn}
\Psi_{q_k}(
q_k ^{2   \varepsilon  \alpha } 
 Y_{k}(s)^{\varepsilon})^{\varepsilon}
  .
\end{align}
Here, we temporarily consider the product in the completion of $\calF_{\Omega(s)}$
with respect to $ Y_{k}(s)^{\varepsilon}$.
Meanwhile,
by repeatedly applying the difference relation \eqref{eq:qdrec1},
we have
\begin{align}
\label{eq:Psidif1}
\begin{split}
&\
\Psi_{q_k}(
q_k ^{  2 \varepsilon \alpha } Y_{k}(s)^{\varepsilon})^{\varepsilon}
\\
= &\ 
(1+q_k^{ \varepsilon  \sgn(\alpha)(2|\alpha|-1)}Y_k(s)^{\varepsilon})^{\sgn(\alpha)}
\Psi_{q_k}(
q_k ^{  2 \varepsilon \sgn(\alpha)(|\alpha|-1) } Y_{k}(s)^{\varepsilon})^{\varepsilon}
\\
= &\ 
\biggl(\, \prod_{u=1}^{|\alpha|}
(1+q_k^{\varepsilon  \sgn(a)(2u-1)}Y_k(s)^{\varepsilon})^{\sgn(a)}
\biggr)
\Psi_{q_k}(
 Y_{k}(s)^{\varepsilon})^{\varepsilon}.
\end{split}
\end{align}
By \eqref{eq:YPsiY1} and \eqref{eq:Psidif1},
we obtain the equality \eqref{eq:AdPsi1}.
\end{proof}

We have the following identification of automorphisms.
\begin{prop}
[{\cite[Lemma 3.4]{Fock03}}]
\label{prop:qrhoAd1}
We have
\begin{align}
\label{eq:qrhoAd1}
\rho_q(s)&=
\rmAd[\Psi_{q_{k_s}}(Y_{k_s}(s)^{\varepsilon_s})^{-\varepsilon_s}],
\\
\label{eq:qqAd1}
\frakq_q(s)&=
\rmAd[\Psi_{q_{k_s}}(Y^{\bfc^+_{k_s}(s)})^{-\varepsilon_s}].
\end{align}
\end{prop}
\begin{proof}
By setting $\varepsilon=\varepsilon_s$ and $k=k_s$
  in
 \eqref{eq:AdPsi1}
 and comparing it with  \eqref{eq:qrho1},
 we obtain the equality \eqref{eq:qrhoAd1}.
 The equality \eqref{eq:qqAd1} follows from
 \eqref{eq:qrhoAd1} and Proposition \ref{prop:qtauy1}.
 \end{proof}

By combining \eqref{eq:qFG1} and \eqref{eq:qrhoAd1}, we have
\begin{align}
\label{eq:qFG2}
\mu_q{(s)}=
\rmAd[\Psi_{q_{k_s}}(Y_{k_s}(s)^{\varepsilon_s})^{-\varepsilon_s}] \circ \tau_q{(s)}.
\end{align}
In fact, this is the original definition of the quantum mutation by \cite{Fock03}
(in our convention),
where $\varepsilon_s$ was set to 1 therein.

Let us summarize the correspondence between the classical and quantum mutations.
\begin{enumerate}
\item
The Poisson bracket \eqref{eq:Poi4} is equivalently written as
\begin{align}
\{y(s)^{\bfn}, y(s)^{\bfn'}\}_s = 
\{\bfn,\bfn'\}_{\Omega(s)} y(s)^{\bfn+\bfn'}.
\end{align}
This is replaced with
the noncommutative product \eqref{eq:Ycom3}.

\item
The time-one flow
of the Hamiltonian
$- {\varepsilon_s}{\d_{k_s}}(- \mathrm{Li}_2(-y_{k_s}(s)^{\varepsilon_s}))$
in \eqref{eq:Hs1}
is replaced with
the map $\rmAd[\Psi_{q_{k_s}}(Y_{k_s}(s)^{\varepsilon_s})^{-\varepsilon_s}]$.
See  \eqref{eq:asymd1}
for the account of the minus signs for $- \mathrm{Li}_2(-x)$ in the above.
\end{enumerate}

\section{Algebraic formulation of quantum mutations}
\label{sec:algebraic1}

As given in \eqref{eq:dm1},
the description of the mutations $\frakq(s)$ by
the Poisson bracket and the Hamiltonian
was further replaced with 
the action of the dilogarithm elements in the group $G=G_{\Omega}$.
It is possible to do it in the quantum case $\frakq_q(s)$ as well;
moreover, it provides a more intrinsic formulation of the quantum mutations
as in the classical case.
Throughout the section, we set
\begin{align}
\Omega=\Omega(0).
\end{align}

\subsection{Exponential group}
We first construct the exponential group $G_q=G_{q,\Omega}$ corresponding to $G=G_{\Omega}$ in
\eqref{eq:exG1}.
This is straightforward because the group $G_q$ is still in the same class of groups
constructed in \cite{Kontsevich13}.
Namely, in the construction of $G$ in Sections \ref{sec:Lie1} and \ref{sec:exponential1}, 
we only need to replace
the Lie bracket in \eqref{eq:Xcom1} with
\begin{align}
\label{eq:Xcom2}
[X_{\bfn}, X_{\bfn'}]=[\{\bfn,\bfn'\}_{\Omega}]_q X_{\bfn+\bfn'},
\end{align}
where $[a]_q$ is the $q$-number in \eqref{eq:qnum1}
extended to any rational number $a$.
By \eqref{eq:bilin1}, $[\{\bfn,\bfn'\}_{\Omega}]_q  \in \bbQ(q^{1/\delta_0})$.
Accordingly, we  set the ground field
of the corresponding Lie algebra
$\frakg_q$  as $\bbQ(q^{1/\delta_0})$.

\subsection{$y$-representation}
We consider an analog of the $y$-representa\-tion $\rho_y$ in
Section \ref{sec:y-representation1}.
We first replace the representation space $\bbQ[[\bfy]]$
to the noncommutative algebra $\overline \calA_{\Omega}$,
which was
introduced in the end of Section \ref{sec:Fock2}.
Then, in parallel with \eqref{eq:tildeX1},  we define the action of $X_{\bfn}\in \frakg_q$ by
\begin{align}
\label{eq:qtildeX1}
\tilde X_{\bfn}(Y^{\bfn'}):=[ \{\bfn,\bfn'\}_{\Omega}]_q Y^{\bfn'+\bfn}
=\frac{q^{2\{\bfn,\bfn'\}_{\Omega}}-1}{q-q^{-1}} Y^{\bfn'} Y^{\bfn}.
\end{align}
It is easy to check that this is an action of $\frakg_q$
and also a derivation on $\overline \calA_{\Omega}$.
Thus, we have a group homomorphism \cite{Nakanishi22b}
\begin{align}
\label{eq:qXn2}
\begin{matrix}
{\rp \rho_{y}}\colon&  G_q &\rightarrow & \mathrm{Aut}(\overline \calA_{\Omega}) \\
& \exp(X) &  \mapsto &\mathrm{Exp} (\tilde{X}),
\end{matrix}
\end{align}
which is parallel with the one in \eqref{eq:Xn2}.
We call it the \emph{$y$-representation}\index{$y$-representation} of $G_q$.
It is injective if $\Omega$ is nonsingular.

\subsection{Quantum dilogarithm elements}

Following \cite{Nakanishi22b},
we introduce
an analog of the dilogarithm elements in
\eqref{3eq:gei1}
with some extra parameters related to $q$.

\begin{defn}[Quantum Dilogarithm element]
\label{3defn:qdiloge1}
For any $\bfn\in N^+$,
$a\in (1/\delta_0)\bbZ_{>0}$, and
$b\in (1/\delta_0)\bbZ$,
we define
\begin{align}
\label{3eq:qgei1}
\Psi_{a,b}[\bfn]:=\exp
\Biggl(\,
\sum_{j=1}^{\infty} \frac{(-1)^{j+1}}{j [ja]_q} q^{jb}X_{j \bfn}
\Biggr)
\in G_q,
\end{align}
We call $\Psi_{a,b}[\bfn]$ the \emph{quantum dilogarithm element}\index{quantum dilogarithm!element} for $\bfn$
with \emph{quantum data}\index{quantum!data} $a$, $b$.
We also write $\Psi_{a,0}[\bfn]$ as $\Psi_a[\bfn]$, for simplicity.
\end{defn}
Roughly speaking, the parameters $a$ and $b$
control the \emph{interval} and the \emph{shift} 
of the powers of $q$
in \eqref{3eq:qgei1}, respectively.
In contrast to the asymptotic behavior  \eqref{eq:asymd1} of the quantum dilogarithm $\Psi_q(x)$,
the element
$\Psi_{a,b}[\bfn]\in G_q$ is simply  reduced to a power of the classical one $\Psi[\bfn]\in G$
in the limit $ q\rightarrow 1$ as
\begin{align}
\label{eq:QDEtoDE1}
\lim_{q\rightarrow 1} 
\Psi_{a,b}[\bfn]=
\Psi[\bfn]^{1/a}.
\end{align}
Thus, the parameter $a$ is regarded as the ``inverse exponent'',
though the above factorization does not occur when $q\neq 1$.

By \eqref{eq:qtildeX1},
the  element $\Psi_{a,b}[\bfn]$ acts on $\overline\calA_{\Omega}$
under the $y$-representation as
\begin{align}
\label{eq:qPsiact1}
\Psi_{a,b}[\bfn](Y^{\bfn'})
=
Y^{\bfn'}
\exp
\biggl(
\sum_{j=1}^{\infty} 
\frac{q^{2j\{\bfn,\bfn'\}_{\Omega}}-1}{q^{2ja}-1}
\frac{(-1)^{j+1}}{j} q^{ja}q^{jb}Y^{j\bfn}
\biggr).
\end{align}
In particular, we have the following formula.
\begin{prop}
[{\cite[\S2.3]{Nakanishi22b}}]
\label{prop:qdmut1}
Under the $y$-representation of $G_q$, we have
\begin{align}
\label{eq:qPsiact2}
\Psi_{1/\delta_{k_s}}[\bfc^+_{k_s}(s)]^{\varepsilon_s} (Y^{\bfn})
=Y^{\bfn}
\prod_{u=1}^{|\alpha|}
(1+ q_{k_s}^{\varepsilon_s \sgn(\alpha)(2u-1)}Y^{\bfc^+_{k_s}(s)})^{\sgn(\alpha)},
\end{align}
where $\alpha=\{\delta_{k_s} \bfc_{k_s}(s), \bfn\}_{\Omega}\in \bbZ$.
\end{prop}
\begin{proof}
We set $a=1/\delta_{k_s}$, $b=0$, $\bfn=\bfc^+_{k_s}(s)$, and $\bfn'=\bfn$
in \eqref{eq:qPsiact1}.
We have
\begin{align}
\frac{q^{2j\{\bfn,\bfn'\}_{\Omega}}-1}{q^{2ja}-1}
=
\frac{q^{2j\varepsilon_{s} \alpha/\delta_{k_s}}-1}{q^{2j/\delta_{k_s}}-1}
=
\begin{cases}
\displaystyle
\sum_{u=1}^{|\alpha|} q^{2j \varepsilon_{s}  (u-1)/\delta_{k_s}}
& \varepsilon_{s}  \alpha >0,
\\
0
&   \alpha = 0,
\\
\displaystyle
-\sum_{u=1}^{|\alpha|} q^{-2j \varepsilon_{s}  u/\delta_{k_s}}
&
\varepsilon_{s}  \alpha < 0.
\end{cases}
\end{align}
Thus, we obtain the equality \eqref{eq:qPsiact2}.
\end{proof}

The adjoint action of the quantum dilogarithm in \eqref{eq:qqAd1} is 
 replaced with the action of the quantum dilogarithm elements.
\begin{prop}
\label{prop:qqs2}
As an automorphism of $\overline \calA_{\Omega}$,
we have
\begin{align}
\label{eq:qqPhi1}
\frakq_q(s)=
{\rp \rho_y}(
\Psi_{1/\delta_{k_s}}[\bfc^+_{k_s}(s)]^{-\varepsilon_s}).
\end{align}
\end{prop}
\begin{proof}
This is immediately obtained 
by comparing \eqref{eq:qfq1} and \eqref{eq:qPsiact2}.
\end{proof}

\subsection{Pentagon relation}
Let us present the pentagon relation among
quantum dilogarithm elements $\Psi_{a,b}[\bfn]$.

We first note a useful formula.
In \eqref{eq:qPsiact1}, suppose that $|\{\bfn, \bfn'\}_{\Omega}|=a$,
where $a\in (1/\delta_0)\bbZ_{>0}$ is the one for $\Psi_{a,b}[\bfn]$ therein.
Then, we have
\begin{align}
\frac{q^{2j\{\bfn,\bfn'\}_{\Omega}}-1}{q^{2ja}-1}
=
\begin{cases}
1 & \{\bfn, \bfn'\}_{\Omega} =a,\\
-q^{-2ja} &  \{\bfn, \bfn'\}_{\Omega} = -a.
\end{cases}
\end{align}
Thus, we obtain
\begin{align}
\label{eq:qPsilem1}
\Psi_{a,b}[\bfn](Y^{\bfn'})
=
\begin{cases}
Y^{\bfn'} (1+ q^{a+b} Y^{\bfn})
&
\{\bfn, \bfn'\}_{\Omega} =a,
\\
Y^{\bfn'} (1+ q^{-a+b} Y^{\bfn})^{-1}
&
\{\bfn, \bfn'\}_{\Omega} =-a.
\end{cases}
\end{align}

We have an analog of Proposition \ref{3prop:pent1}.
\begin{prop}
[{\cite[Theorem 2.6]{Nakanishi22b}}]
\label{prop:qPsipent2}
Let $\bfn_1,\, \bfn_2\in N^+$.
The following relations hold in $G_q$.
\par
(a) (Commutative relation). If $\{\bfn_2,\bfn_1\}_{\Omega}=0$,
then, for any $a_i$ and $b_i$, we have
\begin{align}
\label{eq:Psicom1}
\Psi_{a_2,b_2}[\bfn_2] \Psi_{a_1,b_1}[ \bfn_1] =\Psi_{a_1,b_1}[  \bfn_1 ]  \Psi_{a_2,b_2}[ \bfn_2].
\end{align}

(b) (Pentagon relation).
If $\{\bfn_2,\bfn_1\}_{\Omega}=c$ $(c\in (1/\delta_0) \bbZ_{>0})$,
then, for any $b_i$,
we have
\begin{align}
\label{eq:Psipent2}
\Psi_{c,b_2}[\bfn_2 ] \Psi_{c,b_1}[ \bfn_1]=
\Psi_{c,b_1}[  \bfn_1 ] \Psi_{c,b_1+b_2}[\bfn_1+\bfn_2] \Psi_{c,b_2}[ \bfn_2].
\end{align}
\end{prop}
\begin{proof}
(a). This is clear by \eqref{eq:Xcom2}.
(b).
By the same argument in the proof of Proposition 
\ref{3prop:pent1},
it is enough to prove the equality 
for the action on $Y^{\bfn_1}$ and $Y^{\bfn_2}$
under the $y$-representation.
Let us consider the former case.
Below, we only use  \eqref{eq:qPsilem1}
and the relation $Y^{\bfn_1}Y^{\bfn_2}=q^{-2c}Y^{\bfn_2}Y^{\bfn_1}$.
We have
\begin{align}
\begin{split}
\Psi_{c,b_2}[\bfn_2 ] \Psi_{c,b_1}[ \bfn_1](Y^{\bfn_1})
&= \Psi_{c,b_2}[\bfn_2 ] (Y^{\bfn_1})
\\
&=Y^{\bfn_1}(1+q^{c+b_2}Y^{\bfn_2}).
\end{split}
\end{align}
On the other hand, we have
\begin{align*}
&\
\Psi_{c,b_1}[  \bfn_1 ] \Psi_{c,b_1+b_2}[\bfn_1+\bfn_2] \Psi_{c,b_2}[ \bfn_2]
(Y^{\bfn_1})
\notag
\\
=&\
\Psi_{c,b_1}[  \bfn_1 ] \Psi_{c,b_1+b_2}[\bfn_1+\bfn_2]
(Y^{\bfn_1}(1+q^{c+b_2}Y^{\bfn_2}))
\notag
\\
=&\
\Psi_{c,b_1}[  \bfn_1 ] (
Y^{\bfn_1}(1+q^{c+b_1+b_2}Y^{\bfn_1+\bfn_2})
\notag
\\
&\quad \times
(1+q^{c+b_2}Y^{\bfn_2}
(1+q^{-c+b_1+b_2}Y^{\bfn_1+\bfn_2})^{-1}
)
)
\notag
\\
=&\
\Psi_{c,b_1}[  \bfn_1 ] (
Y^{\bfn_1}(1+q^{c+b_2}Y^{\bfn_2}
 + q^{c+b_1 + b_2}Y^{\bfn_1+ \bfn_2})
)
\\
=&\
Y^{n_1}(1+q^{c+b_2}Y^{\bfn_2}
(1+q^{-c+b_1}Y^{\bfn_1})^{-1}
\notag
\\
&\quad
 + q^{c+b_1 + b_2}Y^{\bfn_1+ \bfn_2}
 (1+q^{-c+b_1}Y^{\bfn_1})^{-1})
 \notag
 \\
= &\  Y^{\bfn_1}(1+q^{c+b_2}Y^{\bfn_2}).
\notag
\end{align*}
Thus, they coincide. The other case is similar.
\end{proof}

\section{Quantization of principally extended cluster patterns}
\label{sec:quantumx1}

Even though our  targets are quantum dilogarithm identities,
where only quantum $y$-variables are involved,
we  also need to work  with  quantum $x$-variables
for the proof of Theorem \ref{thm:qsync2}.
The reader can safely skip the content of this section by accepting Theorem \ref{thm:qsync2}.

  \subsection{Quantum mutations of principally extended $x$-variables}

For the mutation sequence \eqref{eq:mseq1},
let $C(s)$ and $G(s)$ ($s=0,\, \dots,\, P$)
be the corresponding $C$- and $G$-matrices, respectively.
Let $\Lambda(s)$ $(s=0,\, \dots,\, P)$  be the $2n\times 2n$ skew-symmetric matrix defined by
 \begin{align}
 \label{eq:Lam1}
 \Lambda(s) =
 \begin{pmatrix}
 O & -G(s)^T D\\
 D G(s) & - \Omega
 \end{pmatrix},
 \quad
 \Omega=\Omega(0),
 \quad
 D= \Delta^{-1}.
 \end{align}
 Also, let $ \tilde B(s)$ $(s=0,\, \dots,\, P)$ and $\tilde D$ be the $2n \times n$ matrices defined by
 \begin{align}
 \label{eq:tB1}
 \tilde B(s) =
 \begin{pmatrix}
 B(s)\\
C(s) 
 \end{pmatrix}
 ,
 \quad
 \tilde D=
  \begin{pmatrix}
 D\\
 O
 \end{pmatrix}.
 \end{align}
By the dualities \eqref{2eq:dual0} and \eqref{eq:dual2}, they satisfy the matrix relation
 \begin{align}
 \label{eq:compatible1}
 -\Lambda(s) \tilde B(s) =  \tilde D.
 \end{align}
Such a  pair $(\Lambda(s), \tilde B(s))$
 is called a \emph{compatible pair}\index{compatible!pair} in \cite{Berenstein05b}.
 It follows that
 \begin{align}
 \label{eq:BLB1}
  \tilde B(s)^T \Lambda(s)\tilde B(s) = \Omega(s).
 \end{align}
{\rp  Also, by \eqref{2eq:gmut1},
 we have the mutation formula
 \begin{align}
 \label{eq:Lmut1}
 \Lambda(s+1)=\tilde Q(s)^T\Lambda(s) \tilde Q(s),
 \end{align}
 where $\tilde Q(s)=(q_{ij}(s))$ is the $2n\times 2n$ unimodular matrix defined by
 \begin{align}
 q_{ij}(s)=
 \begin{cases}
 \delta_{ij}
 & j\neq k_s,
 \\
 -1
 &
 i=j=k_s,
 \\
[- \ve_s b_{ik_s}(s)]_+
&
i\neq j = k_s.
 \end{cases}
 \end{align}
 For the above $\Lambda(s)=(\lambda_{ij}(s))$,
 let $\delta'_0$ be the smallest positive integer such that
 \begin{align}
 \sum_{i,j=1}^{2n} \bbZ \lambda_{ij}(s) = (1/\delta'_0) \bbZ.
 \end{align}
 Due to the condition \eqref{eq:compatible1}, 
 $\delta'_0$ is a multiple of $\delta_0$ in \eqref{eq:bilin1}.
 Also, it is independent of $s$, thanks to \eqref{eq:Lmut1} and the unimodularity of $\tilde Q(s)$.}
 
 For each $s=0$, \dots, $P$, we introduce the \emph{principally extended $x$-variables}\index{principally extended $x$-variable}
 $\tilde \bfX(s)=(X_1(s), \dots, X_{2n}(s))$ as  noncommutative variables obeying the $q$-commutative
 relation
 \begin{align}
 \label{eq:qXcom1}
 X_i(s) X_j(s) = q^{2\lambda_{ij}(s)} X_j(s)X_i(s),
 \end{align}
  where $X_i(s)$ ($i=n+1$, \dots, $2n$) are \emph{frozen variables}\index{frozen variable}; namely,
  they are never mutated.
  Meanwhile, the unfrozen variables $X_i(s)$ ($i=1$, \dots, $n$) are actually commutative with each other
  by \eqref{eq:Lam1}.
 For $\tilde \bfm =(m_1,\dots, m_{2n}) \in \bbZ^{2n}$,
we define
 \begin{align}
  X(s)^{\tilde \bfm}=q^{\sum_{i<j} \lambda_{ij}(s) m_i m_j} X_1(s)^{m_1} \cdots X_{2n} (s)^{m_{2n}}.
 \end{align}
 Then, the quantum mutation 
  in \cite[Eq.~(4.22)]{Berenstein05b}
is written  as follows:
 \begin{align}
\label{eq:qxmu1}
\begin{matrix}
X_i(s+1)
=
\begin{cases}
\displaystyle
X(s)^{- \bfe_{k_s} + \sum_{j=1}^{2n}  [- \varepsilon_s \tilde b_{jk_s}(s)]_+ \bfe_j }
\\
\quad
+
X(s)^{- \bfe_{k_s} + \sum_{j=1}^{2n}  [\varepsilon_s \tilde b_{jk_s}(s)]_+ \bfe_j }
&i= k_s,
\\
\displaystyle
X_{i}(s)
& i\neq k_s,
\end{cases}
\end{matrix}
\end{align}
where $1\leq k_s \leq n$ and  $\varepsilon_s=\varepsilon_{k_s}(s)$ is the tropical sign as usual.
In Section \ref{subsubsec:FG1}, we will show that the relation \eqref{eq:qXcom1}
is comptible with \eqref{eq:qxmu1}.
Due to the sign-coherence of the $C$-matrices, we have
\begin{align}
\label{eq:signco1}
 \sum_{j=1}^{2n}  [- \varepsilon_s \tilde b_{jk_s}(s)]_+ \bfe_j 
=
 \sum_{j=1}^{n}  [- \varepsilon_s \tilde b_{jk_s}(s)]_+ \bfe_j,
\end{align}
but this is not true for the opposite sign $-\varepsilon_s$.

 For each $s=0$, \dots, $P$,
 the noncommutative variables
 \begin{align}
 \label{eq:qhaty1}
 \hat Y_i (s) =X(s)^{\sum_{j=1}^{2n}   \tilde b_{ji}(s)  \bfe_j }
\quad
( i=1,\, \dots,\, n)
 \end{align}
are  identified with the quantum analogs of the
the $\hat y$-variables $\hat y_i(s)$ with principal coefficients
in Section \ref{sec:tropicalization1}.
By \eqref{eq:compatible1}
and \eqref{eq:BLB1},
we have  relations
 \begin{align}
 \label{eq:qXYrel1}
X_i (s)\hat Y_j(s) &= q_j^{-2 \delta_{ij}}\hat Y_j(s) X_i(s)
\quad
(i=1,\, \dots,\, 2n;\, j=1,\, \dots,\, n),
\\
\label{eq:qhY1}
\hat Y_i (s)  \hat Y_j (s)&= q^{2 \omega_{ij}(s)}  \hat Y_j(s)   \hat Y_i(s)
\quad
 (i,\, j=1,\, \dots,\,  n).
 \end{align}
 Note that the relation \eqref{eq:qhY1} is
 the same as the one \eqref{eq:Ycom1} for $Y_i(s)$'s.
 By \eqref{eq:qXYrel1}, the first case of 
 \eqref{eq:qxmu1} is also  written as
 \begin{align}
 \label{eq:qxmu2}
 X_{k_s}(s+1)=
 X(s)^{- \bfe_{k_s} + \sum_{j=1}^{2n}  [- \varepsilon_s \tilde b_{jk_s}(s)]_+ \bfe_j }
 (1+q_{k_s}^{- \varepsilon_s} \hat Y_{k_s}(s)^{\varepsilon_s}).
 \end{align}
 This is an analog of the first case of  \eqref{2eq:xmut5}.
  
  \begin{ex}[Type $A_2$]
\label{ex:QDIA22}
We use the same convention in Example \ref{ex:QDIA21}.
Let $X^{\tilde  \bfm}=X_{t_0}^{\tilde  \bfm}$.
Then, we have the following result.
\begin{align}
  &
  \begin{cases}
 X_{1;1}=X^{(-1,0,0,0)} + X^{(-1,1,1,0)} ,\\ 
   X_{2;1}=X^{(0,1,0,0)} ,
 \end{cases}
 \allowdisplaybreaks
 \\
  &
  \begin{cases}
 X_{1;2}=X^{(-1,0,0,0)} + X^{(-1,1,1,0)},\\ 
   X_{2;2}= X^{(0,-1,0,0)} + X^{(-1,-1,0,1)}+X^{(-1,0,1,1)},
 \end{cases}
  \allowdisplaybreaks
 \\
  &
  \begin{cases}
 X_{1;3}= X^{(1,-1,0,0)} + X^{(0,-1,0,1)},\\ 
   X_{2;3} =X^{(0,-1,0,0)} + X^{(-1,-1,0,1)}+X^{(-1,0,1,1)},
 \end{cases}
  \allowdisplaybreaks
 \\
  &
  \begin{cases}
 X_{1;4}=X^{(1,-1,0,0)} + X^{(0,-1,0,1)},\\ 
   X_{2;4}=X^{(1,0,0,0)},
 \end{cases}
  \allowdisplaybreaks
  \\
  &
  \begin{cases}
 X_{1;5}=X^{(0,1,0,0)}=X_2,\\ 
   X_{2;5}=X^{(1,0,0,0)}=X_1.
 \end{cases}
 \end{align}
 The pentagon periodicity is preserved again.
 In this example, all coefficients of monomials $X^{\tilde \bfm}$ are 1.
 However,  they are polynomials in $q^{1/\delta'_0}$, in general.
\end{ex}

The following is
the most fundamental result on the quantization of $x$-variables.
\begin{thm}
\label{thm:qLaurent1}\index{quantum!Laurent phenomenon}\index{Laurent! phenomenon (quantum)}
(Quantum Laurent phenomenon, {\cite[Cor.~5.2]{Berenstein05b}}).
Each element $X_i(s)$ is expressed as  a (noncommutative) Laurent polynomial $M_i(\tilde\bfX)$ in $\tilde\bfX=(X_1, \dots, X_{2n})$ with coefficients in $\bbZ[q^{\pm 1/{\delta'_0}}]$.
\end{thm}

Moreover, the following property holds.
 \begin{prop}
[{Bar-invariance, \cite[Proof of Prop.~6.2]{Berenstein05b}}]
 \label{prop:barinv1}\index{bar-invariance}
In the  Laurent polynomial $M_i(\tilde \bfX)$ in Theorem \ref{thm:qLaurent1} each monomial is written in the form
\begin{align}
c X^{\tilde \bfm}
\quad
(c\in \bbZ[q^{\pm 1/{\rp \delta'_0}}],\ \tilde \bfm \in \bbZ^{2n}),
\end{align}
where $c$ is symmetric with respect to $q$ and $q^{-1}$.
\end{prop}

We skip the proofs of Theorem \ref{thm:qLaurent1} and Proposition \ref{prop:barinv1}.

\subsection{Fock-Goncharov decomposition}
\label{subsubsec:FG1}
 The above quantum mutation of $x$-variables
 also admits the Fock-Goncharov decomposition
 in parallel with the $y$-variables.
We do not use the results in this subsection in the rest of the text.
Nevertheless, we present them to show the consistency of the whole picture.

 Let $\{\cdot, \cdot\}_{\Lambda(s)}$ be the skew-symmetric bilinear 
 form on $\bbZ^{2n}$
 defined in the same way as \eqref{eq:Omega2}.
We introduce the noncommutative algebra
$\calA_{\Lambda(s)}$
with generators $X(s)^{\tilde \bfm}$ ($\tilde \bfm \in \bbZ_{\geq 0}^{2n}$)
over $\bbQ(q^{1/{\rp \delta'_0}})$
obeying the relations
\begin{align}
\label{eq:qXcom3}
X(s)^{\tilde \bfm }X(s)^{\tilde \bfm'}=q^{\{\tilde \bfm,\tilde \bfm'\}_{\Lambda(s)}} X(s)^{\tilde \bfm+\tilde \bfm' }
= q^{2 \{\tilde \bfm,\tilde \bfm'\}_{\Lambda(s)}}
X(s)^{\tilde \bfm' } X(s)^{\tilde \bfm}.
\end{align}
By setting
$X(s)^{\bfe_i(s)}=X_i(s)$ ($i=1,\, \dots,\, 2n$), we recover the relation \eqref{eq:qXcom1}.

  Let $ \calF_{\Lambda(s)}$ be the 
skew-field of all fractions of $\calA_{\Lambda(s)}$.
  We formulate the quantum mutation 
  in \eqref{eq:qxmu1} (see also \eqref{eq:qxmu2}) as
  the following skew-field isomorphism
   \begin{align}
\label{eq:qxmu3}
\begin{matrix}
{\rp \mu_q^x{(s)}}\colon &
\calF_{\Lambda(s+1)}
& \rightarrow & \calF_{\Lambda(s)},
\\
&
X_{i}(s+1)
& \mapsto &
\begin{cases}
\displaystyle
X(s)^{-  \bfe_{k_s} + \sum_{j=1}^{2n}  [- \varepsilon_s \tilde b_{jk_s}(s)]_+  \bfe_j}
\\
\quad \times
 (1+q_{k_s}^{- \varepsilon_s} \hat Y_{k_s}(s)^{\varepsilon_s})
& i=k_s,
\\
X_i(s)
&i\neq k_s.
\end{cases}
\end{matrix}
\end{align}
The map $\mu_q^x(s)$ has the
\emph{Fock-Goncharov decomposition}\index{Fock-Goncharov decomposition!of quantum mutation}
\begin{align}
\label{eq:FGqx1}
\mu_q^x(s)=\rho_q^x(s)\circ\tau_q^x(s).
\end{align}
Here, $\tau_q^x(s)$ is a skew-field isomorphism 
defined by
\begin{align}
\label{eq:qxtau1}
\begin{matrix}
\tau_q^x{(s)}: &\calF_{\Lambda(s+1)}& \rightarrow & \calF_{\Lambda(s)}
\\
&
X(s+1)^{\tilde \bfm}
&
\mapsto
&
X(s)^{T^x(s)(\tilde \bfm)},
\end{matrix}
\end{align}
where $T^x(s)$ is a linear isomorphism defined by
 \begin{align}
\label{eq:qxTau1}
\begin{matrix}
T^x{(s)}\colon & \bbZ^{2n} & \rightarrow &   \bbZ^{2n}
\\
&
\bfe_i
&
\mapsto
&
\begin{cases}
\displaystyle
-  \bfe_{k_s} +
\sum_{j=1}^{2n} [- \varepsilon_s \tilde b_{j k_s }(s)]_+  \bfe_j
&i= k_s.
\\
\bfe_i
& i\neq k_s.
\end{cases}
\end{matrix}
\end{align}
By \eqref{eq:signco1}, the sum in the first case
can be restricted to $1\leq j \leq n$.
Also,
$\rho_q^x(s)$ is a skew-field automorphism 
defined by
\begin{align}
\label{eq:qxrho1}
\begin{matrix}
\rho_q^x{(s)} : &\calF_{\Lambda(s)} 
\!\!\!
& \rightarrow &  \calF_{\Lambda(s)}
\\
&
X(s)^{\tilde \bfm}
\!\!\!
&
\mapsto
&
\displaystyle
X(s)^{\tilde \bfm}
\prod_{u=1}^{|m_{k_s} |}
(1+ q_{k_s}^{\varepsilon_s \sgn(m_{k_s})(2u-1)}
  \hat Y_{k_s}(s)^{\varepsilon_s})^{-\sgn(m_{k_s})}.
\end{matrix}
\end{align}

\begin{lem}
\label{lem:rqauto2}
(a). The map $\tau_q^x(s)$ is a skew-field isomorphism.
\par
(b). The map $\rho_q^x(s)$ is a skew-field automorphism.
\end{lem}
\begin{proof}
(a).
The desired condition is equivalent to the equality
\begin{align}
\{\bfe_i, \bfe_j\}_{\Lambda(s+1)}
= \{T^x(s)(\bfe_i), T^x(s)(\bfe_j)\}_{\Lambda(s)}.
\end{align}
This is nontrivial only for (i) $i\leq n$ and $j \geq n+1$,
or (ii)  $j\leq n$ and $i \geq n+1$.
In both cases, the equality follows from
 the specialization of \eqref{2eq:gmut2} with the tropical sign $\varepsilon_s$,
 namely,
\begin{align}
\label{eq:gss1}
\bfg_i(s+1)
=
\begin{cases}
-\bfg_i(s) + 
\sum_{j=1}^n [ -\varepsilon_s b_{jk_s}(s)]_+ \bfg_j(s)
&
i = k_s,
\\
\bfg_i(s)
&
i\neq k_s.
\end{cases}
\end{align}
 (b). This is proved in the same way as
 Lemma \ref{lem:rqauto1}.
\end{proof}

By Lemma \ref{lem:rqauto2},  the map $\mu_q^x(s)$ is a skew-field isomorphism.

Let  $\bfg_{i}(s)$ be the $i$th $g$-vector of the $G$-matrix $G(s)$.  
We define $\tilde \bfg_{i}(s)\in \bbZ^{2n}$ by
\begin{align}
\tilde \bfg_{i}(s)
=
\begin{cases}
(\bfg_{i}(s), \bfzero)
&
i=1,\, \dots, \, n,
\\
\bfe_i
&
i=n+1,\, \dots, \, 2n.
\end{cases}
\end{align}
Also, let $\overline G(s)$ be the $2n \times 2n$ matrix defined by
\begin{align}
\overline G(s) = \begin{pmatrix} G(s)& O\\ O & I\end{pmatrix}.
\end{align}
Then, the first duality \eqref{2eq:dual0} is rewritten as
\begin{align}
\label{2eq:tdual0} 
\overline G(s) \tilde B(s) = \tilde B(0)C(s).
\end{align}

We define the composition
$\tau_q^x(s;0)$ in paralel to \eqref{eq:qtaus01}.
\begin{prop}
\label{prop:tauxq1}
The following formulas hold:
\begin{align}
\label{eq:qxtaus11}
\tau_q^x(s;0)(X_i(s+1))=& X^{\tilde \bfg_i(s+1)}
\quad
(i=1\, \dots, \, 2n),
\\
\label{eq:qxtaus12}
\tau_q^x(s;0)(\hat Y_i(s+1))=&\hat Y^{\bfc_i(s+1)}
:=X^{\sum_{j=1}^n c_{j i}(s+1) 
 \tilde \bfb_{j}}
\quad
(i=1\, \dots, \, n),
\end{align}
where $ \tilde \bfb_{j}$ is the $j$th column vector of the matrix $\tilde B=\tilde B(0)$
in \eqref{eq:tB1}.
\end{prop}
\begin{proof}
The formula \eqref{eq:gss1} is extended as
\begin{align}
\tilde \bfg_i(s+1)
=
\begin{cases}
-\tilde\bfg_i(s) + 
\sum_{j=1}^{2n} [ -\varepsilon_s \tilde b_{jk_s}(s)]_+ \tilde\bfg_j(s)
&
i = k_s,
\\
\tilde\bfg_i(s)
&
i\neq k_s.
\end{cases}
\end{align}
Then, in the same way as Proposition \ref{prop:tauy1}, we have
\eqref{eq:qxtaus11}.
The formula \eqref{eq:qxtaus12} is obtained from \eqref{eq:qxtaus11} as
\begin{align}
\begin{split}
(\mathrm{LHS})
\overset {\eqref{eq:qhaty1}}
&{=}
\tau_q^x(s;0)( X(s+1)^{\sum_{j=1}^{2n} \tilde b_{ji}(s+1) \bfe_j})
\\
\overset {\eqref{eq:qxtaus11}}
&{=}
{\rp X}^{\sum_{j=1}^{2n} \tilde b_{ji}(s+1) \tilde \bfg_j(s+1)}
\overset {\eqref{2eq:tdual0}}
{=}
(\mathrm{RHS}).
\end{split}
\end{align}
\end{proof}

For $\Lambda=\Lambda(0)$,
we define the automorphism $\frakq_q^x(s)\colon
 \calF_{\Lambda} \rightarrow \calF_{\Lambda}$
by
 \begin{align}
\label{eq:qPsiact6}
 \frakq_q^x(s) (X^{\tilde \bfm})
=X^{\tilde \bfm}
\prod_{u=1}^{|\beta |}
(1+ q_{k_s}^{\varepsilon_s \sgn(\beta)(2u-1)}
\hat Y^{\bfc^+_{k_s}(s)})^{-\sgn(\beta)},
\end{align}
where $\beta=\langle \delta_{k_s} \bfc_{k_s}(s), \tilde \bfm \rangle
:= \delta_{k_s} \bfc_{k_s}(s)^T \tilde D^T \tilde  \bfm$.

\begin{prop}
\label{prop:qxcom1}
The following commutative diagram holds:
\begin{align}
\label{eq:qxcd1}
\raisebox{25pt}
{
\begin{xy}
(0,0)*+{\calF_{\Lambda(s)}}="aa";
(25,0)*+{\calF_{\Lambda}}="ba";
(0,-15)*+{\calF_{\Lambda(s)}}="ab";
(25,-15)*+{\calF_{\Lambda}}="bb";
{\ar "aa";"ba"};
{\ar "ab";"bb"};
{\ar "aa";"ab"};
{\ar "ba";"bb"};
(12,3)*+{\text{\small $\tau_q^x(s-1;0)$}};
(12,-12)*+{\text{\small $\tau_q^x(s-1;0)$}};
(-5,-7.5)*+{\text{\small $\rho_q^x(s)$}};
(30,-7.5)*+{\text{\small $\frakq_q^x(s)$}};
\end{xy}
}
\end{align}
\end{prop}
\begin{proof}
 Let us specialize 
 $\tilde \bfm= \tilde \bfg_{i}(s)$ in \eqref{eq:qPsiact6}.
 By the duality \eqref{eq:dual2},
 we have
 \begin{align}
 \beta=\langle \delta_{k_s} \bfc_{k_s}(s),  \tilde \bfg_{i}(s) \rangle = \delta_{k_s i}.
 \end{align}
 Thus,
 we have
 \begin{align}
 \label{eq:qqx1}
 \frakq_q^x(s)(X^{\tilde \bfg_{i}(s)})
 =X^{\tilde \bfg_{i}(s)}
(1+ q_{k_s}^{\varepsilon_s }\hat Y^{\bfc^+_{k_s}(s)})^{- \delta_{k_s i}}.
 \end{align}
 Then, by \eqref{eq:qxrho1} and Proposition \ref{prop:tauxq1}, we have
 \eqref{eq:qxcd1}.
 \end{proof}

We define $\mu_q^x(s;0)$ and $\frakq_q^x(s;0)$ in the same way as \eqref{eq:qmus01} and \eqref{eq:qq01}.
Then, we have the \emph{Fock-Goncharov decomposition}\index{Fock-Goncharov decomposition!of composite quantum mutation} for $\mu_q^x(s;0)$,
\begin{align}
\mu_q^x(s;0)=\frakq_q^x(s;0)
\circ \tau_q^x(s;0).
\end{align}
Also, we have a parallel result with Proposition  \ref{prop:qrhoAd1}.
\begin{prop}
\label{prop:qqxAd1}
\begin{align}
\label{eq:qqxAd1}
\rho^x_q(s) &=
\rmAd[\Psi_{q_{k_s}}(\hat Y_{k_s}(s)^{\varepsilon_s})^{-\varepsilon_s}],
\\
\label{eq:qqxAd2}
\frakq_q^x(s) &=
\rmAd[\Psi_{q_{k_s}}(\hat Y^{\bfc^+_{k_s}(s)})^{-\varepsilon_s}].
\end{align}
\end{prop}
\begin{proof}
By \eqref{eq:qXYrel1},
for any $\varepsilon\in \{1, -1\}$, we have
\begin{align}
\hat Y_k(s)^{\varepsilon} X_i(s)=
q_k^{2 {\varepsilon} \delta_{ki}} X_i (s)\hat Y_k(s)^{\varepsilon} .
\end{align}
Then, as in the proof of Lemma \ref{lem:AdPsi1}, we have
\begin{align}
\label{eq:AdPsix1}
\begin{split}
&\
\rmAd[\Psi_{q_k}(\hat Y_{k}(s)^{\varepsilon})^{\varepsilon}](X_i(s))
=
X_i(s)
 (1+ 
 q_k^{\varepsilon}
\hat Y_{k}(s)^{\varepsilon})^{\delta_{ki}}.
 \end{split}
 \end{align}
Comparing it with \eqref{eq:qxrho1}, we have the equality \eqref{eq:qqxAd1}.
The equality \eqref{eq:qqxAd2} follows from \eqref{eq:qqxAd1}
and Proposition \ref{prop:qxcom1}.
\end{proof}

In summary, even though the quantization of $x$-variables requires a little more complicated formulation
(i.e., the principal extension and the compatible pair) than $y$-variables,
they can be treated in a unified way.
In particular, 
the quantum dilogarithm is responsible for the nontropical part $\rho_q^{x}(s)$
of the quantum mutation.
The picture is further extended to the 
algebraic approach by the group $G_q$.
Namely, 
the automorphism
$ \frakq_q^x(s)$ in  \eqref{eq:qPsiact6}
is also identified with the action of the same dilogarithm element
in Proposition \ref{prop:qqs2}  but for a different representation of $G_q$
called the \emph{principal $x$-representation}\index{principal!$x$-representation}  \cite{Nakanishi22b}.

\notes
The quantization of $y$-variables was introduced in \cite{Fock07}
together with the Fock-Goncharov decomposition (with $\varepsilon=1$)
and also its relevance to the quantum dilogarithm.
The   Fock-Goncharov decomposition with the tropical sign 
was appeared in \cite{Keller11,Kashaev11}.
The quantization of $x$-variables of geometric type was introduced in \cite{Berenstein05b}
and further studied by \cite{Tran09}.
Here, we concentrated on the one with principal coefficients,
where the compatibility pair is conveniently given by \eqref{eq:Lam1} and \eqref{eq:tB1}.
Both quantizations are treated in a unified way
by the structure group $G_q$    \cite{Nakanishi22b}
based on the results in \cite{Kontsevich13, Gross14, Mandel15, Davison19}.
The pentagon relation appeared in a closely related but different setup and context in \cite{Kontsevich08}.
The pentagon relation for quantum dilogarithm elements
 in the form of Proposition \ref{prop:qPsipent2}
 is  due to   \cite{Nakanishi22b}.

\chapter{Quantum dilogarithm identities}

Remarkably any period of a classical $Y$-pattern is lifted to a period of the corresponding
quantum $Y$-pattern.
Moreover, there is an associated quantum dilogarithm identity (QDI)
for the quantum dilogarithm.
To be precise, there are two forms of QDIs, that is,
a QDI in tropical form and a QDI in universal form.
Examples of the QDI associated with a loop in a quantum cluster scattering diagram (QCSDs)
are also presented.

\section{Quantum synchronicity}
\label{sec:preservation1}

For the sequences of mutations in \eqref{eq:mseq1} and \eqref{eq:mseq2}
we consider their quantum counterparts.
Let $\Upsilon_q(s)=(\bfY(s),B(s))$ be the quantum $Y$-seed in \eqref{eq:qmseq1},
while $\Sigma_q(s)=(\tilde \bfX(s),B(s))$ be the principally extended quantum seed in
Section \ref{sec:quantumx1}.
 For a permutation $\nu\in S_n$,
 the $\nu$-periodicity condition is given by
 \begin{align}
\label{eq:qsigmap1}
\Sigma_q(P)=\nu \Sigma_q(0),
\\
\label{eq:qsigmap2}
\Upsilon_q(P)=\nu \Upsilon_q(0),
\end{align}
where in \eqref{eq:qsigmap1} the permutation $\nu$ acts on the unfrozen variables $X_1$, \dots, $X_n$ among
the principally extended variables $\tilde \bfX=(X_1,\dots,X_{2n})$. 
  If the condition \eqref{eq:qsigmap1} (resp.~\eqref{eq:qsigmap2}) holds,
 then  the condition\ \eqref{eq:sigmap1} (resp.~\eqref{eq:sigmap2}) holds
 in the limit $q\rightarrow 1$.
 Remarkably, the opposite indication also holds.

 We start from the $x$-variable case,
 which is due to Berenstein-Zelevinsky \cite{Berenstein05b}.
 
  \begin{thm}
 [Quantum synchronicity {\cite[Theorem 6.1]{Berenstein05b}}]
 \label{thm:qsync1}\index{quantum!synchronicity (for $x$-variable)}\index{synchronicity!quantum --- (for $x$-variable)}
 The condition \eqref{eq:sigmap1} implies the 
  condition \eqref{eq:qsigmap1}.
 \end{thm}

\begin{proof}
Assume that the condition \eqref{eq:sigmap1} holds.
Then, the condition \eqref{eq:qsigmap1} is equivalent to the condition
\begin{align}
\label{eq:qqxP2}
\mu_q^x(P-1;0)(X_i(P)) =X_{\nu^{-1}(i)}
\quad
(i=1,\, \dots, \,n).
\end{align}
Since the mutation \eqref{eq:qxmu2} and the $q$-commutative relation
\eqref{eq:qXcom1} do not involve any subtraction,
one can represent it as
\begin{align}
\mu_q^x(P-1;0)(X_i(P)) =K_i(\bfX) L_i(\bfX)^{-1},
\end{align}
where $K_i(\bfX)$ and $L_i(\bfX)$ are (noncommutative) polynomials in $\bfX$
with subt\-rac\-tion-free coefficients in $\bbQ(q^{1/\rp \delta'_0})$.
By Theorem \ref{thm:qLaurent1}, this equals to
a Laurent polynomial $M_i(\bfX)$ in $\bfX$ with coefficients in $\bbZ[q^{\pm 1/{\rp \delta'_0}}]$.
Thus, we have
\begin{align}
\label{eq:KML1}
K_i(\bfX)=M_i(\bfX)L_i(\bfX).
\end{align}
Let $\mathrm{NP}(M_i(\bfX))$ be the \emph{Newton polytope}\index{Newton polytope} of $M_i(\bfX)$, that is,
the convex hull of all exponents of $\bfX$ in $\bbR^{2n}$.
(This is well defined since $\bfX$ are $q$-commutative.)
Let $a_1=c_1 X^{\tilde \bfm_1}$ be a Laurent monomial belonging to $M_i(\bfX)$ such that
$\tilde \bfm_1$ is at a vertex of $\mathrm{NP}(M_i(\bfX))$.
By \eqref{eq:KML1},
there are some
vertices ${\tilde \bfm_2}$ and ${\tilde \bfm_3}$ of $\mathrm{NP}(L_i(\bfX))$ and $\mathrm{NP}(K_i(\bfX))$,
respectively, such that,
for the monomials $a_2=c_2 X^{\tilde \bfm_2}$ and
$a_3=c_3 X^{\tilde \bfm_3}$ belonging to $L_i(\bfX)$ and $K_i(\bfX)$, respectively,
the equality $a_3=a_1 a_2$ holds.
Thus, ${\rp c_1 =c_3/c_2}$ has a subtraction-free expression in $\bbQ(q^{1/{\rp \delta'_0}})$.
In particular, $\lim_{q\rightarrow 1} c_1\neq 0$.
This means that the Newton polytope $\mathrm{NP}(M_i(\bfX))$ does not shrink in the limit $q\rightarrow 1$.
On the other hand,
by the assumption \eqref{eq:sigmap1}, we have $\lim_{q\rightarrow 1} M_i(\bfX)=x_{\nu^{-1}(i)}$.
Therefore, $M_i(\bfX)=cX_{\nu^{-1}(i)}$ for some $c\in
\bbZ[q^{\pm 1/{\rp \delta'_0}}]$.
One can also apply the same argument to the inverse mutation sequence.
Then, we have
\begin{align}
\mu_q^x(P-1;0)^{-1}(X_{\nu^{-1}(i)})=c' {\rp X_i(P)}
\end{align}
for  some $c'\in
\bbZ[q^{\pm 1/{\rp \delta'_0}}]$.
Since $cc'=1$, $c$ is invertible in $\bbZ[q^{\pm 1/{\rp \delta'_0}}]$.
Thus, $c=q^{a /{\rp \delta'_0}}$ for some $a\in \bbZ$.
Finally, by Proposition \ref{prop:barinv1},  we have $c=1$.
\end{proof}

Next, we consider the $y$-variable case,
which is the basis of our approach to quantum dilogarithm identities.
 \begin{thm}
 [Quantum synchronicity {\cite[Lemma 2.22]{Fock07}, \cite[Prop.~3.4]{Kashaev11}}]
 \label{thm:qsync2}\index{quantum!synchronicity (for $y$-variable)}\index{synchronicity!quantum --- (for $y$-variable)}
 The condition \eqref{eq:sigmap2} implies the 
  condition \eqref{eq:qsigmap2}.
 \end{thm}

\begin{proof}
Assume that the condition \eqref{eq:sigmap2} holds.
By the classical synchronicity (Theorem \ref{1thm:synchro1}),
the condition \eqref{eq:sigmap1} holds.
Thus, by Theorem \ref{thm:qsync1},
the condition \eqref{eq:qsigmap1} holds.
It follows that the periodicity of the quantum $\hat y$-variables
$\hat \bfY(P) = \nu \hat \bfY$ holds.
On the other hand,
$\hat \bfY$ and $\bfY$ satisfy the same $q$-commutative relation
and $\hat \bfY(s)$ and $\bfY(s)$ obey the same mutation.
Moreover, if $\hat Y^{\bfn} = \hat Y^{\bfn'}$ holds,
then $\tilde B\bfn=\tilde B\bfn'$ holds. 
Since $\tilde B$ has rank $n$ (as a matrix),
we have $\bfn=\bfn'$.
This means that $\hat Y^{\bfn}$ satisfies only relations
generated by \eqref{eq:qhY1}.
Therefore, we have $ \bfY(P) = \nu  \bfY$.
\end{proof}

\section[QDI associated with period of  cluster pattern]{QDI associated with a period of a  $Y$-pattern}

\label{sec:QDI1}

In this section, we present several  forms of quantum dilogarithm identities
associated with a period of a  $Y$-pattern.

\subsection{QDI for quantum dilogarithm elements}
We start from 
the quantum analogue of Theorem \ref{thm:DIE1} for the quantum dilogarithm elements
$\Psi_{a}[\bfn]=\Psi_{a,0}[\bfn]$ in \eqref{3eq:qgei1}.

\begin{thm}
[QDI for quantum dilogarithm elements]
\label{thm:QDI1}
 \index{quantum dilogarithm identity!for quantum dilogarithm element}
Suppose that
  the  sequence of  mutations
  \eqref{eq:mseq2} is  \emph{$\nu$-periodic}.
  Then, the following relation holds in $G_q$.
\begin{align}
\label{eq:QDI2}
\Psi_{1/\delta_{k_{P-1}}}[\bfc^+_{k_{P-1}}({P-1})]^{\varepsilon_{P-1}}
\cdots
\Psi_{1/\delta_{k_0}}[\bfc^+_{k_0}(0)]^{\varepsilon_0}
=\rmid.
\end{align}
\end{thm}
\begin{proof}
By Theorem \ref{thm:qsync2}, 
the periodicity \eqref{eq:qsigmap2} holds.
{\rp
Namely, we have $\mu_q(P-1;0)=\rmid$.
Also, by Proposition \ref{prop:qtauy1}, we have $\tau_q(P-1;0)=\rmid$.
Therefore, the equality $\frakq_q(P-1;0)=\rmid$ holds.}
Then, we repeat the proof of Theorem \ref{thm:DIE1}
{\rp by using Proposition \ref{prop:qqs2}.}
\end{proof}

By \eqref{eq:QDEtoDE1},
in the limit $q\rightarrow 1$, the classical  identity \eqref{eq:DIE2} is recovered.

\begin{ex}
[Continued from Example \ref{ex:QDIA22}]
\label{ex:QDIA23}
For the pentagon periodicity of type $A_2$,
the relation \eqref{eq:QDI2} is written as
\begin{align}
\Psi_1[\bfe_2]^{-1}
\Psi_1[\bfe_1+\bfe_2]^{-1}
\Psi_1[\bfe_1]^{-1}
\Psi_1[\bfe_2]
\Psi_1[\bfe_1]
=\rmid,
\end{align}
or equivalently,
\begin{align}
\label{eq:qpentY1}
\Psi_1[\bfe_2]
\Psi_1[\bfe_1]
=
\Psi_1[\bfe_1]
\Psi_1[\bfe_1+\bfe_2]
\Psi_1[\bfe_2].
\end{align}
This is a special case of the pentagon relation \eqref{eq:Psipent2}.
\end{ex}

\subsection{QDI in tropical form}
We turn the relation \eqref{eq:QDI2}
into an identity for the quantum dilogarithm $\Psi_q(x)$.
The derivation is more direct than the classical case.

\begin{thm}
[QDI in tropical form {\cite[Theorem~5.11]{Keller11}, \cite[Theorem 3.5]{Kashaev11}}]
\label{thm:QDI2}
\index{quantum dilogarithm identity!in tropical form}
Let $\bfY$ be the initial quantum $Y$-variables
obeying the relation \eqref{eq:Ycom2} with $s=0$.
Suppose that
  the  sequence of  mutations
  \eqref{eq:mseq2} is  \emph{$\nu$-periodic}.
  Then, the following identity holds in $\overline{\calA}_{\Omega}$:
 \begin{align}
  \label{eq:QDItrop1}
\Psi_{q_{k_{P-1}}}(Y^{\bfc^+_{k_{P-1}}({P-1})})^{\varepsilon_{P-1}}
\cdots
\Psi_{q_{k_0}}(Y^{\bfc^+_{k_0}(0)})^{\varepsilon_0 }
=1.
\end{align}
\end{thm}
\begin{proof}
Again, we repeat the proof of Theorem \ref{thm:DIE1}
{\rp by using \eqref{eq:qqAd1}}.
In the proof, the following fact is essential:
Suppose that $B$ is nonsingular.
Then, the equality
\begin{align}
\label{eq:qqAd3}
\begin{split}
\rmAd[
\Psi_{q_{k_{P-1}}}( Y^{\bfc^+_{k_{P-1}}({P-1})})^{\varepsilon_{P-1}}
\cdots
\Psi_{q_{k_{0}}}( Y^{\bfc^+_{k_0}(0)})^{\varepsilon_0}
]= \rmid.
\end{split}
\end{align}
implies the equality
 \begin{align}
  \label{eq:QDItrop2}
\Psi_{q_{k_{P-1}}}( Y^{\bfc^+_{k_{P-1}}({P-1})})^{\varepsilon_{P-1}}
\cdots
\Psi_{q_{k_{0}}}( Y^{\bfc^+_{k_0}(0)})^{\varepsilon_0}
=1.
\end{align}
To see it,
let $g=1+\sum_{\bfn\in N^+} c_{\bfn} Y^{\bfn}$ be the LHS of \eqref{eq:QDItrop2}.
Then, the condition \eqref{eq:qqAd3} implies that
each term $c_{\bfn}  Y^{\bfn}$ commutes with $Y_1$, \dots, $Y_n$.
Since $\Omega$ is nonsingular, this implies  $c_{\bfn}=0$.
\end{proof}

Two identities \eqref{eq:QDI2} and \eqref{eq:QDItrop1}
can be directly translated to each other, so that one
may regard them as alternative expressions of the same identity.

\begin{ex}
[Continued from Example \ref{ex:QDIA23}]
\label{ex:QDIA24}
The corresponding identity to \eqref{eq:QDItrop1} is written as
\begin{align}
\Psi_q(Y_2)
\Psi_q(Y_1)
=
\Psi_q(Y_1)
\Psi_q(Y^{\bfe_1+\bfe_2})
\Psi_q(Y_2).
\end{align}
Note that $Y^{\bfe_1+\bfe_2}=q Y_1Y_2$.
Thus,
it coincides with the pentagon identity in \eqref{eq:qdpent1},
where $u=Y_2$ and $v=Y_1$.
In particular, it gives an alternative proof of  
the pentagon identity  \eqref{eq:qdpent1}.
\end{ex}
\subsection{QDI in universal form}

There is  another form of QDIs first discovered by Volkov \cite{Volkov11}
for the pentagon identity.
The idea is shuffling the product \eqref{eq:QDItrop1}
in the opposite order.
Below  we set
\begin{align}
\tilde Y_{k_{s}}(s) :={\rp \mu_q(s-1;0)}(Y_{k_{s}}(s))
\end{align}
and regard $\tilde Y_{k_{s}}(s)^{\varepsilon_s}$ and
$\Psi_{q_{k_{s}}}(
\tilde Y_{k_{s}}(s)^{\varepsilon_s})$ as elements of $\overline \calA_{\Omega}$.

\begin{lem}
[Shuffle formula, {\cite[Lemma 3.6]{Kashaev11}}]
\label{lem:shuffling1}
\index{shuffle formula}
The following equality holds in $\overline \calA_{\Omega}$:
\begin{align}
\label{eq:QDIU3}
\begin{split}
&\
\Psi_{q_{k_{s}}}(
Y^{\bfc^+_{k_{s}}({s})})^{\varepsilon_{s}}
\Psi_{q_{k_{s-1}}}( Y^{\bfc^+_{k_{s-1}}({s-1})})^{\varepsilon_{s-1}}
\cdots
\Psi_{q_{k_{0}}}( Y^{\bfc^+_{k_0}(0)})^{\varepsilon_0}
\\
=&\ 
\Psi_{q_{k_{0}}}(\tilde Y_{k_{0}}(0)^{\varepsilon_0})^{\varepsilon_{0}}
\Psi_{q_{k_{1}}}( \tilde Y_{k_{1}}(1)^{\varepsilon_1})^{\varepsilon_1}
\cdots
\Psi_{q_{k_{s}}}(
\tilde Y_{k_{s}}(s)^{\varepsilon_s})^{\varepsilon_{s}}.
\end{split}
\end{align}
\end{lem}

\begin{proof}
By the Fock-Goncharov decomposition
\eqref{eq:qdecom1}
and the equality \eqref{eq:qqAd1}, we have
\begin{gather}
\label{eq:QDIU1}
\tilde Y_{k_{s}}(s)^{\varepsilon_s}
=
\rmAd[
\Psi_{q_{k_{s-1}}}( Y^{\bfc^+_{k_{s-1}}({s-1})})^{\varepsilon_{s-1}}
\cdots
\Psi_{q_{k_{0}}}( Y^{\bfc^+_{k_0}(0)})^{\varepsilon_0 }
]^{-1}
(Y^{\bfc^+_{k_{s}}({s})}).
\end{gather}
Thus, we have the equality
\begin{align}
\label{eq:QDIU2}
\begin{split}
&\
\Psi_{q_{k_{s}}}(
Y^{\bfc^+_{k_{s}}({s})})^{\varepsilon_{s}}
(
\Psi_{q_{k_{s-1}}}( Y^{\bfc^+_{k_{s-1}}({s-1})})^{\varepsilon_{s-1}}
\cdots
\Psi_{q_{k_{0}}}( Y^{\bfc^+_{k_0}(0)})^{\varepsilon_0 }
)
\\
=&\ 
(
\Psi_{q_{k_{s-1}}}( Y^{\bfc^+_{k_{s-1}}({s-1})})^{\varepsilon_{s-1}}
\cdots
\Psi_{q_{k_{0}}}( Y^{\bfc^+_{k_0}(0)})^{\varepsilon_0 }
)
\Psi_{q_{k_{s}}}(
\tilde Y_{k_{s}}(s)^{\varepsilon_s})^{\varepsilon_{s}}.
\end{split}
\end{align}
By repeating it, we obtain the equality \eqref{eq:QDIU3}.
\end{proof}

\begin{thm}
[QDI in universal form { \cite[\S3]{Volkov11}, \cite[Cor.~3.7]{Kashaev11}}]
\label{thm:QDI3}
\index{quantum dilogarithm identity!in universal form}
Suppose that
  the  sequence of  mutations
  \eqref{eq:mseq2} is  \emph{$\nu$-periodic}.
  Then, the following identity holds in $\overline \calA_{\Omega}$:
 \begin{align}
  \label{eq:QDIuniv1}
\Psi_{q_{k_{0}}}(\tilde Y_{k_0}(0)^{\varepsilon_0})^{\varepsilon_0 }
\cdots
\Psi_{q_{k_{P-1}}}(\tilde Y_{k_{P-1}}(P-1)^{\varepsilon_{P-1}})^{\varepsilon_{P-1}}
=1.
\end{align}
\end{thm}
\begin{proof}
This is obtained from \eqref{eq:QDItrop1} and \eqref{eq:QDIU3}.
\end{proof}

\begin{ex}
[Continued from Example \ref{ex:QDIA24}]
\label{ex:QDIA25}
By the result in Example  \ref{ex:QDIA21}, we have
\begin{align}
\begin{split}
&\Psi_q(Y_1)
\Psi_q(Y_2 (1+ q^{-1}Y_1))
\Psi_q((1+ q^{-1} Y_2+ Y_1Y_2)^{-1}Y_1)^{-1}
\\
& \quad \times
\Psi_q(q^{-1} (1+ q^{-1} Y_2)^{-1}Y_2Y_1)^{-1}
\Psi_q(Y_2)^{-1}=1.
\end{split}
\end{align}
\end{ex}

\section{More examples of QDIs}
\label{sec:more1}
Let us consider the QDIs for the periodicities of
types $B_2$ and $G_2$ in Example \ref{ex:BC1},
which involve nontrivial factors $\delta_i$.

(a) Type $B_2$. 
We set
\begin{align}
\label{eq:qBB2}
\Omega=\begin{pmatrix}
0 & -1
\\
1 & 0
\end{pmatrix}
,
\quad
\Delta=\begin{pmatrix}
1 & 0
\\
0 & 2
\end{pmatrix}
.
\end{align}
We use the notation
$[n_1,n_2]_{a,b}$ and $[n_1,n_2]_{a}$ (or column-wise)
 for  quantum dilogarithm elements  $\Psi[\bfn]_{a,b}$ and $\Psi[\bfn]_a$
 as in the classical case in Section \ref{sec:ordering1}.
From the data in Example \ref{ex:BC1},
we read  the following QDI for quantum dilogarithm elements:
\begin{align}
\label{eq:QDIB1}
\begin{bmatrix}
0\\
1\\
\end{bmatrix}_{\frac{1}{2}}
\begin{bmatrix}
1\\
0\\
\end{bmatrix}_{1}
=
\begin{bmatrix}
1\\
0\\
\end{bmatrix}_{1}
\begin{bmatrix}
1\\
1\\
\end{bmatrix}_{\frac{1}{2}}
\begin{bmatrix}
1\\
2\\
\end{bmatrix}_{1}
\begin{bmatrix}
0\\
1\\
\end{bmatrix}_{\frac{1}{2}}.
\end{align}
Accordingly,
the QDI in the tropical from is given by
\begin{align}
\label{eq:QDIB2}
\Psi_{q^{1/2}}(Y_2)
\Psi_{q}(Y_1)
=\Psi_{q}(Y_1)\Psi_{q^{1/2}}(Y^{(1,1)}) \Psi_{q}(Y^{(1,2)})
\Psi_{q^{1/2}}(Y_2).
\end{align}
The QDI in the universal form is given by
 \begin{align}
 \begin{split}
 \Psi_{q}(\tilde Y_1(0))
\Psi_{q^{1/2}}(\tilde Y_2(1))
 &= 
\Psi_{q^{1/2}}(\tilde Y_2(5)^{-1})
\Psi_{q}(\tilde Y_1(4)^{-1})
\\
& \qquad\times
\Psi_{q^{1/2}}(\tilde Y_2(3)^{-1})
\Psi_{q}(\tilde Y_1(2)^{-1}),
\end{split}
\end{align}
where
\begin{align*}
\begin{split}
\tilde Y_1(0)&=Y_1,\\
\tilde Y_2(1)&=Y_2(1+q^{-1}Y_1),\\
\tilde Y_1(2)&=Y_1^{-1}(1+q^{-1/2}Y_2+q^{1/2}Y_1Y_2)(1+q^{-3/2}Y_2+q^{-1/2}Y_1Y_2),\\
\tilde Y_2(3)&=q Y_1^{-1}Y_2^{-1}(1+(q^{-1/2}+q^{-3/2})Y_2+q^{-2}Y_2^2+qY_1Y_2^2),\\
\tilde Y_1(4)&=q^2 Y_1^{-1}Y_2^{-2}(1+q^{-1/2}Y_2)(1+q^{-3/2}Y_2),
\\
\tilde Y_2(5)&=Y_2^{-1}.
\end{split}
\end{align*}
As an additional perspective, 
let us demonstrate that
 the relation \eqref{eq:QDIB1} is derived only from the pentagon relation \eqref{eq:Psipent2} \cite{Nakanishi22b}.
We use the following \emph{fission-fusion formula}\index{fission-fusion formula} \cite[Prop.~2.8]{Nakanishi22b}:
\begin{align}
\Psi_{a,b}[n]=\prod_{t=1}^p \Psi_{pa,b+(2t-p-1)a}[n]
\quad
(p \in \bbZ_{\geq 1}).
\end{align}
Applying the pentagon relation \eqref{eq:Psipent2} repeatedly, we have
\begin{align}
\begin{split}
\begin{bmatrix}
0\\
1\\
\end{bmatrix}_{\frac{1}{2}}
\begin{bmatrix}
1\\
0\\
\end{bmatrix}_{1}
=&\ 
\begin{bmatrix}
0\\
1\\
\end{bmatrix}_{1,\frac{1}{2}}
\begin{bmatrix}
0\\
1\\
\end{bmatrix}_{1,-\frac{1}{2}}
\begin{bmatrix}
1\\
0\\
\end{bmatrix}_{1,0}
\\
=&\ 
\begin{bmatrix}
0\\
1\\
\end{bmatrix}_{1,\frac{1}{2}}
\begin{bmatrix}
1\\
0\\
\end{bmatrix}_{1,0}
\begin{bmatrix}
1\\
1\\
\end{bmatrix}_{1,-\frac{1}{2}}
\begin{bmatrix}
0\\
1\\
\end{bmatrix}_{1,-\frac{1}{2}}
\\
 =&\
\begin{bmatrix}
1\\
0\\
\end{bmatrix}_{1,0}
\begin{bmatrix}
1\\
1\\
\end{bmatrix}_{1,\frac{1}{2}}
\begin{bmatrix}
0\\
1\\
\end{bmatrix}_{1,\frac{1}{2}}
\begin{bmatrix}
1\\
1\\
\end{bmatrix}_{1,-\frac{1}{2}}
\begin{bmatrix}
0\\
1\\
\end{bmatrix}_{1,-\frac{1}{2}}
\\
 =&\
 \begin{bmatrix}
1\\
0\\
\end{bmatrix}_{1,0}
\begin{bmatrix}
1\\
1\\
\end{bmatrix}_{1,\frac{1}{2}}
\begin{bmatrix}
1\\
1\\
\end{bmatrix}_{1,-\frac{1}{2}}
\begin{bmatrix}
1\\
2\\
\end{bmatrix}_{1,0}
\begin{bmatrix}
0\\
1\\
\end{bmatrix}_{1,\frac{1}{2}}
\begin{bmatrix}
0\\
1\\
\end{bmatrix}_{1,-\frac{1}{2}}
\\
 =&\
\begin{bmatrix}
1\\
0\\
\end{bmatrix}_{1}
\begin{bmatrix}
1\\
1\\
\end{bmatrix}_{\frac{1}{2}}
\begin{bmatrix}
1\\
2\\
\end{bmatrix}_{1}
\begin{bmatrix}
0\\
1\\
\end{bmatrix}_{\frac{1}{2}},
\end{split}
\end{align}
which is in parallel with the classical one \eqref{eq:order3}.
In other words, the QDI \eqref{eq:QDIB1} is \emph{finitely reducible}
by the pentagon relation  \eqref{eq:Psipent2}.

(b) Type $G_2$. 
We set
\begin{align}
\label{eq:qBG2}
\Omega=\begin{pmatrix}
0 & -1
\\
1 & 0
\end{pmatrix}
,
\quad
\Delta=\begin{pmatrix}
1 & 0
\\
0 & 3
\end{pmatrix}
.
\end{align}
From the data in Example \ref{ex:BC1},
we read off the following QDI for the dilogarithm elements:
\begin{align}
\label{eq:QDIC1}
\begin{bmatrix}
0\\
1\\
\end{bmatrix}_{\frac{1}{3}}
\begin{bmatrix}
1\\
0\\
\end{bmatrix}_{1}
=
\begin{bmatrix}
1\\
0\\
\end{bmatrix}_{1}
\begin{bmatrix}
1\\
1\\
\end{bmatrix}_{\frac{1}{3}}
\begin{bmatrix}
2\\
3\\
\end{bmatrix}_{1}
\begin{bmatrix}
1\\
2\\
\end{bmatrix}_{\frac{1}{3}}
\begin{bmatrix}
1\\
3\\
\end{bmatrix}_{1}
\begin{bmatrix}
0\\
1\\
\end{bmatrix}_{\frac{1}{3}}.
\end{align}
Accordingly,
the QDI in the tropical form is given by
\begin{align}
\label{eq:QDIG2}
\begin{split}
&\
\Psi_{q^{1/3}}(Y_2)
\Psi_{q}(Y_1)
\\
=&\
\Psi_{q}(Y_1)
\Psi_{q^{1/3}}(Y^{(1,1)}) 
\Psi_{q}(Y^{(2,3)})
\Psi_{q^{1/3}}(Y^{(1,2)}) 
\Psi_{q}(Y^{(1,3)})
\Psi_{q^{1/3}}(Y_2).
\end{split}
\end{align}
The QDI in the universal form is given by
\begin{align}
 \begin{split}
 \Psi_{q}(\tilde Y_1(0))
\Psi_{q^{1/3}}(\tilde Y_2(1))
 &= 
\Psi_{q^{1/3}}(\tilde Y_2(7)^{-1})
\Psi_{q}(\tilde Y_1(6)^{-1})
\Psi_{q^{1/3}}(\tilde Y_2(5)^{-1})
\\
& \qquad\times
\Psi_{q}(\tilde Y_1(4)^{-1}),
\Psi_{q^{1/3}}(\tilde Y_2(3)^{-1})
\Psi_{q}(\tilde Y_1(2)^{-1}),
\end{split}
\end{align}
where
\begin{align*}
\notag
\tilde Y_1(0)&=Y_1,\\
\notag
\tilde Y_2(1)&=Y_2(1+q^{-1}Y_1),\\
\notag
\tilde Y_1(2)&=Y_1^{-1}(1+q^{-1/3}Y_2+q^{2/3}Y_1Y_2)(1+q^{-1}Y_2+Y_1Y_2)\\
\notag
& \quad \times (1+q^{-5/2}Y_2+q^{-2/3}Y_1Y_2),
\\
\notag
\tilde Y_2(3)&=q Y_1^{-1}Y_2^{-1}
(1+
(q^{-4/3}+q^{-2}+q^{-8/3})(qY_2 +Y_2^2+q^3Y_1Y_2^2)
\\
\notag
&\quad 
 + q^{-3}Y_2^3
+(q^2+1)Y_1Y_2^3+
q^{5}Y_1^2Y_2^3
),
\\
\tilde Y_1(4)&=q^6 Y_1^{-2}Y_2^{-3}
(1+(q^{-7/3}+q^{-3})Y_2+q^{-16/3}Y_2^2 + q^{5/3}Y_1Y_2^2)
\\
\notag
&\quad 
\times
(1+(q^{-1/3}+q^{-11/3})Y_2+q^{-4}Y_2^2 + q Y_1Y_2^2)
\\
\notag
&\quad 
\times
(1+(q^{-1}+q^{-5/3})Y_2+q^{-8/3}Y_2^2 + q^{1/3}Y_1Y_2^2),
\\
\notag
\tilde Y_2(5)&=q^2 Y_1^{-1}Y_2^{-2}(1+(q^{-1/3}+q^{-1}+q^{-5/3})Y_2
\\
\notag
&\quad
\times
(q^{-4/3}+q^{-2} + q^{-8/3})Y_2^2+q^{-3}Y_2^3+q^2Y_1Y_2^3),\\
\notag
\tilde Y_1(6)&=q^3 Y_1^{-1}Y_2^{-3}(1+q^{-1/3}Y_2)(1+q^{-1}Y_2)(1+q^{-5/3}Y_2),
\\
\notag
\tilde Y_2(7)&=Y_2^{-1}.
\end{align*}
 Again, the relation \eqref{eq:QDIC1} is 
 finitely reducible by the pentagon relation.

\section{QDI associated with a loop in a QCSD}

As in the classical case,
it is possible to extend QDIs for \emph{quantum cluster scattering diagrams (QCSDs)}\index{quantum!cluster scattering diagram (QCSD)}.

A QCSD $\frakD_{\fraks}^q$ is defined  in parallel with a CSD \cite{Mandel15,Davison19}
by replacing the structure group $G$ with $G_q$ and 
the incoming walls $\bfw_i$ in Definition \ref{defn:CSD1} with
\begin{align}
\bfw_i:=(e_i^{\perp}, \Psi_{1/\delta_i}[e_i])_{e_i}.
\end{align}
By applying the results and methods in \cite{Kontsevich13, Gross14}, 
one can establish the existence and the uniqueness (up to equivalence)
of a QCSD for each skew-symmetrizable integer matrix $B$.
However, not much is known about the structure of QCSDs.

Here, we list some open problems.

\begin{prob}
\par
(1) \emph{Positivity and nonpositivity.}
It is known \cite{Nakanishi22b} that,
any QCSD has  a realization that
every wall element has the form
\begin{align}
\label{eq:Psipower1}
\Psi_{ 1/\delta(h\bfn), b} {[h\bfn]^{s}}
\quad
(\bfn\in N^+_{\rmpr};\, b\in (1/\delta_0)\bbZ;\, h \in \bbZ_{\geq 1};\, s\in \bbZ).
\end{align}
The analog of the positivity property in Theorem \ref{3thm:pos1}
claims that the only positive powers $s$ appear in \eqref{eq:Psipower1}.
This is known  to be true by \cite[Theorem~2.15]{Davison19}
when $B$ is skew-symmetric,
Also, this is true for the types $B_2$ and $G_2$,
which are nonskew-symmetric,
as presented in Section \ref{sec:more1}.
On the other hand, several counter-examples are known in the nonskew-symmetric case \cite{Lee14,Cheung20,Nakanishi22b}.
The exact condition for the positivity is not known.
\par
(2) \emph{Preservation of the support.}
The quantum synchronicity in Section  \ref{sec:preservation1} suggests that
the support of a QCSD remains unchanged from the classical counterpart.
This is not obvious due to a possible \emph{cancellation of walls} of a QCSD in the classical limit
such as $\lim_{q\rightarrow 1} 
\Psi_{a,b}[\bfn]\Psi_{a,b'}[\bfn]^{-1}=1$.
If the positivity in (a) holds, this does not happen; therefore, the support is preserved \cite{Nakanishi22b}.
However, in general, this is yet to be proved.
\par
(3)
\emph{Infinite reducibility.}
An analog of the ordering property in Proposition \ref{prop:ordering1}
is not known yet.
Therefore, the infinite reducibility  of QDIs by the pentagon relation 
as in Theorem \ref{3thm:struct1} is  not clear.
On the other hand,
 there are examples with infinitely or finitely reducibility, including the ones in \eqref{eq:QDIB2} and \eqref{eq:QDIG2}.
\end{prob}

Despite the above problems,
the QDI associated with a loop $\gamma$ in a QCSD is formulated just as in the classical case.
We simply replace the dilogarithm elements in \eqref{eq:sgseq2} with the quantum ones 
that appear on the walls crossed by $\gamma$.
Below we present examples,
which are ``well-behaved''  in the sense that
the positivity holds and the properties 
 (2) and (3) are satisfied.
\begin{ex}
[Affine type of rank 2 \cite{Nakanishi22b}]
\label{ex:QDIaffine1}
Let us consider the affine type of rank 2.
We use the same convention in Example \ref{ex:CSDaffine1}.
\par
(a)
 Type $A_1^{(1)}$. Let $(\d_1,\d_2)=(2,2)$. This is a skew-symmetric case.
 The support of
a QCSD $\frakD^q_{\fraks}$ with minimal support is the same as
the classical one in \eqref{3eq:a111}--\eqref{3eq:a114}.
The  QDI for quantum dilogarithm elements along a loop around the origin is given by
\begin{align}
\label{eq:a115}
\begin{split}
\begin{bmatrix}
0\\
1
\end{bmatrix}
_{\frac{1}{2}}
\begin{bmatrix}
1\\
0
\end{bmatrix}
_{\frac{1}{2}}
&
=
\begin{bmatrix}
1\\
0
\end{bmatrix}
_{\frac{1}{2}}
\begin{bmatrix}
2\\
1
\end{bmatrix}
_{\frac{1}{2}}
\begin{bmatrix}
3\\
2
\end{bmatrix}
_{\frac{1}{2}}
\cdots
\Biggl(
\prod_{j=0}^{\infty}
\begin{bmatrix}
2^j\\
2^j
\end{bmatrix}
_{2^{j-1},-2^{j-1}}
\begin{bmatrix}
2^j\\
2^j
\end{bmatrix}
_{2^{j-1},2^{j-1}}
\Biggr)
\\
&\qquad \times
\cdots
\begin{bmatrix}
2\\
3
\end{bmatrix}
_{\frac{1}{2}}
\begin{bmatrix}
1\\
2
\end{bmatrix}
_{\frac{1}{2}}
\begin{bmatrix}
0\\
1
\end{bmatrix}
_{\frac{1}{2}}.
\end{split}
\end{align}
This is the simplest example of a QDI
involving an infinite product.
One can prove the infinite reducibility as in the classical case  \cite{Nakanishi22b}.
The quantum dilogarithm version of the formula coincides with the identity in \cite[Eq.\ (1.3)]{Dimofte09}
that describes the wall crossing in  $N=2$ super Yang-Mills theory.

\medskip
(b) Type $A_2^{(2)}$. Let $(\d_1,\d_2)=(1,4)$.
This is a nonskew-symmetric case.
 The support of
a QCSD $\frakD^q_{\fraks}$ with minimal support is the same as
the classical one in \eqref{3eq:a221}--\eqref{3eq:a226}.
The QDI  for quantum dilogarithm elements along a loop around the origin is given by
\begin{align}
\label{eq:qa227}
\begin{split}
\begin{bmatrix}
0\\
1
\end{bmatrix}
_{\frac{1}{4}}
\begin{bmatrix}
1\\
0
\end{bmatrix}
_{1}
&=
\begin{bmatrix}
1\\
0
\end{bmatrix}
_{1}
\begin{bmatrix}
1\\
1
\end{bmatrix}
_{\frac{1}{4}}
\begin{bmatrix}
3\\
4
\end{bmatrix}
_{1}
\begin{bmatrix}
2\\
3
\end{bmatrix}
_{\frac{1}{4}}
\begin{bmatrix}
5\\
8
\end{bmatrix}
_{1}
\begin{bmatrix}
3\\
5
\end{bmatrix}
_{\frac{1}{4}}
\cdots
\\
&
\qquad
\times
\Biggl(
\begin{bmatrix}
1\\
2
\end{bmatrix}
_{\frac{1}{2}}
\prod_{j=0}^{\infty}
\begin{bmatrix}
 2^j\\
 2^{j+1}
\end{bmatrix}
_{2^{j-1}, -2^{j-1}}
\begin{bmatrix}
 2^j\\
 2^{j+1}
\end{bmatrix}
_{2^{j-1}, 2^{j-1}}
\Biggr)
\\
&
\qquad 
\times
\cdots
\begin{bmatrix}
5\\
12
\end{bmatrix}
_{1}
\begin{bmatrix}
2\\
5
\end{bmatrix}
_{\frac{1}{4}}
\begin{bmatrix}
3\\
8
\end{bmatrix}
_{1}
\begin{bmatrix}
1\\
3
\end{bmatrix}
_{\frac{1}{4}}
\begin{bmatrix}
1\\
4
\end{bmatrix}
_{1}
\begin{bmatrix}
0\\
1
\end{bmatrix}
_{\frac{1}{4}}
.
\end{split}
\end{align}
One can prove the infinite reducibility as in the classical case  \cite{Nakanishi22b}.
\end{ex}

\begin{ex}
[Type $A_2^{(1)}$]
\label{ex:a211}
As a little more nontrivial example, we consider the
type $A_2^{(1)}$ case,
which is  of rank 3 affine type.
We take
\begin{align}
B=\Omega=\begin{pmatrix}
0 & -1 & -1
\\
1 & 0 & -1
\\
1 & 1 & 0
\end{pmatrix},
\quad
\Delta=I.
\end{align}
 The support of
a QCSD $\frakD^q_{\fraks}$ with minimal support is the same as
the classical one, and it is depicted in Figure \ref{fig:CSDA12}.
The diagram is the projection of the support of $\frakD^q_{\fraks}$
on the unit sphere $S^2$ in $\bbR^3$. Furthermore, it is projected on the tangent plane at $(1/\sqrt{3})(1,1,1)$ by the stereographic projection.
For example,
the triangle in the center represents the positive orthant in $\bbR^3$.
In this example,
all joints
are of type $A_2$ or $A^{(1)}_1$.
In particular, every QDI 
is infinitely reducible.

\begin{figure}
\centering
\includegraphics[width=100mm]{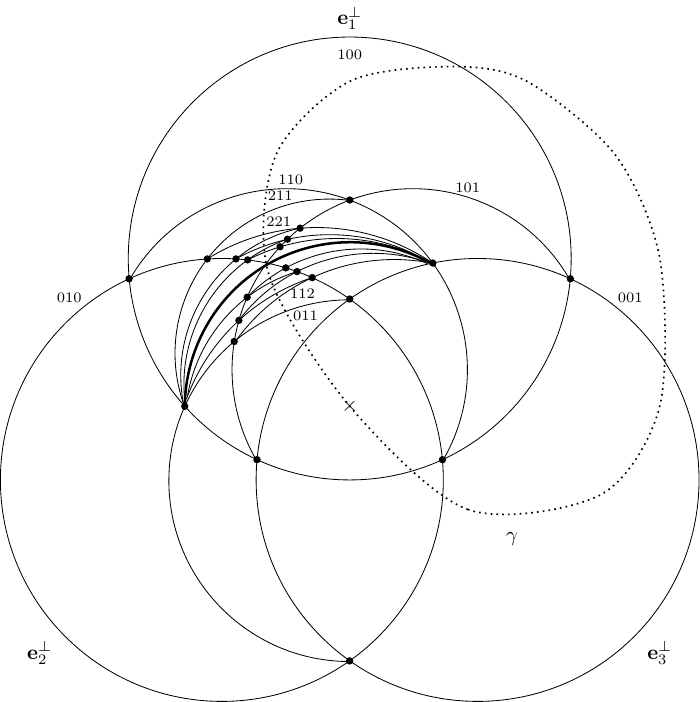}
\caption{The support of a QCSD of type $A_2^{(1)}$.
The three-digit number $101$, for example, represents the normal vector $(1,0,1)$
of a wall.
There are actually infinitely many walls converging to the thick wall
whose normal vector is $(1,1,1)$, so that
the thick wall is unreachable by cluster mutations.
}
\label{fig:CSDA12}
\end{figure}

Let us give one specific example of a QDI.
We choose a loop $\gamma$ depicted in Figure \ref{fig:CSDA12},
where $\gamma$ starts and ends in the positive orthant
and it is counterclockwise.
The loop $\gamma$ is \emph{not} admissible because it passes through the common
joints of walls with three normal vectors
\begin{align}
\label{eq:nv1}
\bfn_1=(0,1,0),
\quad
\bfn_2=(1,0,1),
\quad
\bfn_3=(1,1,1).
\end{align}
However, we have
\begin{align}
\{\bfn_i,\bfn_j\}=0.
\end{align}
So, this is a parallel joint.
Therefore, the quantum dilogarithm elements $\Psi_{a,b}[\bfn_i]$ ($i=1,\, 2,\,3)$ commute
with each other, so that the path-ordered product along $\gamma$ is well-defined.
The intersections of the loop $\gamma$ with walls are separated 
in four steps.
\begin{enumerate}
\item
First,
the loop $\gamma$ crosses the incoming walls $\bfe_1^{\perp}$,
$\bfe_2^{\perp}$, $\bfe_3^{\perp}$ in this order with  positive intersection sign.
(The intersection sign is read from the convexity of the walls at the intersection.
It is positive if and only if the path goes from the convex side to the non-convex side.)
\item
Next, $\gamma$ crosses infinitely many walls
with  negative intersection sign,
where the normal vectors of walls are
\begin{align}
(1,0,0),\
(1,1,0),\
(2,1,1),\
(2,2,1),\
(3,2,2),\
(3,3,2),\
\cdots.
\end{align}
\item
Then, $\gamma$ crosses walls
with  negative intersection sign,
where the normal vectors of walls are $\bfn_1$, $\bfn_2$, $\bfn_3$
in \eqref{eq:nv1}.
The wall with normal vector $\bfn_3$ is presented in the thick curve
in Figure \ref{fig:CSDA12}. Its wall element has
the same structure as the one with normal vector $(1,1)$  in Example \ref{ex:QDIaffine1}.
\item
Finally,
$\gamma$ crosses infinitely many walls
with  negative intersection sign,
where the normal vectors of walls are
\begin{align}
\cdots,
(2,3,3),\
(2,2,3),\
(1,2,2),\
(1,1,2),\
(0,1,1),\
(0,0,1),\
\end{align}
and returns to the initial point.
\end{enumerate}
Thus, we obtain the following QDI along $\gamma$:
\begin{align}
\label{eq:QDIA121}
\begin{split}
\begin{bmatrix}
0
\\
0
\\
1
\end{bmatrix}_1
\begin{bmatrix}
0
\\
1
\\
0
\end{bmatrix}_1
\begin{bmatrix}
1
\\
0
\\
0
\end{bmatrix}_1
&
=
\begin{bmatrix}
1
\\
0
\\
0
\end{bmatrix}_1
\begin{bmatrix}
1
\\
1
\\
0
\end{bmatrix}_1
\begin{bmatrix}
2
\\
1
\\
1
\end{bmatrix}_1
\begin{bmatrix}
2
\\
2
\\
1
\end{bmatrix}_1
\cdots
\\
&\qquad\times
\begin{bmatrix}
0
\\
1
\\
0
\end{bmatrix}_1
\begin{bmatrix}
1
\\
0
\\
1
\end{bmatrix}_1
\left(
\prod_{j=0}^{\infty}
\begin{bmatrix}
2^j
\\
2^j
\\
2^j
\end{bmatrix}_{2^j, -2^j}
\begin{bmatrix}
2^j
\\
2^j
\\
2^j
\end{bmatrix}_{2^j, 2^j}
\right)
\\
&\qquad\times
\cdots
\begin{bmatrix}
1
\\
2
\\
2
\end{bmatrix}_1
\begin{bmatrix}
1
\\
1
\\
2
\end{bmatrix}_1
\begin{bmatrix}
0
\\
1
\\
1
\end{bmatrix}_1
\begin{bmatrix}
0
\\
0
\\
1
\end{bmatrix}_1
.
\end{split}
\end{align}
There is one interesting application of this QDI.
Let
\begin{align}
P=
\begin{pmatrix}
0 & 1 & 1
\\
1 & 0 &-1
\end{pmatrix}
,
\quad
\Omega'=
\begin{pmatrix}
0 & 1 
\\
-1 & 0 
\end{pmatrix}
.
\end{align}
Note that $\mathrm{Ker}\, \Omega = \mathrm{Ker}\, P=\langle (1, -1, 1)\rangle$.
Also, we have
\begin{align}
P^T 
\Omega'
P=\Omega.
\end{align}
So, one may reduce the set of positive vectors $N^+\subset \bbZ^3$ for the structure group $G_q$
 with $N'{}^+=P(N^+)\subset \bbZ^2$,
which is spanned by $(0,1)$ and $(1,-1)$.
Then, applying the projection $P$ to the QDI in 
\eqref{eq:QDIA121}, we obtain a QDI in the reduced structured group $G'_q$
\begin{align}
\label{eq:QDIA122}
\begin{split}
\begin{bmatrix}
1
\\
-1
\end{bmatrix}_1
\begin{bmatrix}
1
\\
0
\end{bmatrix}_1
\begin{bmatrix}
0
\\
1
\end{bmatrix}_1
&
=
\begin{bmatrix}
0
\\
1
\end{bmatrix}_1
\begin{bmatrix}
1
\\
1
\end{bmatrix}_1
\begin{bmatrix}
2
\\
1
\end{bmatrix}_1
\begin{bmatrix}
3
\\
1
\end{bmatrix}_1
\cdots
\\
&\qquad\times
\begin{bmatrix}
1
\\
0
\end{bmatrix}_1
\begin{bmatrix}
1
\\
0
\end{bmatrix}_1
\left(
\prod_{j=0}^{\infty}
\begin{bmatrix}
2^{j+1}
\\
0
\end{bmatrix}_{2^j, -2^j}
\begin{bmatrix}
2^{j+1}
\\
0
\end{bmatrix}_{2^j, 2^j}
\right)
\\
&\qquad\times
\cdots
\begin{bmatrix}
4
\\
-1
\end{bmatrix}_1
\begin{bmatrix}
3
\\
-1
\end{bmatrix}_1
\begin{bmatrix}
2
\\
-1
\end{bmatrix}_1
\begin{bmatrix}
1
\\
-1
\end{bmatrix}_1
.
\end{split}
\end{align}
The quantum dilogarithm version of the formula is identified with the first identity in \cite[Eq.\ (1.5)]{Dimofte09}.
It was proved by \cite{Sugawara22} with a very different method in view of the quantum affine algebras.

\end{ex}

\notes

The quantum synchronicity for $x$-variables is due to \cite{Berenstein05b}.
The quantum synchronicity for $y$-variables is due to \cite{Fock07,Kashaev11}.
The QDIs for quantum dilogarithm elements are studied in \cite{Nakanishi22b}.
The QDIs in tropical form were given by \cite{Keller11} with the categorification method and by \cite{Kashaev11} in a more direct approach as presented here.
The QDIs in universal form were given by \cite{Volkov11} for the pentagon identity and by \cite{Kashaev11} in general.
The QCSDs were given by \cite{Mandel15,Davison19} based on the results by \cite{Kontsevich13, Gross14}.
The QDIs for QCSDs are regarded as special cases of  \emph{wall-crossing formulas} studied in a more broad context by  \cite{Kontsevich08}.
The QDIs are also studied in different approaches in \cite{Reineke08,Dimofte09, Gaiotto09, Sugawara22}.
The QDIs in Section \ref{sec:more1} and Example \ref{ex:QDIaffine1} are taken from  \cite{Nakanishi22b}, while Example \ref{ex:a211} is new here.

\fontsize{10pt}{10.5pt}\selectfont

\cleardoublepage
\addcontentsline{toc}{part}{References}

 \bibliographystyle{amsalpha}
\bibliography{../../biblist/biblist.bib}

\newcommand{\etalchar}[1]{$^{#1}$}
\providecommand{\bysame}{\leavevmode\hbox to3em{\hrulefill}\thinspace}
\providecommand{\MR}{\relax\ifhmode\unskip\space\fi MR }
\providecommand{\MRhref}[2]{%
  \href{http://www.ams.org/mathscinet-getitem?mr=#1}{#2}
}
\providecommand{\href}[2]{#2}
\begin{thebibliography}{CFMM24}

\bibitem[Aka23]{Akagi23}
R.~Akagi, \emph{Explicit forms in lower degrees of rank 2 cluster scattering
  diagrams}, 2023, arXiv:2309.15470 [math.CO].

\bibitem[AM78]{Abraham78}
R.~Abraham and J.~E. Marsden, \emph{Foundations of mechanics}, The
  Benjamin/Cummings Publishing Co., Inc., Reading, Massachusetts, 1978.

\bibitem[Ami09]{Amiot09}
C.~Amiot, \emph{Cluster categories for algebras of global dimension 2 and
  quivers with potentialimension 2 and quivers with potential}, Annales de
  l'Institut Fourier \textbf{104} (2009), 2525--2590; arXiv:0805.1035
  [math.RT].

\bibitem[BF12]{Bonfiglioli12}
A.~Bonfiglioli and R.~Fulci, \emph{Topics in noncommutative algebra: The
  theorem of {C}ampbell, {B}aker, {H}ausdorff and {D}ynkin}, Springer-Verlag,
  2012.

\bibitem[BFZ05]{Berenstein05}
A.~Berenstein, S.~Fomin, and A.~Zelevinsky, \emph{Cluster algebras {III}.
  {U}pper bounds and double {B}ruhat cells}, Duke Math. J. \textbf{126} (2005),
  1--52; arXiv:math/0305434 [math.RT].

\bibitem[Blo78]{Bloch78}
S.~Bloch, \emph{Applications of the dilogarithm function in algebraic
  {$K$}-theory and algebraic geometry}, Proceedings of the {I}nternational
  {S}ymposium on {A}lgebraic {G}eometry, Kinokuniya Book Store, Tokyo, 1978,
  pp.~103--114.

\bibitem[Bou89]{Bourbaki89}
N.~Bourbaki, \emph{Lie groups and {L}ie algebras: {C}hapters 1--3},
  Springer-Verlag, 1989.

\bibitem[Bou02]{Bourbaki02}
\bysame, \emph{Lie groups and {L}ie algebras: {C}hapters 4--6},
  Springer-Verlag, 2002.

\bibitem[BR89]{Bazhanov89}
V.~V. Bazhanov and N.~Yu. Reshetikhin, \emph{Critical {RSOS} models and
  conformal field theory}, Int. J. Mod. Phys. \textbf{4} (1989), 115--142.

\bibitem[BR90]{Bazhanov90}
\bysame, \emph{Restricted solid-on-solid models connected with simply laced
  algebras and conformal field theory}, J. Phys. A: Math. Gen. \textbf{23}
  (1990), 1477--1492.

\bibitem[BZ05]{Berenstein05b}
A.~Berenstein and A.~Zelevinsky, \emph{Quantum cluster algebras}, Adv. in Math.
  \textbf{195} (2005), 405--455; arXiv:math.QA/0404446 [math.QA].

\bibitem[Car06]{Carter06}
R.~W. Carter, \emph{Cluster algebras}, Textos de Matem{\'a}tica {V}ol.~37,
  Department de Matem{\'a}tica da Universidade de Coimbra, 2006.

\bibitem[CFMM24]{Cheung20}
M.-W. Cheung, J.~B. Fr{\'\i}as-Medina, and T.~Magee, \emph{Quantization of
  deformed cluster {P}oisson varieties}, Alg. Rep. Theory \textbf{27} (2024),
  381--427; arXiv:2007.02479 [math.QA.

\bibitem[CGT99]{Caracciolo99}
R.~Caracciolo, F.~Gliozzi, and R.~Tateo, \emph{A topological invariant of {RG}
  flows in 2d integrable quantum field theories,}, Int. J. Mod. Phys.
  \textbf{13} (1999), 2927--2932; arXiv:hep--th/9902094.

\bibitem[Cha05]{Chapoton05}
F.~Chapoton, \emph{Functional identities for the {R}ogers dilogarithm
  associated to cluster {Y}-systems}, Bull. London Math. Soc. \textbf{37}
  (2005), 755--760.

\bibitem[CHL20]{Cao17}
P.~Cao, M.~Huang, and F.~Li, \emph{A conjecture on {$C$}-matrices of cluster
  algebras}, Nagoya Math. J. \textbf{238} (2020), 37--46; arXiv:1702.01221
  [math.RA].

\bibitem[CP95]{Chari95a}
V.~Chari and A.~Pressley, \emph{A guide to quantum groups}, Cambridge
  University Press, 1995.

\bibitem[CPS]{Carl10}
M.~Carl, M.~Pumperla, and B.~Siebert, \emph{A tropical view of
  {L}andau-{G}inzburg models}, preprint, 2010, available at
  https://www.math.uni-hamburg.de/home/siebert/preprints/LGtrop.pdf.

\bibitem[DGS11]{Dimofte09}
T.~Dimofte, S.~Gukov, and Y.~Soibelman, \emph{Quantum wall crossing in
  {$\mathcal{N} =2$} gauge theories}, Lett. Math. Phys. \textbf{95} (2011),
  1--25; arXiv:0912.1346 [hep--th].

\bibitem[DK09]{DiFrancesco09a}
P.~{Di Francesco} and R.~Kedem, \emph{Q-systems as cluster algebras {II}:
  {C}artan matrix of finite type and the polynomial property}, Lett. Math.
  Phys. \textbf{89} (2009), 183--216; arXiv:0803.0362 [math.RT].

\bibitem[DM21]{Davison19}
B.~Davison and T.~Mandel, \emph{Strong positivity for quantum theta bases of
  quantum cluster algebras}, Invent. Math. \textbf{226} (2021), 725--843;
  arXiv:1910.12915 [math.RT].

\bibitem[Dri85]{Drinfeld85}
V.~Drinfel'd, \emph{Hopf algebras and the quantum {Y}ang-{B}axter equation},
  Sov. Math. Dokl. \textbf{32} (1985), 254--268.

\bibitem[DWZ10]{Derksen10}
H.~Derksen, J.~Weyman, and A.~Zelevinsky, \emph{Quivers with potentials and
  their representations {II}: {A}pplications to cluster algebras}, J. Amer.
  Math. Soc. \textbf{23} (2010), 749--790; arXiv:0904.0676 [math.RA].

\bibitem[Eul68]{Euler68}
L.~Euler, \emph{Institutionum {C}alculi {I}ntegralis}, vol.~1, Petropoli:
  Impensis Academiae Imperialis Scientiarum, 1768, Euler Archive - All Works.
  342. https://scholarlycommons.pacific.edu/euler-works/342.

\bibitem[FG06]{Fock03b}
V.~V. Fock and A.~B. Goncharov, \emph{Moduli spaces of local systems and higher
  {T}eichm\"uller theory}, Publ. Math. IHES \textbf{103} (2006), 1--211,
  arXiv:math/0311149 [math.AG].

\bibitem[FG07]{Fock05}
\bysame, \emph{Dual {T}eichm\"uller and lamination spaces}, Handbook of
  Teichm\"uller theory, Vol. {I}, European Mathematical Society, 2007,
  pp.~647--684; arXiv:math/0510312 [math.DG].

\bibitem[FG09a]{Fock03}
\bysame, \emph{Cluster ensembles, quantization and the dilogarithm}, Ann. Sci.
  de l'\'Ecole Norm. Sup. \textbf{42} (2009), 865--930; arXiv:math/0311245
  [math.AG].

\bibitem[FG09b]{Fock07b}
\bysame, \emph{Cluster ensembles, quantization and the dilogarithm {II}: {T}he
  intertwiner}, Prog. Math. \textbf{269} (2009), 655--673; arXiv:math.0702398
  [math.AG].

\bibitem[FG09c]{Fock07}
\bysame, \emph{The quantum dilogarithm and representations of quantum cluster
  varieties}, Invent. Math. \textbf{172} (2009), 223--286; arXiv:math/0702397
  [math.QA].

\bibitem[FG19]{Fujiwara18}
S.~Fujiwara and Y.~Gyoda, \emph{Duality between final-seed and initial-seed
  mutations in cluster algebras}, SIGMA \textbf{15} (2019), 040, 24 pages;
  arXiv:1808.02156 [math.RA].

\bibitem[FK94]{Faddeev94}
L.~D. Faddeev and R.~M. Kashaev, \emph{Quantum dilogarithm}, Mod. Phys. Lett.
  \textbf{A9} (1994), 427--434; arXiv:hep--th/9310070.

\bibitem[FS95]{Frenkel95}
E.~Frenkel and A.~Szenes, \emph{Thermodynamic {Bethe} ansatz and dilogarithm
  identities. {I}}, Math. Res. Lett. \textbf{2} (1995), 677--693;
  arXiv:hep--th/9506215.

\bibitem[FST08]{Fomin08}
S.~Fomin, M.~Shapiro, and D.~Thurston, \emph{Cluster algebras and triangulated
  surfaces. {P}art {I}: {C}luster complexes}, Acta Math. \textbf{201} (2008),
  83--146; arXiv:math/0608367 [math.RA].

\bibitem[FV93]{Faddeev93}
L.~D. Faddeev and A.~Yu. Volkov, \emph{Abelian current algebra and the
  {V}irasoro algebra on the lattice}, Phys. Lett. \textbf{315} (1993),
  311--318; arXiv:hep--th/9307048.

\bibitem[FWZ16]{Fomin16b}
S.~Fomin, L.~Williams, and A.~Zelevinsky, \emph{Introduction to cluster
  algebras}, 2016, {C}hapter 1--3, arXiv:1608.05735 [math.CO]; {C}hapter 4--5,
  arXiv:1707.07190 [math.CO]; {C}hapter 6, arXiv:2008.09189 [math.AC];
  {C}hapter 7, arXiv:2106.02160 [math.CO].

\bibitem[FZ02]{Fomin02}
S.~Fomin and A.~Zelevinsky, \emph{Cluster algebras {I}. {F}oundations}, J.
  Amer. Math. Soc. \textbf{15} (2002), 497--529 (electronic);
  arXiv:math/0104151 [math.RT].

\bibitem[FZ03a]{Fomin03a}
\bysame, \emph{Cluster algebras {II}. {F}inite type classification}, Invent.
  Math. \textbf{154} (2003), 63--121; arXiv:math/0208229 [math.RA].

\bibitem[FZ03b]{Fomin03b}
\bysame, \emph{Y-systems and generalized associahedra}, Ann. of Math.
  \textbf{158} (2003), 977--1018; arXiv:hep--th/0111053.

\bibitem[FZ07]{Fomin07}
\bysame, \emph{Cluster algebras {IV}. {C}oefficients}, Compositio Mathematica
  \textbf{143} (2007), 112--164; arXiv:math/0602259 [math.RT].

\bibitem[GHKK18]{Gross14}
M.~Gross, P.~Hacking, S.~Keel, and M.~Kontsevich, \emph{Canonical bases for
  cluster algebras}, J. Amer. Math. Soc. \textbf{31} (2018), 497--608;
  arXiv:1411.1394 [math.AG].

\bibitem[GL23]{Grafnitz23}
T.~Gr\"anitz and P.~Luo, \emph{Scattering diagrams: Polynomiality and the dense
  region}, 2023, arXiv:2312.13990 [math.AG].

\bibitem[GMN13]{Gaiotto09}
D.~Gaiotto, G.~W. Moore, and A.~Neitzke, \emph{Wall-crossing, {H}itchin
  systems, and the {WKB} approximation}, Adv. in Math. \textbf{234} (2013),
  239--403; arXiv:0907.3987 [hep--th].

\bibitem[GNR17]{Gekhtman16}
M.~Gekhtman, T.~Nakanishi, and D.~Rupel, \emph{Hamiltonian and {L}agrangian
  formalisms of mutations in cluster algebras and application to dilogarithm
  identities}, J. Integrable Syst. \textbf{2} (2017), 1--35; arXiv:1611.02813
  [math.RA].

\bibitem[GPS10]{Gross09}
M.~Gross, R.~Pandharipande, and B.~Siebert, \emph{The tropical vertex}, Duke
  Math. J. \textbf{153} (2010), 297 -- 362; arXiv:0902.0779 [math.AG].

\bibitem[GR17]{Glick17}
M.~Glick and D.~Rupel, \emph{Introduction to cluster algebras}, Symmetries and
  Integrability of Difference Equations (Levi D., Rebelo R., and P.~Winternitz,
  eds.), CRM Series in Mathematical Physics, Springer, 2017, pp.~325--358;
  arXiv:1803.08960 [math.CO].

\bibitem[Gro11]{Gross11}
M.~Gross, \emph{Tropical geometry and mirror symmetry}, CBMS Regional Conf.
  Ser. in Math., no. 114, Amer. Math. Soc., 2011.

\bibitem[GS11]{Gross07}
M.~Gross and B.~Siebert, \emph{From affine geometry to complex geometry},
  Annals of Math. \textbf{174} (2011), 95--138; arXiv:math/0703822.

\bibitem[GSV03]{Gekhtman02}
M.~Gekhtman, M.~Shapiro, and A.~Vainshtein, \emph{Cluster algebras and
  {P}oisson geometry}, Mosc. Math. J. \textbf{3} (2003), 899--934;
  arXiv:math/0208033 [math.QA].

\bibitem[GSV10]{Gekhtman10}
\bysame, \emph{Cluster algebras and {P}oisson geometry}, Mathematical Surveys
  and Monographs, no. 167, American Mathematical Society, 2010.

\bibitem[GT95]{Gliozzi95}
F.~Gliozzi and R.~Tateo, \emph{{ADE} functional dilogarithm identities and
  integrable models}, Phys. Lett. \textbf{B348} (1995), 677--693;
  arXiv:hep--th/9411203.

\bibitem[GT96]{Gliozzi96}
\bysame, \emph{Thermodynamic {B}ethe ansatz and three-fold triangulations},
  Int. J. Mod. Phys. \textbf{A11} (1996), 4051--4064; arXiv:hep--th/9505102.

\bibitem[Hum90]{Humphreys90}
J.~E. Humphreys, \emph{Reflection groups and {C}oxeter groups}, Cambridge
  University Press, 1990.

\bibitem[IIK{\etalchar{+}}10]{Inoue10c}
R.~Inoue, O.~Iyama, A.~Kuniba, T.~Nakanishi, and J.~Suzuki, \emph{Periodicities
  of {T} and {Y}-systems}, Nagoya Math. J. \textbf{197} (2010), 59--174;
  arXiv:0812.0667 [math.QA].

\bibitem[IIK{\etalchar{+}}13a]{Inoue10a}
R.~Inoue, O.~Iyama, B.~Keller, A.~Kuniba, and T.~Nakanishi, \emph{Periodicities
  of {T} and {Y}-systems, dilogarithm identities, and cluster algebras {I}:
  {T}ype {$B_r$}}, Publ. RIMS \textbf{49} (2013), 1--42; arXiv:1001.1880
  [math.QA].

\bibitem[IIK{\etalchar{+}}13b]{Inoue10b}
\bysame, \emph{Periodicities of {T} and {Y}-systems, dilogarithm identities,
  and cluster algebras {II}: {T}ypes {$C_r$}, {$F_4$}, and {$G_2$}}, Publ. RIMS
  \textbf{49} (2013), 43--85; arXiv:1001.1881 [math.QA].

\bibitem[Jac79]{Jacobson79}
N.~Jacobson, \emph{Lie algebras}, Dover Publications, New York, 1979.

\bibitem[Jim85]{Jimbo85}
M.~Jimbo, \emph{A $q$-difference analogue of $u(\hat{\mathfrak{g}})$ and the
  {Y}ang--{B}axter equation}, Lett. Math. Phys. \textbf{10} (1985), 63--69.

\bibitem[JMO88]{Jimbo88}
M.~Jimbo, T.~Miwa, and M.~Okado, \emph{Local state probabilities of solvable
  lattice models: {A}n ${A^{(1)}_{n-1}}$ family}, Nucl. Phys. \textbf{300}
  (1988), 74--108.

\bibitem[Kac90]{Kac90}
V.~G. Kac, \emph{Infinite dimensional {L}ie algebras}, 3rd ed., Cambridge
  University Press, 1990.

\bibitem[Kel10]{Keller08}
B.~Keller, \emph{Cluster algebras, quiver representations and triangulated
  categories}, Triangulated categories (T.~Holm, P.~J{\o}rgensen, and
  R.~Rouquier, eds.), Lecture Note Series, vol. 375, London Mathematical
  Society, Cambridge University Press, 2010, pp.~76--160; arXiv:0807.1960
  [math.RT].

\bibitem[Kel11]{Keller11}
\bysame, \emph{On cluster theory and quantum dilogarithm identities},
  Representations of algebras and related topics (A.~Skowro\'nski and
  K.~Yamagata, eds.), EMS Series of Congress Reports, European Mathematical
  Society, 2011, pp.~85--116; arXiv:1102.4148 [math.RT].

\bibitem[Kel13a]{Keller12}
\bysame, \emph{Cluster algebras and derived categories}, Derived Categories in
  Algebraic Geometry (Y.~Kawamata, ed.), EMS Ser. Congr. Rep., Eur. Math. Soc.,
  2013, pp.~123--183; arXiv:1202.4161 [math.RT].

\bibitem[Kel13b]{Keller10}
\bysame, \emph{The periodicity conjecture for pairs of {D}ynkin diagrams}, Ann.
  of Math. \textbf{177} (2013), 111--170; arXiv:1001.1531 [math.RT].

\bibitem[Kir89]{Kirillov89}
A.~N. Kirillov, \emph{Identities for the {R}ogers dilogarithm function
  connected with simple {L}ie algebras}, J. Sov. Math. \textbf{47} (1989),
  2450--2458.

\bibitem[Kir95]{Kirillov95}
\bysame, \emph{Dilogarithm identities}, Prog. Theor. Phys. Suppl. \textbf{118}
  (1995), 61--142; arXiv:hep--th/9408113.

\bibitem[KM90]{Klassen90}
T.~Klassen and E.~Melzer, \emph{Purely elastic scattering theories and their
  ultraviolet limits}, Nucl. Phys. \textbf{B338} (1990), 485--528.

\bibitem[KN92]{Kuniba92}
A.~Kuniba and T.~Nakanishi, \emph{Spectra in conformal field theories from the
  {R}ogers dilogarithm}, Mod. Phys. Lett. \textbf{A7} (1992), 3487--3494;
  arXiv:hep--th/9206034.

\bibitem[KN11]{Kashaev11}
R.~M. Kashaev and T.~Nakanishi, \emph{Classical and quantum dilogarithm
  identities}, SIGMA \textbf{7} (2011), 102, 29 pages; arXiv:1104.4630
  [math.QA].

\bibitem[KNS94]{Kuniba94a}
A.~Kuniba, T.~Nakanishi, and J.~Suzuki, \emph{Functional relations in solvable
  lattice models: {I}. {F}unctional relations and representation theory}, Int.
  J. Mod. Phys. \textbf{A9} (1994), 5215--5266; arXiv:hep--th/9309137.

\bibitem[KNS10]{Kuniba10}
\bysame, \emph{T-systems and {Y}-systems in integrable systems}, J. Phys. A:
  Math. Theor. \textbf{44} (2010), 103001 (146pp); arXiv:1010.1344 [hep--th].

\bibitem[KQ14]{Kimura12}
Y.~Kimura and F.~Qin, \emph{Graded quiver varieties, quantum cluster algebras
  and dual canonical basis}, Adv. in Math. \textbf{262} (2014), 261--312;
  arXiv:1205.2066.

\bibitem[KR86]{Kirillov86}
A.~N. Kirillov and N.~Y. Reshetikhin, \emph{Exact solution of the {H}eisenberg
  {XXZ} model of spin $s$}, J. Sov. Math. \textbf{35} (1986), 2627--2643.

\bibitem[KR90]{Kirillov90}
\bysame, \emph{Representations of {Y}angians and multiplicities of the
  inclusion of the irreducible components of the tensor product of
  representations of simple {L}ie algebras}, J. Sov. Math. \textbf{52} (1990),
  3156--3164.

\bibitem[KRS81]{Kulish81}
P.~P. Kulish, N.~Yu. Reshetikhin, and E.~K. Sklyanin, \emph{{Y}ang-{B}axter
  equation and representation theory: {I}}, Lett. Math. Phys. \textbf{5}
  (1981), 393--403.

\bibitem[KS06]{Kontsevich06}
M.~Kontsevich and Y.~Soibelman, \emph{Affine structures and non-{A}rchimedean
  analytic spaces}, Prog. Math. \textbf{244} (2006), 321--385;
  arXiv:math/0406564.

\bibitem[KS08]{Kontsevich08}
\bysame, \emph{Stability structures, motivic {D}onaldson-{T}homas invariants
  and cluster transformations}, 2008, arXiv:0811.2435 [math.AG].

\bibitem[KS14]{Kontsevich13}
\bysame, \emph{Wall-crossing structures in {D}onaldson-{T}homas invariants,
  integrable systems and mirror symmetry}, Homological mirror symmetry and
  tropical geometry, Lect. Notes Unione Ital., vol.~15, Springer-Verlag, 2014,
  pp.~197--308; arXiv:1303.3253 [math.AG].

\bibitem[Kun93]{Kuniba93a}
A.~Kuniba, \emph{Thermodynamics of the {$U_q(X^{(1)}_r)$} {B}ethe ansatz system
  with $q$ a root of unity}, Nucl. Phys. \textbf{B389} (1993), 209--244.

\bibitem[KY20]{Kim16}
H.K. Kim and M.~Yamazaki, \emph{Comments on exchange graphs in cluster
  algebras}, Experimental Math. \textbf{29} (2020), 79--100; arXiv:1612.00145
  [math.QA].

\bibitem[Lew81]{Lewin81}
L.~Lewin, \emph{Polylogarithms and associated functions}, North-Holland, 1981.

\bibitem[LLRZ14]{Lee14}
K.~L. Lee, L.~Li, D.~Rupel, and A.~Zelevinsky, \emph{Greedy bases in rank 2
  quantum cluster algebras}, Proc. Natl. Acad. Sci. USA \textbf{111} (2014),
  9712--9716; arXiv:1405.2311 [math.QA].

\bibitem[LS15]{Lee15}
K.~Lee and R.~Schiffler, \emph{Positivity for cluster algebras}, Ann. of Math.
  \textbf{182} (2015), 72--125; arXiv:1306.2415 [math.CO].

\bibitem[Man21]{Mandel15}
T.~Mandel, \emph{Scattering diagrams, theta functions, and refined tropical
  curves}, J. London Math. Soc. \textbf{104} (2021), 2299--2334;
  arXiv:1503.06183 [math.QA].

\bibitem[Mar13]{Marsh13}
R.~J. Marsh, \emph{Lecture notes on cluster algebras}, Zurich Lectures in
  Advanced Mathematics, European Mathematical Society, Z{\"u}rich, 2013.

\bibitem[Mat21]{Matsushita21}
K.~Matsushita, \emph{Consistency relations of rank 2 cluster scattering
  diagrams of affine type and the pentagon relation}, 2021, arXiv:2112.04743
  [math.QA].

\bibitem[MSW11]{Musiker09}
G.~Musiker, R.~Schiffler, and L.~Williams, \emph{Positivity for cluster
  algebras from surfaces}, Adv. in Math. \textbf{227} (2011), 2241--2308;
  arXiv:0906.0748 [math.CO].

\bibitem[Nag13]{Nagao10}
K.~Nagao, \emph{Donaldson-{T}homas theory and cluster algebras}, Duke Math. J.
  \textbf{162} (2013), 1313--1367; arXiv:1002.4884 [math.AG].

\bibitem[Nak11a]{Nakanishi09}
T.~Nakanishi, \emph{Dilogarithm identities for conformal field theories and
  cluster algebras: Simply laced case}, Nagoya Math. J. \textbf{202} (2011),
  23--43; arXiv:math.0909.5480 [math.QA].

\bibitem[Nak11b]{Nakanishi10c}
\bysame, \emph{Periodicities in cluster algebras and dilogarithm identities},
  Representations of algebras and related topics (A.~Skowro\'nski and
  K.~Yamagata, eds.), EMS Series of Congress Reports, European Mathematical
  Society, 2011, pp.~407--444; arXiv:1006.0632 [math.QA].

\bibitem[Nak12]{Nakanishi11c}
\bysame, \emph{Tropicalization method in cluster algebras}, Contemp. Math.
  \textbf{580} (2012), 95--115; arXiv:1110.5472 [math.QA].

\bibitem[Nak18]{Nakanishi16}
\bysame, \emph{Rogers dilogarithms of higher degree and generalized cluster
  algebras}, Jour. Math. Soc. Japan \textbf{70} (2018), 1269--1304;
  arXiv:1605.04777.

\bibitem[Nak21]{Nakanishi19}
\bysame, \emph{Synchronicity phenomenon in cluster patterns}, J. London Math.
  Soc. \textbf{103} (2021), 1120--1152; arXiv:1906.12036 [math.RA].

\bibitem[Nak22]{Nakanishi22b}
\bysame, \emph{Pentagon relation in quantum cluster scattering diagrams}, 2022,
  arXiv:2202.01588 [math.QA].

\bibitem[Nak23]{Nakanishi22a}
\bysame, \emph{Cluster algebras and scattering diagrams}, MSJ Mem., vol.~41,
  Mathematical Society of Japan, 2023, Part I. arXiv:2201.11371 [math.CO], Part
  II. arXiv:2103.16309 [math.CO], Part III. arXiv:2111.00800 [math.CO].

\bibitem[Nak24a]{Nakanishi21d}
\bysame, \emph{Dilogarithm identities in cluster scattering diagrams}, Nagoya
  Math. J. \textbf{253} (2024), 1--22; arXiv:2111.09555 [math.CO].

\bibitem[Nak24b]{Nakanishi24}
\bysame, \emph{Local and global patterns of rank 3 {$G$}-fans of
  totally-infinite type}, 2024, arXiv:2411.16283 [math.CO].

\bibitem[NK09]{Nahm09}
W.~Nahm and S.~Keegan, \emph{Integrable deformations of {CFT}s and the discrete
  {Hirota} equations}, 2009, arXiv:0905.3776 [hep-th].

\bibitem[NS16]{Nakanishi12c}
T.~Nakanishi and S.~Stella, \emph{Wonder of sine-{G}ordon {$Y$}-systems},
  Trans. Amer. Math. Soc. \textbf{368} (2016), 6835--6886; arXiv:1212.6853.

\bibitem[NT10]{Nakanishi10b}
T.~Nakanishi and R.~Tateo, \emph{Dilogarithm identities for sine-{G}ordon and
  reduced sine-{G}ordon {Y}-systems}, SIGMA \textbf{6} (2010), 085, 34 pages;
  arXiv:1005.4199 [math.QA].

\bibitem[NZ12]{Nakanishi11a}
T.~Nakanishi and A.~Zelevinsky, \emph{On tropical dualities in cluster
  algebras}, Contemp. Math. \textbf{565} (2012), 217--226, arXiv:1101.3736
  [math.RA].

\bibitem[Pir21]{Pirio21}
L.~Pirio, \emph{On webs, polylogarithms and cluster algebras}, 2021,
  arXiv:2105.01543 [math.DG].

\bibitem[Pla11]{Plamondon10b}
P.-G. Plamondon, \emph{Cluster algebras via cluster categories with
  infinite-dimensional morphism spaces}, Compos. Math. \textbf{147} (2011),
  1921--1954; arXiv:1004.0830 [math.RT].

\bibitem[Pla18]{Plamondon16}
\bysame, \emph{Cluster characters}, Homological methods, representation theory,
  and cluster algebras, CRM Short Courses, Springer-Verlag, 2018, pp.~101--125;
  arXiv:1610.07546 [math.RT].

\bibitem[Rea14]{Reading12}
N.~Reading, \emph{Universal geometric cluster algebras}, Mathematische
  Zeitschrift \textbf{277} (2014), 499--547; arXiv:1209.3987 [math.RA].

\bibitem[Rea20]{Reading17}
\bysame, \emph{Scattering fans}, Int. Math. Res. Notices \textbf{23} (2020),
  9640--9673; arXiv:1712.06968 [math.CO].

\bibitem[Rei09]{Reineke08}
M.~Reineke, \emph{Poisson automorphisms and quiver moduli}, J. Inst. Math.
  Jussieu \textbf{9} (2009), 653--667; arXiv:0804.3214 [math.RT].

\bibitem[Rog07]{Rogers07}
L.~J. Rogers, \emph{On function sum theorems connected with the series
  $\sum_{1}^{\infty} x^n/n^2$}, Proc. London Math. Soc. \textbf{4} (1907),
  169--89.

\bibitem[RTV93]{Ravanini93}
R.~Ravanini, R.~Tateo, and A.~Valleriani, \emph{Dynkin {TBA}'s}, Int. J. Mod.
  Phys. \textbf{A8} (1993), 1707--1727; arXiv:hep--th/9207040.

\bibitem[{Sag}]{Sage94}
{Sage~Mathematics~Software~(Version 9.4)}, The Sage Development Team, 2021,
  https://www.sagemath.org.

\bibitem[Sch53]{Schutzenberger53}
M.~P. Sch\"utzenberger, \emph{Une interpr\'etation de certaines solutions de
  l'\'equation fonctionnelle ${F}(x+y)={F}(x){F}(y)$}, C. R. Acad. Sci. Paris
  \textbf{236} (1953), 352--353.

\bibitem[SS09]{Speyer04}
D.~Speyer and B.~Sturmfels, \emph{Tropical mathematics}, Mathematics Magazine
  \textbf{82} (2009), 163--173; arXiv:math/0408099 [math.CO].

\bibitem[Sug23]{Sugawara22}
M.~Sugarawa, \emph{Quantum dilogarithm identities arising from the product
  formula for universal {$R$}-matrix of quantum affine algebras}, Publ. RIMS
  \textbf{59} (2023), 769--819; arXiv:2210.17109 [math.QA].

\bibitem[Sze09]{Szenes09}
A.~Szenes, \emph{Periodicity of {Y}-systems and flat connections}, Lett. Math.
  Phys. \textbf{89} (2009), 217--230; arXiv:math/0606377 [math.RT].

\bibitem[Tak08]{Takhtajan08}
L.~A. Takhtajan, \emph{Quantum mechanics for mathematicians}, American
  Mathematical Society, Providence, 2008.

\bibitem[Tat95]{Tateo95}
R.~Tateo, \emph{New functional dilogarithm identities and sine-{G}ordon
  {Y}-systems}, Phys. Lett. \textbf{B355} (1995), 157--164;
  arXiv:hep--th/9505022.

\bibitem[TF79]{Takhtajan79}
L~A. Takhtajan and L.~D. Faddeev, \emph{The quantum method of the inverse
  problem and the {H}eisenber {XYZ} model}, Russian Math. Surveys \textbf{34}
  (1979), 11--68.

\bibitem[Tra11]{Tran09}
T.~Tran, \emph{F-polynomials in quantum cluster algebras}, Algebras and
  Representation Theory \textbf{14} (2011), 1025--1061; arXiv:0904.3291
  [math.RA].

\bibitem[Vol07]{Volkov07}
A.~Yu. Volkov, \emph{On the periodicity conjecture for {Y}-systems}, Commun.
  Math. Phys. \textbf{276} (2007), 509--517; arXiv:hep--th/0606094.

\bibitem[Vol11]{Volkov11}
\bysame, \emph{Pentagon identity revisited}, Int. Math. Res. Not. \textbf{2012}
  (2011), 4619--4624; arXiv:1104.2267 [math.QA].

\bibitem[Wei83]{Weinstein83}
A.~Weinstein, \emph{The local structure of {P}oisson manifolds}, J.
  Differential Geom. \textbf{18} (1983), 523--557.

\bibitem[Wei98]{Weinstein98}
\bysame, \emph{Poisson geometry}, Differential Geom. Appl. \textbf{9} (1998),
  213--238.

\bibitem[Wil14]{Williams12}
L.~Williams, \emph{Cluster algebras: an introduction}, Bull. Amer. Math. Soc.
  \textbf{51} (2014), 1--26; arXiv:1212.6263 [math.RA].

\bibitem[Woj96]{Wojtkowiak96}
J.~Wojtkowiak, \emph{Functional equations of iterated integrals with regular
  singularities}, Nagoya Math. J. \textbf{142} (1996), 145--159.

\bibitem[Zag07]{Zagier07}
D.~Zagier, \emph{The dilogarithm function}, Frontiers in number theory,
  physics, and geometry {II}, Springer-Verlag, 2007, pp.~3--65.

\bibitem[Zam91a]{Zamolodchikov91}
{Al}.~B. Zamolodchikov, \emph{On the thermodynamic {B}ethe ansatz equations for
  reflectionless {ADE} scattering theories}, Phys. Lett. \textbf{B253} (1991),
  391--394.

\bibitem[Zam91b]{Zamolodchikov91b}
\bysame, \emph{Thermodynamic {B}ethe ansatz for {RSOS} scattering theories},
  Nucl. Phys. \textbf{B358} (1991), 497--523.

\end{thebibliography}

\cleardoublepage
\addcontentsline{toc}{part}{Index}
\printindex

\end{document}